\documentclass[openany, 11pt, a4paper]{book}
\usepackage{tabularx, longtable}

\usepackage{inputenc} 
\usepackage[english]{babel}
\usepackage{dsfont, csquotes}

\usepackage[totalheight=24 true cm, totalwidth=17 true cm]{geometry}
\usepackage{enumitem}
\usepackage[toc,page]{appendix}

\setlist[itemize]{noitemsep}
\setlist[enumerate]{noitemsep,label=\textup{(\roman*)}}

\usepackage{bbding}
\usepackage{amsmath,multirow}
\usepackage{amsthm}
\usepackage{amssymb}
\usepackage{subcaption}
\usepackage{mathrsfs}
\usepackage{mathtools}
\usepackage{ stmaryrd }
\usepackage{url}
\usepackage{wrapfig}
\usepackage{starfont}
\usepackage{pifont}
\usepackage{eurosym}

\usepackage{imakeidx}
\makeindex[]

\usepackage{multicol}
\usepackage{relsize}
\usepackage{float}
\usepackage{tikz,pgf}
\usepackage{tikz-cd}
\usepackage{wasysym}
\usepackage{graphicx}
\usepackage{fancyhdr}
\usepackage{verbatim}

\usepackage{pgfplots}

\pgfplotsset{compat=1.15}
\usepackage{mathrsfs}
\usetikzlibrary{arrows}

\newcommand{\liftmat}{\mathcal{M}_q^\gamma(M)}
\newcommand{\dimm}[2]{\dim_#1^#2}
\newcounter{cases}
\newcounter{subcases}
\newcounter{subsubcases}
\newenvironment{mycases}
  {%
    \setcounter{cases}{0}%
    \def\case
      {%
        \par\noindent
        \refstepcounter{cases}%
        \textbf{Case \thecases.}
      }%
  }
  {%
    \par
  }

\newcommand\restr[2]{{
  \left.\kern-\nulldelimiterspace 
  #1 
  \vphantom{|} 
  \right|_{#2} 
  }}
\newcommand{\CC}{\mathbb{C}}

\newcommand{\VMLIFT}{V_{M}^{\text{lift}}}
\newcommand{\rk}{\text{rk}}

\newcommand{\CCC}{\mathcal{C}}

\newcommand{\Lift}{\textup{Lift}}
\newcommand{\VCM}{V_{\mathcal{C}(M)}}
\newcommand{\VCN}{V_{\mathcal{C}(N)}}

\usepackage[final]{pdfpages}
\usepackage{eurosym}
\newtheorem{theorem}{Theorem}[section]

\newtheorem{corollary}[theorem]{Corollary}
\newtheorem{lemma}[theorem]{Lemma}

\theoremstyle{definition}
\newtheorem{definition}[theorem]{Definition}
\newtheorem{proposition}[theorem]{Proposition}
\newtheorem{example}[theorem]{Example}
\newtheorem{construction}[theorem]{Construction}
\newtheorem{notation}[theorem]{Notation}
\newtheorem{remark}[theorem]{Remark}
\newcommand{\gluing}[4]{#1 \amalg_{#3,#4} #2}
\usepackage[all,dvips,knot,web,arc,curve,color,frame]{xy}
\usepackage[all,knot,arc,color,web]{xy}
\xyoption{arc}
\xyoption{web}
\xyoption{curve}

\newtheorem{theoremA}{Theorem}

\newtheorem{theoremB}{Theorem}

\newtheorem{strategy}{Strategy}
\numberwithin{equation}{section}

\theoremstyle{plain}

\newtheorem{question}[theorem]{Question}
\usepackage{hyperref}
\hypersetup{bookmarksnumbered=true,
     colorlinks = true,
     linkcolor = blue,
     anchorcolor = blue,
     citecolor = teal,
     filecolor = blue,
     urlcolor = blue
     }
\everymath{\displaystyle}
\allowdisplaybreaks

\makeatletter
\def\mathcenterto#1#2{\mathclap{\phantom{#1}\mathclap{#2}}\phantom{#1}}
\let\old@widetilde\widetilde
\def\widetildeto#1#2{\mathcenterto{#2}{\old@widetilde{\mathcenterto{#1}{#2\,}}}}
\let\old@widehat\widehat
\def\widehatto#1#2{\mathcenterto{#2}{\old@widehat{\mathcenterto{#1}{#2\,}}}}
\makeatother

\newcommand{\ip}[2]{\langle  #1,#2 \rangle} 

\newcommand{\size}[1]{\left| #1 \right|} 
\newcommand{\pare}[1]{\left( #1 \right)} 
\newcommand{\set}[1]{{\left\{ #1 \right\}}} 

\newcommand{\corch}[1]{\left[ #1 \right]} 

\newcommand*\closure[1]{\overline{#1}}

\DeclareMathOperator{\rank}{rk}

\def\dim{\operatorname{dim}}

\DeclareMathOperator{\sspan}{span}

\DeclareMathOperator{\RR}{\mathbb{R}}

\hypersetup{
    colorlinks=true,
    linkcolor=blue,
    filecolor=magenta,      
    urlcolor=cyan,
    pdftitle={Overleaf Example},
    pdfpagemode=FullScreen,
    }
    
\setlist[itemize]{noitemsep}

\normalfont

\usepackage{bbm}
\usepackage{dsfont}

\usepackage{url}
\usepackage{amsmath,blkarray,booktabs, bigstrut}
\usepackage{tikz}
\usepackage{multirow}
\usepackage{xcolor}
\usepackage{color}
\title{
Algebraic Geometry of Cactus, Pascal, and Pappus Matroids}
\author{Emiliano Liwski, Fatemeh Mohammadi, and Lisa Vandebrouck}
\date{}
\newenvironment{equation1}
  {\begin{equation}}
  {\end{equation}}

\newenvironment{equation2}
  {\begin{equation}}
  {\end{equation}}

\numberwithin{equation}{section}

\begin{document}
\pagenumbering{roman}
\includepdf[pages = {-}]{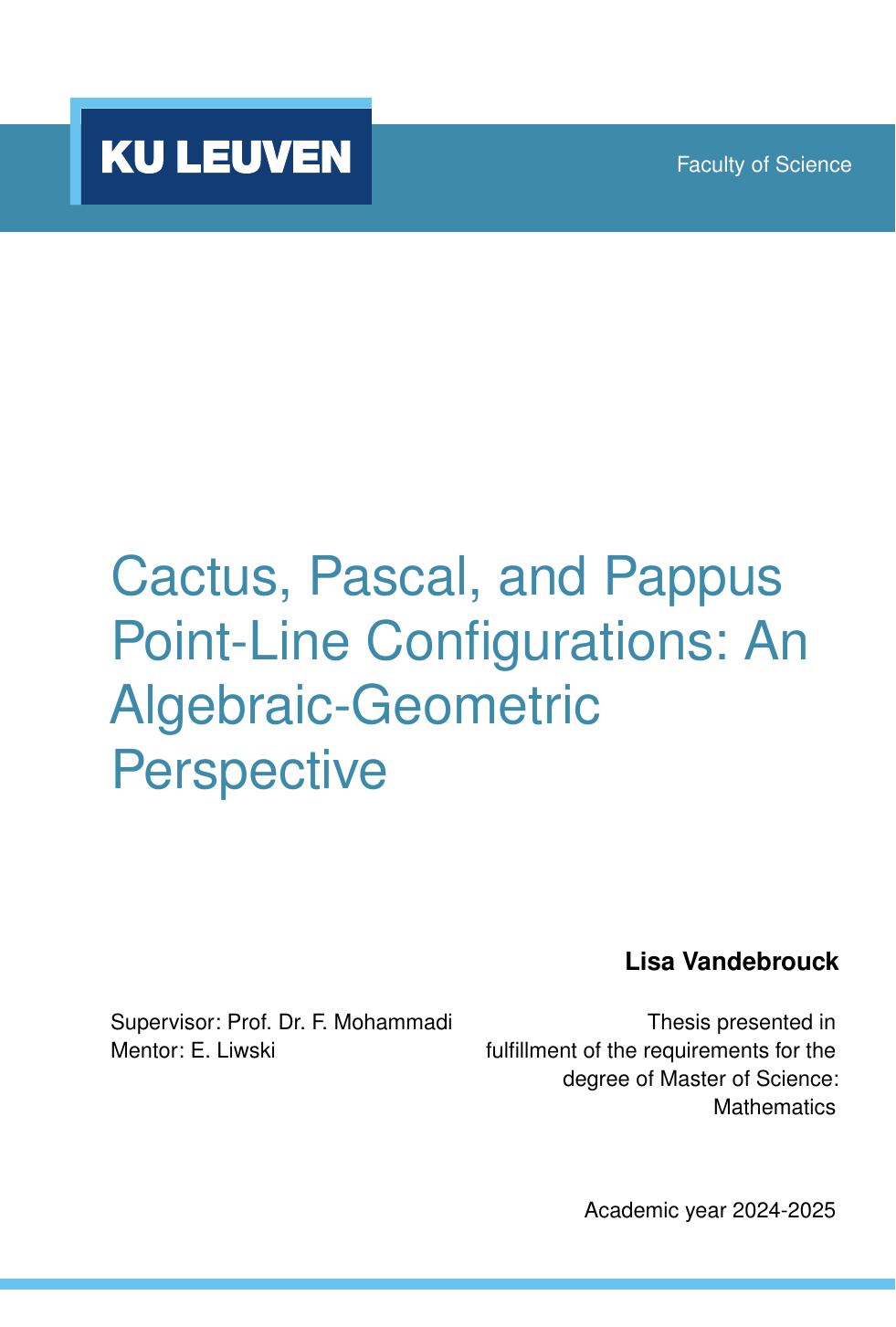}
\chapter*{Preface}\label{ch:preface}

The findings and results presented here are written in the paper \cite{vandebrouckliwskimohammadi}. Fatemeh Mohammadi and Emiliano Liwski introduced me to the topic and offered valuable support with the content. Moreover, they gave suggestions to improve the language and presentation of some of my results to make them more suitable for the paper, especially in the sections \ref{sec: related work}, \ref{sec: cactus} and \ref{sec: pascal & pappus}.
The ideas in Remark \ref{remark: emiliano} are based on a proof by Emiliano Liwski, and Lemma \ref{lemma e} is proven by Emiliano Liwski as well.

\medskip

While working on this thesis, I had the privilege to collaborate with some wonderful people. First and foremost, I would like to sincerely thank my supervisor Professor Mohammadi, whose invaluable guidance made this thesis possible. I greatly appreciate the introduction to the intriguing theory of matroids and your constructive feedback. I am especially grateful for the trust and opportunity to turn this thesis into a paper.

\medskip

I would also like to thank my daily supervisor, Emiliano, for always being available to answer my questions. Due to your many insights and support, I never lost my enthusiasm for this thesis. You were an incredible supervisor; I feel truly fortunate to have worked with someone as dedicated as you.

\medskip

Furthermore, I also want to express my gratitude towards Prof. Dr. Jannik Matuschke and Dr. Rémi Prébet for reading and evaluating this thesis. Your valuable questions further advanced the development of this thesis. Moreover, without the implementation of the algorithm to identify minimal matroids by courtesy of Rémi Prébet, many results in this thesis would not have been possible.

\medskip

Finally, I am deeply grateful for my family and friends.
\cleardoublepage

\chapter*{Abstract}                                 \label{ch:abstract}
In this thesis, we study point-line configurations and their associated matroid and circuit varieties. We aim to find a finite set of defining equations for matroid varieties and an irreducible decomposition for circuit varieties.

To solve the former problem, we use some classical techniques from algebraic geometry, including the Grassmann-Cayley algebra and the liftability technique. From these techniques, we can respectively derive the Grassmann-Cayley ideal, introduced by Sidman, Traves and Wheeler, and the lifting ideal, introduced by Liwski, Mohammadi, Clarke and Masiero. Since the circuit ideal, the Grassmann-Cayley ideal and the lifting ideal are contained in the matroid ideal and explicit generators are known for them, it is a natural question to identify point-line configurations for which a generating set of the matroid ideal is formed by the circuit polynomials, Grassmann-Cayley polynomials and lifting polynomials. For these point-line configurations, we obtain an explicit and finite description of the matroid variety. In this thesis, we prove that the matroid ideal of cactus configurations, the Pascal configuration and the Pappus configuration can be generated by these three types of polynomials. 
\medskip

To find an irreducible decomposition for the circuit varieties of point-line configurations, we use the decomposition strategy developed by Clarke, Grace, Mohammadi and Motwani. If the point-line configuration has some points lying on at most two lines, we develop a shorter alternative as well. For some point-line configurations, we give two different proofs for the decomposition of their circuit varieties, one using the decomposition strategy, one using our new method, so that we can compare them.
We find such a decomposition for
cactus configurations, up to irredundancy. Moreover, we find an irreducible decomposition for the Pascal configuration and the third configuration $9_3$, which is a point-line configuration with nine points and nine lines, such that every point is on three lines and every line contains three points. 
\cleardoublepage
\chapter*{List of Symbols}              \label{List of Symbols} 
Unless otherwise specified, $M$ is a matroid on the ground set $[d]$.
\newline
\begin{longtable}{@{}ll@{}}
$ \mathbb{C}$ & The complex numbers. \\   \addlinespace
$\mathbb{P}^2$ & The complex projective plane. \\   \addlinespace
$\mathcal{P}(A)$ & The power set of a set $A$. \\   \addlinespace
$\sqrt{I}$ & The radical of an ideal $I$. \\
  \addlinespace
  $I:J^\infty$ & The saturation of an ideal $I$ with respect to another ideal $J$. \\
  \addlinespace
  $\top$ & The transpose of a vector or matrix. \\
  \addlinespace
  $\overline{V}$ & The Zariski closure of a set $V$. \\
  \addlinespace
  $\wedge$ & The meet in the Grassmann–Cayley Algebra. \\
  \addlinespace
  $\vee$ & The join in the Grassmann–Cayley Algebra. \\
  \addlinespace
  $\mathcal{D}(M)$ & The dependent sets of $M$. \\
  \addlinespace
  $\mathcal{I}(M)$ & The independent sets of $M$. \\ \addlinespace
  $\mathcal{C}(M)$ & The circuits of $M$. \\ \addlinespace
   $\mathcal{L}_{M}$ & The lines of a point-line configuration $M$. \\ \addlinespace
   $\mathcal{L}_{p}$ & The lines of a point-line configuration containing a point $p$. \\ \addlinespace
  $I_M$ & The matroid ideal of $M$. \\
  \addlinespace
  $G_M$ & The Grassmann–Cayley ideal of $M$. \\
  \addlinespace
  $I_M^{\textup{lift}}$ & The lifting ideal of $M$. \\
  \addlinespace
  $I_{\mathcal{C}(M)}$ & The circuit ideal of $M$. \\
  \addlinespace
  $\Gamma_M$ & The realization space of $M$. \\
  \addlinespace
  $V_M$ & The matroid variety of $M$. \\
  \addlinespace
  $V(G_M)$ & The Grassmann–Cayley variety of $M$. \\
  \addlinespace
  $V(I_M^{\mathrm{lift}})$ & The lifting variety of $M$. \\
  \addlinespace
  $V_{C(M)}$ & The circuit variety of $M$. \\
  \addlinespace
  $\mathcal{M}_q(N)$ & The liftability matrix of $M$ from $q$. \\
  \addlinespace
  $S_M$ & The points of $[d]$ lying on at least two lines of $M$. \\
  \addlinespace
  $Q_M$ & The points of $[d]$ lying on at least three lines of $M$. \\
  \addlinespace
  $M(i)$ & The matroid derived from $M$ by setting the point $i \in [d]$ to be a loop.
  \\
  \addlinespace
  $\pi_M^i$ & The point-line configuration derived from the point-line configuration \\ & $M$ on $[d]$ by placing all the points except $i$ on a line, and identifying \\ & points which are on a  common line of $M$ with $i$.
  \\ \addlinespace
  $N \amalg_{p,q} M$ & The free gluing of matroids $N$ and $M$ which identifies the points $p$ and $q$. \\
  \addlinespace
  $\Lift_{M,q}(\gamma)$ & For a matroid $M$ on $[d]$ and a collection of vectors $\{\gamma_1, \ldots, \gamma_d\} \in \VCM$, \\ & this is the subspace with all vectors $(z_{1},\ldots,z_{d}) \in \CC^d$ such that \\ & $\widetilde{\gamma}=\{\gamma_{i}+z_{i}q: i \in [d]\}$ belongs to $V_{\CCC(M)}$.
\end{longtable}
\cleardoublepage
\chapter*{Summary}
Throughout this summary, we will assume that $M$ is a point-line configuration on the ground set $[d]$, as defined in Definition~\ref{point-line}. We first recall some important definitions, combining Definition~\ref{realization space} and~\ref{cir}.
\begin{definition} For a point-line configuration $M$ on $[d]$, we define:
\begin{itemize}
\item A realization of $M$ is a set of vectors $\gamma=\{\gamma_{1},\ldots,\gamma_{d}\}\subset \CC^{3}$ such that
\[\{i_{1},\ldots,i_{p}\} \ \text{is a dependent set of $M$} \Longleftrightarrow \{\gamma_{i_{1}},\ldots,\gamma_{i_{p}}\}\ \text{is linearly dependent.}\]
The set
$\Gamma_{M}=\{\gamma\subset \CC^{3}: \gamma \ \text{is a realization of $M$}\}$ is called the {\em realization space} of $M$.
The {\em matroid variety} $V_M$ of $M$ is defined as $\closure{\Gamma_M}$, the Zariski closure of $\Gamma_M$. The matroid ideal $I_{M}=I(V_{M})$ is defined as the associated ideal.
\smallskip 
\item We say that a collection of vectors $\gamma=\{\gamma_{1},\ldots,\gamma_{d}\}\subset \CC^{3}$ includes the dependencies of $M$ if it satisfies:
\[\set{i_{1},\ldots,i_{p}}\  \text{is a dependent set of $M$} \Longrightarrow \set{\gamma_{i_{1}},\ldots,\gamma_{i_{p}}}\ \text{is linearly dependent}. \] 
The {\em circuit variety} of $M$ is defined as \[V_{\mathcal{C}(M)}=\{\gamma:\text{$\gamma$ includes the dependencies of $M$}\}.\]
The ideal of $V_{\mathcal{C}(M)}$ is denoted by $I_{\CCC(M)}$.
\end{itemize}
\end{definition}
One of the main problems that this thesis addresses concerns the circuit variety:
\begin{question} \label{question decompo}
    Find an irreducible decomposition for $\VCM$.
\end{question} 
This problem is already studied, for instance in \cite{liwskimohammadialgorithmforminimalmatroids}. For many results, we first present a proof in line with the decomposition strategy as outlined in that paper, and then we provide a shorter and more elegant alternative when possible. Another question in this thesis concerns the matroid ideal. 
\begin{question} \label{question gen}
    Find a set of generators for $I_M$, up to radical.
\end{question}
It is already challenging to find explicit non-trivial polynomials within $I_M$, constructing a complete set of generators up to radical even more so. Since $I_{\CCC(M)} \subset I_M$, we can already include these circuit polynomials in the generating set. To find other generating polynomials, we consider two other ideals which lie in $I_M$. Both are constructed using classical geometrical techniques:
\begin{itemize}
    \item The Grassmann-Cayley ideal, $G_M$, is introduced in \cite{Sidman}. Its generating polynomials can be obtained using the Grassmann-Cayley algebra by considering triples of concurrent lines in the configuration. This is further discussed in Subsection~\ref{Grassmann Cayley ideal}.
    \item The lifting ideal, $I_M^{\textup{lift}}$, is introduced in \cite{clarke2024liftablepointlineconfigurationsdefining}. It is constructed using a geometric liftability technique, as further outlined in Subsection \ref{liftability}.
\end{itemize}

As complete generating sets are known for $I_{\CCC(M)}, G_M$ and $I_M^{\text{lift}}$, it makes sense to explore for which families of matroids the sum of these three ideals generates $I_M$, up to radical. So instead of finding a generating set for matroid ideals in full generality, we will focus on the following question. 
\begin{question} \label{quest: identify}
    Which point-line configurations $M$ satisfy the following equality:
\begin{equation1} \label{eq: intro generators}
    I_M = \sqrt{I_{\CCC(M)}+ G_M+I_M^{\text{lift}}}?
\end{equation1}
\end{question}
There is no step-by-step plan or algorithmic way to prove that a certain point-line configuration satisfies Equation~\eqref{eq: intro generators}, but we present some general framework. This perturbation argument is discussed in more detail in Subsection \ref{subsection: perturbation strategy}.
\begin{strategy}
\leavevmode\par
\begin{itemize} 
\item To prove that Equation~\eqref{eq: I_M gen} holds for a certain point-line configuration $M$, we will prove the statement \begin{equation2} \label{eq: intersec V_M intro}
V_M = \VCM \cap V(G_M) \cap V(I_{M}^{\textup{lift}}).\end{equation2} Only the inclusion $\supseteq$ in Equation~\eqref{eq: intersec V_M intro} is non-trivial. So consider an arbitrary $\gamma \in \VCM \cap V(G_M) \cap V(I_M^{\textup{lift}}).$
\item Decompose the circuit variety of $M$ into its  irreducible components. Then $\gamma$ should be contained in one of them.
\item Perturb $\gamma$ to a collection of vectors in $\Gamma_M$. This proves that $\gamma$ is in the Euclidean closure of $\Gamma_M$, thus it must be in the Zariski closure of $\Gamma_M$, and therefore in $V_M$ as well.
\end{itemize}
\end{strategy}

\noindent In this thesis, we resolve Questions~\ref{question decompo} and~\ref{quest: identify} for three specific families of point-line configurations: cactus configurations, the Pascal configuration, and the Pappus configuration; see Definition~\ref{cact irr}, Figure~\ref{fig:pascal 1} (Left), and Figure~\ref{fig:Pappus, Pappus, I_i, J_i} (Left). We obtain the following main results:

\begin{theoremA}

The following point-line configurations satisfy Equation~\eqref{eq: intro generators}:
\begin{enumerate}[itemsep=0pt, parsep=0pt]
\item Cactus configurations\hfill\textup{(Theorem~\ref{thm: main theorem cactus matroid ideal})}
\item Pappus configuration\hfill\textup{(Theorem~\ref{thm:Pappus})}
\item Pascal configuration.\hfill\textup{(Theorem~\ref{pascal generators})}
\end{enumerate}
\end{theoremA}
\noindent Theorem (A) results in an explicit set of generators for the Pascal and Pappus configurations. The following table provides an overview of the number of each type of generating polynomials.
\begin{table}[h!]
\centering
\begin{tabular}{|l|c|c|c|}
\hline
\textbf{Configuration} & \textbf{Circuit} & \textbf{Grassmann-Cayley} & \textbf{Lifting} \\
& {\bf Polynomials} & {\bf Polynomials} & {\bf Polynomials} \\
\hline
Pascal & 7 & 7 & 708588 \\ 
Pappus & 9 & 9 & 2361960 \\
\hline
\end{tabular}
\end{table}
\newline
\noindent Moreover, we also derive the following results. The third configuration is shown in Figure \ref{fig:Third Config A_1 B}. This point-line configuration has nine points and nine lines, such that every point is on three lines and every line contains three points.

\begin{theoremB}
Let $M$ denote any cactus configuration, let $N$ be the Pascal configuration and $K$ the third configuration $9_3$. The following statements hold:
\begin{enumerate}[itemsep=0pt, parsep=0pt]
\item $M$ is realizable and $V_{M}$ is irreducible. \hfill{(Theorem~\ref{cact irr})}
\item Denote $Q_M$ the subset of points in $M$ that lie on more than two lines. Then $\VCM$ has at most $2^{\size{Q_{M}}}$ irredundant, irreducible components, obtained from $M$ by setting a subset of $Q_{M}$ to be loops. \hfill{(Theorem~\ref{thm: decomposition of circuit variety of a cactus})} 
\item The circuit variety of the Pascal configuration has the following irreducible decomposition:
\[V_{\CCC(N)}=V_{N}\cup V_{U_{2,9}}\cup_{i=7}^{9}V_{N(i)},\tag{Theorem~\ref{proposition: decomposition pascal configuration}}\]
where $U_{2,9}$ is the uniform matroid of rank two on $[9]$ and $N(i)$ denotes the matroid obtained from $N$ by setting $i$ to be a loop.
\item The decomposition of $V_{\CCC(K)}$ has 46 irreducible components. \hfill{(Theorem~\ref{thm decompo third 9_3})}
\end{enumerate}
\end{theoremB}
\noindent

\noindent{\bf Chapter \ref{sec: intro}} introduces the topic and discusses this work in a broader context. {\bf Chapter \ref{sec: prelim}} discusses some necessary background concerning matroids and matroid varieties, and some classical techniques. 
In {\bf Chapter \ref{sec: cactus}}, we will introduce cactus configurations. In this chapter, we will prove the first statement of Theorem (A) and the first and the second statement of Theorem (B). In {\bf Chapter \ref{sec: pascal & pappus}} we study two classical point-line configurations, the Pascal and the Pappus configuration, and we prove the last two points of Theorem (A) and the last point of Theorem (B). In {\bf Chapter \ref{Decomposition of the Circuit Variety of the third configuration 93}}, we study the configuration $9_3$, and we prove the last point of Theorem (B). In {\bf Chapter \ref{sec: decompo revisited}}, we provide an alternative proof for the decomposition of the circuit variety of the Pappus configuration with one loop and the last point of Theorem (B) using the decomposition strategy. 
\cleardoublepage
\tableofcontents
\pagenumbering{arabic}
\numberwithin{theorem}{chapter}
\cleardoublepage
\chapter{Introduction}\label{sec: intro}
\section{Motivation}
Matroids were first introduced by Hassler Whitney in \cite{Hassler1935-ko}. They are a combinatorial generalization of linear dependence among vectors: a collection of linearly dependent subsets of a finite set of vectors in a vector space determines a matroid. They are often studied in mathematics, see for instance \cite{Oxley, piff1970vector}. As discussed in \cite{tutte1966matroids}, this abstraction has led to many interesting applications, for instance in instance network-flow theory \cite{ gale1959transient,minieka1973maximal} and linear and integer programming \cite{bland1974orientability,BLAND197733}. 
\newline
\newline
One can also use this combinatorial notion to define varieties in algebraic geometry, as follows. A realization of a matroid $M$ is a finite collection of vectors for which the induced matroid is $M$. We say that a collection of vectors includes the dependencies of $M$ if their induced matroid has at least the dependencies of $M$. The realization space of a matroid $M$, denoted by $\Gamma_M$, is defined as the set containing all the realizations of $M$, and the matroid variety, denoted by $V_M$, is defined as the Zariski closure of $\Gamma_M$. The circuit variety of $M$ is the set containing each collection of vectors which includes the dependencies of $M$. The matroid and circuit varieties will be the main objects of study in this thesis. In particular, we aim to find an irreducible decomposition for the circuit variety and a finite set of defining equations for the matroid ideal, for various matroids.
\newline
\newline
Matroid varieties were introduced in~\cite{gelfand1987combinatorial} and they have since become a central theme in modern algebraic geometry, because their study allows for the combination of classical tools from algebraic geometry and techniques arising from their combinatorial structure, leading to interesting results~\cite{Vakil, Sidman, liwski2024pavingmatroidsdefiningequations, sturmfels1989matroid, feher2012equivariant, knutson2013positroid}. The behavior of these matroid varieties can be very complex by the Mn\"ev-Sturmfels universality theorem. In particular, for any singularity appearing in a semi-algebraic set, one can find a matroid variety which has the same singularity, up to stable equivalence \cite{Mnev1985, Mnev1988, sturmfels1989matroid}. The interest in both matroid and circuit varieties is not only caused by their rich geometric structure, but also by their natural presence in many other settings. They appear for instance in the study of determinantal varieties~\cite{bruns2003determinantal, clarke2022matroidstratificationshypergraphvarieties, herzog2010binomial, pfister2019primary, ene2013determinantal}, rigidity theory~\cite{jackson2024maximal, whiteley1996some, graver1993combinatorial, maximum}, and conditional independence models~\cite{Studeny05:Probabilistic_CI_structures, DrtonSturmfelsSullivant09:Algebraic_Statistics, hocsten2004ideals, clarke2022conditional, caines2022lattice, alexandr2025decomposing}. In these situations, a common challenge is to decompose conditional independence ideals or determinantal ideals into their primary components. Matroid varieties often play a central role in such decompositions, appearing as the irreducible components. Therefore, if we can find defining equations of matroid varieties, this result has many applications in the understanding of the primary decomposition of these mentioned ideals. 
\section{Main objects and results}
As it is a hard problem to find an irreducible decomposition of circuit varieties and a generating set for matroid ideals, we will focus on a specific type of matroid, called point-line configurations, as defined in Definition \ref{point-line}. We will study for instance cactus configurations, which form a family of point-line configurations generalizing the forest configurations. The latter is often studied in literature, see for instance \cite{clarke2022matroidstratificationshypergraphvarieties}. Moreover, we study the Pascal and Pappus configurations, two classical point-line configurations obtained from the corresponding theorems in incidence geometry. They are shown in Figure \ref{fig:pascal 1} (Left) and \ref{fig:Pappus, Pappus, I_i, J_i} (Left). Finally, we also study the third configuration $9_3$, as shown in Figure \ref{fig:Third Config A_1 B} (Left), which is a point-line configuration with nine points and nine lines, such that every line contains three points and every point is on three lines. This point-line configuration is obtained from a relaxation of the conditions of a Steiner system, \cite{Van_Lint2012-bn}.
\newline
\newline
A first goal of this thesis is to determine an irreducible decomposition of various circuit varieties. This problem is already studied in \cite{liwskimohammadialgorithmforminimalmatroids, liwski2025efficient} for some point-line configurations. In this thesis, we find such a decomposition for some computationally challenging cases, including:
\begin{itemize}
    \item Cactus configurations (up to irredundancy)
    \item Pascal configuration
    \item Third configuration $9_3$.
\end{itemize}
For the cactus configurations, we use a similar method as in \cite{liwski2024solvablenilpotentmatroidsrealizability}. For the Pascal configuration, we give a proof using a new method that we developed, which is applicable if a point-line configuration has \enquote{enough} points of degree two. For the third configuration $9_3$, we use a combination of the decomposition strategy, see \cite{clarke2022matroidstratificationshypergraphvarieties, liwskimohammadialgorithmforminimalmatroids, liwski2025efficient}, and our new methods. Later, we also give a proof of the decomposition of the third configuration $9_3$ following the decomposition strategy completely. One can notice that the latter proof is remarkably longer than the former. 
\newline
\newline
We will also prove that the matroid variety of cactus configurations is non-empty and irreducible, extending the results of \cite{liwski2024solvablenilpotentmatroidsrealizability}.
\newline
\newline
Another goal is to find a finite set of defining equations for the matroid variety, or equivalently, to find a complete finite set of generators for the corresponding ideal, called the matroid ideal, up to radical. Using the liftability technique and the Grassmann-Cayley algebra, we can define two new ideals, the lifting ideal $I_M^{\textup{lift}}$, as in \cite{clarke2024liftablepointlineconfigurationsdefining}, and the Grassmann-Cayley ideal $G_M$, as in \cite{Sidman}. $I_{\CCC(M)}, I_M^{\textup{lift}}, G_M$ are contained in $I_M$, and we know explicit generators for them. Therefore, if we can identify point-line configurations satisfying \begin{equation}\label{eq in main objects}
I_M = \sqrt{I_{\CCC(M)}+I_M^{\textup{lift}}+G_M},\end{equation}
then we have found a generating set for the matroid ideal of these point-line configurations. We use the perturbation strategy, as outlined in Subsection \ref{subsection: perturbation strategy}, to prove that a point-line configuration satisfies Equation \ref{eq in main objects}.
In this thesis, we find that this equation holds for: \begin{itemize}
    \item Cactus configurations
    \item Pascal configuration
    \item Pappus configuration.
\end{itemize}

To my knowledge, the Pascal and Pappus configurations are the first matroids for which Equation~\eqref{eq in main objects} has been proven where both $I_M^{\textup{lift}}$ and $G_M$ are irredundant in the generating set. 
\newline
\newline
From this result, one can derive a generating set for the matroid ideal of the Pascal and Pappus configurations. Table \ref{table} provides an overview of the number of each type of generating polynomial in the generating set, as in Remark \ref{remark pascal} and \ref{remark pappus}.
\newline
\begin{table}[h!]
\centering
\caption{Number of each type of polynomial in the generating set for $I_M$}
\label{table}
\begin{tabular}{|l|c|c|c|}
\hline
\textbf{Configuration} & \textbf{Circuit} & \textbf{Grassmann-Cayley} & \textbf{Lifting} \\
& {\bf Polynomials} & {\bf Polynomials} & {\bf Polynomials} \\
\hline
Pascal & 7 & 7 & 708588 \\ 
Pappus & 9 & 9 & 2361960 \\
\hline
\end{tabular}
\end{table}
\newline
It may seem that the results presented in this thesis are limited, since they only hold for specific families of matroids. However, it is a particularly hard problem to find an irreducible decomposition of a circuit variety and explicit generators for a matroid ideal, as noticed in \cite{pfister2019primary}, for instance since few methods are known for the construction of such generators. To date, there are only specific types of matroids for which an explicit set of generators has been found. The technique that we used as an alternative for the decomposition strategy is already applicable to many other point-line configurations and enables the efficient decomposition of circuit varieties for which the calculations using the decomposition strategy become too extensive. Similarly, we hope that the techniques and results presented in this thesis concerning the generators of a matroid ideal will stimulate further research in this direction.  

\section{Related work} \label{sec: related work}
Some research in this direction has already been done; we highlight some important work.
\begin{itemize}
\item In \cite{pfister2019primary}, the authors established an algorithm to compute defining equations for the matroid variety associated with the $3\times 4$ grid configuration. This configuration contains sixteen circuits of size three. They derived a finite generating set for the corresponding ideal containing 44 polynomials, using {\tt Singular}. Moreover, they showed that the circuit variety has two components. However, the computations push the limits of current computer algebra systems, and there is no clear combinatorial interpretation for the resulting components.
\item In \cite{clarke2024liftablepointlineconfigurationsdefining}, the geometric liftability technique was used to find a generating set for the matroid ideals of the quadrilateral set and the $3\times 4$ grid. The approach in this paper offers a geometric interpretation for the resulting polynomials.
\item In \cite{liwski2024pavingmatroidsdefiningequations}, the authors defined the lifting ideal $I_M^{\textup{lift}} \subset I_{M}$ using the geometric liftability technique. They showed that, for paving matroids without points of degree greater than two, the ideal $I_{\CCC(M)}+I_{M}^{\textup{lift}}$ defines the associated matroid variety.
\item In \cite{Sidman}, the authors used the Grassmann–Cayley algebra to define the ideal $G_{M}\subset I_{M}$. From this, the authors constructed seven new polynomials contained in the matroid ideal of the Pascal configuration.
\item In~\cite{liwski2024solvablenilpotentmatroidsrealizability}, the authors proved that for forest configurations, the radical of the ideal $I_{\CCC(M)}+G_M$ equals the matroid ideal $I_M$. Moreover, they find an irreducible decomposition for forest configurations and they discuss the decomposition strategy.
\item In 
\cite{liwskimohammadialgorithmforminimalmatroids,liwski2025efficient}, various circuit varieties are decomposed.
\item The decomposition of hypergraph varieties is studied in \cite{bruns2003determinantal}. In
\cite{herzog2010binomial, diaconis1998lattice}, they discuss the decomposition of varieties corresponding to binomial edge ideals. 
\end{itemize}

\numberwithin{theorem}{chapter}

\chapter{Preliminaries} \label{sec: prelim}

In this chapter, we discuss some essential properties of matroids and their associated varieties. For additional information, we refer the reader to \cite{Oxley, gelfand1987combinatorial, victorreinerlecturematroids} and \cite{liwski2024pavingmatroidsdefiningequations, liwski2024solvablenilpotentmatroidsrealizability, liwskimohammadialgorithmforminimalmatroids}. Throughout, we use the notation $[n] = \{1, \ldots, n\}$ and $\textstyle \binom{[d]}{n}$ to denote the set of $n$-element subsets of $[d]$. 
\section{Matroids}

We present several results on matroids. For a complete introduction to matroid theory, see \cite{Oxley, piff1970vector}. There are several equivalent ways to define matroids; we present four of them. 

\begin{definition} \label{Definition1matroids}
    A {\it matroid} $M$ is formed by a finite non-empty set $[d]$, called {\it the ground set}, and a collection of subsets of $[d]$, called the {\it independent sets} $\mathcal{I}$, such that the sets $I \in \mathcal{I}$ satisfy the following three properties:
    \begin{itemize}
        \item $\emptyset \in \mathcal{I}$.
        \item If $I_1 \in \mathcal{I}$ and $I_2 \subset I_1$, then $I_2 \in \mathcal{I}$.
        \item If $I_1, I_2 \in \mathcal{I}$ and $|I_1| < |I_2|$, then there exists $e \in I_2 \setminus I_1$ such that $I_1 \cup \{e\} \in \mathcal{I}$.
    \end{itemize}
    This last property is called the {\it exchange axiom.}
\end{definition}
We can also define matroids by their bases.
\begin{definition}\label{Definition2matroids}
    A {\it matroid} $M$ is formed by a non-empty set $[d]$, called the {\it ground set}, and a collection of subsets of $[d]$, called the {\it bases} $\mathcal{B}$, such that the sets $B \in \mathcal{B}$ satisfy the following three properties:
    \begin{itemize}
        \item $B \neq \emptyset$.
        \item $B_1, B_2 \in \mathcal{B}$ and $x \in B_1 \setminus B_2$ implies that there exists $y \in B_2 \setminus B_1$ such that $(B_1 \setminus \{x\}) \cup \{y\} \in \mathcal{B}$.
    \end{itemize}
    This last property is again called the {\it exchange axiom}.
\end{definition}
We can define a matroid as well in terms of its circuits.
\begin{definition}\label{Definition3matroids}
Given a {\it matroid} $M$ with {\it ground set} $[d]$, a collection $\mathcal{C}$ of subsets of $[d]$ is called the collection of {\it circuits} of $M$ if $\mathcal{C}$ satisfies the following three properties:
\begin{itemize}
\item $\emptyset \notin \mathcal{C}$.
\item If $C_1, C_2 \in \mathcal{C}$ and $C_1 \subset C_2$, then $C_1=C_2$.
\item If $C_1, C_2 \in \mathcal{C}$, $C_1 \neq C_2$ and $e \in C_1 \cap C_2$, then there is a $C_3 \in \mathcal{C}$ such that $C_3 \subset C_1 \cup C_2 \setminus \{e\}$.
\end{itemize} The last property is called {\it circuit elimination}.
\end{definition}
We denote by $\mathcal{C}(M), \mathcal{I}(M), \mathcal{D}(M)$ respectively the collection of circuits of $M$, the independent sets of $M$ and the dependent sets of $M$.
Finally, we present a definition of a matroid in terms of its rank function.
\begin{definition}\label{Definition4matroids}
    A matroid $M$ consists of a finite ground set $[d]$, together with a mapping: $\rk: \mathcal{P}([d]) \to \mathbb{N}$, satisfying the following three properties:
   \begin{itemize}
       \item For $X \subset [d]$, $0 \leq \rk(X) \leq |X|$.
       \item For $X \subset Y \subset [d]$, $\rk(X) \leq \rk(Y).$
       \item For $X, Y \subset [d]$, the following inequality holds:
       $$\rk(X \cup Y)+ \rk(X \cap Y) \leq \rk(X)+\rk(Y).$$
   \end{itemize}
   The last property is known as {\em submodularity} of the rank function.
\end{definition}
\begin{remark} \label{remark on definitions matroids}
The previous four definitions are equivalent if one uses the following connections between them.
Let $M$ be a matroid on the ground set $[d]$. Then:
\begin{itemize}
\item The dependent subsets are the subsets of $[d]$ containing a circuit. So the circuits are the minimal dependent subsets, with respect to inclusion.
\item The independent sets are the subsets of $[d]$ which are not dependent, in other words, the subsets which do not contain a circuit. Equivalently, the independent sets are the subsets of $[d]$ for which the rank equals the cardinality.
\item A basis is a maximal independent subset of $[d]$, with respect to inclusion. 
\item The rank of $F \subset [d]$ is the size of the largest independent set contained in $F$. Note that $\rank([d])$ is the size of any basis. We define $\rank(M)$  as $\rank([d])$.
\end{itemize}
\end{remark}

\begin{example}\normalfont\label{uniform 3}
The uniform matroid $U_{n, d}$ of rank $n$ is defined as the matroid on $[d]$ whose independent sets are precisely the subsets $S\subset [d]$ with $|S| \leq n$. 
\end{example}

Given a matroid $M$ on the ground set $[d]$, we define the operations of restriction and deletion.

\begin{definition}
\normalfont \label{subm}
Given a subset $S \subset [d]$, the {\em restriction} of $M$ to $S$, denoted by $M|S$, is the matroid 
on the ground set $S$ whose rank function is the restriction of the rank function of $M$ to $S$. Unless otherwise specified, we will assume that the subsets of $[d]$ inherit this structure and we call them \textit{submatroids} of $M$. To simplify notation, we may write $S$ instead of $M|S$ when the context allows. We define the {\em deletion} of $S$, denoted $M\setminus S$, as $M|([d]\setminus S)$.
\end{definition}
\begin{remark}
    Using Remark~\ref{remark on definitions matroids}, it is clear that for a matroid $M$ on the ground set $[d]$ and $S \subset [d]$, $M | S$ is the matroid with independent sets $$\mathcal{I}(S) = \{I \in \mathcal{I}(M) : I \subset S\},$$ and thus with dependent sets $$\mathcal{D}(S) = \{D \in \mathcal{D}(M) : D \subset S\}.$$
\end{remark}
\begin{example}
    Consider the matroid $QS$ with ground set $[d]$ and dependent sets $$\mathcal{D}(QS) = \bigl\{\{1,2,3\},\{1,5,6\}, \{2,4,6\}, \{3,4,5\}\bigr\} \cup \binom{[6]}{4}.$$ This matroid will show up again in Example \ref{three lines} as the quadrilateral set. Then $M | [4]$ is the matroid with dependent sets $$\mathcal{D}(M|[4]) = \bigl\{\{1,2,3\}, \{1,2,3,4\}\bigr\}.$$
\end{example}

Finally, we introduce the notion of loops and simple matroids.
\begin{definition}
    An element $p$ in the ground set $[d]$ of a matroid $M$ is called a \emph{loop} if $\{p\}$ is a circuit. In figures, it will be shown with a rectangle. For a matroid $M$ on the ground set $[d]$ and $S \subset [d]$, one can construct a new matroid on $[d]$ with circuits $$\min\{\mathcal{C}(M) \cup \{\{s\}:s \in S\}\},$$
    where $\min$ denotes the minimal elements with respect to the inclusion order. We say that this matroid is derived from $M$ by {\em setting the points in $S$ to be loops}, and we denote it by $M(S)$.
\end{definition}
\begin{definition}
    We say that two elements $x,y\in [d]$ are {\em parallel}, or form a {\em double point}, if $\rank(\{x,y\})=1$. A {\em simple matroid} is a matroid without loops or parallel elements.
\end{definition}

\section{Point-line configurations}
We review some definitions and introduce some notation for point-line configurations, following \cite{liwski2024pavingmatroidsdefiningequations}.
\begin{definition} \label{point-line}
    A {\it point-line configuration} is a simple matroid of rank at most three.
\end{definition}
\begin{remark}
    In \cite{liwski2024pavingmatroidsdefiningequations}, they introduce {\em paving matroids} as matroids of rank $n$ such that every circuit has size $n$ or $n+1$. It immediately follows that point-line configurations are paving matroids of rank three and the uniform matroids $U_{2,d}$.
\end{remark}
We introduce some notation for point-line configurations $M$ on the ground set $[d]$: 
\begin{itemize}
    \item A \emph{line} $l$ is a maximal dependent subset of $[d]$ with at least three elements in the ground set of $M$, such that every subset of three elements forms a circuit. The collection containing all lines of $M$ is denoted by $\mathcal{L}_M$. For point-line configurations, the dependent hyperplanes are called {\em lines}.
    \item We refer to the elements in the ground set $[d]$ as {\em points}. Given a point $p \in [d]$, $\mathcal{L}_p$ denotes the collection of dependent hyperplanes that contain $p$. The number of elements in $\mathcal{L}_p$ is called the {\em degree} of $p$. 
\end{itemize}

\begin{example}\label{three lines}
Consider the point-line configuration $M$ depicted in Figure~\ref{fig:combined} (Left). The set of points is $[7]$, and the lines are $$\mathcal{L}_M=\{   \{{1},{2},{7}\} , \{{3},{4},{7}\}  , \{{5},{6},{7}\} \}.$$ For the quadrilateral set $\text{QS}$, shown in Figure~\ref{fig:combined} (Center), the set of points is $[6]$ and the set of lines is given by $$\mathcal{L}_{\text{QS}}=\bigl\{\{1,2,3\},\{1,5,6\},\{2,4,6\},\{3,4,5\}\bigr\}.$$
\end{example}
\begin{figure}
    \centering
    \includegraphics[width=0.9\linewidth]{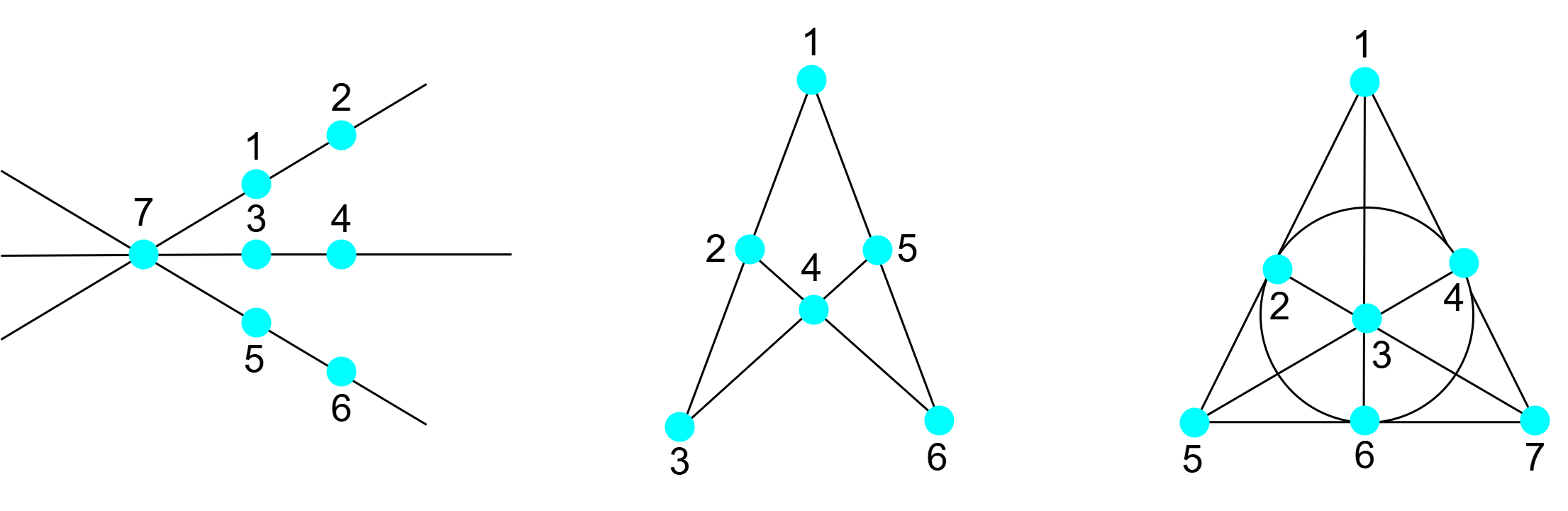}
    \caption{(Left) Three concurrent lines; (Center) Quadrilateral set; (Right) Fano Plane.}
    \label{fig:combined}
\end{figure}
\begin{figure}

    \centering
\includegraphics[width=0.9\linewidth]{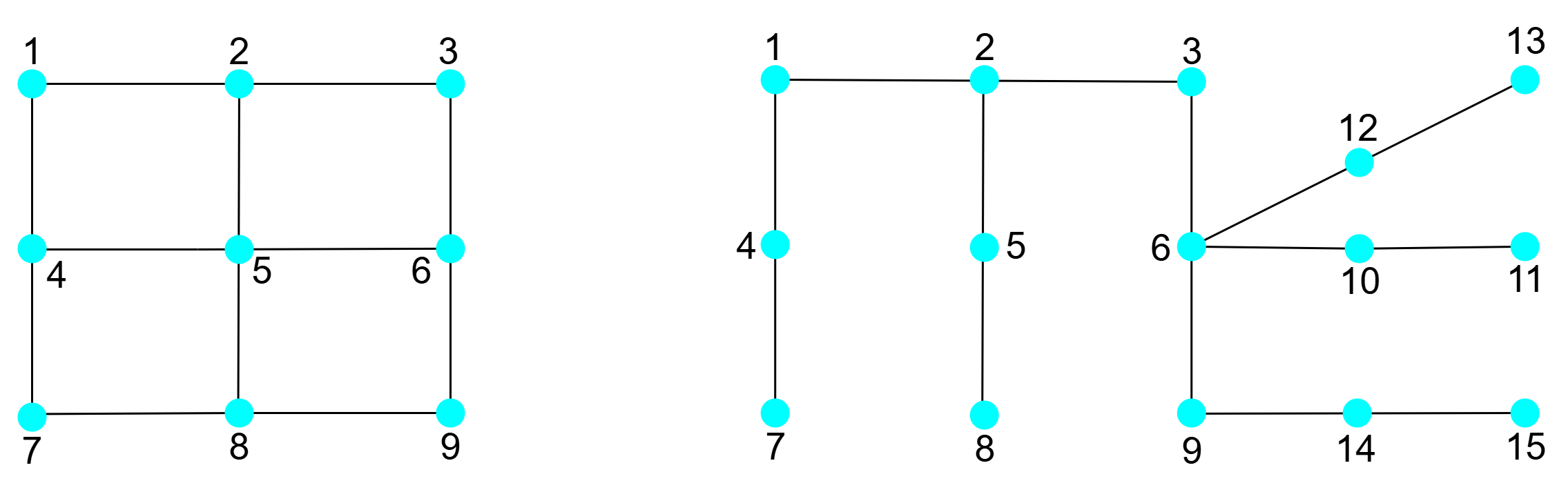}
\caption{(Left) The $3\times 3$ grid; (Right) Example of a forest configuration.}
\label{fig:preliminaries grid, fano, cactus}
\end{figure}

We define two families of point-line configurations that play a central role in this thesis, see \cite{liwski2024solvablenilpotentmatroidsrealizability}. 

\begin{definition} \label{def: solvable and nilpotent}\normalfont 
For a point-line configuration $M$ on $[d]$, we define 
\[S_{M}=\{p\in [d]: \size{\mathcal{L}_{p}}\geq 2\},\quad \text{and} \quad Q_{M}=\{p\in [d]: \size{\mathcal{L}_p} \geq n\}.\] 
We then consider the following chains of submatroids of $M$: 
\begin{equation*}
\begin{aligned}
&M_{0}=M, \ \quad M_{1}=M|S_{M},\quad \text{and }\quad M_{j+1}=M_j|S_{M_{j}} \quad \ \text{ for all $j\geq 1$.}\\
& M^{0}=M,\quad M^{1}=M|Q_{M},\quad \text{and }\quad M^{j+1}=M_j|Q_{M^{j}} \quad \text{ for all $j\geq 1$.}
\end{aligned}
\end{equation*}
We say that $M$ is {\em nilpotent} if $M_{j}=\emptyset$ for some $j$, and we call that $j$ the {\em length of the nilpotency chain}, denoted by $l_n(M)$. 

We say that $M$ is {\em solvable} if $M^{j}=\emptyset$ for some $j$. 
\end{definition}

\begin{example} \label{example nil-> sol}
    It immediately follows from Definition~\ref{def: solvable and nilpotent} that every nilpotent point-line configuration is solvable. The $3 \times 3$ grid, as in Figure \ref{fig:preliminaries grid, fano, cactus} (Left), is solvable but not nilpotent. The Fano plane, as in Figure \ref{fig:combined} (Right), is not solvable.
    
\end{example}

\section{The bracket algebra} \label{The bracket algebra}

We review the bracket algebra from \cite{InvariantMethodsinDiscreteandComputationalGeometry}. First we introduce some notation:
 \begin{notation}
 Throughout, we adopt the following conventions:
     \begin{itemize}
\item Consider an $n\times d$ matrix $X=\pare{x_{i,j}}$ of indeterminates. We denote by $\CC[X]$ the polynomial ring in the variables $x_{ij}$. 
\item For subsets $A\subset[n]$ and $B\subset[d]$ with $\size{A}=\size{B}$, we denote by $[A|B]_{X}$ the minor of $X$ determined by the rows indexed by $A$ and the columns indexed by $B$. 
\item \label{determinant} Consider the vectors $v_{1},\ldots,v_{n}\subset \CC^{n}.$ Then $[v_{1},\ldots,v_{n}]$ denotes the determinant of the matrix with columns $v_1,\ldots,v_n$.
\item \label{brackets notation} Let $M$ be a matroid of rank $n$ on $[d]$. We can associate $X$, an $n\times d$ matrix of indeterminates to $M$. Given a subset $\{i_{1},\ldots,i_{k}\}\subset [d]$ with $k\leq n$, and vectors $v_{1},\ldots,v_{n-k}\in \CC^{n}$, 
we denote by \[[i_{1},\ldots,i_{k},v_{1},\ldots,v_{n-k}]\in \CC[X]\] the determinant of the $n\times n$ matrix constructed by considering the columns of $X$ indexed by $i_{1},\ldots,i_{k}$ 
and the vectors $v_{1},\ldots,v_{n-k}$ as additional columns. Note that $[i_{1},\ldots,i_{k},v_{1},\ldots,v_{n-k}]$ is not a number, but a polynomial in $\CC[X]$. We call this determinant a {\em bracket}.
\end{itemize}
 \end{notation}

\begin{definition} 
    Considering $X$ as in the previous notation, the \emph{bracket ring} $B$ over $\CC^n$ is the subring of the polynomial ring $\CC[X]$ that is generated by all the brackets of the form $[i_1, \ldots, i_n].$
\end{definition}
We end this section with an important result of invariant theory. For a proof, we refer to Theorem~6.19 of \cite{Olver_1999}.
\begin{theorem}[The Second Fundamental Theorem of Invariant Theory]\label{thm: 2nd fundamental thm of invariant theory}
    The following relations hold in the bracket algebra. Moreover, all other relations among the brackets are a consequence of these relations.
    \begin{itemize}
    \item $[a_1, a_2, \ldots, a_n] = 0$ if any $a_j = a_k, j \ne k.$
    
    \item $[a_1, a_2, \ldots, a_n] = \operatorname{sign}(\sigma)[a_{\sigma(1)}, a_{\sigma(2)}, \ldots, a_{\sigma(n)}]$ for any permutation $\sigma$ of $\{1, 2, \ldots, n\}.$
    
    \item $[a_1, a_2, \ldots, a_n][b_1, b_2, \ldots, b_n]$
    \[
    = \sum_{j=1}^{n} [a_1, a_2, \ldots, a_{n-1}, b_j][b_1, b_2, \ldots, b_{j-1}, a_n, b_{j+1}, \ldots, b_n].
    \]
    \end{itemize}
\end{theorem}
\section{Realization spaces of matroids, circuit and matroid varieties}

We start by revisiting the definitions of the realization space of a matroid, the matroid variety and the circuit variety.  
\begin{definition} \label{realization space} \normalfont
Consider a matroid $M$ of rank $n$ on the ground set $[d]$. A {\em realization} of $M$ is a set of vectors $\gamma=\{\gamma_{1},\ldots,\gamma_{d}\}\subset \CC^{n}$ for which the following equivalence holds:
\[\{i_{1},\ldots,i_{p}\} \ \text{is a dependent set of $M$}\Longleftrightarrow \{\gamma_{i_{1}},\ldots,\gamma_{i_{p}}\}\ \text{is linearly dependent.}\]
The {\em realization space} of $M$ is defined as 
$$\Gamma_{M}=\{\gamma\subset \CC^{n}: \gamma \ \text{is a realization of $M$}\}.$$
Each element of $\Gamma_{M}$ can be identified with an $n\times d$ matrix over $\CC$. Throughout, we will implicitly use this identification. A matroid is called {\em realizable} if there exists a realization of $M$. The {\em matroid variety} $V_M$ of $M$ is defined as the Zariski closure of $\Gamma_M$ in $\CC^{nd}$. The {\em matroid ideal} is defined as the associated ideal $I_{M}=I(V_{M})$.
\end{definition}
\begin{definition}\normalfont\label{cir}
Let $M$ be a matroid of rank $n$ on the ground set $[d]$.A collection of vectors $\gamma \subset \CC^{n}$ indexed by $[d]$ is said to {\it include the dependencies of $M$} if:
\[\set{i_{1},\ldots,i_{k}}\  \text{is a dependent set of $M$} \Longrightarrow \set{\gamma_{i_{1}},\ldots,\gamma_{i_{k}}}\ \text{is linearly dependent}. \] 
The {\em circuit variety} of $M$ is defined as \[V_{\mathcal{C}(M)}=\{\gamma:\text{$\gamma$ includes the dependencies of $M$}\}.\]
It is clear that $\gamma \in \VCM$ if and only if for every circuit $c = \{i_1, \ldots, i_m\}$ with $m \leq n$, the minors of the matrix $(\gamma_{i_1}, \ldots, \gamma_{i_m})$ vanish. Thus the circuit variety is indeed a closed variety in the Zariski topology. Its corresponding ideal is called the {\em circuit ideal} and is denoted by $I_{\CCC(M)}$. Clearly, the circuit ideal is generated by all the minors of the submatrices of $X$ with the circuits as columns:
    $$I_{\CCC(M)} = \langle [A|B]_X: B \in \CCC(M) \text{ and } A \subset [n] \text{ such that } |A| = |B|\rangle.$$
\end{definition}
\begin{remark}\label{generating set ICM}
Since $V_{M}\subset V_{\CCC(M)}$, it immediately follows that $I_{\CCC(M)}\subset I_M$.
\end{remark}
Having set these definitions, we can introduce some notation which will be used throughout this thesis.
\begin{notation} \label{notation thesis}
Throughout this thesis, we will freely alternate between projective and affine terminology. We adopt the following conventions:
\begin{itemize}
\item As each non-zero vector $\gamma_{i}$ determines a point in $\mathbb{P}^{2}$, we use the terms {\em vector} and {\em point} interchangeably when $\gamma_{i}\neq 0$, depending on whether $\gamma_i$ is considered in $\CC^{3}$ or $\mathbb{P}^{2}$.
    \item We say that two non-zero vectors $\gamma_1, \gamma_2 \in \CC^n$ \emph{coincide} or are the \emph{same point}, denoted by $\gamma_1 = \gamma_2$, if they are linearly dependent. This reflects that they coincide in $\mathbb{P}^2$.
    \item Consider a collection of vectors $\gamma = (\gamma_1, \ldots, \gamma_d) \in \VCM$ and a line $l$ of $M$. We denote $\gamma_l$ for the following two-dimensional subspace
\[\gamma_{l}=\sspan \{\gamma_{i}:i\in l\} \in \CC^3.\]
Moreover, for points $i,j$ in the ground set, we write \[\gamma_{ij} = \sspan \{\gamma_{i}, \gamma_j\} \in \CC^3.\]
More generally, given a collection of vectors $\{\gamma_{1},\ldots,\gamma_{d}\}\subset \CC^{3}$ and a subset $A\subset [d]$, we denote by $\gamma_{A}\subset \CC^{3}$ the linear span of the vectors $\{\gamma_{a}:a\in A\}$.
\item We will call the subspace $\gamma_{l}$ a \emph{line}, as it determines a unique line in the projective plane $\mathbb{P}^2$.
\item A zero vector is called a {\em loop}.
\item It may seem natural to work entirely in the projective setting, but this simplification is not possible because of the possible occurrence of zero vectors.
 \item Consider the transformation $\phi: \CC^n \to \CC^n: x \to AxD$ with $D$ a non-zero diagonal matrix and $A$ an invertible matrix. We will call $\phi$ a {\em linear transformation} as well.
 
\item Consider an $n \times d$-matrix $A$. If there is stated that \enquote{the collection of vectors $\gamma = \{\gamma_i:i \in [d]\}$ is given by $A$ where the columns from left to right correspond to $\{m_1, \ldots, m_d\}$}, this means that $$(\gamma_{m_1}, \gamma_{m_2}, \ldots, \gamma_{m_d}) = A,$$

where the left-hand side denotes the matrix with $\gamma_{m_1}, \ldots, \gamma_{m_d}$ as columns. We will sometimes write this as \enquote{$\gamma = A$, where the columns from left to right correspond to $\{m_1, \ldots, m_d\}.$}
\end{itemize}
\end{notation}
The following examples provide an illustration of the notions realizability, matroid varieties and circuit varieties.
\begin{example} \label{ex: fano plane not realizable}
    The Fano plane $M$, as shown in Figure \ref{fig:combined} (Right), is not realizable. Assume for contradiction that $\gamma = \{\gamma_1, \ldots, \gamma_7\}$ is a realization of $M$. Since each subset of three elements in $\{\gamma_1, \gamma_2, \gamma_3,\gamma_4\}$ is linearly independent, there is a linear transformation such that $$(\gamma_1,\gamma_2, \gamma_3,\gamma_4) = \begin{pmatrix}
            1 & 0 & 0 & 1 \\
            0 & 1 & 0 & 1 \\
            0 & 0 & 1 & 1
        \end{pmatrix}
$$    
    Since $\gamma \in \Gamma_M$, $\gamma_5$ belongs to $\sspan\{\gamma_{1}, \gamma_{2}\}$ and $\sspan\{\gamma_{3}, \gamma_{4}\}$. It follows that $$\gamma_5 = \begin{pmatrix}
        1 \\ 1 \\ 0
    \end{pmatrix}.$$ In the same way, $$\gamma_6 = \begin{pmatrix}
        1 \\ 0 \\ 1
    \end{pmatrix}.$$ $\gamma_7$ belongs to $\sspan\{\gamma_{2}, \gamma_{3}\}, \sspan\{\gamma_{1}, \gamma_{4}\}$ and $\sspan\{\gamma_{5}, \gamma_{6}\}$. These lines only intersect in the origin. But then $\gamma \notin \Gamma_M$, since $\{7\}$ is not a circuit of $M$. 
    \newline
    \newline
    On the other hand, the quadrilateral set $QS$, as in Figure \ref{fig:combined} (Center) is realizable. Consider the matrix $$\begin{pmatrix}
    1 & 0 & 0 & 1 & 1 & 1  \\
    0 & 1 & 0 & 1 & 1 & 0 \\
    0 & 0 & 1  & 1 & 0 & 1
    \end{pmatrix}$$ and let the columns correspond to $\gamma_i$, where $i$ from left to right ranges over the points $\{1,2,3,4,5,6\}$. Then the collection of vectors $\gamma = \{\gamma_i:i \in [6]\}$ clearly lies in $\Gamma_M$.
\end{example}
\begin{example}
    For the matroid $QS$ corresponding to the quadrilateral set, as in Figure \ref{fig:combined} (Center), consider the matrix
    $$A=\begin{pmatrix}
        1 & 0 & 3 & 0 & 2 & 0  \\
        0 & 1 & 0 & 0 & 0 & 0 \\
        0 & 0 & 0 & 1 & 0 & 0 
    \end{pmatrix}.$$
    Construct the collection of vectors $\gamma = \{\gamma_i:i \in [d]\}$ by letting the columns of $A$ correspond to $\{1,2,3,4,5,6\}.$
     Since $\{\gamma_1, \gamma_2, \gamma_3\},$ $\{\gamma_1, \gamma_5, \gamma_6\}, $ $\{\gamma_2, \gamma_4, \gamma_6\}, $ and $\{\gamma_3, \gamma_4, \gamma_5\}$ are linearly dependent, it follows that $\gamma \in \VCM$. However, $\gamma \notin \Gamma_M$, since $\{\gamma_1, \gamma_3\}$ is a linearly dependent set, while $\{1,3\} \in \mathcal{I}(M)$.
\end{example}
Given the central role of matroid ideals in this thesis, it is worth mentioning that there exists an explicit formula for the matroid ideal of any matroid, as outlined below. However, since this formula includes a saturation, computer algebra systems such as {\tt Macaulay2} cannot finish the calculations for $I_M$ for modest matroids such as the Pascal configuration which only has nine points. Therefore, in this thesis, we will not follow this approach.
\begin{definition}
    Let $M$ be a matroid. The {\it bases ideal} $J_M$ is defined as follows:
    $$J_M = \sqrt{\prod_{B \text{ basis of }M} \big\langle [A|B]_X:\ A\subset [d], |A| = |B|  \big\rangle}.$$
\end{definition}
We also recall saturations as in \cite{IdealsVarietiesandAlgorithms}.
\begin{definition}
    Consider two ideals $I, J$ in $\CC[x_1, \dots, x_n]$, then we define the \emph{saturation} of $I$ with respect to $J$, denoted by $I:J^{\infty}$ as:
\begin{equation*}
\left\{ f \in \CC[x_1, \dots, x_n] : \text{for all } g \in J, \text{ there exists } N \in \mathbb{N}_0 \text{ such that } f g^N \in I \right\}.
\end{equation*}
\end{definition}

\begin{proposition}[Proposition~2.1.3 in \cite{Sidman}] \label{prop: expression gamma_M}
    Consider a matroid $M$. The ideal of the matroid variety can be obtained by saturating the circuit ideal with respect to the bases ideal:
$$I_M = \sqrt{I_{\mathcal{C}(M)}:J_M^\infty}.$$
\end{proposition}
\begin{proof}
    As in \cite{IdealsVarietiesandAlgorithms}, $$V(I_{\CCC(M)}:J_M^\infty) = \overline{\VCM \setminus V(J_M)}.$$ So it suffices to show that $\VCM \setminus V(J_M) = \Gamma_M$. 
    We introduce the following notation. We say that $\gamma$ includes the independencies of $M$ if:
    $$
\scalebox{0.95}{$
\{i_1, \ldots, i_k\} \text{ is an independent set of } M 
\implies 
\{\gamma_{i_1}, \ldots, \gamma_{i_k}\} \text{ is linearly independent.}
$}
$$
    We claim that $\gamma \notin J_M$ is equivalent with $\gamma$ including the independencies of $M$. Assume that $\gamma$ includes the independencies of $M$. Let $B = \{b_1, \ldots, b_k\}$ be a basis for $M$. Since $B$ is an independent set as well, $\{\gamma_{b_1}, \ldots, \gamma_{b_k}\}$ is linearly independent. Thus $\gamma \notin J_M$. On the other hand, assume $\gamma \notin J_M$. Let $\{i_1, \ldots, i_k\}$ be an independent set. This set is contained in a basis $B$. Since $\gamma \notin J_M$, it follows that $\{\gamma_j:j \in B\}$ is a linearly independent set, thus $\{\gamma_{i_1}, \ldots, \gamma_{i_k}\}$ is a linearly independent set as well. We conclude that $\gamma$ includes the dependencies of $M$. This proves the claim.
    \newline
    \newline
     Note that $\gamma$ is a realization of $M$ if and only if $\gamma$ includes both the dependencies and the independencies of $M$. The first statement is equivalent with $\gamma \in \VCM$, while the second statement is equivalent with $\gamma \notin J_M$ by the previous claim. Thus $\VCM \setminus V(J_M) = \Gamma_M$.
\end{proof}

We recall the following theorem from \cite{liwski2024pavingmatroidsdefiningequations, liwski2024solvablenilpotentmatroidsrealizability} on nilpotent and solvable point-line configurations, which will be used in the subsequent sections.
\begin{theorem}
[Proposition~4.30 in \cite{liwski2024pavingmatroidsdefiningequations} and Theorem~4.12, Theorem~4.15, Theorem~5.13 in \cite{liwski2024solvablenilpotentmatroidsrealizability}] \label{nil coincide} 

Let $M$ be a point-line configuration on $[d]$. 
\begin{itemize}
\item[{\rm (i)}] If $M$ is nilpotent and has no points of degree greater than two, then $V_{\mathcal{C}(M)}=V_{M}$.
\item[{\rm (ii)}] If every proper submatroid of $M$ is nilpotent and $M$ has no points of degree greater than two, then $V_{\mathcal{C}(M)}=V_{M}\cup V_{U_{2,d}}$.
\item[{\rm (iii)}] If $M$ is nilpotent, then $M$ is realizable and $V_M$ is irreducible.
\item[{\rm (iv)}] If $M$ is solvable, then $V_{M}$ is irreducible. 
\end{itemize}
\end{theorem}

\begin{example}\label{ej quad}
For the point-line configuration $\text{QS}$ corresponding to the quadrilateral set illustrated in Figure~\ref{fig:combined} (Center), every proper submatroid of $M$ is nilpotent. It follows that $V_{\mathcal{C}(\text{QS})}=V_{\text{QS}}\cup V_{U_{2,6}}$. 
\end{example}

\section{Liftability} \label{liftability}

We now recall the notion of {\em liftable} point-line configurations, as in \cite{clarke2024liftablepointlineconfigurationsdefining,liwski2024pavingmatroidsdefiningequations}. In this subsection, let $M$ be a point-line configuration on $[d]$. 

\begin{definition} \label{def liftable} We introduce the following notions:
\begin{itemize}
\item Let $\gamma = \{\gamma_{i} : i\in [d]\}\subset \CC^{3}$ be a collection of vectors and $q\in \CC^{3}$ a vector. The collection of vectors $\widetilde{\gamma}=\{\widetilde{\gamma}_{i}:i\in [d]\}\subset \CC^{3}$ is a {\em lifting} of $\gamma$ from the vector $q$ if, for each $i\in [d]$, $\widetilde{\gamma}_i \neq q$ is on the line through $\gamma_i$ and $q$, or equivalently, there exists for each $i \in [d]$ a scalar $z_{i}\in \CC$ such that $\widetilde{\gamma}_{i}=\gamma_{i}+z_{i}q$. Moreover, we call the lifting {\em non-degenerate} if $\rk(\widetilde{\gamma}) = 3$.
\item Given a collection of vectors $\gamma \subset \CC^3$ and a vector $q \in \CC^3$, $\gamma$ is called {\it liftable} from $q$ if there exists a non-degenerate lifting of $\gamma$ from $q$ to a collection $\widetilde{\gamma} \in \VCM$.
\item $M$ is called \textit{liftable} if, for any collection of vectors $\gamma = \{\gamma_i : i \in [d]\}$ in $\mathbb{C}^3$ spanning a hyperplane $H$, and for any vector $q \notin H$, $\gamma$ is liftable from $q$.
\end{itemize}
\end{definition}

The following relation is established between liftable and nilpotent point-line configurations.

\begin{proposition}[Proposition~4.13 in \cite{liwski2024solvablenilpotentmatroidsrealizability}] \label{nilp lift}
Any nilpotent point-line configuration is liftable.
\end{proposition} 

In the next construction, we define the {\it liftability matrix}.
\begin{definition} \label{matrix lift}
Consider a vector $q \in \CC^n$. To construct the \textit{liftability matrix} of $M$ and $q$, denoted as $\mathcal{M}_q(M)$, index the columns by the points $[d]$ and the rows by the circuits of size three. For the row corresponding to a circuit $c = \{c_1, c_2, c_3\} \subset [d]$, define the $c_1^{\textup{th}}, c_2^{\textup{th}}, c_3^{\textup{th}}$ coordinate of the liftability matrix respectively as
$$[c_2,c_3,q] \quad -[c_1,c_3,q] \quad [c_1,c_2,q],$$
and let the other entries be zero. Notice that these entries are polynomials instead of numbers. For a collection of vectors $\gamma = \{\gamma_i:i \in [d]\}$, $\mathcal{M}_q^\gamma(M)$ denotes the matrix obtained from $\mathcal{M}_q(M)$ by substituting the three variables associated to index $i$ with the entries of $\gamma_i$, for each $i \in [d]$.
\end{definition}
The motivation for this definition is the following lemma.
\begin{lemma}[Lemma~4.31 in \cite{liwski2024solvablenilpotentmatroidsrealizability}]\label{lemma: motivation liftability} 
  For a collection of vectors $\gamma = \{\gamma_i:i \in [d]\} \subseteq \CC^{n}$, a vector $q \in \CC^{n}$ and 
  a vector $(z_1, \ldots, z_d) \in \CC^d,$ the following equivalence holds: $\gamma \in \ker(\liftmat)$ if and only if the collection of vectors $\{\gamma_i + z_i\cdot q : i \in [d]\}$ is in $\VCM$.
\end{lemma}
\begin{example}
    For the quadrilateral set $QS$ as in Figure \ref{fig:combined}, the liftability matrix for $QS$ for an arbitrary vector $q$ is:
    $$\mathcal{M}_q(QS) = \begin{pmatrix}
        [2,3,q] & -[1,3,q] & [1,2,q] & 0 & 0 & 0 \\
        [5,6,q] & 0 & 0 & 0 & -[1,6,q] & [1,5,q] \\
        0 & [4,6,q] & 0 & -[2,6,q] & 0 & [2,4,q] \\
        0 & 0 & [4,5,q] & -[3,5,q] & [3,4,q] & 0 \\
    \end{pmatrix}.$$
\end{example}

\begin{definition}
    The {\it lifting ideal}, denoted by $I_M^{\text{lift}}$, is the ideal that all the $(|N|-2)$-minors of the liftability matrices $\mathcal{M}_q(N)$ generate, where $N$ ranges over all full-rank submatroids of $M$ and $q$ varies over $\CC^3$.
\end{definition}

We have the following fundamental result for the lifting ideal from \cite{liwski2024pavingmatroidsdefiningequations}. 

\begin{theorem}[Theorem~3.9 and Lemma~3.6 in \cite{liwski2024pavingmatroidsdefiningequations}] \label{thm: gamma i V_M^lift => liftable}
Let $M$ be a point-line configuration on $[d]$.
\begin{itemize}
\item $I_M^\text{lift} \subset I_M$.
\item If $\gamma=\{\gamma_{1},\ldots,\gamma_{d}\}\subset \CC^{3}$ is a collection of vectors in $V(I_{M}^{\text{lift}})$ lying in a hyperplane $H$, then for each full-rank submatroid $N \subseteq M$ on the ground set $[d']$, $\gamma|_{[d']}$ is non-degenerately liftable from any $q\in \CC^{3}$ for which $q \notin H$. Moreover, this lifting can be made arbitrarily small.
\end{itemize}
\end{theorem}

\section{Grassmann-Cayley algebra}\normalfont
We recall the notion of the Grassmann-Cayley algebra from \cite{AlgorithmsInInvariantTheory,liwski2024solvablenilpotentmatroidsrealizability}.
\begin{definition} The Grassmann-Cayley algebra is defined as the exterior algebra $\bigwedge (\mathbb{C}^{d})$, together with two operations: the \textit{join}, denoted by $\vee$ and the \textit{meet}, denoted by $\wedge$. 
\begin{itemize}
\item If $v$ can be written as the join of $k$ vectors $v_{1}, \dots, v_{k}$, so $v=v_1 \vee \ldots \vee v_k$, then we say that $v$ is an {\em extensor}. In that case, one often omits $\vee$ in the notation and just writes: $v = v_1\ldots v_k$.
\item Consider two extensors $v = v_{1} \ldots v_{k}$ and $w = w_{1} \ldots w_{j}$, with lengths $k$ and $j$ respectively, where $j + k \geq d$. Then one can define the meet of $v$ and $w$ as
$$v\wedge w=\sum_{\sigma\in \mathcal{S}(k,j,d)}
\corch{v_{\sigma(1)}\ldots v_{\sigma(d-j)}w_{1}\ldots w_{j}}\cdot v_{\sigma(d-j+1)}\ldots v_{\sigma(k)},$$
where $ \mathcal{S}(k,j,d)$ denotes the set of all permutations of $\corch{k}$ that satisfy $\sigma(1)<\ldots <\sigma(d-j)$ and $\sigma(d-j+1)<\ldots <\sigma(d)$. When $j + k < d$, the meet is defined to be zero. 
\end{itemize}
\end{definition}

There is a connection between the extensor $v = v_{1} \ldots v_{k}$ and the subspace $\overline{v} = \langle v_{1}, \ldots, v_{k} \rangle$ it generates. The following properties hold:

\begin{lemma} [Lemma~2.12 in \cite{liwski2024solvablenilpotentmatroidsrealizability}]
\label{klj}
Consider two extensors $v=v_{1}\ldots v_{k}$ and $w=w_{1}\ldots w_{j}$ such that $j+k\geq d$. Then we have:
\begin{itemize}
\item The extensor $v$ is equal to zero if and only if the vectors $v_{1},\ldots,v_{k}$ are linearly dependent.
\item $\overline{v}$ uniquely determines the extensor $v$ (up to a scalar multiple).
\item The meet of two extensors is again an extensor.
\item By the previous point, $\overline{v \wedge w}$ is an extensor, so it is well-defined. $v\wedge w\neq 0$ if and only if $\ip{\overline{v}}{\overline{w}}=\CC^{d}$. In this case, we have $\overline{v}\cap\overline{w}=\overline{v\wedge w}.$
\end{itemize}
\end{lemma}
\begin{example} \label{example: 3 concurrent lines}
    In $\CC^3$, one can apply this lemma to establish the following correspondence  between algebra and geometry: 
    \begin{itemize}
        \item Three vectors $v_1, v_2, v_3 \in \CC^3$ span a space of rank two if and only if $\det(v_1, v_2, v_3) = 0$. This is, of course, the well-known result that three vectors in $\mathbb{C}^3$ are linearly dependent if and only if their determinant vanishes.
        \item If the point $v_5$ lies both in $\sspan\{v_1, v_2\}$ and  $\sspan\{v_3,v_4\}$, then $$v_5 = v_1v_2 \wedge v_3v_4 =  \det(v_1,v_2,v_3)v_4-\det(v_1,v_2,v_4)v_3.$$
    \end{itemize}
\end{example}
\subsection{Grassmann-Cayley ideal}\label{Grassmann Cayley ideal} 
We define the Grassmann-Cayley ideal using a construction from \cite{liwski2024solvablenilpotentmatroidsrealizability}. 
First we briefly review Noetherian modules, as in \cite{atiyah}.
\begin{definition}
    Consider a ring $R$ and an ascending chain of ideals in $R$, with respect to the inclusion. The ring $R$ is {\em Noetherian} if every such chain stabilizes, or equivalently, if every ideal of $R$ is finitely generated.
\end{definition}
The following classical result is proven in \cite{Hilbert1890}. For a modern proof of that result, we refer to Theorem 7.5 in \cite{atiyah}.
\begin{theorem}[Hilbert's Basis Theorem, Theorem I in \cite{Hilbert1890}]\label{thm: hilberts basis}
    If $A$ is a Noetherian ring, then the polynomial ring $A[x]$ is Noetherian.
\end{theorem}
Since $\CC$ is a field, $\CC$ is Noetherian, and so is $\CC[x]$.
\newline
\newline
We now give the construction of the ideal Grassmann-Cayley ideal.

\begin{construction}\label{GM}
Let $M$ be a point-line configuration. 
\begin{itemize}
\item Let $X_0$ be the set with the generators for $I_{\mathcal{C}(M)}$, as in Definition~\ref{cir}. 
\item Assume $x \in [d]$, $l_1 \neq l_2 \in \mathcal{L}_x$ and $p_1, p_2 \in l_1$, $p_3, p_4 \in l_2$. Choose a polynomial $P \in I_M$. We can use Example \ref{example: 3 concurrent lines}, since $x \in l_1 \cap l_2$. Thus for any $\gamma \in \Gamma_M$, we have that $$\gamma_x = [\gamma_{p_1}, \gamma_{p_2}, \gamma_{p_3}] \gamma_{p_4} - [\gamma_{p_1}, \gamma_{p_2}, \gamma_{p_4}] \gamma_{p_3}.$$
Denote $$y = [p_1, p_2, p_3] p_4 - [p_1, p_2, p_4] p_3,$$ and let $P'$ be the polynomial constructed from $P$ where the variable $x$ is replaced by $y$. Notice that $P' \in I_M$.
\item For $j \geq 1$, construct the set $X_j$ as the union of the set $X_{j-1}$ and the polynomials $P'$ for $P \in X_{j-1}$ for some variables $x$. 
\item Define $I_j$ as the ideal generated by the polynomials in $X_j$. Obviously, $X_0 \subset X_1 \subset \ldots$, so this implies that $I_0 \subset I_1 \subset \ldots$. Using Theorem~\ref{thm: hilberts basis}, this chain of ideals stabilizes. Denote the stabilized ideal as $G_M$. 
\end{itemize}
\begin{definition}
$G_M$ is called the {\it Grassmann-Cayley ideal}.
\end{definition}
\end{construction}

\begin{proposition}\label{inc G}
For any point-line configuration $M$, we have $G_{M}\subset I_{M}.$ 
\end{proposition}
\begin{proof}
It suffices to prove that $I_j \subset I_M$ for each $j \in \mathbb{N}$. Since $I_0 = I_{\CCC(M)} \subset I_M,$ the basis step holds. For the induction step, assume that $I_j \subset I_M$. Using Construction \ref{GM}, we can conclude that $I_{j+1} \subset I_M.$
\end{proof}
\begin{example}\label{gc3}
Let $M$ be the point-line configuration shown on the left of Figure~\ref{fig:combined} (Left). In any realization, the intersection of the lines $\gamma_{12}, \gamma_{34}$ and $\gamma_{56}$ is non-trivial, so $\gamma_{12} \cap \gamma_{34}\subset \gamma_{56}$. Thus by Example \ref{example: 3 concurrent lines}:
$$(34)\wedge (12)\vee 56=(\corch{3,1,2}4-\corch{4,1,2}3)\vee 56=\corch{1,2,3}\corch{4,5,6}-\corch{1,2,4}\corch{3,5,6} = 0$$ for any $\gamma \in \Gamma_M$. This is a polynomial in $G_M$.
In a similar way, one can construct the polynomials $$\{[1,2,5][6,3,4]-[1,2,6][5,3,4], [3,4,5][6,1,2]-[3,4,6][5,1,2]\}$$ in $G_M$. Using Theorem~\ref{thm: 2nd fundamental thm of invariant theory}, these polynomials all coincide, so $$G_M = I_{\CCC(M)} + \langle\corch{1,2,3}\corch{4,5,6}-\corch{1,2,4}\corch{3,5,6} \rangle$$ in this case.
\end{example}
One can deduce the following proposition from these observations.
\begin{proposition} \label{gm toegevoegd}
    Consider lines $l_{1},l_{2},l_{3}$ of $M$ which contain a common point $x\in [d]$.  
For any $\gamma\in V_{\CCC(M)}\cap V(G_M)$, we have $(\gamma_{l_{1}}\wedge \gamma_{l_{2}})\vee\gamma_{l_{3}}=0$, implying that the lines $\gamma_{l_{1}},\gamma_{l_{2}},\gamma_{l_{3}}$ are concurrent in $\mathbb{P}^{2}$. Thus, if $\gamma_{x}=0$, one may redefine $\gamma_{x}$ to be the point in $\mathbb{P}^{2}$ lying in the intersection of the three lines $\gamma_{l_{1}},\gamma_{l_{2}}$ and $\gamma_{l_{3}}$.
\end{proposition}
\begin{example}
    Consider the Fano plane as in Figure \ref{fig:combined} (Right). The point 5 lies on the lines $\bigl\{\{1,2,5\},\{3,4,5\},\{5,6,7\}\bigr\}$, implying that the polynomial $$[1,2,3][4,6,7]-[1,2,4][3,6,7]$$ is a polynomial within $I_1$. The point $3$ lies on the lines $\{3,4,5\}$ and $\{1,3,6\}$, therefore
     $$[1,2,6][4,5,1][4,6,7]+[1,2,4][1,6,7][4,5,6] \in I_2.$$
\end{example}

\section{Forest configurations} \label{subsec: forest configurations}
In this section, we review a specific family of point-line configurations, called the forest configurations as described in \cite{clarke2022matroidstratificationshypergraphvarieties, liwski2024solvablenilpotentmatroidsrealizability}. These configurations will motivate the introduction of  cactus configurations later on. The definition of forest configurations is based on that of forest graphs. We will first define how to associate a graph to a point-line configuration.
\begin{definition}
Consider a point-line configuration $M$ with an ordering on the ground set $[d]$. $\phi: [d] \to [d]$ maps a point in the ground set to its ordering number. The {\it corresponding graph} $G(M)$ has vertex set $[d]$ and edges $\{i,j\}$, such that:
    \begin{itemize}
        \item $i,j$ lie on the same line $l$ of $M$.
        \item There is no point $k \in l$ with $\phi(i) < \phi(k) < \phi(j)$ or $\phi(j) < \phi(k) < \phi(i)$.
    \end{itemize}
\end{definition}
This construction turns points of the configuration into vertices and paths into edges. It can be used to define a forest configuration.
\begin{definition} \label{def: forest}
    A {\em forest} is a graph without cycles. A {\em forest configuration} is a point-line configuration $M$ such that $G(M)$ is a forest.
\end{definition}
Lemma~5.2 of \cite{clarke2022matroidstratificationshypergraphvarieties} proves that this definition does not depend on the choice of the map $\phi$. 
\begin{example}
    The point-line configuration associated to three concurrent lines in Figure \ref{fig:combined} (Left) is a forest configuration. Consider $\phi$ to be the identity map on $[7]$. $G(M)$ is the graph with vertex set $[7]$ and the edges are $\bigl\{\{1,2\}, \{3,4\}, \{5,6\}, \{2,7\}, \{4,7\}, \{6,7\}\bigr\}$. Since this is a forest graph, the matroid with three concurrent lines is a forest configuration.
\end{example}
One can also introduce the notion of a cycle for point-line configurations, to state an equivalent characterization of forest configurations, as in \cite{Connectednessandcombinatorialinterplayinthemodulispaceoflinearrangements}.
\begin{definition}\label{definition n-cycle}
Suppose $M$ has $n$ lines. $M$ is called a {\em cycle} if there exists a subset of points $\{p_{1},\ldots,p_{n}\}\subset [d]$ together with an ordering $\{l_{1},\ldots,l_{n}\}$ of its lines, such that:
\begin{itemize}
\item Every point $p_{i}$	
  is incident to precisely two lines, namely $\{l_{i},l_{i+1}\}$, for each $i\in [n]$, where we identify $l_{n+1}$ with $l_{1}$.
\item Every other point $p\in [d]\setminus \{p_{1},\ldots,p_{n}\}$ is incident to exactly one line, so it has degree one.
\end{itemize}

\end{definition}
From Definition~\ref{def: forest}, one can derive the following lemma:
\begin{lemma} \label{lemma: forest config}
   Let $M$ be a point-line configuration. $M$ is a forest configuration if and only if $M$ has no cycles.
\end{lemma}
\begin{example}
    Using Lemma~\ref{lemma: forest config}, it follows that the quadrilateral set, as in Figure~\ref{fig:combined} (Center), is not a forest configuration, since there is a cycle formed by the points $\{1,2,4,5\}$ and lines $\bigl\{\{1,2,3\}, \{2,4,6\}, \{3,4,5\}, \{1,5,6\}\bigr\}.$
\end{example}
There is another characterization of forest-configurations, which shows that it can be constructed by an inductive process. First, we introduce the notion of a free gluing.
\begin{definition} \label{def: free gluing}
Consider the point-line configurations $M$ and $N$ with ground sets $[d_{1}]$ and $[d_{2}]$, respectively,  and let $p\in [d_{1}]$ and $q\in [d_{2}]$. The {\em free gluing} of $M$ and $N$ at $p$ and $q$ is defined as the point-line configuration on the ground set 
\[([d_{1}]\setminus \{p\})\amalg ([d_{2}]\setminus \{q\})\cup \{P\},\]
such that the points $p$ and $q$ are identified in $P$. The set of lines of $M \amalg_{p,q} N$ is given by
\[
\mathcal{L}_{M \amalg_{p,q} N}= \{(l\setminus \{p\})\cup \{P\}:l\in \mathcal{L}_{p}\}\cup
\{(l\setminus \{q\})\cup \{P\}:l\in \mathcal{L}_{q}\}
\cup (\mathcal{L}_{M}\setminus \mathcal{L}_{p})\cup (\mathcal{L}_{N}\setminus \mathcal{L}_{q}).
\]
We denote this free gluing by $M \amalg_{p,q} N$.
\end{definition}
The free gluing $M \amalg_{p,q} N$ is in some sense the point-line configuration constructed by identifying the configurations $M$ and $N$ at the points $p,q$ in the most \enquote{free} or independent way, since it preserves all dependencies of $M$ and $N$ without introducing new ones. 

There is a close connection between the notion of a free gluing and the classical concept of {\em amalgamation} of matroids, as for instance in \cite{poljak1984amalgamation,hochstattler2017sticky}. For matroids $M$ and $N$ on ground sets $[d_{1}]$ and $[d_{2}]$, an amalgamation is a matroid whose restrictions to $[d_{1}]$ and $[d_{2}]$ agree with $M$ and $N$, respectively.

We illustrate the definition with two examples. 

\begin{example}
Consider the point-line configurations associated to the $3 \times 3$ grid and to the Fano plane, as depicted in Figure~\ref{fig:combined} (Right) and Figure~\ref{fig:preliminaries grid, fano, cactus} (Left), respectively. Figure~\ref{example free gluing} (Left) shows their free gluing at the points $7$ and $7'$. Similarly, consider the point-line configuration associated to three concurrent lines, shown in Figure~\ref{fig:combined} (Left), and to the quadrilateral set, shown in Figure~\ref{fig:combined} (Center). Their free gluing along the points $3$ and $3'$ is shown in Figure~\ref{example free gluing} (Right).
\end{example}

\begin{figure}[ht]
    \centering
    \includegraphics[width=0.9\linewidth]{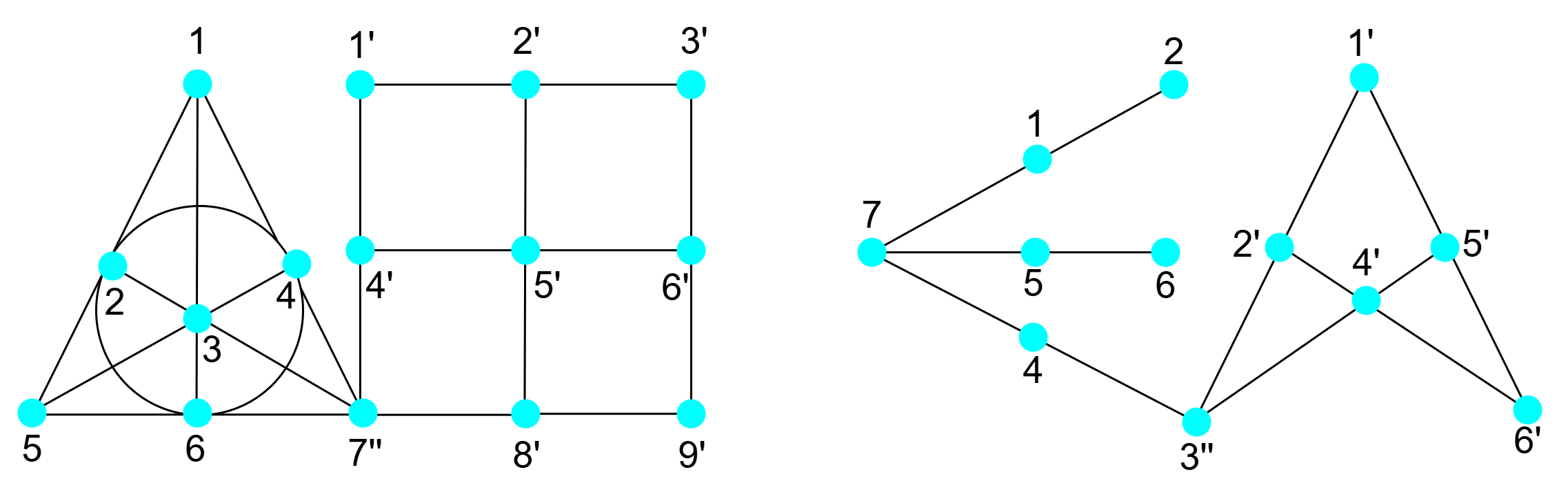}
    \caption{(Left) Free gluing of the point-line configurations associated to the Fano plane and $3 \times 3$ grid; (Right) Free gluing of the point-line configurations associated to the three concurrent lines and the quadrilateral set.}
    \label{example free gluing}
\end{figure}

\begin{definition}
    A {\em connected forest configuration} is a point-line configuration which is formed by inductively freely gluing lines.
\end{definition}
Using Proposition~4.7 of \cite{liwski2024solvablenilpotentmatroidsrealizability}, we can deduce the following equivalent characterization for forest configurations.
\begin{lemma}
    A point-line configuration is a forest configuration if and only if its connected components are connected forest configurations.
\end{lemma}
\begin{example}
    The point-line configuration in Figure \ref{fig:preliminaries grid, fano, cactus} (Right) can be constructed by the free gluing of the lines $\{1,4,7\}$ and $\{1,2,3\}$ at $1$, and then successively gluing the following lines to the point-line configuration of the previous step: the line $\{2,5,8\}$ at 2, the line $\{3,6,9\}$ at 3, $\{6,10, 11\}$ at $6$, $\{6,12,13\}$ at 6 and finally $\{9,14,15\}$ at 9. Thus this point-line configuration is a forest.
\end{example}
One can prove the following result for forest configurations.
\begin{theorem}[Theorem~4.22, \cite{liwski2024solvablenilpotentmatroidsrealizability}]
A forest configuration is nilpotent.
\end{theorem}
Using Theorem~\ref{nil coincide}, one can deduce that a forest configuration is realizable and the corresponding matroid variety is irreducible.
In \cite{clarke2022matroidstratificationshypergraphvarieties}, the irreducible decomposition of a forest configuration is studied. Recall that for a matroid $M$ on the ground set $[d], Q_M$ denotes the set of points of degree at least three, and $M(S)$ denotes the matroid obtained from $M$ by setting the points in $S$ to be loops.
\begin{theorem}[Theorem~5.14 in \cite{clarke2022matroidstratificationshypergraphvarieties}]\label{thm: decomposition circuit varieties forest} 
For a forest configuration $M$ on the ground set $[d]$, the circuit variety $\VCM$ has at most $2^{|Q_M|}$ irreducible components, each one obtained from $M$ by setting a subset of $Q_M$ to be loops:  $$\VCM = V_M \cup_{S \subset Q_M} V_{M(S)}.$$
\end{theorem}

\begin{example} \label{ex: decomposition circuit variety 3 concurrent lines}
    For the point-line configuration with three concurrent lines as in Figure \ref{fig:combined} (Left), one can use Theorem~\ref{thm: decomposition circuit varieties forest} to show that:
    $$\VCM = V_M \cup V_{M(7)}.$$ For the point-line configuration in Figure \ref{fig:preliminaries grid, fano, cactus} (Right), we have $\VCM = V_M \cup V_{M(6)}.$
\end{example}
\section{Methods}
In this section, we discuss some results and we derive concrete methods to find an irreducible decomposition of circuit varieties and to find a generating set for the matroid ideal.
\subsection{Decomposition strategy} \label{subsec: decomposition strategy}
We outline the methods to decompose a circuit variety into irreducible components. We first introduce minimal matroids.

\begin{definition}\label{dependency}
Consider the matroids $N_{1}$ and $N_{2}$ on $[d]$. We write $N_{1}\leq N_{2}$ if $\mathcal{D}(N_{1})\subset \mathcal{D}(N_{2})$. This determines a partial order on matroids, which we will call the \emph{dependency order}. 
\end{definition}

\begin{definition}\label{def min}
The collection of {\em minimal matroids} for a matroid $M$ is defined as: \[\min(M)=\min\{N:N>M\}.\] 
\end{definition}

\begin{proposition}[Proposition~4.1 in \cite{liwskimohammadialgorithmforminimalmatroids}]\label{deco circ}
Let $M$ be a matroid. Then
$$V_{\mathcal{C}(M)}=\bigcup_{N\in \min(M)}V_{\mathcal{C}(N)}\ \cup \ V_{M}.$$
\end{proposition}
To decompose a circuit variety of a point-line configuration $N$ into irreducible components, we can use the following strategy, as in \cite{liwskimohammadialgorithmforminimalmatroids}. 
\newline
\newline
{\bf Decomposition Strategy:}
\begin{itemize}
 \item Determine the minimal matroids over $N$ by applying Algorithm~3.31 from \cite{liwskimohammadialgorithmforminimalmatroids}. For this, we can use the code implemented by Rémi Prébet, available at the following link: \text{bit.ly/49GYXxC}.
 \item Decompose the circuit variety using Proposition~\ref{deco circ}.
 \item Carry out the previous two steps for all circuit varieties in the equation above. Repeat this
 process iteratively for any new circuit variety that shows up, until each circuit variety satisfies
 one of the following conditions:
 \begin{itemize}
 \item Condition 1: If $N'$ is a nilpotent point-line configuration with no points of degree greater than two, then $V_{\mathcal{C}(N')} = V_{N'}$, by Theorem~\ref{nil coincide}.
 \item Condition 2: If $N'$ is a point-line configuration with no points of degree greater than two, such that every proper submatroid of $N'$ is nilpotent, then $V_{\mathcal{C}(N')} = V_{N'} \cup V_{U_{2,d}}$, by Theorem~\ref{nil coincide}. 

 \end{itemize}
 When we encounter a non-simple rank-three matroid $N$, we
instead consider implicitly its reduced version $N_{\text{red}}$, obtained from $N$ by removing loops and identifying double
points. $N_{\text{red}}$ is a point-line configuration. This approach is valid since the minimal matroids of $N$ correspond to the minimal matroids of $N_{\text{red}}$ up to removing double points and loops.
To reconstruct the minimal matroids of $N$ again, we have to reintroduce these loops and double points
again.

After this step, we obtain a decomposition of $\VCM$ as a union of matroid varieties.
 \item For each matroid variety obtained at the end of the previous step, we can check its irreducibility by considering
 the solvable matroids involved, using Theorem~\ref{nil coincide}. 
 \item Remove any redundant components from the decomposition. After completing this steps, we have found the irreducible components of $\VCM$.
\end{itemize}

\subsection{Perturbation strategy} \label{subsection: perturbation strategy}
In the next chapters, we will find a generating set for the matroid ideal of various point-line configurations, up to radical. We will outline the strategy to prove that the matroid ideal is generated by circuit polynomials, lifting polynomials and Grassmann-Cayley polynomials. We first state Hilbert's Nullstellensatz, as in \cite{Hartshorne}.
\begin{theorem}[Hilbert's Nullstellensatz] \label{thm: hilbert nullstellensatz}
    Denote by $\CC[X]$ the polynomial ring in $n$ variables. There is a one-to-one correspondence between the closed subvarieties of $\CC^n$ and the radical ideals in $\CC[X]$: the map sending a closed subvariety $Y$ to $I(Y)$ is bijective with as inverse map the one sending a radical ideal $I$ to $V(I)$.
\end{theorem}
We introduce the notion of a {\em perturbation}, which is a motion that can be made arbitrarily small. This will play a central role in the perturbation strategy. 
\begin{remark}
    $\lVert.\rVert$ denotes the Euclidean norm of a vector in $\CC^n$.
\end{remark}

\begin{definition}
Consider a finite collection of vectors $\gamma = \{\gamma_1,\ldots, \gamma_d\} \in \CC^n$ and a specific property $X$. We say that we can find a {\em perturbation} of $\gamma$ to obtain a new collection of vectors satisfying $X$, if the following holds: for an arbitrary $\epsilon > 0$ and for each $i \in [d]$, we can find a vector $\widetilde{\gamma}_i$ satisfying $\lVert \gamma_i - \widetilde{\gamma}_i \rVert < \epsilon$, such that the property $X$ holds for the new collection $\widetilde{\gamma}.$ Moreover, we can apply a perturbation to a $k$-dimensional subspace $S \subset \CC^n$, by considering a basis $\{v_1, \ldots, v_k\}$ of it and perturbing the vectors of this basis to find a new $k$-dimensional subspace.
\end{definition}
\noindent We remark that this definition of perturbation is also used in \cite{hocsten2004ideals}. 
\newline
Having set this important notion, we can outline the perturbation strategy to prove that \begin{equation} \label{eq: I_M gen}
    I_M = \sqrt{I_{\CCC(M)}+G_M+I_M^{\textup{lift}}}
\end{equation} for a given point-line configuration $M$.

\medskip

 \noindent {\bf Perturbation Strategy:}
\begin{itemize}
\item To prove that Equation~\eqref{eq: I_M gen} holds, we will prove the statement \begin{equation} \label{eq: intersec V_M}
V_M = \VCM \cap V_{G_M} \cap V(I_{M}^{\textup{lift}}),\end{equation}
which is equivalent by Theorem~\ref{thm: hilbert nullstellensatz}. By Remark~\ref{generating set ICM}, Theorem~\ref{thm: gamma i V_M^lift => liftable} and Proposition~\ref{inc G}, the inclusion $\supset$ is obvious in Equation~\eqref{eq: I_M gen}, so we only have to prove the inclusion $\subset$ in Equation~\eqref{eq: intersec V_M}.
\item We decompose the circuit variety of $M$ into its irreducible components, using the decomposition strategy as outlined in Subsection \ref{subsec: decomposition strategy}.
\item Consider an arbitrary $\gamma \in \VCM \cap V_{G_M} \cap V(I_M^{\textup{lift}}).$ Since $\gamma \in \VCM$, $\gamma$ should be in one of its irreducible components.
\item Perturb $\gamma$ to a collection of vectors $\Tilde{\gamma} \in \Gamma_M.$ To achieve this, we may use the properties stated in Theorem \ref{thm: gamma i V_M^lift => liftable} and Proposition \ref{gm toegevoegd}. 
\item Since there is a perturbation of $\gamma$ to the realization space of $M$, we can construct a sequence in $\Gamma_M$ converging to $\gamma$ in the Euclidean topology. Therefore, $\gamma$ is in the Zariski closure of $\Gamma_M$, and since the Zariski topology is coarser than the Euclidean topology, $\gamma$ is in the Zariski closure of $\Gamma_M$ which
is by definition $V_M$.
\end{itemize}
When perturbing in a proof, we implicitly fix an arbitrary $\epsilon > 0$ at the beginning, and then we prove that for a collection of vectors $\gamma = \{\gamma_1, \ldots, \gamma_d\}$, we can find a collection of vectors $\Tilde{\gamma}$ which satisfies $X$ and for which \begin{equation} \label{eq:pert strat hulp}\lVert \gamma_i - \widetilde{\gamma}_i \rVert < \epsilon.\end{equation} It will be often useful that we can take $\epsilon > 0$ small enough such that no additional linear dependencies are created in the new collection $\widetilde{\gamma}$. We will prove this in the following lemma, first we introduce the obvious notion of a distance between two collections of vectors and the induced distance on matrices.
\begin{definition}
    For a collection of $d$ vectors, $\gamma = \{\gamma_1, \ldots, \gamma_d\}$, we define its norm as $$\lVert\gamma\rVert = \sum_{i=1}^n \lVert\gamma_i\rVert.$$
    Similarly, for a $n \times d$-matrix $N$, we define its norm $N$ as $$\lVert N\rVert = \sum_{j=1}^d\sqrt{\sum_{i=1}^nN_{ij}^2}.$$ 
\end{definition}

\begin{lemma}\label{lemma:finite epsilon to not create dependencies}
    For every collection of $d$ vectors $\gamma \subset \CC^n$, there is an $\epsilon > 0$, such that for every collection of vectors $\widetilde{\gamma}$ with $\lVert\gamma - \widetilde{\gamma}\rVert < \epsilon$, $\widetilde{\gamma}$ does not have more dependencies than $\gamma$. 
\end{lemma}
For a proof, we refer to the appendix. So if one only considers the values of $\widetilde{\epsilon}$ in Equation \eqref{eq:pert strat hulp} for which $\widetilde{\epsilon} < \frac{\epsilon}{d},$ there are no additional dependencies created.
\setcounter{MaxMatrixCols}{20}

\numberwithin{theorem}{chapter}

\chapter{Cactus configurations} \label{sec: cactus}
In this chapter, we introduce the notion of cactus point-line configurations. We examine the defining equations of their associated matroid varieties and the decomposition of the corresponding circuit varieties. Note that every forest configuration is a cactus configuration. We examine whether the results of Section \ref{subsec: forest configurations} generalize to cactus configurations. This will be indeed the case. The main results of this chapter are Theorem~\ref{cact irr}, where we prove that each cactus configuration is realizable and its matroid variety is irreducible, Theorem~\ref{thm: main theorem cactus matroid ideal}, in which we find a finite set of generators for the matroid ideal, up to radical, and Theorem~\ref{thm: decomposition of circuit variety of a cactus}, in which we find a possibly redundant yet irreducible decomposition of $\VCM$.
\section{Irreducibility of circuit varieties of cactus configurations}
In this section, we introduce several equivalent characterizations of cactus point-line configurations. Moreover, we prove that their associated matroid varieties are irreducible.

\medskip

\noindent
Recall Definitions \ref{def: free gluing} and \ref{definition n-cycle}, defining the notions of a free gluing and a cycle.
\begin{definition}\label{cactus}
A point-line configuration $M$ is called a \emph{connected cactus} configuration if it can be formed by inductively gluing lines and cycles. Thus we can construct a sequence of point-line configurations $N_{1},\ldots,N_{k}$ such that:
\begin{itemize}
\item $N_{1}$ is either a line or a cycle.
\item For each $i\in [k-1]$, $N_{i+1}$ is the free gluing of $N_{i}$ with a line or a cycle.
\item $N_{k}=M$.
\end{itemize}
More generally, a {\em cactus configuration} is a point-line configuration such that each of its connected components is a connected cactus configuration. Note that \enquote{connected component} and \enquote{component} have a completely different meaning.

\end{definition}
\begin{figure}
    \centering
    \includegraphics[width=0.9\linewidth]{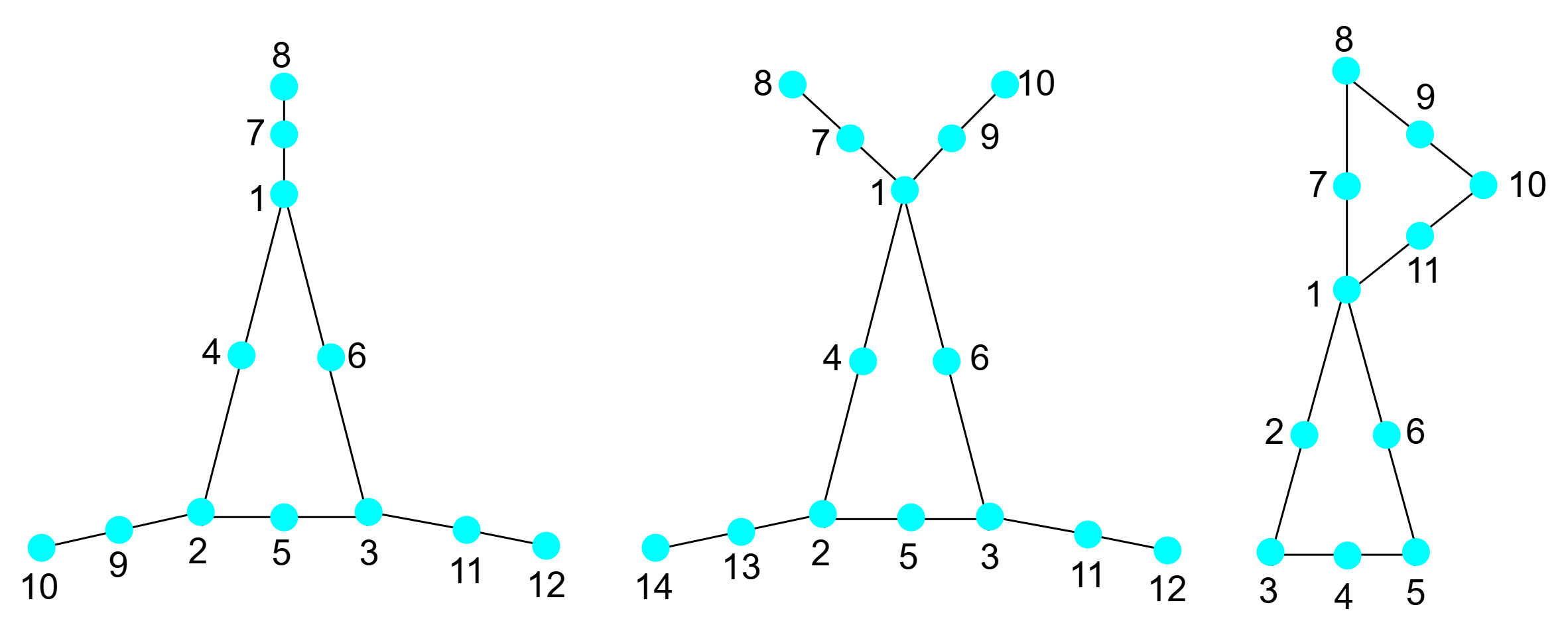}
    \caption{Examples of cactus configurations.}
    \label{fig:2 cactus}
\end{figure}
\begin{example} \label{example cactus decompo}
    One can construct the point-line configuration in Figure \ref{fig:2 cactus} (Left) starting from the cycle with lines $\{1,2,4\}, \{2,3,5\}$ and $\{1,3,6\}$, and successively gluing the lines $\{1,7,8\}$ at 1, $\{2,9,10\}$ at $2$ and $\{3,11,12\}$ at 3. Hence, the resulting point-line configuration is a cactus configuration.
\end{example}

We prove an equivalent characterization of cactus configurations, which is useful to check directly if a configuration is a cactus. First we introduce a lemma and some notation.
\begin{lemma}\label{lemma: 2 cycles points in common}
    Let $M$ be a point-line configuration on $[d]$. Let $K$ and $H$ be cycles, such that $K$ and $H$ are submatroids of $M$ and $K \neq H$. Assume every line of $M$ is in at most one cycle. Then $K$ and $H$ do not intersect in more than one point.
\end{lemma}
\begin{proof}
    Assume $p,q \in [d]$ lie in both $K$ and $H$, with $p \neq q$. Let $l_1, \ldots, l_a$ be the different lines connecting $p$ and $q$ in $K$ and $m_1, \ldots, m_b$ the distinct lines connecting $p$ and $q$ in $H$. By assumption $\mathcal{L}_K \cap \mathcal{L}_H = \emptyset$, so the lines $l_1, \ldots, l_a, m_1 \ldots, m_b$ are all distinct and form a cycle $C$. $C \neq K$ since $m_1 \in \mathcal{L}_C \setminus \mathcal{L}_K$. $l_1 \in \mathcal{L}_K \cap \mathcal{L}_M$, so this is a contradiction.
\end{proof}
\begin{notation} \label{notation: components}
    For a cactus configuration $M$ on $[d]$ with component $K$, $M \setminus K$ denotes the point-line configuration $$M \setminus \{p \in [d]: p \text{ is on at least one line of } \mathcal{L}_K \text{ and } P \text{ is not on a line in } \mathcal{L}_M \setminus \mathcal{L}_K\}.$$
\end{notation}
\begin{proposition} \label{prop: cactus-like = every line in at most 1 n-cycle}
    A point-line configuration $M$ is a cactus configuration if and only if every line $l \in \mathcal{L}_M$ is in at most 1 cycle.
\end{proposition}
\begin{proof}
    First assume $M$ is a connected cactus configuration. Then we can use induction on the number of lines and cycles $n$ which are freely glued.
    \newline
    \newline
    {\bf Basis step}: If $n=0$, then $M$ is a cycle or a line, so clearly every line is in at most 1 cycle.
    \newline
    \newline
    {\bf Induction step}: Assume that the statement holds for point-line configurations with at most $n$ lines and cycles, let $M$ be a cactus configuration with $n+1$ lines and cycles. Then $M = \gluing{M'}{K}{p_1}{p_2}$, where $M$ consists of $n$ lines and cycles and $K$ is a cycle or a line. Let $l \in \mathcal{L}_M$ arbitrary. Since $M$ is the free gluing of $M'$ and $K$, $l \in \mathcal{L}_{M'} \setminus \mathcal{L}_K$ or $l \in \mathcal{L}_K \setminus \mathcal{L}_{M'}$, so $l$ lies in at most 1 cycle by the induction hypothesis.
    Since the implication holds for connected cactus configurations, it is clear that it holds for cactus configurations.
    \newline
    \newline
    To prove the other direction, assume that every line in a point-line configuration $M$ is in at most one cycle. We can define components on $M$, such that every line $l \in \mathcal{L}_M$ is in a unique component:
    \begin{itemize}
        \item Choose $l \in \mathcal{L}_M$.
        \item If $l$ is in a cycle, define that cycle to be a component. Note that by assumption $l$ is in at most one cycle.
        \item If $l$ is not in a cycle, then let the component be $l$ itself.
    \end{itemize}
    We use induction on the number of components $d$.
    \newline
    \newline
    {\bf Basis step:} If there is no component, the statement holds trivially.
    \newline
    \newline
    {\bf Induction step:} Assume the theorem holds for configurations with $d$ components. Let $M$ be a configuration with $d+1$ components. If we assume that there is a component $K$ such that it intersects $M\setminus K$ in at most one point, then we can use the inductive hypothesis on $M \setminus K$ to conclude that it is a cactus. $M$ is the free gluing or union of $M \setminus K$ and $K$, so this shows that $M$ is a cactus. 
    \newline
    \newline
    We end the proof by showing that there is indeed a component $K$ that intersects $M\setminus {K}$ in at most one point. Assume not, then consider an arbitrary $K_1$. We can find another $K_2$ such that $K_1$ and $K_2$ have a point $p_1$ in common. We can find by assumption another $K_3$, such that $K_3$ and $K_2$ intersect in $p_2$ such that $p_1 \neq p_2$. Repeating this procedure, we can find a sequence $p_1 p_2 \ldots$, such that $p_l \neq p_{l+1}$ and $p_l, p_{l+1} \in K_l$, $K_l \neq K_{l+1}$. Moreover, if at the $l+1^{\textup{th}}$ step, there exist a component $K_j$ for $j < l$ which intersects $K_{l}$ in a point different from $p_{l}$, then define $K_{l+1}$ to be that $K_j$ and $p_{l+1}$ to be the intersection point.
    
    Since there are only finitely many components, there is a component which recurs. By taking the smallest such sequence, we can assume without loss of generality that $K_1, \ldots, K_{j+1}$, with $K_1 = K_{j+1}$ and corresponding points $p_1, \ldots, p_{j}$ is a sequence such that $p_i, p_{i+1} \in K_i$, $p_i \neq p_{i+1}$.
    Moreover, due to its construction, two non-consecutive points are not in a common component, except for $p_1$ and $p_j$. 
    \newline
    \newline
    Define for every $k \in \{1, \ldots, j-1\}$ the set of lines $$A_k = \{l :l \text{ is on the path in } K_k \text{ between }p_k \text{ and } p_{k+1}\},$$
    and define 
    $$A_{j} = \{l :l \text{ is on the path in } K_1 \text{ between }p_1 \text{ and } p_{j+1}\}.$$ Let $$A= \cup_{k=1}^j A_k.$$
    Moreover, define the set of points $$B_k = \{p:p \text{ is on a line in } A\}.$$
    
    The points $B$ and lines $A$ form a cycle $D$. 
    One can observe that all the components $K_1, \ldots, K_j$ are cycles: if $K_l$, $l \in \{i, \ldots,j\}$ is a line, then by construction of the components and $D$, its component should be the cycle $D$, which is a contradiction. 
    \newline
    \newline
    By Lemma \ref{lemma: 2 cycles points in common}, $j > 2$. Since $K_2 \neq K_3$, we can assume without loss of generality that $K_2 \neq D$. Since $p_1 \neq p_2$ and $p_1,p_2 \in K_1 \cap D$, this is a contradiction with Lemma \ref{lemma: 2 cycles points in common}.

    \end{proof}

\begin{example}
    The point-line configuration corresponding to the $3 \times 3$ grid in Figure \ref{fig:preliminaries grid, fano, cactus} is not a cactus, since the line $\{2,5,8\}$ forms a cycle together with the lines $\{1,2,3\}, \{1,4,7\}$ and $\{7,8,9\}$, and another cycle with the lines  $\{1,2,3\}, \{3,6,9\}$ and $\{7,8,9\}$.
\end{example}
There is a close connection between our definition of cactus configurations and the concept of cactus graphs in graph theory, a well-studied family of graphs in the literature, see \cite{Husimi1950-ph}. This link is further outlined below. We first introduce the definition of cactus graphs.

\begin{definition}
A connected graph in which any two simple cycles share at most one vertex is called a {\em cactus}. One can equivalently define it as a connected graph where each edge belongs to at most one simple cycle.
\end{definition}

We recall the following well-known property for {\em cactus graphs}.

\begin{lemma}[Theorem 5 in \cite{Janczewski2016}]\label{property}
Every cactus $G$ contains a subgraph $H$, which is either an edge or a cycle, such that at most one vertex of $H$ is adjacent to a vertex in $G\backslash H$.
\end{lemma}
The following lemma provides the connection between cactus graphs and cactus configurations.

\begin{lemma} \label{lemma e}
We can associate a simple graph $G(M)$ to each point-line configuration $M$ in the following way:
\begin{itemize}
\item The vertices of $G(M)$ correspond to the points in $M$ of degree at least two.
\item Two vertices of $G(M)$ are joined by an edge if and only if there is a line in $\mathcal{L}_M$ containing the corresponding points in $M$. 
\end{itemize}
Then $M$ is a cactus configuration if and only if $G(M)$ is a cactus graph.
\end{lemma}

\begin{proof}
We may assume that $M$ is connected, since the argument applies to each connected component separately. 

We first suppose that $G(M)$ is a cactus. We use induction on the number of vertices of $G(M)$ to prove that $M$ is a cactus configuration.

\medskip
{\bf Base case}: If $G(M)$ has only one vertex, then the statement is trivial.

\medskip
{\bf Inductive step:} There is a subgraph $H$, which is either an edge or a cycle, such that at most one of its vertices is adjacent to a vertex in $G(M)\backslash H$, by Lemma~\ref{property}. Denote this vertex by $x$. We then obtain that $M$ is of the form
\[M=M\rq \amalg_{p,q} N,\]
where \[G(M')=H, \qquad  \text{and} \qquad G(N)=(G(M)\setminus H)\cup \{x\},\] 
and the gluing is performed at the point corresponding to $\{x\}$. $M'$ is a line or cycle configuration, because $H$ is an edge or a cycle in $G(M)$. Moreover, since $(G(M)\setminus H)\cup \{x\}$ is a cactus graph, it follows from the inductive hypothesis that $N$ is a cactus configuration. We thus conclude that $M$ is a cactus configuration.

\medskip

Conversely, suppose that $M$ is a cactus configuration. By Definition~\ref{cactus}, $M$ is constructed through a sequence of gluings of lines or cycles. We prove the statement by induction on the number $n$ of such gluings.

\medskip
{\bf Base case}: If $n = 0$, then $M$ consists of a single line or cycle, in which case the statement is clear.

\medskip
{\bf Inductive step}: $M$ is of the form
\[M=M' \amalg_{p,q} N,\]
where $M'$ is a line or a cycle and $N$ is a cactus configuration. Passing to the associated graphs, we see that $G(M')$ and $G(N)$ are subgraphs of $G(M)$ sharing a unique vertex $x$ which corresponds to the point of gluing. Moreover, we can deduce by the inductive hypothesis that $G(M')$ is an edge or a cycle of $G(M)$ and $G(N)$ is a cactus graph. Moreover, because of the gluing condition, there is no edge in $G(M)$ connecting a vertex in  $G(M')\setminus \{x\}$ to a vertex in $G(N)\setminus \{x\}$. Hence, $G(M)$ is also a cactus graph. This completes the proof.
\end{proof}

We will now establish that any cactus configuration is nilpotent.

\begin{proposition} \label{thm: cactus-like config is nilpotent}
Every cactus configuration is nilpotent.
\end{proposition}
\begin{proof}
Consider a cactus configuration $M$. We will assume that $M$ is connected, as the argument applies separately to each connected component of a cactus configuration. By Definition~\ref{cactus}, $M$ is constructed by successively gluing lines or cycles. We establish the result by induction on the number $n$ of such gluings.

\medskip
{\bf Basis step}: If $n = 0$, then $M$ consists of a single line or cycle, in which case the statement is clear. 

\medskip
{\bf Induction step}: Assume that any cactus configuration constructed from at most $n-1$ gluings is nilpotent, and consider a configuration $M$ obtained through $n$ gluings. At the $n^{\text{th}}$-step, $M$ takes the form
\[M=M\rq \amalg_{p,q} N,\]
where $M\rq$ is constructed by $n-1$ successively gluings of lines and cycles, and $N$ is either a line or a cycle. Denote the ground sets of $M\rq$ and $N$ by $[d_{1}]$ and $[d_{2}]$, respectively. The ground set of $M\rq \amalg_{p,q} N$ is then given by
\[([d_{1}]\setminus p)\amalg ([d_{2}]\setminus q)\cup \{P\},\]
where $P$ is the point that identifies both $p$ and $q$.

\medskip
{\bf Case~1.} First assume that $N$ is a line, or equivalently $N=U_{2,d_{2}}$. Then the points in $[d_{1}]\setminus p$ have degree one in $M\rq \amalg_{p,q} N$, which implies that $S_{M}\subset M\rq$ as a submatroid. It follows by induction that $M\rq$ is nilpotent, and since any submatroid of a nilpotent matroid is also nilpotent, we deduce that $S_{M}$ is nilpotent as well. Consequently, the nilpotent chain of $S_{M}$, as defined in Definition~\ref{def: solvable and nilpotent}, eventually reaches the empty set, and the same must hold for $M$.

\medskip
{\bf Case~2.}
Now assume that $N$ is a cycle. If $q$ has degree two in $N$, then for every line $l$ of $N$, we have $\size{S_{M}\cap l}=2$. Since a line should have at least three points, $S_{M}\cap [d_{1}]$ contains no lines. This implies that $S_{S_{M}}\cap [d_{1}]\subset \{P\}$, and thus $S_{S_{M}}$ is a submatroid of $M\rq$. We can conclude that $M$ is nilpotent by applying the same argument as in the previous case. 

If $q$ has degree one in $N$, then there is a line $l$ of $N$ for which $\size{S_{M} \cap l} = 3$, and for all other lines $l \in \mathcal{L}_N$, $\size{S_{M} \cap l} = 2$. Then for each line $l$ of $N$, $\size{S_{S_M} \cap l} \leq 2$. Thus $S_{S_M} \cap \corch{d_1}$ contains no lines, and using notation as in Definition~\ref{def: solvable and nilpotent}, $M_3 \cap [d_1] \subset \{P\}.$ The result again follows by applying the same argument as in the previous case. This completes the proof.
\end{proof}

We now prove the irreducibility of the matroid varieties associated with cactus configurations. 

\begin{theorem}\label{cact irr}
Let $M$ be a cactus configuration. Then $M$ is realizable and the variety $V_{M}$ is irreducible.
\end{theorem}

\begin{proof}
Since $M$ is nilpotent by Proposition~\ref{thm: cactus-like config is nilpotent}, $M$ is solvable as well by Example \ref{example nil-> sol}. The result is then implied by Theorem~\ref{nil coincide}.
\end{proof}

\section{Matroid ideal of cactus configurations}

In this section, we establish a complete set of defining equations for the matroid varieties associated to cactus configurations, or equivalently, a generating set for their matroid ideals up to radical. We first establish a series of lemmas, which are comparable to Lemma~4.23 in \cite{liwski2024solvablenilpotentmatroidsrealizability}. Especially Lemma \ref{Generalization of Lemma 4.21 i} and the first part of Lemma \ref{Generalization of Lemma 4.21 iii} are very similar.

\begin{lemma}\label{Generalization of Lemma 4.21 i}
Let $M$ be point-line configuration on $[d]$ and let $X \subset [d]$ be a subset of points that does not contain a cycle. Then, there exists a point $x \in X$ such that 
\[|\{l \in \mathcal{L}_x: l \cap (X \setminus \{x\}) \neq \emptyset\}| \leq 1.\]
\end{lemma}

\begin{proof}
We prove this lemma by contradiction. Choose an arbitrary point $x_{1} \in X$. We have assumed that there is a line $l_{1} \in \mathcal{L}_{x_{1}}$ containing another point $x_{2} \neq x_{1} \in X$. Invoking the hypothesis again, we can find a distinct line $l_{2} \neq l_{1}$ that contains $x_{2}$ and another point $x_{3} \neq x_{2}$. Iterating this procedure, we construct a sequence of points $x_{1}, x_{2}, \ldots$ and a sequence of lines $l_{1}, l_{2}, \ldots$ such that $x_{k} \neq x_{k+1}$, $l_{k} \neq l_{k+1}$, and $x_{k}, x_{k+1} \in l_{k}$. Since $X$ is finite, some point must appear multiple times in the sequence, forming a cycle. This is not possible, since we assumed that $X$ does not contain a cycle. From this contradiction, we conclude that the statement should hold.
\end{proof}

Recall from Definition~\ref{def: solvable and nilpotent} that $Q_{M}$ denotes the set of points in $M$ with degree at least three.

\begin{lemma}\label{Generalization of Lemma 4.21 ii}
Let $M$ be a cactus configuration, and let $\gamma \in V_{\mathcal{C}(M)}$ such that $\gamma_p \neq 0$ for all $p \in Q_M$. Then, we can perturb $\gamma$ to obtain $\tau \in \Gamma_M$.
\end{lemma}

\begin{proof}
We may assume that $M$ is connected, since the argument can then be applied separately to each connected component. By Definition~\ref{cactus}, $M$ is constructed by successively gluing lines and cycles. We use induction on the number $m$ of such gluings to prove the result.

\medskip
{\bf Basis step:} If $m = 0$, then $M$ is either a line or a cycle. By Proposition~\ref{thm: cactus-like config is nilpotent} and Theorem~\ref{nil coincide} (i), it follows that $\VCM = V_M$, so we can perturb $\gamma$ to obtain $\tau \in \Gamma_M$.

\medskip
{\bf Induction step}: Assume that the statement holds for any cactus configuration constructed by at most $m-1$ gluings, and consider a cactus configuration $M$ obtained through $m$ gluings. At the $n^{\text{th}}$-step, $M$ can be written as
\[M=M\rq \amalg_{p,q} N,\]
where $M\rq$ is formed by $n-1$ gluings, and $N$ is either a line or a cycle. Denote the ground sets of $M\rq$ and $N$ by $[d_{1}]$ and $[d_{2}]$, respectively. The ground set of $M\rq \amalg_{p,q} N$ is then given by
\[([d_{1}]\setminus p)\amalg ([d_{2}]\setminus q)\cup \{P\},\]
where $P$ is the point that identifies both $p$ and $q$.

\medskip
We first show that we may assume $\gamma_{P}\neq 0$. 
\begin{itemize}
\item If $P\in Q_{M}$, then $\gamma_{P}\neq 0$ by assumption. 
\item If $|\mathcal{L}_P| \leq 2$ and $\gamma_{P}=0$, define the set $R =\{\gamma_l:l \in \mathcal{L}_i \textup{ and } \rk(\gamma_l)=2\}$. Since $i$ is on at most two lines, it follows that $|R| \leq 2.$ If $|R| = 0$, then perturb $\gamma$ arbitrarily away from the origin. If $|R|=1$, then perturb $\gamma$ on the unique line in $R$, away from the origin. If $|R|=2$, then perturb $\gamma$ on the intersection of the two distinct lines in $R$. In each case, the resulting configuration lies in $\Gamma_M$, as desired.
\end{itemize}
{\bf Case~1:} Suppose that $N$ is a line. By applying the induction hypothesis, we can perturb the vectors $\{\gamma_r:r\in [d_{1}]\}$ to a collection $\{\tau_r:r\in [d_{1}]\}$ in $\Gamma_{M\rq}$. 
Since $\gamma_{P}\neq 0$, we can further extend $\tau$ to the points in $[d_{2}] \setminus \{q\}$ by applying a perturbation to the vectors $\{\gamma_{r}:r\in [d_{2}]\setminus q\}$, ensuring that these vectors remain on the same line as $\tau_P$ and on no other line. This results in a collection of vectors $\tau \in \Gamma_{M}$, as desired.

\medskip
{\bf Case~2:} Suppose that $N$ is a cycle. Invoking the induction hypothesis, we can perturb the vectors $\{\gamma_r:r\in [d_{1}]\}$ to obtain a collection $\{\tau_r:r\in [d_{1}]\}$ in $\Gamma_{M\rq}$.
By Theorem \ref{nil coincide} (i), we can perturb the vectors $\{\gamma_r:r\in [d_{2}]\}$ to obtain a collection $\{\beta_r:r\in [d_{2}]\}$ in $\Gamma_{N}$.  
Furthermore, since $\gamma_p \neq 0$, we can apply a small rotation from the origin to the vectors in $\beta$ such that $\beta_{q}=\tau_{p}$.
Next, by rotating $\beta$ infinitesimally around $\tau_p$, and noting that each $\beta_r \neq 0$ since $\beta \in \Gamma_N$, we can obtain a collection $\beta$ such that the vectors $\beta_r$ do not lie on any lines of the configuration other than those associated with $N$, for all $r \in [d_{2}]$. 
This results in a collection of vectors in $\Gamma_{M}$, as desired.
\end{proof}
\begin{lemma}\label{Generalization of Lemma 4.21 iii}
Let $M$ be a cactus configuration, such that the points of $Q_M$ do not contain a cycle. If $\gamma \in V_{\mathcal{C}(M)} \cap V(G_{M})$, then there exists $\tau \in V_{\mathcal{C}(M)}$ such that 
\begin{itemize}
 \item If $p \in Q_{M}$, then $\tau_{p} \neq 0$.
\end{itemize}
In particular, $\tau$ can be chosen as a perturbation of $\gamma$.
\end{lemma}
\begin{proof}
Denote $R_{p}$ the set $\{l\in \mathcal{L}_{p}:\dim(\gamma_{l})=2\}$ for each point $p \in [d]$ and consider two cases:
\newline

{\bf Case 1:} 
There is a point $p \in Q_{M}$ with $\gamma_{p} = 0$ and lines $l_{1}, l_{2} \in R_{p}$ such that $\gamma_{l_{1}} \neq \gamma_{l_{2}}$. 
\newline
Then we define the collection of vectors $\tau$ indexed by $p \in [d]$ as follows: $\tau_p = \gamma_{l_1} \cap \gamma_{l_2}$ and $\tau_q = \gamma_q$ if $q \neq p$. 
We will prove that $\tau \in \VCM \cap V(G_M)$. 
\newline
First we show $\tau \in \VCM$. For this we need to show that for all $l \in \mathcal{L}$, $\dim(l) \leq 2$. We have to consider several cases:
\begin{itemize}
    \item $l \in R_p$: because $\gamma \in V(G_M)$, $\tau_p$ lies on all lines $\{\gamma_{l} : l \in R_{p}\}$. 
    \item $l \notin R_p$, $p \in l$: this implies that $\dim(\gamma_l) \leq 1$, so $\dim(\tau_l) \leq 2$.
    \item $l \notin R_p$, $p \notin l$: then $\tau_l = \gamma_l$, so $\dim(\tau_l) = \dim(\gamma_l) \leq 2$.
\end{itemize}
Now we will prove that $\tau \in V(G_M)$. Choose an arbitrary polynomial $r \in G_M$. We have to show that $r(\tau) = 0$. Let $p_1, p_2 \in l_1$ and $p_3, p_4 \in l_2$. Using Construction \ref{GM}, the following equality holds: \begin{equation}
\tau_{p}=\gamma_{l_{1}}\wedge \gamma_{l_{2}}=[{\gamma_{p_{1}},\gamma_{p_{2}},\gamma_{p_{3}}}]\gamma_{p_{4}}- [{\gamma_{p_{1}},\gamma_{p_{2}},\gamma_{p_{4}}}]\gamma_{p_{3}}.
\end{equation}
Define $\overline{r}$ as the polynomial obtained from $r$ by replacing the variable $p$ with:
\[
[{p_{1},p_{2},p_{3}}]p_{4}-[{p_{1},p_{2},p_{4}}]p_{3}.
\]
Then $\overline{r} \in G_{M}$ by definition of $G_M$. By assumption, $\gamma \in V(G_{M})$, so we have $\overline{r}(\gamma) = 0$. Thus $r(\tau) = \overline{r}(\gamma) = 0$. 
After carrying out this procedure iteratively, we obtain a collection such that for every point $p \in Q_M$ with $\gamma_p = 0$, there are no lines $l_1, l_2 \in R_p$ for which $\gamma_{l_1} \neq \gamma_{l_2}$.
\medskip

{\bf Case 2}: For every $p \in Q_{M}$ with $\gamma_{p} = 0$, if $l_{1}, l_{2}$ are lines in $R_{p}$, then $\gamma_{l_{1}} = \gamma_{l_{2}}$.
\newline
Denote $X \subset Q_M$ the set of points $p$ for which $\gamma_p = 0$. We use the principle of induction on $|X|$ to prove that we can perturb the collection of vectors $\gamma$ to $\tau$ such that $\tau_q = \gamma_q$ if $q \notin X$, and $\tau \in \VCM$. 
\newline
\newline
{\bf Basis step}: If $X = 0$, then the statement holds trivially.
\newline
\newline
{\bf Induction step}: Suppose that the statement is true for point-line configurations with $|X| \leq n$ and assume that $|X| = n+1$. Since $X$ does not contain a cycle, we can use Lemma \ref{Generalization of Lemma 4.21 i} to find a point $p \in X$, such that $|\{l \in \mathcal{L}_p: l \cap (X \setminus \{p\}) \neq \emptyset\}| \leq 1$.
Denote the line in this set as $l_1$. If this set is empty, we can adapt the proof by choosing an arbitrary line in $\CC^3$ through $p$ instead of $\tau_{l_1}$ which we will construct. There is a line $s$ in $\CC^3$, such that $\gamma_l = s$ for every line $l$ in $R_p$. We can apply the inductive hypothesis to the submatroid $M\setminus \{p\}$ to find vectors $\tau$ such that $\tau_q = \gamma_q$ if $q \notin X$, $\tau \in V_{\mathcal{C}(M \setminus \{p\})}$ and for all $q \in M \setminus \{p\}$: $\tau_q \neq 0$. Now extend $\tau$ by defining $\tau_p$ as the non-zero vector in $\tau_{l_1} \cap s$. We need to show that $\tau \in \VCM$. As in Case 1, it suffices to show that for all $l \in \mathcal{L}$, $\dim(l) \leq 2$. We distinguish some cases:
\begin{itemize}
\item $l\notin \mathcal{L}_{p}$: $\dim(\tau_{l})\leq 2$ because $\restr{\tau}{M\backslash \{p\}}\in V_{\mathcal{C}(M\backslash \{p\})}$ by the induction hypothesis.
\item $l=l_{1}$: $\dim(\tau_{l})\leq 2$ because $\tau_{p}\in \tau_{l_{1}}$.
\item $l\in R_{p}, l \neq l_1$: $\dim(\tau_{l})\leq 2$ because $\tau_l = s$ and $\tau_{p} \in s$.
\item $l \in \mathcal{L}_{p} \setminus (R_{p} \cup {l_{1}})$: since $l$ is not in $R_{p}$, it holds that $\dim(\gamma_{l}) \leq 1$. Now $p$ is the only point on $l$ in $X$, so $\dim(\tau_{l}) \leq 2$ in this case.
\end{itemize}
The resulting collection has no loops at 
points in $Q_{M}$, so this completes the proof.
\end{proof}

We now establish the main result of this section.

\begin{theorem} \label{thm: main theorem cactus matroid ideal}
   Let $M$ be a cactus configuration, such that the points of $Q_M$ do not contain a cycle. Then $I_M = \sqrt{I_{\mathcal{C}(M)} +G_M}$. 
\end{theorem}
\begin{proof}
Denote by $I$ the ideal on the right-hand side, which is clearly contained in $I_M$, by Remark~\ref{generating set ICM}, Theorem~\ref{thm: gamma i V_M^lift => liftable} and Proposition~\ref{inc G}. To show the reverse inclusion, we will show $V(I) \subset V_M$, using Theorem~\ref{thm: hilbert nullstellensatz}. Let $\gamma$ be a collection of vectors such that $\gamma\in \VCM \cap V(G_M)$. By Lemma~\ref{Generalization of Lemma 4.21 iii}, there is a perturbation of the collection of vectors $\gamma$ to a collection of vectors $\tau \in \VCM$ such that $\tau_p \neq 0$ for all $p \in Q_M$. Applying Lemma~\ref{Generalization of Lemma 4.21 ii}, we obtain a collection of vectors $\tau'\in \Gamma_{M}$ which is a perturbation of $\tau$. This proves that $\gamma \in V_{M}$, as desired.
\end{proof}
\begin{example} 
Consider the cactus configuration as in Figure \ref{fig:2 cactus} (Center). One can calculate that $$G_M = \{[1,2,3],[3,4,5],[1,5,6],[1,7,8],[8,9,10],[10,11, 1],$$ $$[2,3,5][6,7,8]-[2,3,6][5,7,8],[2,3,5][6,10,11]-[2,3,6][5,10,11], $$ $$[2,3,7][8,10,11]-[2,3,8][7,10,11],[5,6,7][8,10,11]-[5,6,8][7,10,11]\}.$$ Using Theorem~\ref{thm: main theorem cactus matroid ideal}, it is clear that $I_M = \sqrt{G_M}.$
\end{example}
\begin{example} \label{example f} We will show that we cannot omit the assumption that the points of $Q_M$ do not contain a cycle to establish the result in Theorem~\ref{thm: main theorem cactus matroid ideal}. 
Let $M$ be the matroid depicted in Figure~\ref{fig:2 cactus} (Center). 
    Let 
\[
\gamma = \left(
\begin{array}{cccccccccccccc}
1 & 0 & 0 & 1 & 1 & 1 & 1 & 1 & 1 & 1 & 1 & 0 & 0 & 0 \\
0 & 1 & 0 & 1 & 2 & 4 & 5 & 6 & 7 & 9 & 8 & 0 & 0 & 0 \\
0 & 0 & 1 & 1 & 3 & 8 & 7 & 10 & 12 & 21 & 17 & 0 & 0 & 0
\end{array}
\right),
\]

where the columns correspond to the points $$\{4,5,6, 7,8,9,10,11,12,13,14,1,2,3\},$$ respectively. Note that $\gamma \in \VCM \cap V(G_M)$. We will show that $\gamma\not \in V_{M}$, which implies that the matroid ideal is not generated by the circuit ideal and the Grassmann-Cayley ideal, up to radical.

Suppose from the contrary that $\gamma \in V_{M}=\closure{\Gamma_{M}}$. As proven in Section \ref{appendix:proof of lemma pappus}, the Zariski and Euclidean closure coincide for $\Gamma_{M}$, implying that for every $\epsilon$ we can select an element $\gamma_{\epsilon}\in \Gamma_{M}$ with $\lvert \gamma-\gamma_{\epsilon}\rvert <\epsilon$. After a perturbation, we may assume $\gamma_{\epsilon}$ takes the form

{\footnotesize
\[
\begin{pmatrix}
1 & 0 & 0 & 1 & 1 & 1 & 1 & 1 & 1 & 1 & 1 & x_1 & x_2 & x_3 \\
0 & 1 & 0 & 1 & 2 + \epsilon_1 & 4 + \epsilon_3 & 5 + \epsilon_5 & 6 + \epsilon_7 & 7 + \epsilon_9 & 9 + \epsilon_{11} & 8 + \epsilon_{13} & y_1 & y_2 & y_3 \\
0 & 0 & 1 & 1 & 3 + \epsilon_2 & 8 + \epsilon_4 & 7 + \epsilon_6 & 10 + \epsilon_8 & 12 + \epsilon_{10} & 21 + \epsilon_{12} & 17 + \epsilon_{14} & z_1 & z_2 & z_3
\end{pmatrix},
\]
}
where the columns again correspond to $\{4,5,6, 7,8,9,10,11,12,13,14,1,2,3\}$ and $\epsilon(i)\in \CC$ satisfies $\lvert \epsilon(i)\rvert <\epsilon$ for all $i\in [14]$.

Since $\{1,9,10\}$ and $\{1,7,8\}$ are lines in $M$, the point ${\gamma_{\epsilon}}_{1}$ lies in the intersection of the lines spanned by $\{{\gamma_{\epsilon}}_{9},{\gamma_{\epsilon}}_{10}\}$ and $\{{\gamma_{\epsilon}}_{7},{\gamma_{\epsilon}}_{8}\}$. Therefore, 

\[{\gamma_{\epsilon}}_{1}={\gamma_{\epsilon}}_{9}{\gamma_{\epsilon}}_{10}\wedge {\gamma_{\epsilon}}_{7}{\gamma_{\epsilon}}_{8}.\]

Taking the limit $\epsilon \to 0$, we obtain

\[\lim_{\epsilon \to 0} {\gamma_{\epsilon}}_{1}=\gamma_{9}\gamma_{10}\wedge \gamma_7 \gamma_8=(1,13/3,23/3)^{\top}. \]

Similarly, $\{3,1,6\}$ and $\{3,11,12\}$ are lines of $M$, so ${\gamma_{\epsilon}}_{3}$ lies in the intersection of the lines spanned by $\{{\gamma_{\epsilon}}_{1},{\gamma_{\epsilon}}_{6}\}$ and $\{{\gamma_{\epsilon}}_{11},{\gamma_{\epsilon}}_{12}\}$. It follows that

\[{\gamma_{\epsilon}}_{3}={\gamma_{\epsilon}}_{1}{\gamma_{\epsilon}}_{6}\wedge {\gamma_{\epsilon}}_{11}{\gamma_{\epsilon}}_{12}.\]

Taking the limit $\epsilon \to 0$, we obtain
\[\lim_{\epsilon \to 0} {\gamma_{\epsilon}}_{3}=\lim_{\epsilon \to 0}{\gamma_{\epsilon}}_{1}{\gamma_{\epsilon}}_{6}\wedge {\gamma_{\epsilon}}_{11}{\gamma_{\epsilon}}_{12}=(1,13/3,23/3)^{\top}\gamma_{6}\wedge \gamma_{11} \gamma_{12}=(1,13/3,20/3)^{\top}. \]
Similarly, $\{2,3,5\}$ and $\{2,13,14\}$ are lines of $M$. Therefore, we have that the triples of points $\{{\gamma_{\epsilon}}_{2},{\gamma_{\epsilon}}_{3},{\gamma_{\epsilon}}_{5}\}$ and $\{{\gamma_{\epsilon}}_{2},{\gamma_{\epsilon}}_{13},{\gamma_{\epsilon}}_{14}\}$ are collinear. Consequently, we have 

\[{\gamma_{\epsilon}}_{2}={\gamma_{\epsilon}}_{3}{\gamma_{\epsilon}}_{5}\wedge {\gamma_{\epsilon}}_{13}{\gamma_{\epsilon}}_{14}.\]

If we again take the limit $\epsilon \to 0$, we obtain
\[\lim_{\epsilon \to 0} {\gamma_{\epsilon}}_{2}=\lim_{\epsilon \to 0}{\gamma_{\epsilon}}_{3}{\gamma_{\epsilon}}_{5}\wedge {\gamma_{\epsilon}}_{13}{\gamma_{\epsilon}}_{14}=(1,13/3,20/3)^{\top}\gamma_{5}\wedge \gamma_{13} \gamma_{14}=(1,65/12,80/12)^{\top}. \]

Finally, since $\{1,2,4\}$ is a line of $M$, and $\gamma_{\epsilon}\in \Gamma_{M}$, it must hold that \[\det({\gamma_{\epsilon}}_{1},{\gamma_{\epsilon}}_{2},{\gamma_{\epsilon}}_{4})=0.\]
However, evaluating the limits give
\begin{equation*}
\begin{aligned}
\lim_{\epsilon \to 0}\det({\gamma_{\epsilon}}_{1},{\gamma_{\epsilon}}_{2},{\gamma_{\epsilon}}_{4})&=\det(\lim_{\epsilon \to 0}{\gamma_{\epsilon}}_{1},\lim_{\epsilon \to 0}{\gamma_{\epsilon}}_{2},\gamma_{4})\\
&=\det((1,13/3,23/3)^{\top},(1,65/12,80/12)^{\top},\gamma_{4})\\
&=-455\neq 0.
\end{aligned}
\end{equation*}

This is a contradiction. We conclude that $\gamma \notin V_{M}$, completing the proof.
\end{example}
\section{\texorpdfstring{Bound on the number of irreducible components of $\VCM$}{Bound on the number of irreducible components of VCM}}
We now provide a bound on the number of irreducible components of the circuit variety of a cactus configuration.
\begin{theorem} \label{thm: decomposition of circuit variety of a cactus}
    Let $M$ be a cactus configuration on the ground set $[d]$. Then $\VCM$ has at most $2^{\size{Q_{M}}}$ irreducible components, each obtained from $M$ by setting a subset of $Q_{M}$ to be loops. 
\end{theorem}

\begin{proof}
Let $[d]$ be the ground set of $M$. For each subset $J\subset Q_{M}$, let $M(J)$ denote the matroid obtained from $M$ by setting the points of $J$ to be loops.

\medskip
{\bf Claim.} The following equality holds:
\begin{equation}\label{vcm}
V_{\mathcal{C}(M)}=\bigcup_{J\subset Q_{M}}V_{M(J)}.
\end{equation}
\medskip
The inclusion $\supset$ immediately holds since the matroids $M(J)$ have more dependencies than $M$. Therefore, $\Gamma_{M(J)} \subseteq \VCM$, thus $V_{M(J)} \subseteq \VCM$, since $V_{M(J)}$ is the smallest Zariski closed subset containing $\Gamma_{M(J)}.$ To prove the reverse inclusion, consider an arbitrary $\gamma \in \VCM$. We have to find a matroid variety on the right-hand side of Equation~\eqref{vcm} containing $\gamma$.

Let $X$ be the set
\[\{x \in Q_M : \gamma_x = 0\}.\]
Note that the submatroid $M\setminus X$ is still a cactus configuration. Since $\gamma_{p}\neq 0$ for all $p\in [d]\setminus X$, it follows that $\gamma_{p}\neq 0$ for all $p\in Q_{M\setminus X}$. Therefore, the conditions of Lemma \ref{Generalization of Lemma 4.21 ii} are satisfied, which implies that $\restr{\gamma}{[d]\setminus X}\in V_{M\setminus X}$. So we can conclude that $\gamma\in V_{M(X)}$. 

\medskip
By Theorem~\ref{cact irr}, each variety on the right-hand side of Equation~\eqref{vcm} is irreducible. It follows that $V_{\mathcal{C}(M)}$ has at most $2^{\size{Q_{M}}}$ irreducible components.
\end{proof}
\begin{example} \label{ex: cactus} 
Let us consider the matroid $M$ as in Figure \ref{fig:2 cactus} (Left). Using Theorem~\ref{thm: decomposition of circuit variety of a cactus}, one can immediately conclude that \begin{equation} \label{eq: example of cactus 3 lines 1 cycle} \VCM = V_M \cup_{i=1}^3 V_{M(i)} \cup_{1 \leq i< j \leq 3} V_{M(i,j)} \cup V_{M(1,2,3)}.\end{equation} Notice that if we would use the decomposition strategy as explained in Subsection \ref{subsec: decomposition strategy}, the first step of the strategy already results in 132 minimal matroids over $M$. Therefore, the decomposition strategy is not suitable in this case.
\newline
\newline
We cannot conclude immediately from Theorem~\ref{thm: main theorem cactus matroid ideal} that $I_M = \sqrt{I_{\CCC(M)} + G_M}$, since the points $1,2,3 \in Q_M$ form a cycle. However, we can still prove this result using Equation~\eqref{eq: example of cactus 3 lines 1 cycle}. As outlined in the perturbation strategy, see Subsection \ref{subsection: perturbation strategy}, we prove an equivalent result for the corresponding varieties: we show that $V_M = \VCM \cap V(G_M).$ The inclusion $\subset$ is obvious from Remark \ref{generating set ICM} and Proposition~\ref{inc G}. For the other inclusion, using Equation~\eqref{eq: example of cactus 3 lines 1 cycle}, we have to prove that $$(V_M \cap V(G_M)) \cup_{i=1}^3 (V_{M(i)} \cap V(G_M)) \cup_{i=1}^3\cup_{j\neq i=1}^3 (V_{M(i,j)} \cap V(G_M)) $$ $$\cup 
(V_{M(1,2,3)} \cap V(G_M))\subset V_M.$$
For all components in the union except $(V_{M(1,2,3)} \cap V(G_M))$, this is obvious from the proof of Lemma \ref{Generalization of Lemma 4.21 iii} and Lemma \ref{Generalization of Lemma 4.21 ii}. For the final component $V_{M(1,2,3)} \cap V(G_M)$, one can notice that the point-line configuration $M(1,2,3)$ consists of 3 loops, 7 points and no lines. Let $\gamma$ be an arbitrary collection of vectors in $\CC^3$ where the linearly dependent sets are exactly the sets with at most four elements. Assume that we can perturb $\gamma$ such that there exist vectors $\gamma_1, \gamma_2,\gamma_3 \in \CC^3 \setminus \{0\}$ for which:
\begin{equation} \label{eq: system for example}
\left\{
    \begin{array}{lll}
        \left[1,7,8\right] &=& 0  \\
        \left[1,6,3\right] &=& 0 \\
        \left[3,11,12\right] &=& 0 \\
        \left[2,5,3\right] &=& 0 \\
        \left[2,9,10\right] &=& 0 \\
        \left[2,4,1\right] &=& 0.
    \end{array}
\right.
\end{equation}
Then we can use Lemma \ref{Generalization of Lemma 4.21 ii} to perturb this collection of vectors to $\gamma \in V_M$, which shows that $V_{M(1,2,3)} \cap V(G_M) \subset V_M.$ This concludes the proof.

We will end this example by proving the assumption. After a linear transformation, we can assume that our original $\gamma$ can be written as follows:
$$
\begin{pmatrix}
    1 & 0 & 0 & 1 & 0 & 0 & 0 & 1 & 1 & 1 & 1 & 1 & 1 \\
    0 & 1 & 0 & 1 & 0 & 0 & 0 & y_{7} & y_{8} & y_{9} & y_{10} & y_{11} & y_{12} \\
    0 & 0 & 1 & 1 & 0 & 0 & 0 & z_{7}& z_{8} & z_{9} & z_{10} & z_{11} & z_{12} 
\end{pmatrix},
$$
where the columns respectively correspond to $\{4,5,6,7,1,2,3,8,9,10,11,12\}$. Since $1$ is on a line with $7$ and $8$, define $$\gamma_1 = \begin{pmatrix}
    1 +l_1 \\
    1 + l_1 y_{8} \\
   1 + l_1 z_{8}
\end{pmatrix}$$ for a variable $l_1$. Define $\gamma_3 = \gamma_1 \gamma_6 \wedge \gamma_{11} \gamma_{12}.$ This is an expression for $\gamma_3$ which has degree one in the only variable $l_1$. In the same way, one can define $\gamma_2 = \gamma_3 \gamma_5 \wedge \gamma_{9} \gamma_{10}$. This is again an expression of degree one in $l_1$. For the explicit expressions of $\gamma_3$ and $\gamma_2$, please refer to the appendix. If $[1,2,4]= 0$, then $\gamma$ satisfies all conditions of Equation \eqref{eq: system for example}. This final condition results in a quadratic polynomial in $l_1$:
$Al_1^2+Bl_1+C$, where the explicit expressions for $A,B$ and $C$ can be found in Appendix \ref{appendix: example}. This quadratic equation has at least one non-zero solution if the discriminant $B^2-4AC \neq 0$ and $A \neq 0$. These are two complements of hypersurfaces in $\CC^{10}$, with variables $y_{8}, z_{8},y_{9}, z_{9},y_{10}, z_{10},y_{11}, z_{11}, y_{12}, z_{12}.$ The intersection of these complements is dense in $\CC^{10}$, as in \cite{Hartshorne}. This implies that there is a perturbation of the given $y_{8}, \ldots, z_{12}$ for which $B^2-4AC \neq 0$ and $A \neq 0$. Thus we have found a perturbation of $\gamma$ such that we can define additional vectors $\gamma_1, \gamma_2, \gamma_3$ satisfying Equation \eqref{eq: system for example}.
\end{example}
    \numberwithin{theorem}{chapter}

\chapter{Matroid ideals of Pascal and Pappus configurations} \label{sec: pascal & pappus}
In this chapter, we determine a complete set of defining equations, up to radical, for the matroid ideals associated to the Pascal and Pappus configurations. First we present a section which bundles some useful lemmas, including a necessary criterion for liftability. In the next section, we find an irreducible decomposition for the circuit variety of the Pascal configuration as in Theorem~\ref{proposition: decomposition pascal configuration}. Moreover, we find a finite generating set for the matroid ideal of the Pascal configuration, as in Theorem~\ref{pascal generators}. In Theorem~\ref{thm:Pappus}, we find a finite set of generators for the matroid ideal of the Pappus configuration.
\section{Some technical results for the decomposition of circuit varieties}
We start by listing some lemmas which will be often used later on in the thesis. Then we prove a result providing a sufficient criterion for liftability.
\begin{lemma} \label{lemma: third config loop matroid variety}
    Let $M$ be a matroid, and $i$ a point of $[d]$ such that $i$ has degree at most 2 in $M$. Then $V_{M(i)} \subset V_M$.
\end{lemma}

\begin{proof}
    We show that any collection $\gamma \in \Gamma_{M(i)}$ can be perturbed to a collection in $\Gamma_{M}$, thereby establishing that $\gamma\in V_{M}$. Let $\gamma \in \VCM$ and consider the set $R =\{\gamma_l:l \in \mathcal{L}_i \textup{ and } \rk(\gamma_l)=2\}$. Since $i$ has degree at most two, it follows that $|R| \leq 2.$ If $|R| = 0$, then perturb $\gamma$ arbitrarily away from the origin. If $|R|=1$, then perturb $\gamma$ on the unique line in $R$, away from the origin. If $|R|=2$, then perturb $\gamma$ on the intersection of the two distinct lines in $R$. In each case, the resulting configuration lies in $\Gamma_M$, as desired.
\end{proof}
\begin{lemma} \label{lemma: if gamma_i gamma_j wedge gamma_k gamma_l =0}
Let $M$ be a point-line configuration on the ground set $[d]$, 5 a point with $\mathcal{L}_5 = \{\{1,2,5\}, \{3,4,5\}\}$, $\gamma \in \VCM$ and $v \in \CC^3$. If $\gamma_1 \gamma_2 \wedge \gamma_3 \gamma_4= 0$ and $\gamma_5 = v$, then there is a perturbation of $\gamma_5$ to a vector different from $v$, such that the resulting configuration is still in $\VCM$.
  \end{lemma}  
\begin{proof}
Since $\gamma_1 \gamma_2 \wedge \gamma_3 \gamma_4= 0$, there is at most one line of rank two in the set $\{\gamma_{12}, \gamma_{34}\}$. If there is exactly one line in that set, one can perturb $\gamma_5$ away from $v$ on that line. If there is no line in that set, one can perturb $\gamma_5$ to an arbitrary point. In both cases, the subspaces $\gamma_{\{1,2,5\}}$ and $\gamma_{\{3,4,5\}}$ have rank at most two, thus $\gamma$ is still in $\VCM$.
\end{proof}
\begin{lemma} \label{lemma: meet of 2 extensors arbitrary close}
                 Let $v_1, v_2, v_3, v_4 \in \mathbb{C}^3$, such that $v_1 v_2 \wedge v_3 v_4 \neq 0$, and consider $\widetilde{v}_1, \widetilde{v}_2, \widetilde{v}_3, \widetilde{v}_4$ perturbations of respectively $v_1, v_2, v_3, v_4$. Fix $v_5$ on $\langle v_1, v_2\rangle \cap \langle v_3, v_4\rangle$. Then we can choose $\widetilde{v_5}$ on $\langle \widetilde{v}_1, \widetilde{v}_2\rangle \cap \langle \widetilde{v}_3, \widetilde{v}_4\rangle$, such that it is arbitrarily close to $v_5$.
\end{lemma}
\begin{proof}
Since $v_1 v_2 \wedge v_3 v_4 \neq 0$, we can use the formula $\langle v_1, v_2\rangle \cap \langle v_3, v_4\rangle = \overline{v_1v_2 \wedge v_3 v_4}$ by Lemma \ref{klj}. Since the determinant is a continuous function, $\widetilde{v}_1\widetilde{v}_2 \wedge \widetilde{v}_3 \widetilde{v}_4$ is arbitrary close to $v_1v_2 \wedge v_3 v_4$, so there exists a vector $\widetilde{v}_5$ on $\overline{\widetilde{v}_1\widetilde{v}_2 \wedge \widetilde{v}_3 \widetilde{v}_4}$ arbitrarily close to $v_5$.
\end{proof}
\subsection{A sufficient criterion for liftability} \label{subsec: a sufficient criterion for liftability}

In this subsection, our aim is to prove Proposition~\ref{prop: nilpotent add point liftable}, which provides a sufficient condition for an affirmative answer to the following question:

\begin{question}\normalfont\label{question}
Consider a point-line configuration $M$ on $[d]$ and a collection of points $\{\gamma_{1},\ldots,\gamma_{d-1}\}\subset \mathbb{\CC}^{3}$ lying on a common line $l\subset \mathbb{\CC}^{3}$ which is liftable in a non-degenerate way from a point $q\in \mathbb{\CC}^{3}$ to a collection in $V_{\CCC(M \setminus \{d\})}$. 
Is there a point $\gamma_{d}\in l$ such that the extended collection $\{\gamma_{1},\ldots,\gamma_{d}\}$ is non-degenerately liftable from $q$ to a collection in $V_{\CCC(M)}$?  
\end{question}
In the following chapters, we will apply Proposition~\ref{prop: nilpotent add point liftable} to compute the matroid ideals associated to the Pascal and Pappus configurations. We begin by introducing a definition.
\begin{definition}\label{definition dim lift}
    Let $M$ be a point-line configuration on $[d]$ and let $\gamma$ be a collection of vectors in $V_{\CCC(M)}$ and $q\in \CC^{3}$ a vector in general position with respect to $\gamma$. In accordance with Definition~\ref{def liftable}, we denote by $\Lift_{M,q}(\gamma)\subset \CC^{d}$ the subspace consisting of all vectors $(z_{1},\ldots,z_{d})$ such that the collection of vectors given by $\widetilde{\gamma}_{i}=\gamma_{i}+z_{i}q$ belongs to $V_{\CCC(M)}$. We also introduce the notation
    \[\dimm{q}{\gamma}(M)=\dim(\Lift_{M,q}(\gamma)),\]
for the dimension of the space of liftings of $\gamma$. Using Lemma~\ref{lemma: motivation liftability}, we can deduce that this number is equal to
\[\dim(\ker(\liftmat)),\]
where $\liftmat$ is the matrix defined in Definition~\ref{matrix lift}.
For any submatroid $M|S$, we denote the number $\dim(\Lift_{M|S,q}(\restr{\gamma}{S}))$ by $\dimm{q}{\gamma}(M|S)$. 
\end{definition}
Moreover, we introduce the following notation.
\begin{notation}
Recall the notation $S_M$ from Definition~\ref{def: solvable and nilpotent}. For a line $l \in \mathcal{L}_M$, we introduce the notation $S_l = S_M \cap l$. Furthermore, we use the notation $M \setminus l$ as defined in Notation~\ref{notation: components}.
\end{notation}
\begin{proposition}\label{proposition: dim(ker(liftmat))}
    Let $M$ be a point-line configuration with ground set $[d]$, $\gamma \in \CC^{3}$ be a collection of vectors on a line $l$, without loops or parallel points, and $q$ a point not on that line in general position with respect to $\gamma$. In this case \begin{equation}\label{eq:dim_q(gamma) is cte}\dimm{q}{\gamma}(M) = \dimm{q}{\gamma}(M|S_M) + \sum_{l \in \mathcal{L}_M}(2 - \rk(S_l))+ |\{p \in [d]: \mathcal{L}_p = \emptyset\}|.\end{equation}
\end{proposition}

\begin{proof}
    By Lemma \ref{lemma: motivation liftability}, $\widetilde{\gamma} = \{\gamma_i+\lambda_iq: i \in [d]\} \in \VCM$ if and only $\lambda = (\lambda_{1}, \ldots, \lambda_{d})^\top \in \Lift_{M,q}(\gamma)$. 
    Assume that for a vector $\gamma$, we have already found a lifting $\widetilde{\gamma}|_{[b]}$ for a submatroid
    $[b] \subset [d]$. The following remarks hold.
    \newline
    \newline
    {\bf Remark 1}: If $l$ is a line containing at least two points $p_1$ and $p_2$ of ${[b]}$ and a point $b+1 \in [d] \setminus [b]$, then $\widetilde{\gamma}|_{b+1} = \widetilde{\gamma}_l \cap \langle \gamma_{b+1}, q\rangle,$ since $\gamma_p$ is not a loop and all $\gamma_i$ are different. So in that case, $\lambda_{b+1}$ is completely determined by $\lambda_{p_1}$ and $\lambda_{p_2}$, thus $$\dimm{q}{\gamma}(M|[b+1]) = \dimm{q}{\gamma}(M|[b]).$$ 
    {\bf Remark 2}:
    If $p$ is not on any line of $M|[b]
    $, then there is no constraint on $\widetilde{\gamma}_{b+1}$ except for being on the line between $q$ and $\gamma_p$, so $$(\lambda_1, \ldots, \lambda_{b}, 0), (\lambda_1, \ldots, \lambda_{b}, 1) \in \Lift_{M|[b+1],q}(\gamma).$$
    Therefore: 
    $$\dimm{q}{\gamma}(M|[b+1]) = \dimm{q}{\gamma}(M|[b])+1.$$
    
    {\bf Remark 3:} The points in $M \setminus S_M$, are exactly the points of degree zero and one. Therefore, we have to consider precisely the points of degree zero and the points on one line for which $\rk(S_l)=0$, $\rk(S_l) = 1$ or $\rk(S_l) = 2$. 
    \newline
    \newline
    First consider the points in $M$ of degree one on a line with $|S_l| = 0$. Fix such a line $l$. The first point $p_1$ is on no line of $M|S_M$, and the second point is on no line of $M|(S_M \cup \{p_1\})$ as well.
    For the other points, they are on a line of $M|(S_M \cup \{p_1,p_2\})$, so by Remark 1 and 2, $$\dimm{\gamma}{q} (M |(S_M \cup \{p \in M : p \in l\})) = \dimm{\gamma}{q}(M|S_M) +2.$$
    We can repeat this argument for each line in $\mathcal{L}_M$ with $|S_l| =0$. 
    \newline
    \newline
    Now assume $$N = S_M \cup \{p \in [d] :p \text{ is on exactly one line } l \text{ for which } |S_l| = 0\}.$$
    We will consider the points in $M$ of degree one on a line with $|S_l| = 1$. Let $l \in \mathcal{L}_M$ be such a line. Since a line contains at least three points, $l$ should contain at least one point of degree one. By adding one point $p_{1} \in [d] \setminus S_M$ of $l$:
    $$\dimm{q}{\gamma}(M|(N \cup \{p_{1}\})) =\dimm{q}{\gamma}(M|N) +1,$$
    since $p_1$ is on no line of $M|N$. Adding the other points of the line does not increase the dimension since they are on a line in $M|(N \cup \{p_1\}).$
    Analogously, every point of degree 1 on a line with $|S_l| \geq 2$ does not increase the dimension of the kernel of the liftability matrix.
    \newline
    \newline
    Finally, every point $p$ which is not on any line increases the dimension of the liftability matrix with one. 
    Thus using remark 3, we can conclude that \begin{align*}
\dim_q^\gamma(M) 
&= \dim_q^\gamma(M|S_M) 
+ \sum_{l \in \mathcal{L}_M,\, |S_l| = 0} 2 
+ \sum_{l \in \mathcal{L}_M,\, |S_l| = 1} 1 \\
&\quad + \sum_{l \in \mathcal{L}_M,\, |S_l| = 0} 0 
+ |\{p \in [d] : \mathcal{L}_p = \emptyset\}| \\
&= \dim_q^\gamma(M|S_M) 
+ \sum_{l \in \mathcal{L}_M} (2 - \rk(S_l)) 
+ |\{p \in [d] : \mathcal{L}_p = \emptyset\}|.
\end{align*}
    \end{proof}
\begin{theorem} \label{thm independent}
    If $M$ is a nilpotent point-line configuration, $\gamma \in \CC^3$ with no coinciding points or loops, then $\dimm{q}{\gamma}(M)$ is independent of $q$ and $\gamma$.
\end{theorem}
\begin{proof}
    We use induction on the length $l_n(M)$ of the nilpotency chain $$M_0 = M \supseteq M_1 = S_M \supseteq M_2 \supseteq \ldots \supseteq M_k = \emptyset.$$
    {\bf Basis step}: If $l_n(M) = 0$, then $M = \emptyset$. In that case, $\dimm{q}{\gamma}(M) = 0$ is clearly independent of $q$ and $\gamma$.
    \newline
    \newline
    {\bf Induction step}: Assume that the theorem holds for matroids with $l_n(M) = n$, and let $l_n(M) = n+1$. By Proposition~\ref{proposition: dim(ker(liftmat))},
    $$\dimm{q}{\gamma}(M) = \dimm{q}{\gamma}(M|S_M) + \sum_{l \in \mathcal{L}}(2 - \rk(S_l))+ |\{p \in [d]: \mathcal{L}_p = \emptyset\}|.$$
    $\dimm{q}{\gamma}(M|S_M)$ is independent of $q$ and $\gamma$ because of the induction hypothesis, since $l_n(S_M) = n$. The other terms are independent of $q$ and $\gamma$ as well. So by the principle of induction, $\dim_q^\gamma(M)$ is independent of $q$ and $\gamma$.
    \end{proof}
If $\dim_\gamma^q(M)$ is independent of $q$ and $\gamma$, we denote it as $\dim(M).$
If $M$ is nilpotent, we can restate Proposition~\ref{proposition: dim(ker(liftmat))} in a more elegant way, as explained in the following remark.

\begin{remark} \label{remark: emiliano}
    Consider a nilpotent point-line configuration $M$ on the ground set $[d]$ and an ordering $w=(p_{1},\ldots,p_{d})$ of its elements. Denote by $M_{w,i}$ the restriction $M\mid \{p_{1},\ldots,p_{i}\}$, and define $w_{i}$ to be the degree of $p_{i}$ in $M_{w,i}$.
If $M$ is nilpotent, there is an ordering $w$ satisfying the following inequality:
\begin{equation}\label{max}\text{max}\{w_{i}:i\in [d]\}\leq 1.\end{equation}
One can construct for instance the following ordering: let $$M = M_0 \supset M_1 \ldots \supset M^l = \emptyset$$ be the chain as in Definition~\ref{def: solvable and nilpotent}, such that $M^{l-1} \neq \emptyset$. For a point-line configuration $N$, denote the number of points in the ground set by $|N|$. Define $p_1, \ldots, p_{|M^{l-1}|}$ to be all the distinct points in $M^{l-1}$, with an arbitrary ordering. Clearly, Equation~\eqref{max} holds in this case. Analogously, define $p_{|M^{l-1}|+1}, \ldots, p_{|M^{l-2}|}$ to be all the points in $M^{l-2} \setminus M^{l-1}$. Continuing this process, one finds an ordering $w = (p_1, \ldots, p_d)$ satisfying Equation~\eqref{max}.
If $M$ is nilpotent, we can adapt the proof of Proposition~\ref{proposition: dim(ker(liftmat))} in the following way. Using remark 1 and 2 of that proof for $\{p_1, \ldots, p_b\} \subseteq [d]$ iteratively for $b$ ranging from 1 to $d$, it is clear that \begin{equation} \dim(\liftmat) = \#\{i\in [d]:w_i = 0\} = \sum_{i \in [d]} d-w_i. \end{equation} From this reformulation, one can immediately deduce that $\dim(\liftmat)$ is independent of $q$ and $\gamma$ if $M$ is a nilpotent point-line configuration.
\end{remark}

\begin{proposition} \label{prop: degree < 3 add point liftable}
   Let $M$ be a point-line configuration on $[d]$. Then,  $\deg(d)\leq 2$ is a sufficient condition for an affirmative answer to Question~\ref{question}. 
\end{proposition}
\begin{proof}
Denote $\widetilde{\tau}$ a lifting of $\tau$ in $V_{\CCC(M \setminus \{d\})}$ and $R = \{\widetilde{\tau}_l:l \in \mathcal{L}_p, \rk(\widetilde{\tau}_l) = 2\}.$

\medskip

If $|R|=0$, let ${\tau}_d$ be an arbitrary point in $\CC^3$. The resulting collection is liftable from $q$ since any lifting that preserves the points $\{\tau_{1},\ldots,\tau_{d-1}\}$ and lifts the point $\tau_{d}$ represents a lifting to a collection in $V_{\CCC(M)}$.

\medskip

If $|R|=1$, then again any arbitrary point $\tau_{d}\in l$ results in a collection which is liftable from $q$. To see this, consider the lifting of the points $\{\tau_{1},\ldots,\tau_{d-1}\}$ from $q$ to a collection $\widetilde{\tau}$ in $V_{\CCC(M\setminus d)}$. If $\widetilde{\tau}_l \in R$, lift the point $\tau_{d}$ from $q$ to a point in $\widetilde{\tau}_{l}$, resulting in a non-degenerate collection in $V_{\CCC(M)}$. 

\medskip

If $|R|=2$, let $l_1$ and $l_2$ be the lines containing $p$. Then define $\widetilde{\tau}_d$ to be on $\widetilde{\tau}_{l_1} \cap \widetilde{\tau}_{l_2}$, arbitrarily close to the origin. In each case, define $\tau_d$ to be the projection of $\widetilde{\tau}_d$ from $q$ on $l$. The extended $\tau$ can be lifted non-degenerately from $q$ to $\widetilde{\tau}$.

\end{proof}
Before continuing to the final property of this section, we first recall Proposition~I.7.1 of \cite{Hartshorne}: 
\begin{proposition} \label{prop: hartsorne}
    Consider an algebraically closed field $k$, and two varieties $Y,Z$ of dimensions $r,s$ respectively. Then every irreducible component $W$ of $Y\cap Z$ has dimension $\geq r+s-n.$
\end{proposition}

\begin{proposition}\label{prop: nilpotent add point liftable}
Let $M$ be a point-line configuration on $[d]$. Then, within the framework of Question~\ref{question}, the answer is affirmative under the following assumptions:
\begin{enumerate}
\item $M \setminus \{d\}$ is nilpotent and satisfies $\dim(M \setminus \{d\}) \geq 1+\deg(d).$
\item $\gamma$ has no coinciding points.
\end{enumerate} 
\end{proposition}
\begin{proof}
We first prove that we can find a lifting of $\tau$ to $\widetilde{\tau} \in \VCM$ such that the lines of $M$ in $\{\widetilde{\tau}_m:m\in \mathcal{L}_d\}$ are concurrent. More precisely, we will prove that there exists a vector $z =(z_1, \ldots, z_d) \in \CC^d$ for which $$\widetilde{\tau}^z =(\tau_{1}+z_{1}q,\ldots,\tau_{d-1}+z_{d-1}q)$$ satisfies the conditions of Question~\ref{question}. 
\newline
First we consider the concurrency condition. Pick two points for each $m_i \in \mathcal{L}_d$, denote them as $p_{1m_i}$ and $p_{2m_i}$. The lines through these points should be concurrent, which is translated by Lemma \ref{klj} into the following conditions.
    \begin{equation}\label{eq:concurrency conditions}
\left\{
    \begin{array}{ll}
        t_3^z = (\widetilde{\tau}^z_{p_{1m_1}} \vee \widetilde{\tau}^z_{p_{2m_1}}) \wedge (\widetilde{\tau}^z_{p_{1m_2}} \vee \widetilde{\tau}^z_{p_{2m_2}}) \vee (\widetilde{\tau}^z_{p_{1m_3}} \vee \widetilde{\tau}^z_{p_{2m_3}}) = 0 \\
        \vdots \\
        t_{\deg(d)}^{z} = (\widetilde{\tau}^z_{p_{1m_1}} \vee \widetilde{\tau}^z_{p_{2m_1}}) \wedge (\widetilde{\tau}^z_{p_{1m_2}} \vee \widetilde{\tau}^z_{p_{2m_2}}) \vee (\widetilde{\tau}^z_{p_{1m_{\deg(d)}}} \vee \widetilde{\tau}^z_{p_{2m_{\deg(d)}}}) = 0 \\
    \end{array}
\right.
\end{equation}
So define $$Z:=\{z:\widetilde{\tau}^z \text{ satisfies the conditions of Equation } \ref{eq:concurrency conditions}\}.$$
Since $Z$ is determined by the vanishing locus of $\deg(d)-2$ polynomials, we can use Proposition~\ref{prop: hartsorne} to conclude that:
$$\dim(Z) \geq d+2-\deg(d).$$
Define $$T :=\{z:\widetilde{\tau}^z \in V_{\CCC(M \setminus \{d\}}\}.$$
By definition, $T$ is precisely the linear space $\Lift_{M\setminus \{d\},q}(\gamma)$ introduced in Definition~\ref{definition dim lift}, and thus
\[\dim(T)=\dim_{q}^{\gamma}(M\setminus \{d\})=\dim(M\setminus \{d\}),\]
where the final equality follows from Theorem~\ref{thm independent}. Note that $T \cap Z$ is non-empty, as it contains the zero vector. Considering the dimension of $T \cap Z$, we can again use Proposition~\ref{prop: hartsorne} to conclude that:
$$\dim(T \cap Z) \geq 2-\deg(d)+\dim(M \setminus \{d\}) \geq 3.$$
Since $T \cap Z$ has dimension at least three, we can find a $d$-tuple $z \in T \cap Z$ for which $\widetilde{\tau}^z_d$ has rank three. 

Thus we can find a non-degenerate $\widetilde{\tau}$, which is a lifting of $\tau$ from $q$, such that $\widetilde{\tau}_{m_1}, \widetilde{\tau}_{m_2}, \ldots, \ldots, \widetilde{\tau}_{m_k}$ are concurrent in a point $x \in \CC^3$.
\newline
\newline
Now project $x$ from $q$ onto the line $l$ and define the resulting point to be $\tau_d$. This yields the desired result.
\end{proof}
\section{Pascal Configuration}
In this section, we study the Pascal configuration, the point-line configuration illustrated in Figure~\ref{fig:pascal 1} (Left). Throughout this section, we denote this point-line configuration by $M$. Specifically, we determine the irreducible decomposition of its circuit variety and we find a generating set for its matroid variety, up to radical.

We recall the definition of a hypergraph automorphism, as in \cite{bretto2013hypergraph}, a key concept in the following chapters.
\begin{definition}
A {\em hypergraph automorphism} of a hypergraph $H$ with vertices $V$ and hyperedges $E$ is a permutation of the vertex set, which induces a bijection from the set of hyperedges.
\end{definition}
Note that, when proving a result, it suffices to prove it for one representative from each equivalence class under the equivalence relation induced by the automorphisms.
\begin{remark}
For each pair of points $i,j \in [6]$, one can find a hypergraph automorphism of $M$ which maps $i$ to $j$. A similar result holds for each pair of points $i,j \in \{7,8,9\}$. We say that the points in $[6]$ are symmetric to each other, and that the points in $\{7,8,9\}$ are symmetric to each other. 
\end{remark}
We also introduce the following notation for the remainder of the thesis:
\begin{notation}\label{notation pi}
Consider a point-line configuration $N$ on $[d]$.
\begin{itemize}
\item $\pi_N^i$ denotes the point-line configuration with one line containing all the points in $[d]$, except for $i$. Moreover, we identify the points that share a common line with $i$ in $N$. See Figure~\ref{fig:pascal 1} (Center) for an illustration of $\pi_{M}^{1}$.
\item Recall that $N(i)$ denotes the point-line configuration obtained by making $i$ a loop. More generally, for a subset $\{i_{1},\ldots,i_{k}\}\subset [d]$, $N(i_{1},i_{2}, \ldots, i_{k})$ denotes the point-line configuration obtained by turning all points in $\{i_{1},i_{2}, \ldots, i_{k}\}$ into loops. Note that this point-line configuration is isomorphic to $M\setminus \{i_{1},\ldots,i_{k}\}$, if we remove the loops.
    \end{itemize}
\end{notation}
Moreover, the following remark will be used throughout the remainder of the paper.

\begin{remark}
    When perturbing, we will always assume that the perturbation is small enough such that no new dependencies are introduced, as in Lemma \ref{lemma:finite epsilon to not create dependencies}.
\end{remark}
\subsection{Irreducible decomposition of $V_{\mathcal{C}(M)}$}

In this subsection, we determine an irreducible decomposition for the circuit variety of $M$. Rather than using the decomposition strategy in this section, we develop a shorter alternative.
We start this subsection with a remark.
\begin{remark} \label{remark reduced}
For matroids $M$ of rank at most three with loops or parallel elements, the results from Theorem \ref{nil coincide} still apply.
Consider $M_{\text{red}}$, obtained by removing loops and identifying double points, then Theorem \ref{nil coincide} holds for $M_{\text{red}}$. Putting the double points and loops back, if $M$ is nilpotent without points of degree at least three, then $$\VCM = V_M.$$ Moreover, if every proper submatroid of $M$ is nilpotent and $M$ has no points of degree at least three, then $$\VCM = V_M \cup V_{U_{2,d}}',$$ where $V_{U_{2,d}}'$ is the matroid obtained from the uniform matroid with $d$ points such that it has the same double points and loops as the original matroid. 
\end{remark}
We will introduce Lemma \ref{props}, which will play a fundamental role in our discussion. Therefore, we first prove another lemma, for which the proof is deferred to Section \ref{appendix:proof of lemma pappus}.
\begin{lemma} \label{auxiliary lemma}
    Denote the point-line configuration $M \setminus \{1\}$ as $N$. This is depicted in Figure~\ref{fig:pascal 2} (Left). Let $A$ be the point-line configuration on the ground set $\{2,\ldots,9\}$, in which $2=3=9$ and the points $\{2,4,7,8\}$ form a line, as in Figure~\ref{fig:pascal 1} (Right). Then $V_A \subseteq V_N$.
\end{lemma}
\begin{lemma}\label{props}
    Denote the point-line configuration $M \setminus \{1\}$ as $N$, shown in Figure~\ref{fig:pascal 2} (Left). Then, we have \[V_{\mathcal{C}(N)} = V_{N} \cup V_{N(9)}.\]
\end{lemma}
\begin{proof}
Let $\gamma \in \VCN$. We will prove that $\gamma \in V_N$ or $\gamma \in V_{N(9)}$. We first suppose that $\gamma_{9}=0$. Since $N\setminus \{9\}$, shown in Figure \ref{fig:pascal 2} (Center), is nilpotent and all points of degree have degree at most two, Theorem~\ref{nil coincide} (i) implies that $V_{\CCC(N\setminus \{9\})}=V_{N\setminus \{9\}}$, hence $V_{\CCC(N(9))}=V_{N(9)}$. From this we conclude that $\gamma\in V_{N(9)}$.

Now suppose that $\gamma_9 \neq 0$. Then by Lemma~\ref{lemma: third config loop matroid variety}, we may assume that $\gamma$ has no loops.

\medskip
{\bf Case~1.} Suppose that $\gamma_2 \neq \gamma_9$ or $\gamma_3\neq \gamma_{9}$. We may assume without loss of generality that $\gamma_2 \neq \gamma_9$. By Theorem~\ref{nil coincide} (i), we can apply a perturbation to the vectors $\{\gamma_{2},\gamma_3, \gamma_4, \gamma_{5}, \gamma_7,\gamma_{8}, \gamma_9\}$ to obtain a collection $\{\widetilde{\gamma}_{2},\widetilde{\gamma}_{3},\widetilde{\gamma}_{4},\widetilde{\gamma}_{5},\widetilde{\gamma}_{7}, \widetilde{\gamma}_{8}, \widetilde{\gamma}_{9}\}\in \Gamma_{N \setminus \{6\}}$.

Since $\gamma_{2}\neq \gamma_{9}$, the line $\widetilde{\gamma}_{2,9}$ is a perturbation of the line $\gamma_{2,9}$. Since the only line in $N$ containing the point $6$ is $\{2,9,6\}$, we can then perturb $\gamma_6$ so that it lies on the line $\widetilde{\gamma}_{2,9}$. This yields a collection of vectors in $\Gamma_{N}$, showing that $\gamma\in V_{N}$.

\medskip
{\bf Case~2.} Suppose that $\gamma_2= \gamma_3= \gamma_9$. Since the subsets $\{2,7,4\},\{7,8,9\}$ and $\{3,4,8\}$ are dependent in $N$ and $\gamma\in V_{\CCC(N)}$, it follows that the sets of vectors $\{\gamma_2,\gamma_7,\gamma_4\},\{\gamma_7,\gamma_8,\gamma_9\}$ and $\{\gamma_3,\gamma_4,\gamma_8\}$ are linearly dependent. Moreover, as $\gamma_2= \gamma_3= \gamma_9$, we obtain that the following sets of vectors are linearly dependent as well:
\[\{\gamma_2,\gamma_7,\gamma_4\}, \{\gamma_2,\gamma_7,\gamma_8\}, \{\gamma_2,\gamma_4,\gamma_8\}.\]
From this we conclude that the points $\{\gamma_2,\gamma_4,\gamma_7,\gamma_8\}$ are on a line.

Let $A$ be the point-line configuration as in Lemma \ref{auxiliary lemma}. Given that $\rank \{\gamma_2,\gamma_4,\gamma_7,\gamma_8\}\leq 2$, it follows that $\gamma \in V_{\CCC(A)}$.

If we identify the double point $\{2,3,9\}$ as a simple point in $A$, then the resulting matroid is nilpotent and all of its points have degree at most two. By Theorem~\ref{nil coincide} (i), we conclude that $V_{\CCC(A)} = V_A$. By Lemma \ref{auxiliary lemma}, $V_A \subseteq V_N$, so $\gamma \in V_N$, which finishes the proof.
\end{proof}

We now present the irreducible decomposition of the circuit variety associated with the Pascal configuration.

\begin{theorem} \label{proposition: decomposition pascal configuration}
The circuit variety $V_{\mathcal{C}(M)}$ admits the following irreducible decomposition:
\begin{equation}\label{equ}V_{\mathcal{C}(M)} = V_M \cup V_{U_{2,9}} \cup_{i=7}^9 V_{M(i)}.\end{equation}
\end{theorem}

\begin{proof}
The inclusion $\supseteq$ is clear, since the matroids on the right-hand side have more dependencies
than M. To prove the reverse inclusion, we must show that any $\gamma\in V_{\mathcal{C}(M)}$ lies in the variety on the right-hand side of~\eqref{equ}. 

First, suppose that there exists a point $i \in \{7,8,9\}$ satisfying $\gamma_i = 0$. Without loss of generality, we may assume that $i=7$. $M(7)$ is shown in Figure~\ref{fig:pascal 2} (Right). Since the matroid $M\setminus \{7\}$ is nilpotent and has no points of degree at least three, it follows that $V_{\mathcal{C}(M(7))} = V_{M(7)}$. Consequently, we have $\gamma \in V_{M(7)}$, as desired. Hence, we can assume that $\gamma_{7},\gamma_{8},\gamma_{9}\neq 0$. Using Lemma \ref{lemma: third config loop matroid variety}, we can then assume that every vector of $\gamma$ is non-zero. 

\medskip
{\bf Case~1.} Suppose that there exists a point $p\in \{1,\ldots,6\}$ such that $\gamma_{l_1} \wedge \gamma_{l_2} \neq 0$, where $\{l_1, l_2\} =\mathcal{L}_p$. We may assume without loss of generality that $p=1$.

Since $M\setminus \{1\}$ is the matroid from Lemma~\ref{props}, applying this result yields $\restr{\gamma}{\{2,\ldots,9\}}\in V_{M\setminus \{1\}}$. Consequently, we can perturb the vectors $\{\gamma_{2},\ldots,\gamma_{9}\}$ to obtain a collection $\{\widetilde{\gamma}_{2},\ldots,\widetilde{\gamma}_{9}\}\in \Gamma_{M\setminus \{1\}}$. We can extend $\widetilde{\gamma}$ by defining 
\[\widetilde{\gamma}_1 = \widetilde{\gamma}_5 \widetilde{\gamma}_7 \wedge \widetilde{\gamma}_6 \widetilde{\gamma}_8,\]
then $\widetilde{\gamma}\in \Gamma_{M}$. It follows from Lemma \ref{lemma: meet of 2 extensors arbitrary close} that $\widetilde{\gamma}$ is arbitrarily close to $\gamma$, thus $\gamma\in V_{M}$.

\medskip
{\bf Case~2.} Suppose that for each point $p \in \{1, \ldots, 6\}$, we have $\gamma_{l_1} \wedge \gamma_{l_2} = 0$, where $\{l_1, l_2\} =\mathcal{L}_p$. 

By applying Lemma \ref{lemma: if gamma_i gamma_j wedge gamma_k gamma_l =0}, we can find a perturbation of $\gamma$ in $V_M$, such that for each $i \in [6], j \in [9]$ with $i \neq j$, we can assume that $\gamma_i \neq \gamma_j$. Observe that for every line $l$ of $M$ distinct from $\{7,8,9\}$, the set $\{\gamma_{p}:p\in l\}$ contains three different points and no loops, thus $\rank(\gamma_{l})=2$. By our assumption in Case~2, these six lines must coincide, that is,
\[\gamma_{\{1,6,8\}} = \gamma_{\{1,5,7\}} = \gamma_{\{2,6,9\}} = \gamma_{\{2,7,4\}} = \gamma_{\{3,4,8\}}=\gamma_{\{3,5,9\}}.\]

Thus the vectors $\{\gamma_{1},\ldots,\gamma_{9}\}$ all lie in a common two-dimensional subspace. We conclude that $\gamma\in V_{\mathcal{C}(U_{2,9})}=V_{U_{2,9}}$.

\medskip
The irreducibility of each component follows from Theorem~\ref{nil coincide}. The irredundancy of $V_{U_{2,9}}$ follows from a similar argument as in \cite{clarke2024liftablepointlineconfigurationsdefining}. 
The irredundancy of $V_{M(i)}$ is proven in Section \ref{appendix:proof of lemma pappus}.
    \end{proof}
    \begin{figure}
    \centering
    \includegraphics[width=0.9\linewidth]{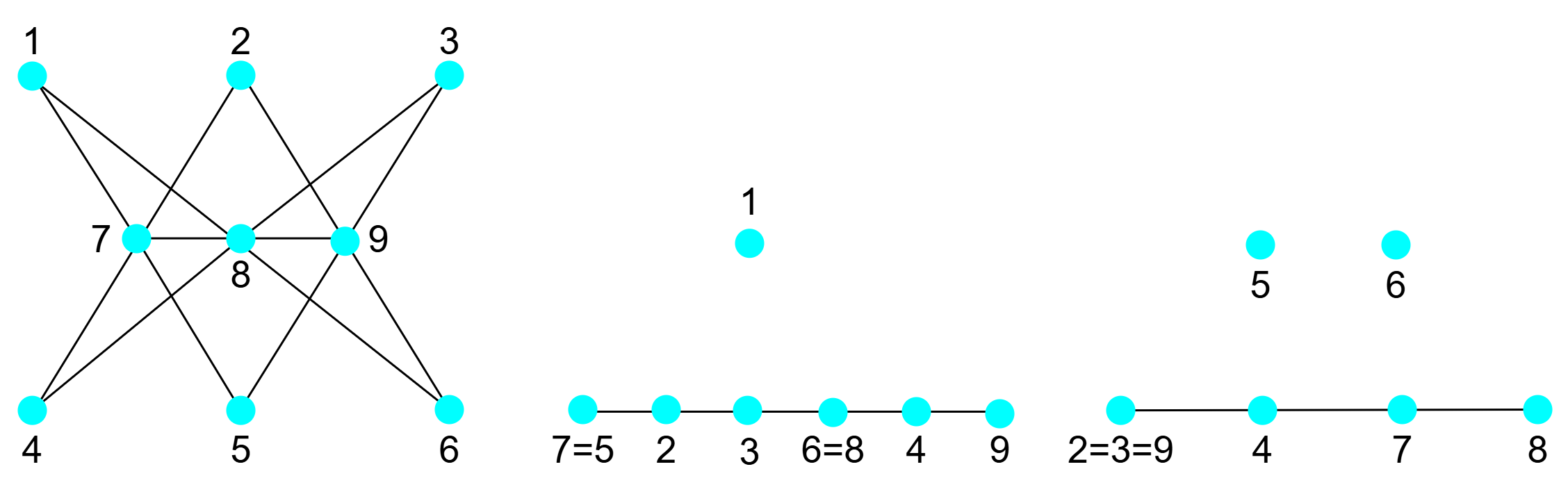}
    \caption{(Left) Pascal configuration; (Center) $\pi_M^1$; (Right) $A$.}
    \label{fig:pascal 1}
\end{figure}
\begin{figure}
    \centering
    \includegraphics[width=0.9\linewidth]{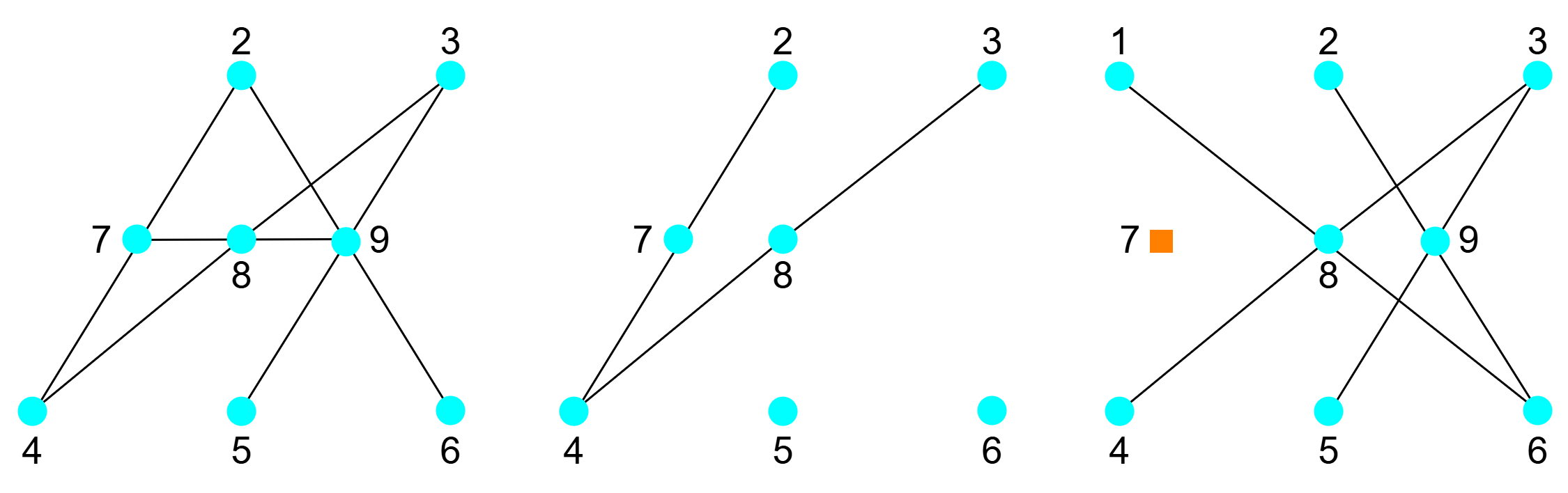}
    \caption{(Left) $M \setminus \{1\}$; (Center) $M \setminus \{1,9\}$; (Right) $M(7)$.}
    \label{fig:pascal 2}
\end{figure}

\subsection{Defining equations of $V_{M}$}

In this section, we provide a complete set of defining equations for the matroid variety of the Pascal configuration, up to radical. First, we state a lemma, for which the proof is deferred to Section \ref{appendix:proof of lemma pappus}.
\begin{lemma} \label{lemma pascal V_B subset V_M}
    Consider the matroid $B$ for which all points coincide. Then $V_B \subseteq V_M$.
\end{lemma}
\begin{theorem} \label{pascal generators}
    Let $M$ be the matroid associated to the Pascal configuration. Then $$I_{M}=\sqrt{I_{\mathcal{C}(M)}+G_{M}+I_{M}^{\textup{lift}}}.$$
\end{theorem}

\begin{proof}
To prove the equality of the statement, we will show the following equivalent equality of varieties
\begin{equation}\label{lifti}
V_M = V_{\mathcal{C}(M)} \cap V(G_M) \cap V_{M}^{\textup{lift}}.
\end{equation}

The inclusion $\subset$ in~\eqref{lifti} follows directly from Remark~\ref{generating set ICM}, Theorem~\ref{thm: gamma i V_M^lift => liftable} and Proposition~\ref{inc G}. To establish the reverse inclusion, consider $$\gamma \in V(G_M) \cap V(I_M^{\text{lift}}) \cap \VCM.$$ We must show that $\gamma\in V_{M}$. Since $\gamma\in V_{\mathcal{C}(M)}$, it follows from Theorem~\ref{proposition: decomposition pascal configuration} that
\begin{equation}\label{gamma in}
\gamma\in V_M \cup V_{U_{2,9}} \cup_{i=7}^9 V_{M(i)}.
\end{equation}

\medskip
{\bf Case~1.} Suppose that $\rk(\gamma)=3$.

\medskip
{\bf Case~1.1.} Suppose that $\gamma_{7},\gamma_{8},\gamma_{9}\neq 0$.  Then it immediately follows from Theorem~\ref{proposition: decomposition pascal configuration} that $\gamma\in V_{M}$, as desired. 

\medskip
{\bf Case~1.2} Suppose that there is a loop $\gamma_i$ for some point $i \in \{7,8,9\}$. If there are at least two distinct lines $l \in \mathcal{L}_i$ with $\rk(\gamma_l) = 2$, we redefine $\gamma_i$ as the intersection of these lines, using $\gamma \in V(G_M)$. After applying this procedure to all such loops, we may assume that for each remaining loop $i$, there is at most one line $l\in \mathcal{L}_i$ with $\rk(l) = 2$.

Now perturb $\gamma_i$ to a non-zero vector in $\gamma_{l}$ for some $l \in \mathcal{L}_i$ with $\rk(l) = 2$. If no such line exists, we redefine $\gamma_i$ arbitrarily. The resulting configuration still lies in $\VCM$, since any line passing through $\gamma_i$ with rank one will still have rank at most two after this perturbation. Following this process for all points in $\{7,8,9\}$, we redefine $\gamma$ such that $\gamma_{7},\gamma_{8},\gamma_{9}\neq 0$ and the configuration remains in $\VCM$. At this stage, we fall into Case~1.1, from which it follows that $\gamma\in V_{M}$, as desired. 

\medskip
{\bf Case~2.} Suppose that $\rk(\gamma) =2$. In this case, all the vectors of $\gamma$ lie on a single line, which we denote by $l$. 

\medskip
{\bf Case~2.1.} Suppose that $\gamma_i \neq 0$ for all $i \in \{1, \ldots, 9\}$. Since $\gamma \in V(I_M^{\text{lift}})$, we lift $\gamma$ to $\widetilde{\gamma} \in \VCM$, such that $\rk(\widetilde{\gamma}) = 3$ and $\widetilde{\gamma}$ does not have any zero vectors. By Case~1.1, this ensures that $\widetilde{\gamma} \in V_M$, and consequently $\gamma\in V_{M}$, as desired.

\medskip
{\bf Case~2.2.} Suppose that at least one vector $\gamma_i$ is a loop. If multiple elements in $\{\gamma_{i}\}$ are loops, redefine all loops but one as arbitrary points on $l$. Consequently, we may assume that $\gamma$ has one loop and due to symmetry reasons, we can assume that $\gamma_1$ or $\gamma_7$ is the zero vector. First assume that $\gamma_1$ is the only zero vector. Since $M \setminus\{1\}$ is nilpotent, $\gamma|_{\{2,\ldots, 9\}}$ is liftable from any $q$. Thus by Proposition~\ref{prop: degree < 3 add point liftable}, we can redefine $\gamma_1$ such that $\gamma$ can be lifted to a collection of vectors of rank three in $\VCM$. By Case~1.1, this implies $\gamma \in V_M$.

\medskip Now assume that $\gamma_7 = 0$. Let $q\in \CC^3$ be any point outside $l$. By Proposition~\ref{proposition: dim(ker(liftmat))}, we obtain
\[\dim_q^\gamma(M \setminus \{7\}) = \dim_q^\gamma(S_M) = 4.\] 
Applying Proposition~\ref{prop: nilpotent add point liftable}, we can redefine $\gamma_7$ as a point on $l$ such that the resulting configuration is liftable to a non-degenerate collection of vectors in $\VCM$. By Case~1.1, $\gamma \in V_M$.

\medskip

{\bf Case~3.} Suppose that $\rk(\gamma) = 1$. In this case, all points either coincide with a fixed point or are loops. If at least one point is a loop, perturb the remaining points to reduce to Case~2.2. If there are no loops, then Lemma \ref{lemma pascal V_B subset V_M} implies that $\gamma \in V_M$.

\end{proof}

Theorem~\ref{lemma pascal V_B subset V_M} establishes that the ideals $I_{\CCC(M)}, G_M$ and $I_M^{\textup{lift}}$ generate $I_{M}$ up to radical, but it does not present an explicit generating set. This will be done in the following remark.
\begin{remark} \label{remark pascal}
By Theorem~\ref{lemma pascal V_B subset V_M}, to find explicit generators for $I_M$, up to radical, it suffices to find explicit generators for $I_{\CCC(M)}, G_M$ and $I_M^{\textup{lift}}$.
\begin{itemize}[noitemsep]
\item The circuit polynomials are precisely the bracket polynomials corresponding to the circuits of $M$, as explained in Definition~\ref{cir}. In this case, they are:
$$\bigl\{[1,6,8],[1,5,7],[2,4,7],[2,6,9],[3,4,8],[3,5,9],[7,8,9]\bigr\}.$$
\item Concerning the Grassmann–Cayley polynomials, we can deduce from the proof of Theorem~\ref{lemma pascal V_B subset V_M} that it suffices to consider the following:
\begin{itemize}
\item[{\rm (i)}] The polynomial
\[[1,5,3][1,4,2][5,4,6][3,2,6]-[1,5,4][1,3,2][5,3,6][4,2,6]\in G_{M}\]
which follows from the Grassmann–Cayley expression
\[(15 \wedge 24) \vee (16 \wedge 34) \vee (35 \wedge 26)=0.\]
\item[{\rm (ii)}] The polynomial 
\[[5,2,6][3,6,1][7,3,4]-[3,2,6][3,6,1][7,5,4]+[3,2,6][4,6,1][7,5,3]\in G_{M}\]
which corresponds to the expression
\[7\vee (53\wedge 26)\vee (34\wedge 61)=0,\]
together with the two analogous polynomials in $G_M$ obtained from the expressions
\[8\vee (51\wedge 24)\vee (35\wedge 62)=0,\qquad \text{and} \qquad 9\vee (43\wedge 16)\vee (24\wedge 51)=0.\]
\item[{\rm (iii)}] The polynomial 
\[[7,4,9][3,6,1]-[4,6,1][7,3,9]\in G_M\]
corresponding to the expression
\[7\vee 9 \vee (34\wedge 61)=0,\]
together with its two analogues derived from
\[7\vee 8 \vee (35\wedge 62)=0,\qquad \text{and} \qquad 9\vee 8 \vee (15\wedge 42)=0.\]
\end{itemize}
These are precisely the seven Grassmann–Cayley polynomials that the authors of \cite{Sidman} introduced in Theorem~3.0.2. For a detailed treatment of their construction and properties, we refer the reader to that paper.
\item One can notice that in the proof of Theorem~\ref{lemma pascal V_B subset V_M}, we only use the assumption $\gamma \in V(I_M^{\textup{lift}})$ to provide the existence of a non-degenerate lifting of the vectors $\{\gamma_1,\ldots,\gamma_9\}$ to a collection in $V_{\CCC(M)}$. Thus, by Lemma~\ref{lemma: motivation liftability}, it suffices to consider the $7\times 7$ minors of the liftability matrices $\mathcal{M}_{q}(M)$.

Since $q$ ranges over $\CC^{3}$, the set of all such minors is not finite. However, the discussion in Section~6 of \cite{liwski2024pavingmatroidsdefiningequations} shows that one obtains a finite set of generators by selecting vectors $q_1, \ldots q_9\in \CC^{3}$ and replacing the vector $q$ appearing in the brackets of column $i$ with $q_i$. This also gives us a polynomial inside $I_M$. By multilinearity of the determinant, it is enough to take each $q_i$ from the canonical basis $\{e_1, e_2, e_3\}$ of $\CC^3$. The resulting polynomials are the $7\times 7$ minors of the matrices 

 $$
\scalebox{0.95}{$
\begin{pmatrix}
       [68q_1] & 0 & 0 & 0 & 0 & -[18q_6] & 0 & [16q_8] & 0 \\
       [57q_1] & 0 & 0 & 0 & -[17q_5] & 0 & [15q_7] & 0 & 0 \\
       0 & [47q_2] & 0 & -[27q_4] & 0 & 0 & [24q_7] & 0 & 0 \\
       0 & [69q_2] & 0 & 0 & 0 & -[29q_6] & 0 & 0 & [26q_9] \\
       0 & 0 & [48q_3] & -[38q_4] & 0 & 0 & 0 & [34q_8] & 0 \\
       0 & 0 & [59q_3] & 0 & -[39q_5] & 0 & 0 & 0 & [35q_9] \\
       0 & 0 & 0 & 0 & 0 & 0 & [89q_7] & -[79q_8] & [78q_9] \\
\end{pmatrix}
$}
$$
where $q_1, \ldots, q_9\in \{e_1, e_2, e_3\}$. The total number of such polynomials is $\textstyle \binom{9}{7}3^{9}$.
\end{itemize}
\end{remark}
\section{Pappus configuration}
In this section, we discuss the Pappus configuration, which is the point-line configuration illustrated in Figure~\ref{fig:Pappus, Pappus, I_i, J_i} (Left). We find a complete set of defining equations for its matroid variety.
Throughout this section, we denote the Pappus configuration by $M$.
The following theorem from \cite{liwskimohammadialgorithmforminimalmatroids} provides the irreducible decomposition of the Pappus configuration.
\begin{theorem} \label{pappus decompo}
   The circuit variety of $M$ admits the following irredundant, irreducible decomposition:
\begin{equation}\label{equation: decomposition of pappus} \VCM = V_{M} \cup V_{U_{2,9}} \cup_{i=1}^{18} V_{I_i} \cup_{i=1}^3 V_{J_i} \cup_{i=1}^9 V_{\pi_M^i}, \end{equation}
where the matroids in the decomposition are the following:

\begin{itemize}
\item $U_{2,9}$ denotes the uniform matroid of rank two on the ground set $[9]$.
\item The matroids $I_i$ for $i\in [18]$ are obtained from $M$ by making one of its points a loop and adding one of the circuits $\{1, 4, 9\}, \{2,5,8\}$ or $\{3,6,7\}$, see Figure~\ref{fig:Pappus, Pappus, I_i, J_i} (Center). 
\item $J_{1},J_{2},J_{3}$ denote the matroids derived from $M$ by making all three points of one of the triples $\{1,4,9\}, \{3,6,7\}$ or $ \{2,5,8\}$ loops, see Figure~\ref{fig:Pappus, Pappus, I_i, J_i} (Right).
\item The matroids $\pi_M^i$ are defined in Notation~\ref{notation pi}. See Figure~\ref{fig:Pappus, K_i, L__i, N} (Left) for an illustration.
\end{itemize}
\end{theorem}

\begin{figure}
    \centering
    \includegraphics[width=0.9\linewidth]{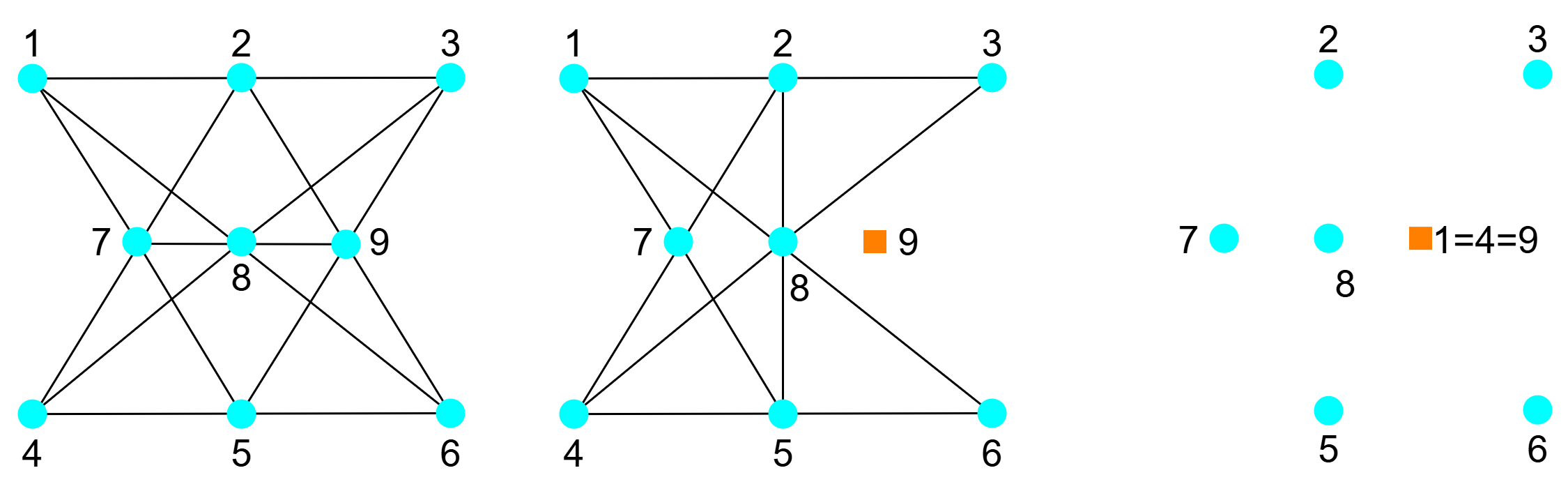}
    \caption{(Left) Pappus Configuration, (Center) $I_9$, (Right) $J_1$.}
    \label{fig:Pappus, Pappus, I_i, J_i}
\end{figure}
\begin{figure}
    \centering
    \includegraphics[width=0.9\linewidth]{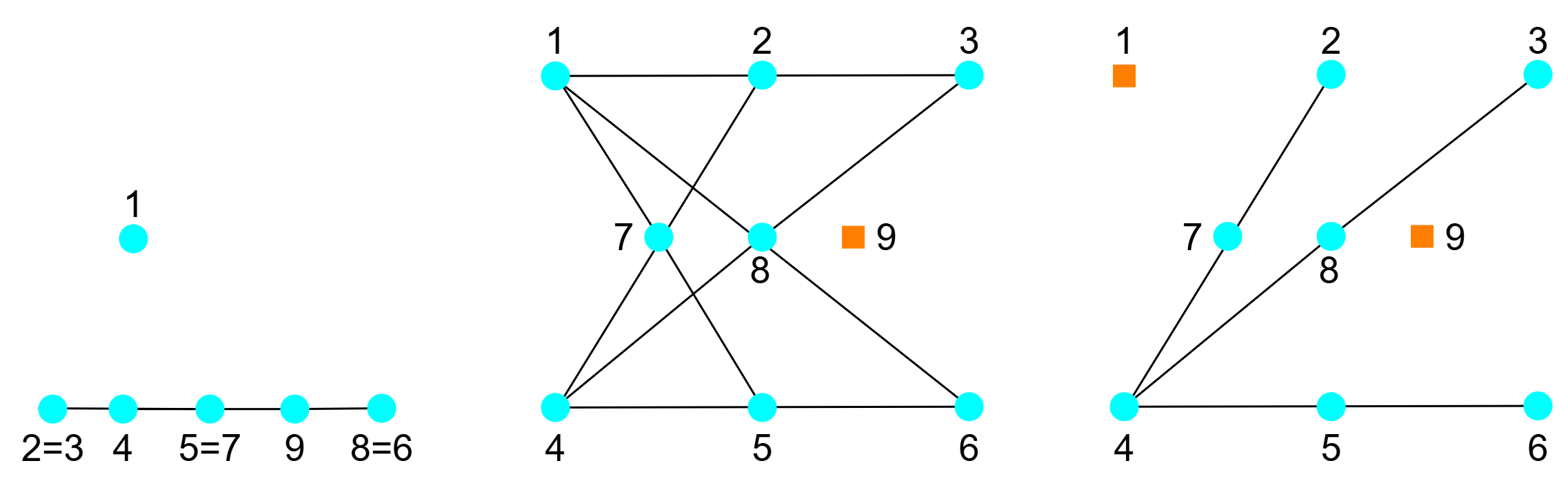}
    \caption{(Left) $\pi_M^i$; (Center) $M(9)$; (Right) $
    M(1,9)$.}
    \label{fig:Pappus, K_i, L__i, N}
\end{figure}
In order to determine the defining equations of $M$, we will use the following proposition concerning the decomposition of the circuit variety of the matroids $M(i)$, derived from $M$ by making one of its points a loop. Since all $i$ are analogous, we focus on $i=9$ in the proposition below, as in Figure \ref{fig:Pappus, K_i, L__i, N} (Center). Note that we do not use the decomposition strategy; instead, we opt for a shorter alternative. A proof of this result with the decomposition strategy will be provided in Chapter \ref{sec: decompo revisited}.
\begin{proposition} \label{prop: Pappus M(9)}
    The following equation holds for the circuit variety of  $M(9)$:
    \begin{equation} \label{eq: VCM(9) final decomposition}
V_{\CCC(M(9))} = V_{M(9)} \cup V_{U_{2,9}(9)} \cup_{j=1,4} V_{M(9,j)} \cup V_{M(1,4,9)}.
\end{equation}
\end{proposition}
\begin{proof}
We will decompose the variety $N = M \setminus \{9\}.$ 
The inclusion $'\supseteq'$ in Equation~\eqref{eq: VCM(9) final decomposition} is obvious, since the matroids on the right have more dependencies than those on the left. For the other inclusion, first assume that $\gamma_1 = 0$ or $\gamma_4 = 0$, we can assume without loss of generality that the former holds. $M(1,9)$ is shown in Figure \ref{fig:Pappus, K_i, L__i, N} (Right). By Example \ref{ex: decomposition circuit variety 3 concurrent lines}, $V_{\mathcal{C}(N(1))} = V_{N(1)} \cup V_{N(1,4)}$, thus $\gamma \in V_{N(1)} \cup V_{N(1,4)}$.
\newline
\newline
Now assume $\gamma \in V_{\mathcal{C}(N)}$, $\gamma_1, \gamma_4 \neq 0$. By Lemma \ref{lemma: third config loop matroid variety}, we can assume that $\gamma$ has no loops.
\begin{mycases}
\medskip
    \case{There is a point $p \in [8] \setminus \{1,4\}$ such that $\gamma_{l_1} \wedge \gamma_{l_2} \neq 0$, where $\{l_1, l_2\} =\mathcal{L}_p$.} 
    Since $N \setminus \{p\}$ is nilpotent and has no points of degree at least three, we can use Theorem~\ref{nil coincide} to perturb $\gamma|_{N \setminus \{p\}}$ to  $\widetilde{\gamma}|_{N \setminus \{p\}} \in \Gamma|_{N \setminus \{p\}}$. Define $\widetilde{\gamma}_p = \widetilde{\gamma}_{l_1} \wedge \widetilde{\gamma}_{l_2},$ with $\mathcal{L}_p = \{l_1, l_2\}.$ Then $\widetilde{\gamma} \in \Gamma_N$, so by Lemma \ref{lemma: meet of 2 extensors arbitrary close}, $\gamma \in V_N$.
    \medskip
    \case{For all points $p \in [8]\setminus\{1,4\}$, $\gamma_{l_1} \wedge \gamma_{l_2} = 0$, where $\{l_1, l_2\} =\mathcal{L}_p$.} 
    
\medskip

After applying Lemma \ref{lemma: if gamma_i gamma_j wedge gamma_k gamma_l =0}, we can assume that $\gamma_i \neq \gamma_j$ for $i\in [8] \setminus \{1,4\}, j \in [8], i \neq j$. Thus for each $l \in \mathcal{L}_N$, the corresponding subspace $\gamma_l$ has rank two.

Because of Case 2, we have that $\gamma_{\{1,2,3\}}$ $= $ $\gamma_{\{2,7,4\}}$ $=$ $\gamma_{\{3,4,8\}}$ $=$ $\gamma_{\{1,6,8\}}$ $= $ $\gamma_{\{4,5,6\}}$ $=$ $\gamma_{\{1,5,7\}}$, implying that the vectors corresponding to the points in $[8]$ are on a line, thus $\gamma \in V_{U_{2,8}}.$ This finishes the proof.
    \end{mycases}
\end{proof}

We now state the main result of this section, which provides a complete set of defining equations for the Pappus configuration.

\begin{theorem}\label{thm:Pappus}
    Let $M$ be the matroid associated to the Pappus configuration. Then $$I_{M}=\sqrt{I_{\mathcal{C}(M)}+G_{M}+I_{M}^{\textup{lift}}}.$$
\end{theorem}

\begin{proof}
To prove the equality of the statement, we will show the following equivalent equality of varieties
\begin{equation}\label{lifting}
V_M = V_{\mathcal{C}(M)} \cap V(G_M) \cap V_{M}^{\textup{lift}}.
\end{equation}

The inclusion $\subset$ in~\eqref{lifting} follows directly from Remark~\ref{generating set ICM}, Theorem~\ref{thm: gamma i V_M^lift => liftable} and Proposition~\ref{inc G}. 
To establish the reverse inclusion, let $\gamma\in V_{G_{M}}\cap V_{M}^{\text{lift}}\cap V_{\mathcal{C}(M)}$ arbitrary. We must show that $\gamma\in V_{M}$. To do so, we consider different cases.  

    \begin{mycases}
\case {\bf $\gamma$ has no loops.}
\newline
Then we can see from Equation~\eqref{equation: decomposition of pappus} that $\gamma \in V_{M}$ or $\gamma \in V_{U_{2,9}}$ or $\gamma \in V_{\pi_M^i}$. We can assume that $\gamma \notin V_M.$

\medskip
{\bf Case 1.1.} If not all the points are collinear, then $\gamma \in V_{\pi_M^i}$.
        So eight points of $\gamma$ are on a line, and one point is outside that line, without loss of generality we can assume this point to be 9.
        Since $M \setminus \{9\}$ is a submatroid of $M$ of full rank, using Theorem~\ref{thm: gamma i V_M^lift => liftable}, we can lift $\gamma|_{[8]}$ from the point $\gamma_9$ to $\widetilde{\gamma}|_{[8]}$ in $V_{\CCC(M\setminus \{9\})}$, such that the points $\widetilde{\gamma}_1, \ldots, \widetilde{\gamma}_8$ are not collinear. Extending $\widetilde{\gamma}|_{[8]}$ to $\widetilde{\gamma}$ by defining $\widetilde{\gamma}_9 = \gamma_9$, it follows that $\widetilde{\gamma} \in \VCM$. We can conclude by Theorem~\ref{pappus decompo} that $\widetilde{\gamma} \in V_{M}$. 
        
        \medskip
        
        \textbf{Case 1.2.} Now assume that all the points are collinear, and there are at least two different points. Since $\gamma \in \VMLIFT$, we can lift this collection infinitesimally from an arbitrary point to a non-degenerate $\widetilde{\gamma} \in \VCM$ by Theorem \ref{thm: gamma i V_M^lift => liftable}. 
        If $\widetilde{\gamma} \in V_{M}$, we are done.
        Else $\widetilde{\gamma} \in V_{\pi_M^i}$. We may assume without loss of generality that all the points in the collection $\widetilde{\gamma}$ except $\widetilde{\gamma}_1$ are collinear. Since a perturbation cannot create new dependencies, the following equations should hold for our original configuration $\gamma$: $$\gamma_2 = \gamma_3, \gamma_5 = \gamma_7, \gamma_6 = \gamma_8.$$ One can perturb $\gamma_1$ arbitrarily such that it is not on the line formed by $\gamma_2, \ldots, \gamma_9$ anymore. Then we can complete the proof as in Case 1.1.

        \medskip
        
     \textbf{Case 1.3.} Assume that all the vectors $\gamma_1, \ldots, \gamma_9$ coincide. 
    Consider the matroid $L_k$ obtained from $M$ by setting $4=6,1=3$ and $7=9$, as in Figure 8 (Right) of \cite{liwskimohammadialgorithmforminimalmatroids}.
    Clearly, $\gamma \in V_{\mathcal{C}(L_k)} = V_{L_k}$ by Theorem~\ref{nil coincide}. By Lemma 5.4 (iii) of \cite{liwskimohammadialgorithmforminimalmatroids}, it follows that $V_{L_k} \subset V_{M},$ thus $\gamma \in V_M.$
    \newline
\case\textbf{$\gamma$ has exactly one loop.} \newline
Assume without loss of generality that the loop is 9. It is proven in Lemma~5.5 (iii) of \cite{liwskimohammadialgorithmforminimalmatroids} that $V_{M(9)} \subset V_M$, so it suffices to prove that $\gamma \in V_{M(9)}.$ The matroid $M(9)$ is shown in Figure \ref{fig:Pappus, K_i, L__i, N} (Center).

Using Proposition~\ref{prop: Pappus M(9)}, $\gamma \in V_{M(9)}$ or $\gamma \in V_{U_{2,9}(9)}$. If the former holds, $\gamma \in V_M$, so we can assume that  $\gamma \in V_{U_{2,9}(9)}$. Since $\gamma \in \VMLIFT$, we can lift the collection of vectors $\gamma|_{[8]}$ to the collection of vectors $\widetilde{\gamma}|_{[8]}$ of dimension three from an arbitrary vector $q$ not on the line. Then $\widetilde{\gamma}|_{[8]}$ does not have any loops, and is non-degenerate, so it lies in $V_{M\setminus\{9\}}$. Extending $\widetilde{\gamma}|_{[8]}$ to $\widetilde{\gamma}$ by defining $\widetilde{\gamma}|_9$ to be the zero vector, it is clear that $\widetilde{\gamma} \in V_{M(9)}.$

\medskip

     \case\textbf{$\gamma$ has exactly two loops}

     \medskip
     
     {\bf Case 3.1.} Suppose that the 
     the two loops do not belong to a common line in $M$. In that case, we can assume without loss of generality that the points $1$ and $9$ are loops. The matroid $M(1,9)$ is illustrated in Figure~\ref{fig:Pappus, K_i, L__i, N} (Right). 

    \medskip
     
     {\bf Case 3.1.1} Assume first that the vectors $\gamma_{2},\ldots,\gamma_{8}$ are collinear. By applying a perturbation to these points within the line they span, we may further assume that they are pairwise distinct. 
     Using Proposition~\ref{proposition: dim(ker(liftmat))}, $\dim_q^\gamma(M \setminus \{1, 9\}) = 1+3 = 4$. Moreover $M \setminus \{1, 9\}$ is nilpotent, $\deg(1) = 3$ and all the points $\gamma_i$ are pairwise distinct and non-zero for $i\in \{2,\ldots,8\}$.
     So we can use Proposition~\ref{prop: nilpotent add point liftable} to redefine $\gamma_1$ on the line through $\gamma_2, \ldots, \gamma_8$, such that $\gamma|_{[8]} \in V_{M \setminus \{9\}}^{\text{lift}}$. This reduces to Case 2, from which we can conclude that $\gamma \in V_M.$
     
     \medskip
     
      {\bf Case 3.1.2} Suppose now that $\rank\{\gamma_{2},\ldots,\gamma_{8}\}=3$.  
Since $\gamma\in V(G_{M})$, one can redefine $\gamma_{1}$ to be a non-zero vector, while keeping the collection of vectors within $V_{\CCC(M(9))}$. Let $\widetilde{\gamma}$ denote this modified collection. Given that $\widetilde{\gamma}$ is in $V_{\CCC(M(9))}$, has rank three, and contains only one loop, it follows from Equation~\eqref{eq: VCM(9) final decomposition} that $\widetilde{\gamma}\in V_{M(9)}$. Since $V_{M(9)}\subset V_{M}$, we deduce that $\widetilde{\gamma}\in V_{M},$ so $\gamma\in V_{M}$.

      \medskip
      
      {\bf Case 3.2.} Suppose that the two loops correspond to elements of $[9]$ that lie on a common line in $M$. In this case, we can use Lemma 5.5 (ii) of \cite{liwskimohammadialgorithmforminimalmatroids} to conclude that the circuit variety of this matroid is contained within that of $V_{M}$.

\medskip

     \case{\bf $\gamma$ has exactly three loops.}

\medskip

{\bf Case~4.1} Suppose that two of the three loops correspond to elements of $[9]$ that lie on a common line in $M$. In this case, we can conclude by Lemma~5.5 (ii) of \cite{liwskimohammadialgorithmforminimalmatroids} that the circuit variety of this matroid is contained within that of $V_{M}$.

\medskip

{\bf Case~4.2} Suppose that no two of the three loops correspond to elements of $[9]$ that lie on a common line in $M$. In this case, we may assume without loss of generality that the loops correspond to the elements $1,4$ and $9$, as shown in Figure \ref{fig:Pappus, Pappus, I_i, J_i}. Since $\gamma \in V_{G_M}$, we can redefine $\gamma_4$ as a non-zero vector, while keeping the collection of vectors within $V_{\CCC(M)}$. This reduces to Case~3, where the argument relies only on the concurrency of the lines $\gamma_{23},\gamma_{68}$ and $\gamma_{57}$, a condition that remains satisfied after the perturbation.
    
    \medskip
     \case{\bf There are at least four loops}
     
     \medskip
     
    In this case at least two of these four loops belong to a common line in $M$. It then follows from Lemma~5.5 (ii) of \cite{liwskimohammadialgorithmforminimalmatroids} that the circuit variety of this matroid is contained in $V_{M}$.
\end{mycases}
\end{proof}

Although Theorem~\ref{thm:Pappus} establishes that the ideals $I_{\CCC(M)}, G_M$ and $I_M^{\textup{lift}}$ generate $I_{M}$ up to radical, it does not present an explicit generating set. The following remark presents such a generating set.
\begin{remark} \label{remark pappus}
By Theorem~\ref{thm:Pappus}, to find explicit generators for $I_M$, up to radical, it suffices to find explicit generators for $I_{\CCC(M)}, G_M$ and $I_M^{\textup{lift}}$.
\begin{itemize}[noitemsep]
 \item The circuit polynomials are precisely the bracket polynomials corresponding to $M$. For the Pappus configuration, they are:
     $$\bigl\{[1,2,3],[1,6,8],[1,5,7],[2,4,7],[2,6,9],[7,8,9],[4,5,6],[3,4,8],[3,5,9]\bigr\}.$$ 
\item For the Grassmann–Cayley polynomials, the proof of Theorem~\ref{thm:Pappus} shows that it suffices to consider the nine polynomials arising from the nine triples of concurrent lines in $M$. These are given by:
\begin{equation*}
\begin{aligned}\bigl\{&[2,3,5][7,6,8]-[2,3,7][5,6,8], [1,3,4][7,6,9]-[1,3,7][4,6,9], \\ &[1,2,4][8,5,9]-[1,2,8][4,5,9],
[2,4,1][5,8,9]-[2,4,5][1,8,9], \\
&
     [7,9,1][6,3,4]-[7,9,6][1,3,4],[2,6,3][5,7,8]-[2,6,5][3,7,8], \\
&[2,7,3][8,5,6]-[2,7,8][3,5,6],[4,6,1][7,3,9]-[4,6,7][1,3,9],\\
&[2,9,1][8,4,5]-[2,9,8][1,4,5]\bigr\}.
\end{aligned}
\end{equation*}
\item From the proof of Theorem~\ref{thm:Pappus}, it follows that the only instances where the assumption $\gamma\in V(I_M^{\textup{lift}})$ is invoked, are: 
\begin{enumerate}
\item To ensure the existence of a non-degenerate lifting of the vectors $\{\gamma_{1},\ldots,\gamma_{9}\}$ to a collection in $V_{\CCC(M)}$.
\item To guarantee that for each $i\in [9]$, the eight vectors $\{\gamma_1, \ldots,\gamma_{i-1}, \hat{\gamma}_i, \gamma_{i+1}, \ldots, \gamma_9\}$ have a non degenerate lifting to a collection in $V_{\CCC(M\setminus \{i\})}$. 
\end{enumerate}
Thus, by Theorem \ref{thm: gamma i V_M^lift => liftable}, it suffices to consider the $7\times 7$ minors of the liftability matrices $\mathcal{M}_q(M)$, together with the $6\times 6$ minors of the liftability matrices $\mathcal{M}_q(M\setminus \{i\})$ for $i\in [9]$.

Although the set of all such minors is not finite, since $q$ ranges over $\CC^{3}$, we may adopt the same strategy as used in Remark~\ref{remark pascal} to obtain a finite generating set. Namely, we select vectors $q_{1},\ldots,q_{9}\in \{e_{1},e_{2},e_{3}\}$ and replace the vector $q$ appearing in the bracket of column $i$ with $q_i$. The resulting polynomials are the $7\times 7$ minors of the matrices

$$
\scalebox{0.95}{$
\begin{pmatrix}
   [23q_1] & [-13q_2] & [12q_3] & 0 & 0 & 0 & 0& 0 & 0 \\
   [68q_1] & 0 & 0 & 0 & 0 & -[18q_6] & 0 & [16q_8] & 0 \\
   [57q_1] & 0 & 0 & 0 & -[17q_5] & 0 & [15q_7] & 0 & 0 \\
   0 & [47q_2] & 0 & -[27q_4] & 0 & 0 & [24q_7] & 0 & 0 \\
   0 & [69q_2] & 0 & 0 & 0 & -[29q_6] & 0 & 0 & [26q_9] \\
   0 & 0 & [48q_3] & -[38q_4] & 0 & 0 & 0 & [34q_8] & 0 \\
   0 & 0 & [59q_3] & 0 & -[39q_5] & 0 & 0 & 0& [35q_9] \\
   0 & 0 & 0 & 0 & 0 & 0 & [89q_7] & -[79q_8] & [78q_9] \\
   0 & 0 & 0 & [56q_4] & -[46q_5] & [45q_6] & 0 & 0 & 0
\end{pmatrix}
$}
$$
where $q_{1},\ldots,q_{9}\in \{e_{1},e_{2},e_{3}\}$, and the $6 \times 6$ minors of the matrices
$$
\scalebox{0.95}{$
\begin{pmatrix}
   [23q_1] & [-13q_2] & [12q_3] & 0 & 0 & 0 & 0 & 0 \\
   [68q_1] & 0 & 0 & 0 & 0 & -[18q_6] & 0 & [16q_8] \\
   [57q_1] & 0 & 0 & 0 & -[17q_5] & 0 & [15q_7] & 0 \\
   0 & [47q_2] & 0 & -[27q_4] & 0 & 0 & [24q_7] & 0 \\
   0 & 0 & [48q_3] & -[38q_4] & 0 & 0 & 0 & [34q_8] \\
   0 & 0 & 0 & [56q_4] & -[46q_5] & [45q_6] & 0 & 0
\end{pmatrix}
$}
$$
   where $q_1, \ldots, q_8 \in \{e_1,e_2,e_3\}$, 
   along with the $6\times 6$ minors of the matrices analogous to this, corresponding to omitting each index $i=1,\ldots,8$. The total number of such polynomials is:
\[\textstyle \binom{9}{7}3^{9}+\binom{8}{6}3^{10}.\]
\end{itemize}
\end{remark}
\numberwithin{theorem}{chapter}
\numberwithin{equation}{chapter}
\chapter{\texorpdfstring{Decomposition of the circuit variety of the third configuration $9_3$}{Decomposition of the circuit variety of the third configuration 9_3}}
\label{Decomposition of the Circuit Variety of the third configuration 93}
We will consider the third configuration $9_3$, as in Figure \ref{fig:Third Config A_1 B} (Left). Throughout this chapter, this point-line configuration is denoted by $M$. We will find an irreducible decomposition of its circuit variety, as in Theorem \ref{thm decompo third 9_3}. In Chapter \ref{sec: pascal & pappus}, we managed to give a proof without using the decomposition strategy for the decomposition of the circuit varieties of the Pascal configuration and the Pappus configuration with one loop. Since the third configuration $9_3$ has no points of degree two, we are not able to develop a similar technique. However, we can use the first step of the decomposition strategy, as outlined in Subsection \ref{subsec: decomposition strategy}, to find the matroids which are minimal over $M$. For these matroids, if it is possible, we use a similar technique as in Chapter \ref{sec: pascal & pappus} to find a decomposition of its circuit variety, else we use again the decomposition strategy.

\medskip
We start this chapter with the following remark about $M$.
\begin{remark}
Note that there are two classes of symmetric points in $M$: $\{3,6,8\}$ and $\{1,2,4,5,7,9\}$. Moreover, for every point $i \in [9]$, there are exactly two points which are not on a line with $i$. Denote these points as $k_i^1$ and $k_i^2.$
\end{remark}
The following remark will be important for the decomposition strategy.
\begin{remark}
When working with the minimal matroids of a non-simple rank-three matroid $N$, we instead consider implicitly its reduced version $N_{\text{red}}$, defined by removing loops and identifying double points. This approach is valid given the correspondence between the minimal matroids of $N$ and $N_{\text{red}}$. To reconstruct the minimal matroids of $N$ again, we have to reintroduce these loops and double points again.
\end{remark}

\noindent

The minimal matroids over $M$ are:
\begin{itemize}
    \item The matroid obtained by identifying the points in $\{7,8,9\}$ as in Figure \ref{fig:Third Config A_1 B} (Center). 
    Denote this matroid by $A_1$, and denote the five other matroids derived from it by applying an automorphism of $M$ by $A_i$.
     \item The matroid obtained by identifying the pairs of points $\{1,2\}$, $\{4,5\}$ and $\{7, 9\}$, as in Figure \ref{fig:Third Config A_1 B} (Right). 
    Denote this matroid by $B$.  
    \item The two matroids obtained by identifying $\{1,5,9\}$ or $\{2,4,7\}$.
    Denote the former by $C_1$, shown in Figure \ref{fig:Third config C pi_M^1 M(1)} (Left) and the latter by $C_2$.  
      \item The three matroids $\pi_M^i$ for $i$ in $\{3,6,8\}$. $\pi_M^1$ is shown in Figure \ref{fig:Third config C pi_M^1 M(1)} (Center).
    \item The matroid obtained by adding $\{3,6,8\}$ as a circuit. Denote this matroid by $D$.
    \item The matroids $M(i)$, where $i$ is in $\{1,2,4,5,7,9\}$. $M(1)$ is shown in Figure \ref{fig:Third config C pi_M^1 M(1)} (Right).
    \end{itemize}
    \begin{figure}
    \centering
    \includegraphics[width=0.9\linewidth]{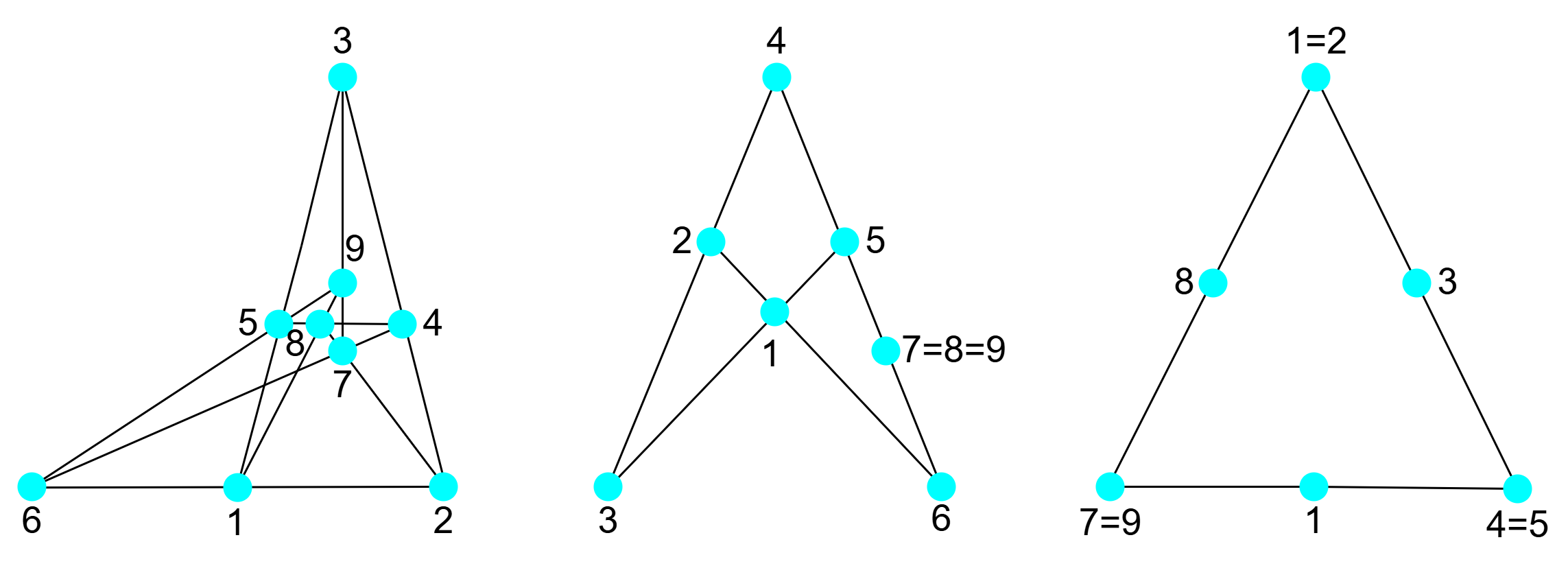}
    \caption{(Left) Third configuration $9_3$; (Center) $A_1$; (Right) $B$.}
    \label{fig:Third Config A_1 B}
\end{figure}
\begin{figure}
    \centering
    \includegraphics[width=0.9\linewidth]{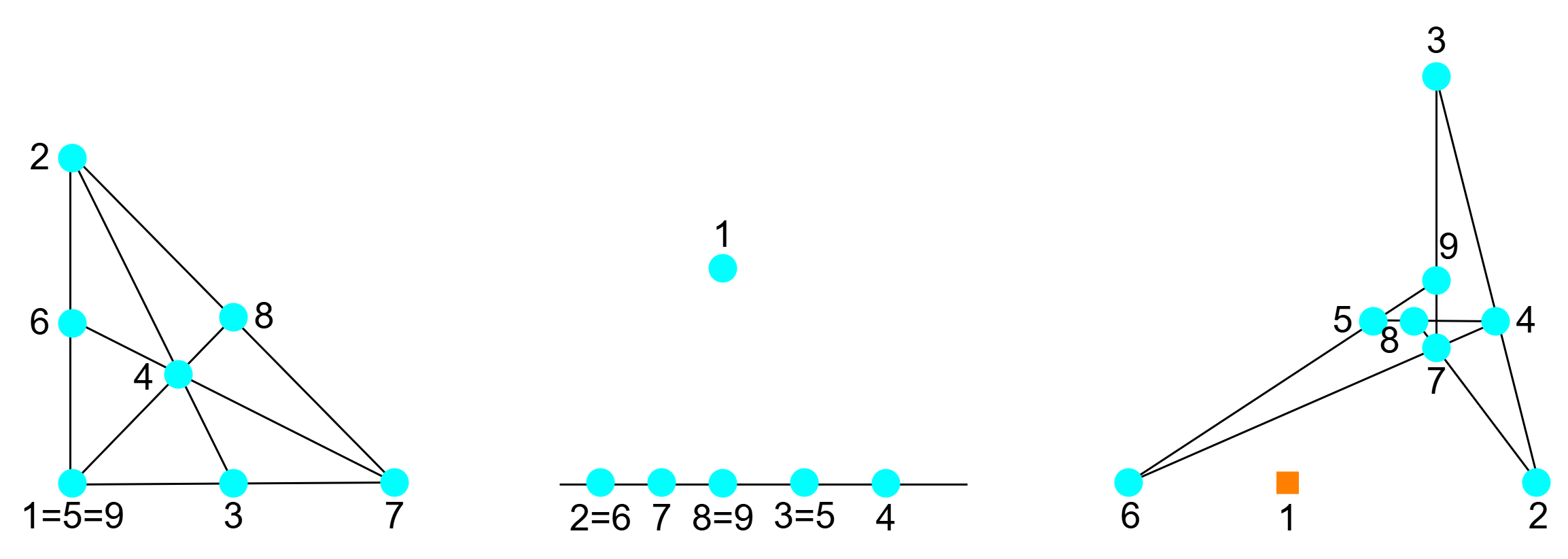}
    \caption{(Left) $C_1$; (Center) $\pi_M^1$; (Right) $M(1)$.}
    \label{fig:Third config C pi_M^1 M(1)}
\end{figure}
As in the decomposition strategy, it follows that:
$$\VCM = V_M \cup_{i=1}^6 V_{\CCC(A_i)} \cup V_{\CCC(B)} \cup_{i=3,6,8} V_{\CCC(\pi^i_M)} $$
\begin{equation} \label{eq: first decomposition for third config 93}
    \cup_{i=1}^2 V_{\CCC(C_i)} \cup V_{\CCC(D)} \cup_{i=1,2,4,5,7,9} V_{\CCC(M(i))}.
\end{equation}
By Theorem~\ref{nil coincide}, it already follows that 
\begin{equation} \label{eq:inc V_B}
V_{\CCC(B)} = V_B.
\end{equation}
We prove the following lemmas.
\begin{lemma} \label{third lemma 1}
For the point-line configuration $A_1$, as in Figure \ref{fig:Third Config A_1 B} (Center), we have that:
    $$V_{\CCC(A_1)} \subset V_{A_1} \cup V_{U_{2,9}}.$$

\end{lemma}
\begin{proof}
Denote $A'$ for the matroid obtained from $A_1$ by omitting the points $8$ and $9$. Then it suffices to prove that
    $$V_{\CCC(A')} = V_{A'} \cup V_{U_{2,7}}.$$
Let $\gamma \in V_{\CCC(A')}$. By Lemma \ref{lemma: third config loop matroid variety}, we can assume that $\gamma$ has no loops. We consider two cases:
    \begin{mycases}
        \case Assume there exists a point $p \in [6]$ such that $\gamma_{l_1} \wedge \gamma_{l_2} \neq 0$, where $\{l_1, l_2\} =\mathcal{L}_p$. Then we can use Theorem~\ref{nil coincide} to perturb $\gamma|_{[7] \setminus \{p\}}$ to  $\widetilde{\gamma}|_{[7] \setminus \{p\}} \in \Gamma_{A' \setminus \{p\}}$. Define $\widetilde{\gamma}_p = \widetilde{\gamma}_{l_1} \wedge \widetilde{\gamma}_{l_2}$. By Lemma \ref{lemma: meet of 2 extensors arbitrary close}, $\widetilde{\gamma}$ is an arbitrary perturbation of $\gamma$. Moreover, $\widetilde{\gamma} \in \Gamma_{A'}$, thus $\gamma \in V_{A'}.$
    \case{For all points $p \in [6]$, $\gamma_{l_1} \wedge \gamma_{l_2} = 0$, where $\{l_1, l_2\} =\mathcal{L}_p$.} 
\newline    
Since the point 7 is on only one line of the point-line configuration, we can assume after a perturbation that $\gamma_7$ does not coincides with another vector of $\gamma$. Moreover, using Lemma \ref{lemma: if gamma_i gamma_j wedge gamma_k gamma_l =0}, we can assume that after a perturbation $\gamma_i \neq \gamma_j$ for $i,j \in [6], i \neq j$. Thus we may assume that each line has rank two. Since $\gamma_{\{2,3,4\}} = \gamma_{\{4,5,6,7\}} = \gamma_{\{1,2,6\}} = \gamma_{\{1,3,5\}}$, the vectors in $\{\gamma_1, \ldots, \gamma_7\}$ lie on a line, thus $\gamma \in V_{U_{2,7}}.$ This concludes the proof.
    \end{mycases}    
\end{proof}
\begin{lemma} \label{third lemma 2}
The following equality holds for the point-line configuration $C_1$ in Figure \ref{fig:Third config C pi_M^1 M(1)} (Left):
    $$V_{\CCC(C_1)} \subset V_{C_1} \cup V_{B} \cup_{i=1,2,4,7} V_{C(i)} \cup V_{U_{2,9}}.$$
\end{lemma}
\begin{proof}
Let $\gamma \in V_{\CCC(C_1)}$.
If there is a point of degree three which is a loop, we can assume without loss of generality that this point is 2. Then $\gamma \in V_{\CCC(C_1(2))} = V_{C_1(2)}$ by Theorem~\ref{nil coincide} and Remark~\ref{remark reduced}. So we can assume that $\gamma$ has no loops of degree three. By Lemma \ref{lemma: third config loop matroid variety}, we can assume that $\gamma$ has no loops at all. We consider two cases:
    \begin{mycases}
        \case{There is a point $p \in \{3,6,8\}$ such that $\gamma_{l_1} \wedge \gamma_{l_2} \neq 0$, where $\{l_1, l_2\} =\mathcal{L}_p$.} \\
Without loss of generality, we may assume that $p=3$. By Theorem~\ref{nil coincide} and Remark \ref{remark reduced}, $V_{\CCC(C_1 \setminus \{3\})} \subset V_{C_1 \setminus \{3\}} \cup V_{U_{2,8}}.$ If $\gamma|_{[9] \setminus \{3\}} \in V_{U_{2,8}}$, this contradicts the assumption that $\gamma_{24} \wedge \gamma_{17} \neq 0$, thus $\gamma|_{[9] \setminus \{3\}} \in V_{C_1 \setminus \{3\}}$. Perturb $\gamma|_{[9] \setminus \{3\}}$ to $\widetilde{\gamma}|_{[9] \setminus \{3\}} \in \Gamma_{C_1 \setminus \{3\}}$. Define $\widetilde{\gamma}_3 = \widetilde{\gamma}_2 \widetilde{\gamma}_4\wedge \widetilde{\gamma}_1 \widetilde{\gamma}_7$. By Lemma \ref{lemma: meet of 2 extensors arbitrary close}, $\widetilde{\gamma}$ is a perturbation of $\gamma$, and $\widetilde{\gamma} \in \Gamma_{C_1}$, thus $\gamma \in V_{C_1}.$
    \case{For all points $p \in \{3,6,8\}$, $\gamma_{l_1} \wedge \gamma_{l_2} = 0$, where $\{l_1, l_2\} =\mathcal{L}_p$.} 
\newline    
By Lemma \ref{lemma: if gamma_i gamma_j wedge gamma_k gamma_l =0}, we can apply a perturbation to $\gamma$ such that we can moreover assume that $\gamma_i \neq \gamma_j$ if $i \neq j$, for $i \in \{3,6,8\}, j \in [9].$
If $\{\gamma_1, \gamma_6, \gamma_2, \gamma_4, \gamma_7\}$ are collinear, then all the vectors $\gamma_i$ for $i \in [9]$ are collinear, thus $\gamma \in U_{2,9}$. The same holds for the set of vectors formed by the lines through the points 6 and 8. So we can assume that the vectors in $\{\gamma_1, \gamma_6, \gamma_2, \gamma_4, \gamma_7\}$ and $\{\gamma_2, \gamma_7, \gamma_8, \gamma_4, \gamma_1\}$ and $\{\gamma_1, \gamma_3, \gamma_7, \gamma_2, \gamma_4\}$ are not collinear.
Since $\gamma_{1} \gamma_2 \wedge \gamma_4 \gamma_7 = 0$, it follows that $\gamma_1= \gamma_2$ or $\gamma_4 = \gamma_7$. Similarly, $\gamma_1 = \gamma_7$ or $\gamma_2 = \gamma_4$, moreover $\gamma_1 =  \gamma_4$ or $\gamma_2 = \gamma_7$.
Assume without loss of generality that $\gamma_1= \gamma_2= \gamma_7$. Then $\gamma$ is a realization of a matroid for which the sets $\{1,4,8\},\{1,4,3\},\{1,4,6\}$ are dependent. Since not all the vectors are collinear, and $\gamma_3 \neq \gamma_6 \neq \gamma_8$, it must hold that $1=4$, in which case $\gamma \in V_{\CCC(B)}=V_B$.
    \end{mycases}    
\end{proof}

The proof of the following lemma is deferred to Section \ref{appendix:proof of lemma pappus}, page \pageref{proof V_C in V_M}, since it uses a technique outlined in that chapter.
\begin{lemma} \label{third lemma 3}
    For the point-line configuration $C_1$ as in Figure \ref{fig:Third config C pi_M^1 M(1)} (Left), the following equation holds: 
    $$V_{C_1} \subset V_M.$$
\end{lemma}
We state a lemma which will be useful to prove another result.
\begin{figure}
    \centering
    \includegraphics[width=0.9\linewidth]{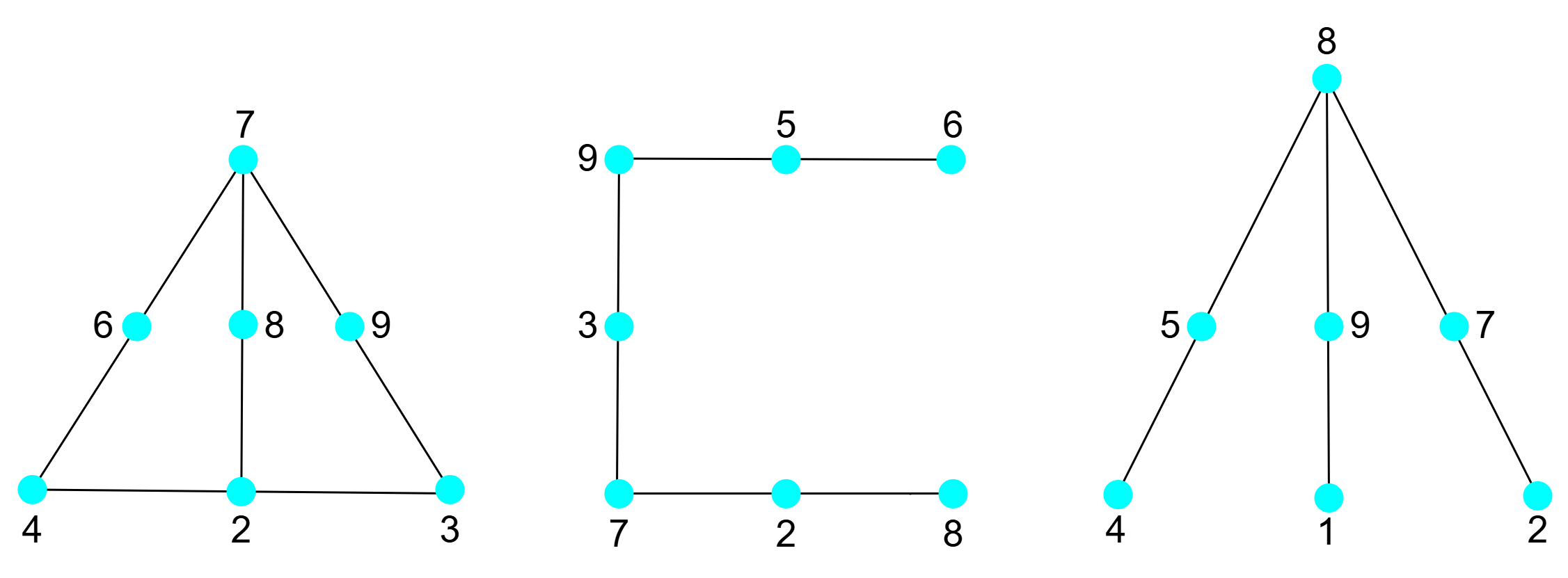}
    \caption{(Left) $Q$; (Center) $M \setminus \{1,4\}$; (Right) $M \setminus \{3,6\}$.}
    \label{fig:toevoegingen thesis}
\end{figure}
\begin{lemma} \label{lemma: third config hulplemma}
    For the matroid $Q$ with lines $\{4,6,7\}, \{2,7,8\}, \{3,7,9\}, \{2,3,4\}$, as in Figure \ref{fig:toevoegingen thesis} (Left), we have
    $$V_{\CCC(Q)} = V_Q \cup V_{Q(7)}.$$
\end{lemma}
\begin{proof} If $\gamma_7 = 0$, then $\gamma \in V_{\CCC(Q(7))} = V_{Q(7)}$, by Theorem~\ref{nil coincide}. Thus we can assume that $\gamma_7 \neq 0$. By Lemma \ref{lemma: third config loop matroid variety}, we may assume that $\gamma$ has no loops. 
\begin{mycases}
\case Suppose that $\gamma_4 \neq \gamma_7$ or $\gamma_3 \neq \gamma_7$ or $\gamma_2 \neq \gamma_7$. Without loss of generality, we may assume that $\gamma_4 \neq \gamma_7$. Perturb $\gamma|_{\{2,3,4,7,8,9\}}$ to $\widetilde{\gamma}|_{\{2,3,4,7,8,9\}} \in \Gamma_{Q \setminus \{6\}}$ using Theorem~\ref{nil coincide}, since $Q \setminus \{6\}$ is nilpotent and has no points of degree greater than two. Then define $\widetilde{\gamma}_6$ on $\widetilde{\gamma}_{47}$. Since $\gamma_4 \neq \gamma_7$, it follows that $\widetilde{\gamma}_6$ is a perturbation of $\gamma_6$. $\widetilde{\gamma} \in \Gamma_Q$, which implies that $\gamma \in V_Q$.
\case    
 Suppose that
    $\gamma_2= \gamma_3 = \gamma_4 = \gamma_7$. Denote the matroid on the ground set $\{2,3,4,6,7,8,9\}$ with $2=3=4=7$ and no lines by $O$. It is clear that $\gamma \in V_{\CCC(O)}.$ Since $V_{\CCC(O)} = V_O$, the proof finishes if we can prove that $V_O \subset V_Q$, as done in 
    Section \ref{appendix:proof of lemma pappus}.
    \end{mycases}
\end{proof}
\begin{lemma} \label{third lemma 4}\label{prop:decompo pascal(1)}
    Let $M$ denote the third configuration $9_3$ and $N=M \setminus \{1\}$. Then $$\VCN = V_N \cup V_{N(4)} \cup V_{N(7)} \cup V_{U_{2,8}}.$$
\end{lemma}
\begin{proof}
    Let $\gamma \in \VCN$. If $\gamma_4 = 0$, then it follows by Theorem~\ref{nil coincide} that $\gamma \in V_{\CCC(N(4))} = V_{N(4)}$, since $N \setminus \{4\}$ is nilpotent and has no points of degree greater than two. $M \setminus \{1,4\}$ is shown in Figure \ref{fig:toevoegingen thesis} (Center). A similar result holds if $\gamma_7 = 0$. Now assume that $\gamma_4 \neq 0 \neq \gamma_7$. By Lemma \ref{lemma: third config loop matroid variety}, we can assume that $\gamma$ has no loops. We consider two cases:
    \begin{mycases}
        \case{There is a point $p \in \{2,3,5,6,8,9\}$ such that $\gamma_{l_1} \wedge \gamma_{l_2} \neq 0$, where $\{l_1, l_2\} =\mathcal{L}_p$.} 
        
        \medskip

        {\bf Case 1.1.}
   If $p \in \{5,9\}$, we may assume without loss of generality that $p=5$. Then we can use Lemma \ref{lemma: third config hulplemma} to perturb $\gamma|_{[9] \setminus \{1,5\}}$ to $\widetilde{\gamma}|_{[9] \setminus \{1,5\}}$, since $\gamma$ has no loops. If we extend this collection by $\widetilde{\gamma}_5 = \widetilde{\gamma}_6 \widetilde{\gamma}_9 \wedge \widetilde{\gamma}_5 \widetilde{\gamma}_8$, then $\widetilde{\gamma} \in \Gamma_N$. By Lemma \ref{lemma: meet of 2 extensors arbitrary close}, $\widetilde{\gamma}$ is an arbitrary perturbation of $\gamma$, so $\gamma \in V_N.$

    \medskip
     
     {\bf Case 1.2.} If $p \in \{2,3,6,8\}$ we can use Theorem~\ref{nil coincide} to perturb $\gamma|_{M \setminus \{p\}}$ to  $\widetilde{\gamma}|_{M \setminus \{p\}} \in \Gamma|_{M \setminus \{p\}}$. Define $\widetilde{\gamma}_p = \widetilde{\gamma}_{l_1} \wedge \widetilde{\gamma}_{l_2}$, then $\widetilde{\gamma} \in \Gamma_N$. By Lemma \ref{lemma: meet of 2 extensors arbitrary close}, $\widetilde{\gamma}$ is an arbitrary perturbation of $\gamma$, so $\gamma \in V_N.$
     \medskip 
    \case{For all points $p \in \{2,3,5,6,8,9\}$, $\gamma_{l_1} \wedge \gamma_{l_2} = 0$, where $\{l_1, l_2\} =\mathcal{L}_p$.} Using Lemma \ref{lemma: if gamma_i gamma_j wedge gamma_k gamma_l =0}, we can assume that $\gamma_i \neq \gamma_j$ for $i \in \{2,3,5,6,8,9\}, j \in [9] \setminus \{1\}, i \neq j$.   
Notice that $\rk(\gamma_l) = 3$ for all lines of $N$ except for possibly $l = \{4,6,7\}$. Because of Case 2, we have that $\gamma_{\{2,3,4\}} = \gamma_{\{2,7,8\}} = \gamma_{\{3,7,9\}} = \gamma_{\{4,5,8\}} = \gamma_{\{5,6,9\}}$, implying that the points in $\{2,3,4,5,6,7,8,9\}$ are on a line, thus $\gamma \in V_{U_{2,8}}.$ This concludes the proof.
    \end{mycases}
\end{proof}
\begin{lemma} \label{third lemma 5}
    Let $N=M \setminus \{3\}$. Then $$\VCN = V_N \cup V_{N(6)} \cup V_{N(8)} \cup V_{N(6,8)} \cup V_{U_{2,8}}.$$
\end{lemma}
\begin{proof}
    Let $\gamma \in \VCN$. If $\gamma_6 = 0$, then by Example \ref{ex: decomposition circuit variety 3 concurrent lines}, $\gamma \in V_{\CCC(N(6))} = V_{N(6)} \cup V_{N(6,8)}$. A similar result holds if $\gamma_8=0$. Now assume that $\gamma_6 \neq 0 \neq \gamma_8$. By Lemma \ref{lemma: third config loop matroid variety}, we may assume that $\gamma$ has no loops at all. We consider two cases:
    \begin{mycases}
        \case There is a point $p \in \{1,2,4,5,7,9\}$ such that $\gamma_{l_1} \wedge \gamma_{l_2} \neq 0$, where $\{l_1, l_2\} =\mathcal{L}_p$. In this case, we can use Proposition~\ref{prop:decompo pascal(1)} to perturb $\gamma|_{[9] \setminus \{p,3\}}$ to  $\widetilde{\gamma}|_{[9] \setminus \{p,3\}} \in \Gamma|_{M \setminus \{p,3\}}$. Define $\widetilde{\gamma}_p = \widetilde{\gamma}_{l_1} \wedge \widetilde{\gamma}_{l_2}$. Then $\widetilde{\gamma} \in \Gamma_N$, so by Lemma \ref{lemma: meet of 2 extensors arbitrary close}, $\Tilde{\gamma} \in V_N.$
        \medskip
    \case{For all points $p \in \{1,2,4,5,7,9\}$, $\gamma_{l_1} \wedge \gamma_{l_2} = 0$, where $\{l_1, l_2\} =\mathcal{L}_p$.} 
    We can assume by Lemma \ref{lemma: if gamma_i gamma_j wedge gamma_k gamma_l =0} that $\gamma_i \neq \gamma_j$ for $i\in \{1,2,4,5,7,9\}, j \in [9] \setminus \{3\}$.
Note that $\rk(\gamma_l) = 3$ for all lines. Because of Case 2, we have that $\gamma_{\{1,2,6\}} = \gamma_{\{1,8,9\}} = \gamma_{\{2,7,8\}} = \gamma_{\{4,5,8\}} = \gamma_{\{4,6,7\}}= \gamma_{\{4,5,8\}}$, implying that the points in $\{1,2,4,5,6,7,8,9\}$ lie on a line, thus $\gamma \in V_{U_{2,8}}.$ This concludes the proof.
    \end{mycases}
\end{proof}

\begin{lemma} \label{third lemma 6}\label{lemma:decompo third D}
The following inclusion holds:
    $$V_{\CCC(D)} \subset \cup_{i=1}^6 V_{\CCC(A_i)} \cup V_B \cup_{i=1}^2 V_{\CCC(C_i)} \cup_{i=1}^9 V_{\pi^i_M} \cup V_{U_{2,9}} \cup_{i=1}^9 V_{\CCC(M(i))}.$$
\end{lemma}
The proof of this lemma, which only uses the decomposition strategy, is given in Section \ref{appendix:proof of lemma pappus} on page \pageref{inc V_D third}. 

\medskip

Combining Equation \eqref{eq: first decomposition for third config 93} with Equation \eqref{eq:inc V_B} and Lemma \ref{third lemma 1}, \ref{third lemma 2}, \ref{third lemma 3}, \ref{third lemma 4}, \ref{third lemma 5}, \ref{third lemma 6}, we have proven the following decomposition for $\VCM$:
$$
    \VCM = V_M \cup_{i=1}^6 V_{A_i} \cup V_B \cup_{i=1}^9V_{\pi_M^i} $$ $$ \cup_{i=1}^9 \bigl( \cup_{j=1,2} V_{M(i,k_j^i)} \bigr)  \cup V_{M(3,6,8)}  \cup V_{U_{2,9}} \cup_{i=1}^9 V_{M(i)}.
$$
Since all the matroids except $M$ appearing in this decomposition are solvable, it follows from Theorem~\ref{nil coincide} that their matroid varieties are irreducible. For the irreducibility of $V_M$, we use Table~4.1 from \cite{corey2025singular}. The proof of the irredundancy of each component is deferred to Section \ref{appendix:proof of lemma pappus}.
\medskip

Thus, summarizing the results of this chapter:
\begin{theorem} \label{thm decompo third 9_3}
    The following is an irredundant, irreducible, decomposition for the point-line configuration $M$ of the third configuration $9_3$, where we use the same notation for the point-line configurations as previously in this chapter.
    $$\VCM = V_M \cup_{i=1}^6 V_{A_i} \cup V_B \cup_{i=1}^9V_{\pi_M^i} $$     \begin{equation}
\cup_{i=1}^9 \bigl( \cup_{j=1,2} V_{M(i,k_j^i)} \bigr)  \cup V_{M(3,6,8)}  \cup V_{U_{2,9}} \cup_{i=1}^9 V_{M(i)}.
\end{equation}
\end{theorem}
\numberwithin{theorem}{chapter}

\chapter{The decomposition of the circuit variety of the Pappus configuration and the third configuration $9_3$ revisited} \label{sec: decompo revisited}
In Chapter \ref{sec: pascal & pappus}, we give a proof of the irreducible decomposition of the circuit variety of the Pascal configuration and of the Pappus configuration with one loop, based on the fact that $M(9)$ has some points of degree two. In Chapter \ref{Decomposition of the Circuit Variety of the third configuration 93}, we give a proof of an irreducible decomposition of the third configuration $9_3$ which uses the decomposition strategy in the first step, but instead of recursively applying this strategy for each minimal matroid encountered, we use a similar method as in Chapter \ref{sec: pascal & pappus}. However, in this chapter we give a proof which only uses the decomposition strategy. One can notice that the proofs based on the decomposition strategy are remarkably more extensive than the ones based on our previous methods. However, the decomposition strategy has the advantage that it also works for matroids without points of degree at most two.
\section{Pappus configuration}
Denote the Pappus configuration by $M$. We will give an alternative proof of Theorem \ref{prop: Pappus M(9)}, which only uses the decomposition strategy as outlined in Subsection \ref{subsec: decomposition strategy}. 

\medskip
{\bf Proof of Proposition \ref{prop: Pappus M(9)}}
The minimal matroids over $M(9)$ are the following. An automorphism means a hypergraph automorphism of $M(9).$ The point 9 is a loop in each point-line configuration in this list.
\begin{itemize} 
\item The matroid $A_1$ obtained by identifying $1=2=7$, and four other matroids $A_i$ derived from $A_1$ via automorphisms, as in Figure \ref{fig:Gemengd Pappus A_i B_i C_i} (First Left).
\item The matroid $A'$ obtained by identifying $2=3=4$, and lines $\{1,5,7\},\{8,1,6\},\{2,5,6\}.$
\item The matroid $B_1$ obtained by identifying the points in $\{1,2,6\}$ and placing them on the line with 4,5 and 7, and the three other matroids $B_i$ formed by automorphism, as in Figure \ref{fig:Gemengd Pappus A_i B_i C_i} (Second Left). 
\item The matroid $B'_1$ obtained by identifying the points in $\{2,4,6\}$ and placing them on the line with the points $\{1,3,8\}$ and a matroid $B_2'$ formed by automorphism.
\item The six matroids obtained by identifying two points of degree two which are on the same line. We call these matroids $C_i$, as in Figure \ref{fig:Gemengd Pappus A_i B_i C_i} (Second Right).
\item The matroid $D_1$ with $1=2$, $4=5$, and three other matroids formed by automorphisms, as in Figure \ref{fig:Gemengd Pappus A_i B_i C_i} (First Right). 
\item The matroid $D_1'$ with $1=7,3=4$ and lines $\{1,2,3\},\{8,1,6\},\{3,5,6\}$, and another matroid obtained by automorphism.
\item The matroid $\widetilde{I_1} = I_9$ as in Figure \ref{fig:Pappus, Pappus, I_i, J_i} (Center) and another matroid $\widetilde{I_2}$ formed by automorphism where 9 is still a loop.
\item $M(9,1)$ as in Figure \ref{fig:Pappus, K_i, L__i, N} (Center), and $M(9,4)$.\end{itemize}
\begin{figure}
    \centering
    \includegraphics[width=0.9\linewidth]{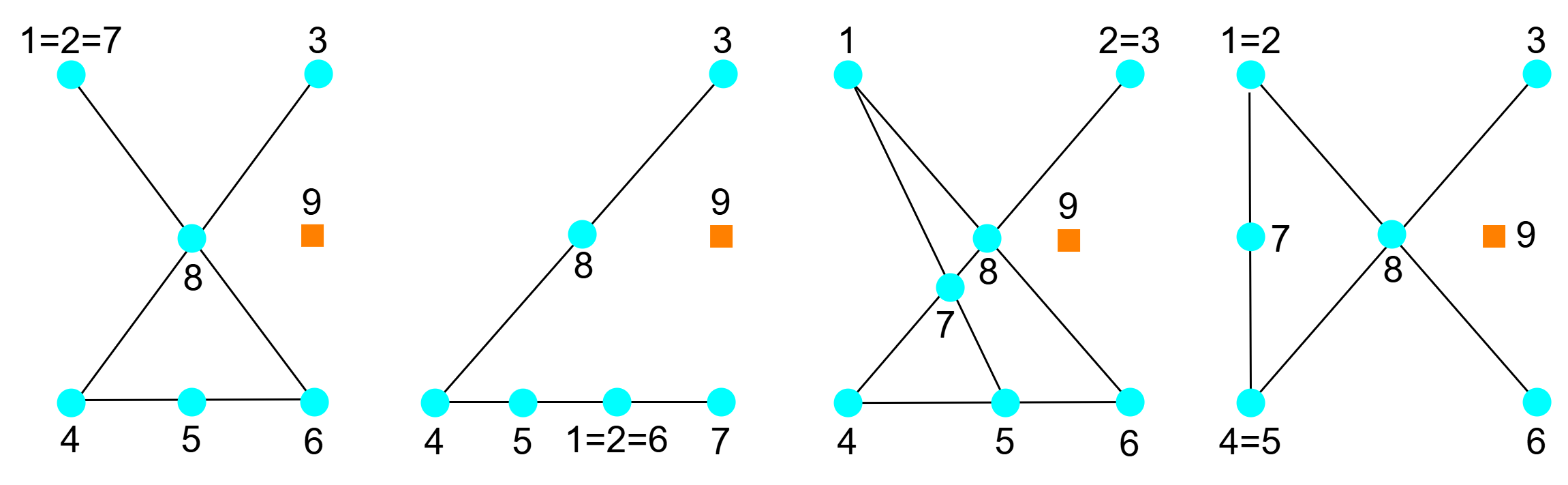}
    \caption{(First Left) $A_1$, (Second Left) $B_1$, (Second Right) $C_1$, (First Right) $D_1$.}
    \label{fig:Gemengd Pappus A_i B_i C_i}
\end{figure}
The point-line configurations $A_i, A', B_i, B', D_i$ and $D_i'$ are nilpotent with no points of degree greater than two and $C_i$ is quasi-nilpotent and has no points of degree greater than two. So from Theorem~\ref{nil coincide}, it follows that 
   $$ V_{\mathcal{C}(M(9))} = V_{M(9)} \cup_{i=1}^5 V_{A_i} \cup V_{A'} \cup_{i=1}^4 V_{B_i} \cup_{i=1}^2 V_{B_i'} $$
    \begin{equation} \label{eq: pappus decomposition M(9) first}
    \cup_{i=1}^6 V_{C_i} \cup_{i=1}^4 V_{D_i} \cup_{i=1}^2 V_{D_i'} \cup_{i=1}^2 V_{\mathcal{C}(\widetilde{I_i})} \cup V_{\CCC(M(9,1))} \cup V_{\CCC(M(9,4))} \cup V_{U_{2,9}(9)}.
\end{equation}
The following inclusions hold:
\begin{enumerate}
    \item For $i \in \{1, \ldots, 5\}$: $V_{A_i} \subset V_{M(9)}$.
    \item $V_{A'} \subset V_{M(9)}$.
    \item For $i \in \{1, \ldots, 4\}$: $V_{B_i} \subset V_{M(9)}$.
    \item For $i \in \{1, 2\}$: $V_{B_i'} \subset V_{M(9)}$.
    \item For $i \in \{1, \ldots, 6\}$: $V_{C_i} \subset V_{M(9)}$.
    \item For $i \in \{1, \ldots, 4\}$: $V_{D_i} \subset V_{M(9)}$.
    \item For $i \in \{1, 2\}$: $V_{D_i'} \subset V_{M(9)}$.
    \item $V_{\CCC(\widetilde{I_i})} \subset V_{M(9)}  \cup V_{U_{2,9}} \cup_{j=1,4} V_{M(9,j)} \cup V_{M(9,1,4)}.$
    \item $V_{\CCC(M(9,1))} = V_{M(9,1)} \cup V_{M(9,1,4)}.$
\end{enumerate}
The proofs of the statements in 1,3,5,6 and 8 are technical, therefore we defer them to Section \ref{appendix:proof of lemma pappus}. The proofs of the other statements in 2, 4, 7, are completely analogous to the ones in respectively 1,3,6. The statement in 9 immediately follows from Example \ref{ex: decomposition circuit variety 3 concurrent lines}.
Thus we have found a decomposition of $V_{M(9)}$:
\begin{equation*} 
V_{\CCC(M(9))} =V_{M(9)} \cup V_{U_{2,9}(9)} \cup V_{M(9,1)} \cup V_{M(9,4)} \cup V_{M(9,1,4)}.
\end{equation*}
\section{Third configuration $9_3$}
In this section, we revisit the decomposition of the circuit variety of the third configuration $9_3$. In Chapter \ref{Decomposition of the Circuit Variety of the third configuration 93}, we give a proof of Theorem \ref{thm decompo third 9_3} which uses the decomposition strategy only in the first step and in the proof of Lemma \ref{lemma:decompo third D}. We can also give a proof which uses the decomposition strategy in each step. We use the same notation for each matroid as in Chapter \ref{Decomposition of the Circuit Variety of the third configuration 93}.

Consider the following lemma.
\begin{lemma} \label{lemma: inclusions third 9_3}
    The following inclusions hold.
    \begin{enumerate}
        \item \begin{equation} \label{eq: lemma 1 for third config}
V_{\mathcal{C}(A_i)} \subset V_{A_i} \cup_{j=1}^9 V_{\pi_M^j} \cup V_{U_{2,9}} \cup_{i=1}^9 V_{\CCC(M(i))}.\end{equation}
\item     \begin{equation} \label{eq: lemma 3 third config}
        V_{\CCC(C_i)} \subset V_M \cup V_B \cup_{j=1}^9 V_{\pi_{M}^j} \cup V_{U_{2,9}} \cup_{j=1}^9 V_{\CCC(M(j))}.
\end{equation}
\item \begin{equation} \label{eq: lemma 4 third config}
    V_{\CCC(D)} \subset \cup_{i=1}^6 V_{\CCC(A_i)} \cup V_{B} \cup_{i=1}^2 V_{\CCC(C_i)} \cup_{i=1}^9 V_{\pi_M^i} \cup V_{U_{2,9}} \cup_{i=1}^9 V_{\CCC(M(i))}.
\end{equation}

\item Let $i$ be a point in $\{1,2,4,5,7,9\}.$ Denote $k_1^i$ and $k_2^i$ for the two points for which there is no line in $\mathcal{L}_M$ containing the pair $\{i,k_1^i\}$ or $\{i,k_2^i\}$. Then  
\begin{equation} \label{eq: lemma 5 third config}
    V_{\CCC(M(i))} \subset V_{M(i)} \cup V_B \cup_{j=1,2} V_{M(i,k_1^i)} \cup_{j=1}^9 V_{\pi_{M}^j} \cup V_{U_{2,9}}. 
\end{equation}
\item Let $i$ be a point in $\{3,6,8\}$. Denote the other two points in $\{3,6,8\}$ as $k_1^i$ and $k_2^i$. Then
\begin{equation} \label{eq: lemma 6 third config}
V_{\CCC(M(i))} \subset V_{M(i)} \cup_{j=1,2} V_{M(i,k_j^i)} \cup V_{M(3,6,8)} \cup_{j=1}^9 V_{\pi_{M}^j} \cup V_{U_{2,9}}.
\end{equation}
    \end{enumerate}
\end{lemma}
A proof of this lemma which only uses the decomposition strategy is given in Section \ref{appendix:proof of lemma pappus}. It is clear that this proves Theorem
\ref{thm decompo third 9_3}.
\chapter{Conclusion and Further Research}

In this thesis, we have found an irreducible decomposition for the circuit varieties of cactus configurations, up to redundancy, and an irreducible decomposition for the circuit variety of the Pascal configuration and the third configuration $9_3$. During these proofs, we also decomposed the circuit varieties of many other point-line configurations. Since the cactus configurations are a family of configurations, Theorem \ref{thm: decomposition of circuit variety of a cactus} enables us to decompose the circuit variety of various point-line configurations immediately, including configurations with numerous points and lines. For these configurations, algorithmic methods such as the decomposition strategy are not suitable, as for instance illustrated in Example \ref{example cactus decompo}. This example shows a cactus configuration which has 132 minimal matroids and only a few
are related by automorphisms. This implies that the decomposition strategy in this case results in a computationally heavy situation.

\medskip

For some matroids, including the Pascal configuration, the Pappus configuration with one loop and the third configuration $9_3$, we have given a proof which is much shorter than the one based on the decomposition strategy. Our technique applies to many other point-line configurations as well. If the point-line configuration has \enquote{enough} points of degree two, then we can immediately apply our method. If not, as for the third configuration $9_3$, we can use the decomposition strategy in a first step: we identify the minimal matroids of our point-line configuration. For these point-line configurations, if they have \enquote{enough} points of degree two, we can apply our developed technique, else we use the decomposition strategy again. We can repeat this process recursively. This makes the proofs significantly faster. To illustrate this, the thesis gives a proof of the decomposition of the circuit variety of the third configuration $9_3$ and the Pappus matroid with one loop using the decomposition strategy as well, and one can notice that these proofs are much longer and involve many other point-line configurations. Therefore, having introduced this technique, we are able to decompose the circuit varieties of many new matroids for which the computation using only the decomposition strategy is too challenging.

\medskip

An evident question extending our work is a generalization to other dimensions. As mentioned, point-line configurations have a natural generalization to larger dimensions, called $n$-paving matroids. This year, there appeared a paper \cite{liwski2025efficient} which studies these decompositions in new dimensions. An interesting question would be the following: 
\begin{question}
Is it possible to extend the idea behind our more efficient method for the decomposition of circuit varieties to other dimensions?
\end{question}
Our technique cannot be carried over directly to other paving matroids, since it implicitly uses some properties of $\mathbb{P}^2$.

\medskip

On the other hand, we also find a generating set for the matroid ideal of cactus configurations and the Pascal and Pappus configurations. Although it may seem that these results are limited since they involve specific types of matroids, the problem of finding a generating set for matroid ideals is notoriously difficult, since few methods to construct explicit generators are known. The Pascal and Pappus configuration are up to my best knowledge the only known cases where the Grassmann-Cayley ideal and the lifting ideal are both irredundant in the generating set. An obvious question that can be investigated is the following. 
\begin{question}
    Identify new matroids for which $$I_M = \sqrt{I_{\CCC(M)}+G_M+I_M^\textup{lift}}.$$ 
\end{question}
However, we have a very strong suspect that the matroid ideal of among others the point-line configuration associated to the affine plane, as in \cite{liwskimohammadialgorithmforminimalmatroids}, is not generated by only circuit, Grassmann-Cayley and lifting polynomials. However, we do not have techniques yet to construct new types of polynomials for the generating set. This leads to the following question. 
\begin{question}
    Develop new methods to construct other ideals irredundant for the generating set of a matroid ideal.
\end{question}
We can also try to generalize our previous results to new dimensions. Instead of point-line configurations, we can study paving matroids. Moreover, the circuit and lifting ideal are well-defined for higher dimensions. However, the construction of Grassmann-Cayley ideal uses specific properties of rank three, and therefore it is not obvious how to redefine it for new dimensions.
\begin{question}
    What is a good definition for $G_M$ in higher dimensions, and are there paving matroids for which $$I_M = \sqrt{I_{\CCC(M)}+G_M+I_M^{\textup{lift}}}$$ holds?
\end{question}
A last question comes up when we look at the explicit set of generating polynomials for the Pascal and Pappus matroid, as in Remark \ref{remark pascal} and \ref{remark pappus}. The number of lifting polynomials in the generating set is remarkably larger than that of the circuit and Grassmann-Cayley polynomials. 
\begin{question}
    Is there a smaller generating set which generates an ideal $I$ for which $I_M^{\textup{lift}} \subset I \subset I_M$?
    \end{question}
    We already used the graph ideal, defined in \cite{liwski2024pavingmatroidsdefiningequations}, instead of the lifting ideal to find a finite set of generators containing $I_M^{\textup{lift}}$, but for any practical applications, it would be very useful to find a smaller generating set.
\cleardoublepage
\appendix
\numberwithin{theorem}{chapter}
\chapter{Appendices}\label{ch:myappendix}
\section{Proof of Lemma \ref{lemma:finite epsilon to not create dependencies}}
We present a proof for Lemma \ref{lemma:finite epsilon to not create dependencies}.
\medskip

    If a set of $d$ vectors is linearly dependent, then all the $\min(n,d) \times \min(n,d)$-minors of the corresponding $n \times d$-matrix vanish. Define $$\phi: \CC^{n \times d} \to \CC^i: A \to (A_j)_{j=1, \ldots, i},$$
    where $A_j$ are all the possible $\min(n,d) \times \min(n,d)$-minors of $A$. Since the determinant is a continuous function for the induced topology by the norm, $\phi$ is continuous as well. Identifying a collection of vectors with a matrix, we can apply $\phi$ as well to a collection of $d$ vectors in $\CC^n$. Assume $\{\gamma_1, \ldots, \gamma_d\}$ is a linearly independent set. 
    Then $\phi(\{\gamma_1, \ldots, \gamma_d\}) = \overline{x} \neq (0, \ldots, 0)$. Denote $x = \lVert\overline{x}\rVert$. There is a $\delta$, such that for any collection of vectors $\Tilde{\gamma}$: \begin{equation} \label{eq: condition on perturbation} \text{If } \lVert\widetilde{\gamma} - \gamma\rVert < \delta, \text{ then } \lVert\phi(\widetilde{\gamma}) - (0, \ldots, 0)\rVert  > x/2.\end{equation} In that case, $\{\widetilde{\gamma}_1, \ldots, \widetilde{\gamma}_n\}$ is a linearly independent set of vectors as well. Consider all the possible linearly independent subsets of vectors in $\gamma$. For each set, define a corresponding $\delta > 0$, such that the property in Equation~\eqref{eq: condition on perturbation} holds. Let $\epsilon$ be the minimal value of the set containing every  such $\delta$. Then if one perturbs $\gamma$ to $\widetilde{\gamma}$ such that $\lVert\gamma- \widetilde{\gamma}\rVert < \delta$, there are no additional dependencies created.
\section[Detailed computations for Example 3.19]{Detailed computations for Example \ref{ex: cactus}}\label{appendix: example}
This section presents some detailed computations for Example \ref{ex: cactus}.
\medskip

The vectors $\gamma_3$ and $\gamma_2$ have the following explicit expressions:
\[ \gamma_3=
\left( 
\begin{array}{c}
- l_1 y_{11} + l_1 y_{12} - y_{11} + y_{12} \\ 
- l_1 y_{8} y_{11} + l_1 y_{8} y_{12} - y_{11} + y_{12} \\ 
- l_1 y_{8} z_{11} + l_1 y_{12} z_{11} + l_1 y_{8} z_{12} - l_1 y_{11} z_{12} + y_{12} z_{11} - y_{11} z_{12} - z_{11} + z_{12}
\end{array} 
\right)
\]
and 
$$
\scalebox{0.95}{$
\begin{aligned}
\gamma_2 = &\;
\left(
\begin{array}{c}
- l_1 y_{11} z_{9} + l_1 y_{12} z_{9} + l_1 y_{11} z_{10} - l_1 y_{12} z_{10} - y_{11} z_{9} + y_{12} z_{9} + y_{11} z_{10} - y_{12} z_{10} \\
-l_1y_{11} y_{10}z_{9}+l_1y_{10}y_{12}z_{9} -l_1y_{9}y_{8}z_{11} + l_1 y_{8} y_{10} z_{11} + l_1 y_{9} y_{12} z_{11} - l_1 y_{10} y_{12} z_{11} \\
- l_1 y_{8} z_{9} z_{11} +l_1 y_{12} z_{9} z_{11} + l_1 y_{8} z_{11} z_{10} - l_1 y_{12} z_{11} z_{10} + l_1 y_{8} z_{9} z_{12}
\end{array}
\right) \\[1em]
&+
\left(
\begin{array}{c}
0 \\
l_1 y_{9} y_{11} z_{10} - l_1 y_{9} y_{12} z_{10} + l_1y_{9}y_8z_{12}-l_1 y_{8} y_{11} z_{12}-l_1y_{9} y_{10} z_{12} + l_1 y_{11} y_{10} z_{12} \\
- l_1 y_{11} z_{9} z_{12} - l_1 y_{8} z_{10} z_{12} +l_1y_{11} z_{10} z_{12} + y_{12} z_{9} z_{11} - y_{12} z_{11} z_{10} - y_{11} z_{9} z_{12}
\end{array}
\right) \\[1em]
&+
\left(
\begin{array}{c}
0 \\
-y_{11}y_{10} z_{9}+y_{10} y_{12} z_{9} + y_{9} y_{12} z_{11} - y_{10} y_{12} z_{11} + y_{9} y_{11} z_{10}-y_{9} y_{12} z_{10}- y_{9} y_{11} z_{12} \\
+ y_{11} z_{10} z_{12} - z_{9} z_{11} + z_{11} z_{10} + z_{9} z_{12} -z_{10} z_{12}
\end{array}
\right) \\[1em]
&+
\left(
\begin{array}{c}
0 \\
y_{11} y_{10} z_{12} - y_{9} z_{11}+y_{10} z_{11} + y_{9} z_{12} - y_{10} z_{12} \\
0
\end{array}
\right)
\end{aligned}
$}
$$

\normalsize
The quadratic equation in Example \ref{ex: cactus} is the following:

\[
\begin{aligned}
& -l_1^2 y_{8}^2 z_9 z_{11}
+ l_1^2 y_8 y_{12} z_9 z_{11}
+ l_1^2 y_{11} y_{10} z_9 z_8
- l_1^2 y_{10} y_{12} z_9 z_8 
 + l_1^2 y_9 y_8 z_{11} z_8 \\
& - l_1^2 y_8 y_{10} z_{11} z_8
- l_1^2 y_9 y_{12} z_{11} z_8
+ l_1^2 y_{10} y_{12} z_{11} z_8 
+ l_1^2 y_{8}^2 z_{11} z_{10}
- l_1^2 y_8 y_{12} z_{11} z_{10} \\ &
- l_1^2 y_9 y_{11} z_8 z_{10}
+ l_1^2 y_9 y_{12} z_8 z_{10} 
- l_1^2 y_{11} y_8 z_9 z_{12}
+ l_1^2 y_{8}^2 z_9 z_{12}
+ l_1^2 y_9 y_{11} z_8 z_{12} \\ &
- l_1^2 y_9 y_8 z_8 z_{12} 
- l_1^2 y_{11} y_{10} z_8 z_{12}
+ l_1^2 y_8 y_{10} z_8 z_{12}
+ l_1^2 y_{11} y_8 z_{10} z_{12}
- l_1^2 y_{8}^2 z_{10} z_{12} \\
&+ l_1\, y_8 y_{12} z_9 z_{11}
+ l_1\, y_{11} y_{10} z_9 z_8
- l_1\, y_{10} y_{12} z_9 z_8
- l_1\, y_9 y_{12} z_{11} z_8 
+ l_1\, y_{10} y_{12} z_{11} z_8 \\
&- l_1\, y_8 y_{12} z_{11} z_{10}
- l_1\, y_9 y_{11} z_8 z_{10}
+ l_1\, y_9 y_{12} z_8 z_{10}
- l_1\, y_{11} y_8 z_9 z_{12}
+ l_1\, y_9 y_{11} z_8 z_{12} \\
&- l_1\, y_{11} y_{10} z_8 z_{12}
+ l_1\, y_{11} y_8 z_{10} z_{12}
+ l_1\, y_{11} y_{10} z_9
- l_1\, y_{10} y_{12} z_9
+ l_1\, y_9 y_8 z_{11}
- l_1\, y_8 y_{10} z_{11} \\
&- l_1\, y_9 y_{12} z_{11}
+ l_1\, y_{10} y_{12} z_{11}
- 2l_1\, y_8 z_9 z_{11}
+ l_1\, y_{12} z_9 z_{11}
+ l_1\, y_9 z_{11} z_8
- l_1\, y_{10} z_{11} z_8 \\
&- l_1\, y_9 y_{11} z_{10}
+ l_1\, y_9 y_{12} z_{10}
+ 2l_1\, y_8 z_{11} z_{10}
- l_1\, y_{12} z_{11} z_{10}
+ l_1\, y_9 y_{11} z_{12}
- l_1\, y_9 y_8 z_{12} \\
&- l_1\, y_{11} y_{10} z_{12}
+ l_1\, y_8 y_{10} z_{12}
- l_1\, y_{11} z_9 z_{12}
+ 2l_1\, y_8 z_9 z_{12}
- l_1\, y_9 z_8 z_{12}
+ l_1\, y_{10} z_8 z_{12} \\
&+ l_1\, y_{11} z_{10} z_{12}
- 2l_1\, y_8 z_{10} z_1
+ y_{11} y_{10} z_9
- y_{10} y_{12} z_9
- y_9 y_{12} z_{11}
+ y_{10} y_{12} z_{11} \\ &
+ y_{12} z_9 z_{11} 
- y_9 y_{11} z_{10}
+ y_9 y_{12} z_{10}
- y_{12} z_{11} z_{10}
+ y_9 y_{11} z_{12}
- y_{11} y_{10} z_{12} \\
& - y_{11} z_9 z_{12}
+ y_{11} z_{10} z_{12}
+ y_9 z_{11}
- y_{10} z_{11}
- z_9 z_{11}
+ z_{11} z_{10} \\
& - y_9 z_{12}
+ y_{10} z_{12}
+ z_9 z_{12}
- z_{10} z_{12}=0.
\end{aligned}
\]

So \[
\begin{aligned}
A = &- y_{8}^2 z_9 z_{11}
+  y_8 y_{12} z_9 z_{11}
+ y_{11} y_{10} z_9 z_8
-  y_{10} y_{12} z_9 z_8 
+  y_9 y_8 z_{11} z_8
-  y_8 y_{10} z_{11} z_8 \\ &
-  y_9 y_{12} z_{11} z_8
+  y_{10} y_{12} z_{11} z_8 
+  y_{8}^2 z_{11} z_{10}
-  y_8 y_{12} z_{11} z_{10}
-  y_9 y_{11} z_8 z_{10}
+  y_9 y_{12} z_8 z_{10} \\
& -  y_{11} y_8 z_9 z_{12}
+  y_{8}^2 z_9 z_{12}
+  y_9 y_{11} z_8 z_{12}
-  y_9 y_8 z_8 z_{12} 
-  y_{11} y_{10} z_8 z_{12}
+  y_8 y_{10} z_8 z_{12} \\ &
+ y_{11} y_8 z_{10} z_{12}
-  y_{8}^2 z_{10} z_{12} \\
\end{aligned} \]

\[
\begin{aligned}
B = &  y_8 y_{12} z_9 z_{11}
+  y_{11} y_{10} z_9 z_8
-  y_{10} y_{12} z_9 z_8
-  y_9 y_{12} z_{11} z_8
+  y_{10} y_{12} z_{11} z_8 \\
&-  y_8 y_{12} z_{11} z_{10}
-  y_9 y_{11} z_8 z_{10}
+  y_9 y_{12} z_8 z_{10}
-  y_{11} y_8 z_9 z_{12}
+  y_9 y_{11} z_8 z_{12} \\
&-  y_{11} y_{10} z_8 z_{12}
+  y_{11} y_8 z_{10} z_{12}
+  y_{11} y_{10} z_9
-  y_{10} y_{12} z_9
+  y_9 y_8 z_{11}
-  y_8 y_{10} z_{11} \\
&-  y_9 y_{12} z_{11}
+  y_{10} y_{12} z_{11}
- 2 y_8 z_9 z_{11}
+  y_{12} z_9 z_{11}
+  y_9 z_{11} z_8
-  y_{10} z_{11} z_8 \\
&-  y_9 y_{11} z_{10}
+  y_9 y_{12} z_{10}
+ 2 y_8 z_{11} z_{10}
-  y_{12} z_{11} z_{10}
+  y_9 y_{11} z_{12}
-  y_9 y_8 z_{12} \\
&-  y_{11} y_{10} z_{12}
+  y_8 y_{10} z_{12}
-  y_{11} z_9 z_{12}
+ 2 y_8 z_9 z_{12}
-  y_9 z_8 z_{12}
+  y_{10} z_8 z_{12} \\
&+  y_{11} z_{10} z_{12}
- 2 y_8 z_{10} z_1
\end{aligned}
\]
\[
\begin{aligned}
C = & y_{11} y_{10} z_9
- y_{10} y_{12} z_9
- y_9 y_{12} z_{11}
+ y_{10} y_{12} z_{11}
+ y_{12} z_9 z_{11} 
- y_9 y_{11} z_{10} \\
&
+ y_9 y_{12} z_{10}
- y_{12} z_{11} z_{10}
+ y_9 y_{11} z_{12}
- y_{11} y_{10} z_{12} 
- y_{11} z_9 z_{12}
+ y_{11} z_{10} z_{12}\\
&
+ y_9 z_{11}
- y_{10} z_{11}
- z_9 z_{11}
+ z_{11} z_{10}  - y_9 z_{12}
+ y_{10} z_{12} 
+ z_9 z_{12}
- z_{10} z_{12}.
\end{aligned}
\]
\section{Identifying redundant and irredundant matroid varieties}  \label{appendix:proof of lemma pappus}
In this section, we will develop methods to solve the following problems, as in \cite{Connectednessandcombinatorialinterplayinthemodulispaceoflinearrangements,liwskimohammadialgorithmforminimalmatroids}.
\begin{question} \label{question appendix}
    Consider two matroids $M$ and $N$ on the same ground set, is $V_N \subset V_M$?
\end{question}
\begin{question} \label{quest irredundant}
    Given a decomposition of a circuit variety of a rank three matroid $M$: $$V_{\CCC(M)} = \cup_{i=1}^n V_{X_i},$$
    which components are irredundant?
\end{question}

We first introduce some definitions. 
\begin{definition}
    Consider a matroid $M$ of rank three on $[d]$. The {\em projective realization space} $\mathcal{R}(M)$ is formed by all the
collections of points $\gamma = \{\gamma_i:i \in [d]\} \subseteq \mathbb{P}^2$ for which the following condition holds:
$$\{\gamma_{i_1}, \ldots, \gamma_{i_k}\} \text{ is linearly dependent} \iff \{i_1, \ldots, i_k\} \text{ is a dependent set of } M.$$
\end{definition}
\begin{definition}
   Consider a matroid $M$. We define the {\em moduli space} as the quotient of this realization space $\mathcal{R}(M)$ by
the action of the projective general linear group $\text{PGL}_3(\CC) = \text{GL}_3(\CC)/\CC^*.$
\end{definition}
We have introduced this moduli space, since if $\{1,2,3,4\}$ is a circuit of $M$, then we can find a unique representative $\gamma \in \mathcal{R}(M)$ for each class in $\mathcal{M}(M)$ for which:
\begin{equation} \label{1234}
\{\gamma_1, \gamma_2, \gamma_3, \gamma_4\} = \begin{pmatrix}
    1 & 0 & 0 & 1 \\
    0 & 1 & 0 & 1 \\
    0 & 0 & 0 & 1 
\end{pmatrix}.
\end{equation}
Thus $\mathcal{R}(M)$ is the set of all collections of $d$ vectors $\gamma \in \mathcal{R}(M)$ for which Equation~\eqref{1234} holds. We need one final concept before we can outline a strategy to study Question \ref{question appendix}. We introduce $m$-perturbations as in \cite{Connectednessandcombinatorialinterplayinthemodulispaceoflinearrangements}. 
\begin{definition}
Consider a point-line configuration $M$ on the ground set $[d]$, and consider a line $l \in \mathcal{L}$ and a point $p \in l$. The point-line configuration on $[d]$ with set of lines:
\[
\mathcal{L}_{\widetilde{M}} = 
\begin{cases} 
(\mathcal{L} \setminus \{l\}) \cup \{l \setminus \{p\}\} & \text{if } |l| > 3, \\ 
\mathcal{L} \setminus \{l\} & \text{if } |l| = 3,
\end{cases}
\]
is called an \emph{elementary perturbation} of $M$ at $\{l, p\}$ and denoted with $\widetilde{M}$. 
An $m$-\emph{perturbation} of $M$ is a sequence of $m$ elementary perturbations,
\[
M_0, M_1, \dots, M_m = M,
\]
where $M_{i}$ is an elementary perturbation of $M_{i-1}$ for each $i \in [m]$ and $M_0$ is solvable.

\end{definition}
To prove the redundancy of matroid varieties in this section, we will use the following proposition. We refer to \cite{Connectednessandcombinatorialinterplayinthemodulispaceoflinearrangements} for a proof.
\begin{proposition}[Theorem~4.5 in \cite{Connectednessandcombinatorialinterplayinthemodulispaceoflinearrangements}] \label{prop appendix}
If a point-line configuration $M$ admits an $m$-perturbation, there exist an element $d \in \mathbb{N}$, an open subset $U \subset \mathbb{C}^d$ and $m$ polynomials $P_1, \dots, P_m$ such that
\[
\mathcal{M}(M) \cong U \cap V(P_1, \dots, P_m).
\]
Moreover, the proof of this theorem is constructive, so we can obtain an explicit description of the open set $U$ and the polynomials $P_1, \dots, P_m$ from it.
\end{proposition}
Using this result, we can outline the following strategy to prove that $V_M \subset V_N$:
\noindent
\begin{strategy} \leavevmode\par
\begin{itemize} 
    \item Describe $\mathcal{M}(N)$ choosing representatives satisfying Equation~\eqref{1234} and Proposition~\ref{prop appendix}.
    \item Describe $\mathcal{M}(M)$ choosing representatives satisfying Equation~\eqref{1234} and Proposition~\ref{prop appendix}.
    \item Pick an arbitrary element $\xi \in \mathcal{M}(M)$ satisfying Equation \ref{1234}. Prove that one can construct a sequence of elements in $\Gamma_N$ converging to $\xi$.
    \item Then $\xi$ lies in the Euclidean closure of the realization space of $M$. Since the Zariski topology is coarser than the Euclidean topology, $\xi$ lies in the Zariski closure of the realization space of $M$ as well, which is by definition the matroid variety. Thus $\Gamma_M \subset V_N$. Since $V_M$ is the closure of $\Gamma_M$ and $V_N$ is closed as well, this implies that $V_M \subset V_N$.
\end{itemize}
\end{strategy}
Now we address question \ref{quest irredundant}. If $V_{\CCC(M)} = \cup_{i=1}^n V_{X_i}$, we say that $V_{X_1}$ is irredundant in this decomposition if $V_{X_1} \not \subset V_{X_j}$ for any $j \neq i$. To prove such a statement, we first introduce a definition from \cite{dolgachev631} and some properties.
\begin{definition}
    Consider a topological space $X$ and a subset $V$. $V$ is called {\em constructible} if it can be written as a disjoint union of finitely many locally closed subsets.
\end{definition}
\begin{proposition}
    The realization space of a matroid $M$ is locally closed in the Zariski topology.
\end{proposition}
\begin{proof}
    As in \cite{dolgachev631}, a set is locally closed if and only if it can be written as $U_1 \setminus U_2$, where $U_1$ and $U_2$ are both closed. In the proof of Proposition~\ref{prop: expression gamma_M}, one writes $$\Gamma_M=\VCM \setminus V(J_M).$$ Since both sets are closed, the result follows.
\end{proof}
We introduced the notion of constructible sets to state the following proposition.
\begin{proposition} [Corollary~4.20 in \cite{sturmfelsbook}]
    For constructible sets, the Zariski closure equals the Euclidean closure.
\end{proposition}
\begin{corollary} \label{corol eucl zar coincide}
    For a matroid $M$, the Euclidean closure of the realization space of $M$ equals the Zariski closure.
\end{corollary}
We state one final remark before introducing a strategy to prove the irredundancy of a component in a decomposition.
    \begin{remark} \label{remark irredundancy}
If the circuit variety of a matroid $M$ can be decomposed as $$V_{\CCC(M)} = \cup_{i=1}^n V_{X_i},$$ we can assume without loss of generality that $X_i$ is realizable for each $i \in \{1, \ldots, n\}.$
To check the irredundancy of $X_1$, it immediately holds that $V_{X_1} \not \subset V_{X_j}$ if $X_1 \not \geq X_j$, using the dependency order. This can be seen as follows: assume $X_1 \not \geq X_j$, then there should be a dependent set $D \in \mathcal{D}(X_j) \setminus \mathcal{D}(X_1)$. Consider $\gamma \in \Gamma_{X_1}$. Then $\{\gamma_i:i \in D\}$ is linearly independent, thus $\gamma \not \in V_{X_j}$. 
\newline
So it suffices to find check that $V_{X_1} \not \subset V_{X_j}$ for point-line configurations $X_j \geq X_1$.
\end{remark}
We outline a strategy to prove that $V_{X_1}$ is not redundant in the decomposition $V_{\CCC(M)}=\cup_{i=1}^nV_{X_i}.$

\noindent \begin{strategy} \label{strategy no subset}
\leavevmode\par
\begin{itemize} 
    \item By Remark \ref{remark irredundancy}, we only have to check that $V_{X_1} \not \subset V_{X_j}$ for matroids $X_j \geq X_1$. Thus identify matroids $X_j$ for which $X_1 \leq  X_j$. For each matroids $X_j$, go through the following steps.
    \item Describe $\mathcal{M}(X_j)$ choosing representatives satisfying Equation~\eqref{1234} and Proposition~\ref{prop appendix}.
    \item Fix an explicit $\xi \in \mathcal{M}({X_1})$ satisfying Equation \ref{1234}. 
    \item Prove that one cannot construct a sequence of elements in $\Gamma_{X_j}$ converging to $\xi$.
    \item If such a sequence does not exist, then $\xi$ does not lie in the Euclidean closure of $X_j$. By Corollary \ref{corol eucl zar coincide}, $\xi$ does not lie in $V_{X_j}$.
\end{itemize}
\end{strategy}
\begin{remark}
    To prove that for a matroid $M$, $V_{M(i)}$ is not contained in $V_M$, we can assume by Lemma \ref{lemma: third config loop matroid variety} that the degree of $i$ is at least three. Assume $\{1,2,i\}, \{3,4,i\},\{5,6,i\} \in \mathcal{L}_i$. If one can construct a $\xi \in V_{M(i)}$ for which the lines $\xi_{12}, \xi_{34}$ and $\xi_{56}$ are not concurrent, then there is no sequence in $\Gamma_N$ converging to $\xi$, by Lemma \ref{lemma: meet of 2 extensors arbitrary close}.
\end{remark}
\begin{remark}
    One can notice that this strategy uses a projective framework, while we are working in the affine space. However, recalling Remark~\ref{notation thesis} and replacing the projective general linear group by linear transformations as defined in that remark, this approach will still work.
\end{remark}
\begin{remark}
    To avoid repetition, we will not explicitly mention it in each proof, but when there is stated that \enquote{an arbitrary element $\gamma \in \Gamma_N$ can be written as}, it is understood to be up to a linear transformation. Moreover, we will extensively make use of the last point in Notation \ref{notation thesis}.
\end{remark}
\subsection{Pascal Configuration}
Throughout this subsection, $M$ denotes the Pascal configuration as in Figure \ref{fig:pascal 1} (Left). In this section, we will prove that some matroid varieties are redundant in the decomposition of $\VCM$. On the other hand, we will also prove the irredundancy of some matroid varieties.

\subsubsection{Identifying redundancies}
In this subsection, we will prove some inclusions of matroid varieties appearing in the decomposition of the circuit variety of the Pascal matroid. More specifically, we prove Lemma~\ref{auxiliary lemma} and Lemma~\ref{lemma pascal V_B subset V_M}. This is useful to prove redundancies of some matroid varieties in a decomposition.
\medskip

\noindent{\bf Proof of lemma \ref{auxiliary lemma}} After a linear transformation, an arbitrary element $\gamma \in \Gamma_N$ can be written as \begin{equation}\label{matrix gama}\gamma = \begin{pmatrix} 
1 & 0 & 0 & 1 & 1+v & 1& 1+v & 1 \\
0 & 1 & 0 & 1 & 1 & 0 & 1+z & w+zw \\
0 & 0 & 1 & 1 & 1 & w & 1 & w 
\end{pmatrix},\end{equation}
where each column correspond to $\gamma_i$ for $i=\{9,4,5,6,2,3,7,8\}$ from left to right, and the minors associated to the bases are non-zero. From now on, we will write this more shortly as {\em the columns correspond to $\{9,4,5,6,2,3,7,8\}$}.
An arbitrary element $\xi \in \Gamma_{A}$ has the form $$\xi = \begin{pmatrix}
1 & 0 & 0 & 1 & 1 & 1 & 1 & 1 \\
0 & 1 & 0 & 1 & 0 & 0 & x & y \\
0 & 0 & 1 & 1 & 0 & 0 & 0 & 0
\end{pmatrix},$$
where, again, the columns correspond to $\{9,4,5,6,2,3,7,8\}$ and the minors corresponding to the bases are non-zero. Fix an arbitrary $\xi \in \Gamma_A$. For any $\epsilon$, define $\gamma_\epsilon\in \Gamma_N$ by substituting the following values into Equation~\eqref{matrix gama}:
\[
\left\{
    \begin{array}{ll}
        1+v = \frac{1}{\epsilon} \\
        w  =  \frac{\epsilon}{x}y \\
        1+z =  \frac{x}{\epsilon}
    \end{array}
\right..    \]
Then $$\gamma_\epsilon = \begin{pmatrix}
    1 & 0 & 0 & 1 & 1 & 1 & 1 & 1 \\
    0 & 1 & 0 & 1 & \epsilon & 0 & x & y \\
    0 & 0 & 1 & 1 & \epsilon & \frac{\epsilon y}{x} & \epsilon & \frac{\epsilon y}{x}
\end{pmatrix},$$
using the conventions of Notation~\ref{notation thesis}. As $\epsilon \to 0$, we find that $\gamma_\epsilon \to\xi$, which implies that $\xi \in V_N$. Since $\xi$ is chosen arbitrarily, it follows that $\Gamma_A \subset V_N$, and therefore $V_{A} \subset V_N$, since $V_N$ is a closed set. This completes the proof.

\medskip

\noindent {\bf Proof of Lemma \ref{lemma pascal V_B subset V_M}}
$B$ is a matroid on ground set $[9]$ such that all points coincide. By applying a linear transformation, we may assume that this common point is $(1,0,0)$.
       An arbitrary element of $\Gamma_M$ can be expressed as $$\begin{pmatrix}
           1 & 1 & 1 & 1 & 1+x & 1+y & 1+y+z & y-x-xz & 2x+1 \\
           0 & \epsilon & 0 & \epsilon & \epsilon & \epsilon & \epsilon + \epsilon z & 0 & (1+x) \epsilon \\
           0 & 0 & \eta & \eta & 0 & \eta & \eta & \eta & \eta x
       \end{pmatrix},$$
       where the columns from left to right correspond to $\{7,8,1,2,9,4,3,5,6\}$ and the minors associated to the bases are non-zero. Letting $x,y,\eta, \epsilon, z \to 0$, we see that the configuration in which all points coincide can be realized as a limit of configurations in $\Gamma_M$. It follows that $V_B \subset V_M$.

       \subsubsection{Identifying irredundancies}
       In this subsection, we prove that $V_{M(i)}$ is irredundant in the decomposition of $\VCM$ as in Theorem \ref{proposition: decomposition pascal configuration}, if $i \in \{7,8,9\}$.

\medskip

\noindent{\bf Irredundancy of $V_{M(i)}$}: 
Without loss of generality, we may assume that $i=7$. Consider $$\gamma = \begin{pmatrix}
    1 & 0 & 0 & 1 &  1 & 1 & 3 & 0 & 0\\
    0 & 1 & 0 & 1 &  2 & 0 & 1 & 1 & 0\\
    0 & 0 & 1 & 1  & 1 & 2 & 1 & 4 & 0\\
\end{pmatrix}.$$
Let the columns correspond to $\{9,8,3,6,1,5,2,4,7\}$ from left to right. Then $\gamma \in V_{M(7)}$. One can calculate that $$\gamma_9\gamma_8 \wedge \gamma_1 \gamma_5 = \begin{pmatrix}
    -1 \\ -4 \\ 0
\end{pmatrix}$$ and $$\gamma_9 \gamma_8 \wedge \gamma_2 \gamma_4 = \begin{pmatrix}
    -12 \\ -3 \\ 0
\end{pmatrix}.$$
Since $\wedge$ is continuous and $\{7,8,9\}, \{7,1,5\},\{7,2,4\}$ are lines of $M$, this implies that there is no perturbation $\widetilde{\gamma}$ of $\gamma$ for which $\widetilde{\gamma}_{\{7,8,9\}}, \widetilde{\gamma}_{\{7,1,5\}}$ and $\widetilde{\gamma}_{\{7,2,4\}}$ intersect non-trivially. Thus $\gamma \notin V_M$.
\subsection{Pappus Configuration} 
Denote $M$ for the matroid associated to the Pappus configuration, as in Figure \ref{fig:Pappus, Pappus, I_i, J_i} (Left). In this subsection, we prove Proposition \ref{prop: Pappus M(9)}. From this we can identify redundant matroid varieties in a decomposition. 
\medskip

\noindent{\bf Proof of Proposition~\ref{prop: Pappus M(9)}}
\newline
1) $V_{A_i} \subset V_{M(9)}$
\newline
If we prove this result for $i=1$, then the result holds for the other values of $i$ as well, since these point-line configurations are obtained from $A_1$ by a hypergraph automorphism of $M(9)$. So
without loss of generality, assume that $i=1$. After a linear transformation, an arbitrary $\gamma \in \Gamma_{M(9)}$ can be written as \begin{equation} \label{eq: eerste vb appendix 1} \gamma = \begin{pmatrix} 
1 & 0 & 0 & 1 & 1 & 0 & 1+y & 1+y & 0\\
0 & 1 & 0 & 1 & 1+x & 1 & 1+x & 0 & 0  \\
0 & 0 & 1 & 1 & 1 & 1 & 1 & 1 & 0
\end{pmatrix},\end{equation} where the columns from left to right correspond to $\{1,4,5,8,3,6,2,7,9\}$ and the minors corresponding to the bases are non-zero. After a linear transformation, an arbitrary $\xi \in \Gamma_{A_1}$ can be written as  \begin{equation} \label{eq: appendix eerste vb} \xi = \begin{pmatrix}
    1 & 0 & 0 & 1 & 1 & 0 & 1 & 1 & 0 \\
    0 & 1 & 0 & 1 & 1+z & 1 & 0 & 0 & 0\\
    0 & 0 & 1 & 1 & 1 & 1 & 0 & 0 & 0
\end{pmatrix},\end{equation}
where the columns from left to right correspond to $\{1,4,5,8,3,6,2,7, 9\}$ as well, and the minors corresponding to the bases are non-zero.
Fix a collection of vectors $\xi \in \Gamma_{A_1}$. 
Set $$
\left\{
    \begin{array}{lll}
       1+y & = &\epsilon^{-1}  \\
    x & = &z
    \end{array}
\right.
$$ in $\gamma$ and denote this $\gamma$ as $\gamma_\epsilon$. 
If $\epsilon \to 0$, then $\gamma_\epsilon \to \xi$. So $\xi \in V_{M(9)}$. Thus $\Gamma_{A_1} \subset V_{M(9)}$ and this implies that $V_{A_1} \subset V_{M(9)}$.
\newline
\newline
3) $V_{B_i} \subset V_N$
\newline
We can assume without loss of generality that $i=1$. An arbitrary $\gamma \in \Gamma_N$ can be written as
$$\gamma = \begin{pmatrix}
    1 & 0 & 0 & 1 & 1  & 1& 1 & 1+x-y & 0 \\
    0 & 1 & 0 &  1 & 1 & 0 & 1+x & -y & 0 \\
    0 & 0 & 1 & 1 & 1+x & y & 1+x & 0 & 0 \\
\end{pmatrix},$$
where the columns from left to right correspond to $\{1,5,3,8,4,2,6,7,9\}$ and the minors corresponding to the bases are non-zero.
\newline
An arbitrary element $\xi \in \Gamma_{B_1}$ can be written as $$\xi = \begin{pmatrix}
    1 & 0 & 0 & 1 & 1 & 1 & 1 & 1 & 0\\
    0 & 1 & 0 & 1 & 1 & 0 & 0  &z & 0 \\
    0& 0 & 1 & 1 & 0 & 0 & 0 & 0 & 0 \\
\end{pmatrix},$$
where the columns from left to right correspond to $\{1,5,3,8,4,2,6,7,9\}$ as well, and the minors corresponding to the bases are non-zero. Fix a $\xi \in \Gamma_{B_1}$.
Write $\gamma_\epsilon$ for $\gamma \in \Gamma_N$ with 
$$\left\{
    \begin{array}{lll}
       1+x & = &\epsilon  \\
    y &=& \frac{z(1+x)}{z-1} 
    \end{array}
\right..$$
If $\epsilon \to 0$, $\gamma_\epsilon \to \xi$. So $\xi \in V_N$. Since $\xi$ is arbitrary, $\Gamma_{B_1} \subset V_N$, thus $V_{B_1} \subset V_N$.
\newline
\newline
5) $V_{C_i} \subset V_N$.
\newline
We can assume without loss of generality that $i=1$. An arbitrary $\gamma \in \Gamma_N$ can be written as
$$\gamma = \begin{pmatrix}
    1 & 0 & 0 & 1 & 1  & 1& 1 & 1+x-y & 0 \\
    0 & 1 & 0 &  1 & 1 & 0 & 1+x & -y & 0 \\
    0 & 0 & 1 & 1 & 1+x & y & 1+x & 0 & 0 \\
\end{pmatrix},$$
where the columns from left to right correspond to $\{1,5,3,8,4,2,6,7,9\}$ and the minors corresponding to the bases are non-zero.
\newline
An arbitrary element $\xi \in \Gamma_{C_1}$ can be written as $$\xi = \begin{pmatrix}
    1 & 0 & 0 & 1 & 1 & 0 & 1 & 1 & 0\\
    0 & 1 & 0 & 1 & 1 & 0 & 1+z  &1 & 0 \\
    0& 0 & 1 & 1 & 1+z & 1 & 1+z & 0 & 0\\
\end{pmatrix},$$
where the columns from left to right correspond to $\{1,5,3,8,4,2,6,7,9\}$ as well, and the minors corresponding to the bases are non-zero. Fix a $\xi \in \Gamma_{C_1}$.
Write $\gamma_\epsilon$ for $\gamma \in \Gamma_N$ with 
$$\left\{
    \begin{array}{lll}
       x & = &z \\
    y &=& \frac{-(1+x)}{\epsilon}
    \end{array}
\right..$$
Then $$\gamma_\epsilon = \begin{pmatrix}
    1 & 0 & 0 & 1 & 1  & \frac{\epsilon}{-1-z}& 1 & 1-\epsilon & 0\\
    0 & 1 & 0 &  1 & 1 & 0 & 1+z & 1 & 0\\
    0 & 0 & 1 & 1 & 1+z & 1& 1+z & 0 & 0\\
\end{pmatrix}.$$
If $\epsilon \to 0$, $\gamma_\epsilon \to \xi$. So $\xi \in V_N$. Since $\xi$ is arbitrary, $\Gamma_{C_1} \subset V_N$, thus $V_{C_1} \subset V_N$.
\newline
\newline
6) $V_{D_i} \subset V_N$.
\newline
We can assume without loss of generality that $i=1$. An arbitrary $\gamma \in \Gamma_N$ can be written as $$\gamma = \begin{pmatrix} 
1 & 0 & 0 & 1 & 1 & 1& 1 & y-1 & 0\\
0 & 1 & 0 & 1 & 1+x & 1+x & 0 & y(1+x) & 0 \\
0 & 0 & 1 & 1 & 1 & 1+x & y & 0 & 0
\end{pmatrix},$$ 
where the columns from left to right correspond to $\{1,5,3,6,4,8,2,7,9\}$ and the minors corresponding to the bases are non-zero.
An arbitrary element $\xi \in \Gamma_{D_1}$ can be written as $$\xi = \begin{pmatrix}
    1 & 0 & 0 & 1 & 0 & 0 & 1 & 1 & 0\\
0 & 1 & 0 & 1 & 1 & 1 & 0 & w & 0\\
0 & 0 & 1 & 1 & 0 & 1 & 0 & 0 & 0
\end{pmatrix},$$
where the minors corresponding to the bases are non-zero and the columns from left to right correspond to $\{1,5,3,6,4,8,2,7,9\}$ as well. Fix a $\xi \in \Gamma_{D_1}$.
Write $\gamma_\epsilon$ for $\gamma \in \Gamma_N$ with 
$$\left\{
    \begin{array}{lll}
       y & = & \frac{w\epsilon}{w-1} \\
    \epsilon^{-1} &=& 1+x
    \end{array}
\right..$$
If $\epsilon \to 0$, $\gamma_\epsilon \to \xi$. So $\xi \in V_N$. Since $\xi$ is arbitrary, $\Gamma_{D_1} \subset V_N$, thus $V_{D_1} \subset V_N$.
\newline
\newline
8) $V_{\CCC(\widetilde{I}_i)} \in V_{M(9)}.$
\newline
Assume without loss of generality that $i=1$, so we have to prove that $\widetilde{I}_1 = I_9 \subseteq V_{M(9)}$. We use again the decomposition strategy. The minimal matroids over $I_9$ are the following. The point $9$ is a loop in each minimal matroid. An automorphism means a hypergraph automorphism from $I_9$. 
\begin{itemize} 
    \item The matroid $A_1$ with 1=2=7=8 and lines $\{4,5,6\},\{1,3,4\}$, as in Figure \ref{fig:pappus I_i 1} (Left), and three other matroids formed by automorphisms.
    \item The matroid $A_1'$ with $2=3=4=5$ and lines $\{8,1,6\},\{1,2,7\}$, and a matroid formed by automorphism.
    \item The matroid $B$ with $1=2=4=5=8$, as in Figure \ref{fig:pappus I_i 1} (Center).
    \item The matroid $C_1$ with $1=2=7, 5=6$ and lines $\{8,1,5\},\{8,3,4\}$, as in Figure \ref{fig:pappus I_i 1} (Right) and three matroids formed by automorphism.
    \item The matroid $C'$ with $1=5=6, 3=8$ and lines $\{1,2,3\},\{2,4,7\}$.
    \item The matroid $C''$ with $2=3=4$ and $6=8$ and lines $\{1,5,7\},\{2,5,6\}$.
    \item The matroid $D_1$ with $1=3, 4=6$ and lines $\{8,1,4\},$ $\{1,5,7\},$ $\{2,4,7\},$ $\{8,2,5\}$, as in Figure \ref{fig:pappus I_i 2} (Left), and another matroid formed by automorphism.
    \item The matroid $D'$ with $1=7, 3=4$ and lines $\{1,2,3\},$ $\{8,1,6\},$ $\{8,2,5\},$ $\{3,5,6\}$.
    \item The matroid $E_1$ with $2=3=8$ and lines $\{1,2,6\},$ $\{1,5,7\},$ $\{2,4,7\},$ $\{4,5,6\}$, as in Figure \ref{fig:pappus I_i 2} (Center), and two other matroids formed by automorphism.
    \item Three matroids of the form $\pi_{I_9}^j$, for $j \in \{3,6,7\}$.
    
    \item The matroid $F$ with lines $\{1,2,3\},$ $\{8,1,6\},$ $\{1,5,7\},$ $\{2,4,7\},$ $\{8,2,5\},$ $\{4,5,6\},$ $\{8,3,4\},$ $\{3,6,7\}$, as in Figure \ref{fig:pappus I_i 1} (Right).
    \item The matroids with an additional loop from the set $\{1,2,4,5,8\}.$
\end{itemize}
The matroids $A_i, A_i', B, C_1, C', C'', D_i, D'$ and $E_i$ are nilpotent or quasi-nilpotent, so by Theorem \ref{nil coincide}, their circuit variety is contained in the union of their matroid variety and the matroid variety of $U_{2,9}.$
\begin{figure}
    \centering
    \includegraphics[width=0.9\linewidth]{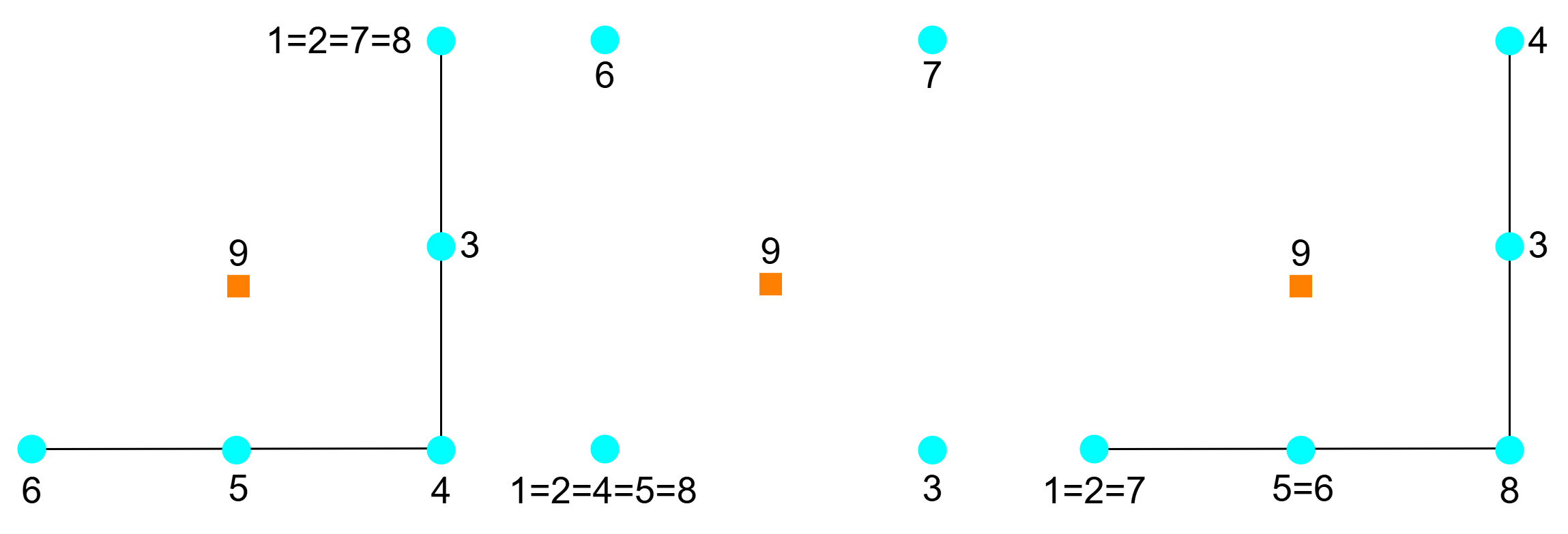}
    \caption{(Left) $A_1$; (Center) $B_1$; (Right) $C_1$.}
    \label{fig:pappus I_i 1}
\end{figure}
\begin{figure}
    \centering
    \includegraphics[width=0.9\linewidth]{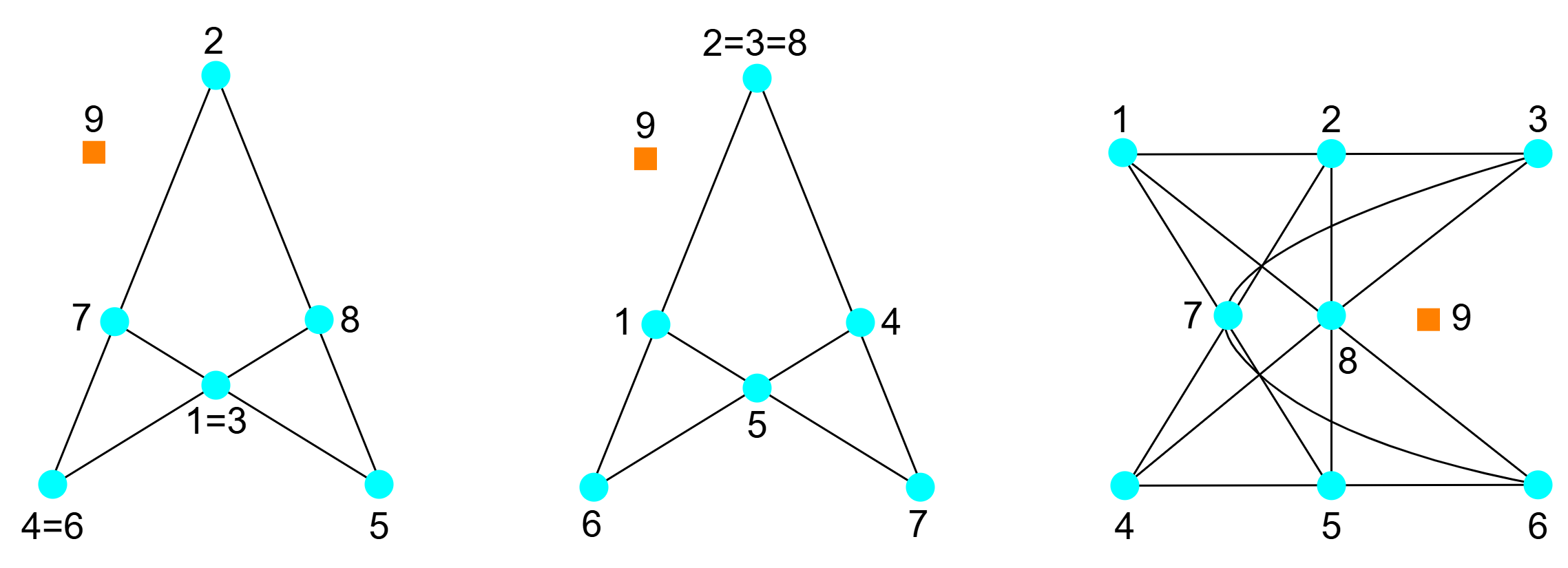}
    \caption{(Left) $D_1$; (Center) $E_1$; (Right) $F.$}
    \label{fig:pappus I_i 2}
\end{figure}
To prove that $V_{I_9} \subset V_{M(9)}$, take an arbitrary $\gamma \in \Gamma_{M(9)}$. It can be written as $$\gamma = \begin{pmatrix}
1 & 0 & 0 & 1 & 1 & 0 & 1 & -z & 0\\
0 & 1 & 0 & 1 & z & 1-z & 0 & y-yz & 0 \\
0 & 0 & 1 & 1 & 0 & 1 & y & y-yz & 0 
\end{pmatrix},$$ where the columns from left to right correspond to $\{1,2,7,8,3,4,5,6,9\}$ and the minors corresponding to the bases are non-zero.
An arbitrary $\xi \in  \Gamma_{I_9}$ can be written  as $$\xi = \begin{pmatrix}
1 & 0 & 0 & 1 & 1 & 0 & 1 & -x & 0\\
0 & 1 & 0 & 1 & x & 1-x & 0 & 1-x & 0 \\
0 & 0 & 1 & 1 & 0 & 1 & 1 & 1-x & 0
\end{pmatrix},$$ where the columns from left to right correspond to $\{1,2,7,8,3,4,5,6,9\}$ as well and the minors corresponding to the bases are non-zero.
Let $$
\left\{
    \begin{array}{ll}
        z = x+\epsilon \\
        y = 1+ \epsilon
    \end{array}
\right..
$$
If $\epsilon \to 0$, then $\gamma \to \xi$. So $\xi \in V_{M(9)}.$ Thus 
\begin{equation} \label{eq: appendix pappus I_i}
V_{I_9} \subset V_{M(9)}. 
\end{equation}
To prove that $V_{A_i} \subset V_{M(9)}$, we can assume without loss of generality that $i=1$. Fix an arbitrary $\xi \in \Gamma_{A_1}$. Define $\widetilde{\xi}_1$ as a perturbation of $\xi_1$ on the line through $\xi_6$ and $\xi_8$, such that $\widetilde{\xi}_1 \neq \xi_1$ and define $\widetilde{\xi}_2 = \widetilde{\xi}_1\xi_3 \wedge \xi_5 \xi_8$. Then define $\widetilde{\xi}_7 = \widetilde{\xi}_2 \xi_4 \wedge \widetilde{\xi}_1 \xi_5$ and extend this collection of vectors to $\widetilde{\xi}$ by defining $\widetilde{\xi_i} = \xi_i$ for the remaining vectors. Since $\xi \in \Gamma_{M(9)}$, we can use Lemma \ref{lemma: meet of 2 extensors arbitrary close} to conclude that $\widetilde{\xi}$ is arbitrarily close to $\xi$. Since $\widetilde{\xi} \in \Gamma_{I(9)}$, it follows that
\begin{equation} \label{eq: appendix Pappus A_1}
V_{A_1} \subset V_{I_9} \subseteq V_{M(9)}.\end{equation} 
The proof that $V_{A_i'} \subset V_{M(9)}$ is completely analogous.
\newline
\newline

To prove that $V_B \subset V_{M(9)}$, take an arbitrary $\gamma \in \Gamma_{I_9}$. It can be written as $$\gamma = \begin{pmatrix} 
1 & 0 & 0 & 1 & 1+x & 1+x & 1+x & 1+y & 0\\
0 & 1 & 0 & 1 & 1 & 1 & 1+x & 1 & 0 \\
0 & 0 & 1 & 1 & 1 & 1+y & 1+y & 1+y & 0
\end{pmatrix},$$ where the columns from left to right correspond to $\{3,6,7,1,2,4,5,8,9\}$ and the minors corresponding to the bases are non-zero.
Moreover, the determinant of the minor corresponding to the columns 2,5,8 is 0, implying $$-yx^2-x^2+xy-y^2=0.$$
An arbitrary $\xi \in  \Gamma_{B}$ can be written  as $$\xi = \begin{pmatrix}
1 & 0 & 0 & 1 & 1 & 1 & 1 & 1 & 0\\
0 & 1 & 0 & 1 & 1 & 1 & 1 & 1 & 0 \\
0 & 0 & 1 & 1 & 1 & 1 & 1 & 1 & 0 
\end{pmatrix},$$
where the columns from left to right correspond to $\{3,6,7,1,2,4,5,8,9\}$ and the minors corresponding to the bases are non-zero.
Let $$
\left\{
    \begin{array}{ll}
        x = \frac{\epsilon +\sqrt{-3-4\epsilon}\epsilon}{2(1+\epsilon)} \\
        y = \epsilon
    \end{array}
\right..
$$
If $\epsilon \to 0$, then $\gamma \to \xi$. Thus 
\begin{equation} \label{eq: appendix Pappus B}
V_{B} \subset V_{M(9)}. 
\end{equation}
To prove that $V_{C_1} \subset V_{M(9)}$, choose $\xi \in \Gamma_{C_1}$ arbitrarily. Define $\widetilde{\xi}_6$ as perturbation of $\xi_6$ on the line through $\xi_4$ and $\xi_5$, such that $\widetilde{\xi}_6 \neq \xi_6$ and define $\widetilde{\xi}_1 = \gamma_2\gamma_3 \wedge \widetilde{\gamma}_6\gamma_8$. Then define $\widetilde{\xi}_7 = \widetilde{\xi}_1 \xi_5 \wedge {\xi}_2 \xi_4$. Extend this collection of vectors to $\widetilde{\xi}$ by defining $\widetilde{\xi_i} = \xi_i$ for the remaining vectors. $\widetilde{\xi}$ is a perturbation of $\xi$ by Lemma \ref{lemma: meet of 2 extensors arbitrary close}. Moreover, since $\widetilde{\xi} \in \Gamma_{I_9} \subset V_{M(9)}$, we can conclude that \begin{equation} \label{eq: appendix Pappus C_1}
V_{C_1} \subset V_{M(9)}.
\end{equation}
The proof that $V_{C'} \subset V_{M(9)}$ and $V_{C''} \subset V_{M(9)}$ is completely analogous.
\newline
\newline
To prove that $V_{D_1} \subset V_{M(9)}$, take an arbitrary $\gamma \in \Gamma_{I_9}$. After a projective transformation, it can be written as \begin{equation} \label{eq: appendix Pappus M(9)}
\gamma = \begin{pmatrix} 
1 & 0 & 0 & 1 & 1 & x & 0 & 1 & 0\\
0 & 1 & 0 & 1 & x & x & 1 & x & 0 \\
0 & 0 & 1 & 1 & 0 & 1 & 1 & 1 & 0
\end{pmatrix},\end{equation} where the columns from left to right correspond to $\{1,2,4,5,3,6,7,8,9\}$ and the minors corresponding to the bases are non-zero. 
An arbitrary $\xi \in  \Gamma_{D_1}$ can be written  as $$\xi = \begin{pmatrix}
1 & 0 & 0 & 1 & 1 & 0 & 0 & 1 & 0\\
0 & 1 & 0 & 1 & 0 & 0 & 1 & 0 & 0 \\
0 & 0 & 1 & 1 & 0 & 1 & 1 & 1 & 0 
\end{pmatrix},$$
where the columns from left to right correspond to $\{1,2,4,5,3,6,7,8,9\}$ as well and the minors corresponding to the bases are non-zero. 
Let $x = \epsilon$. If $\epsilon \to 0$, then $\gamma \to \xi$. So $\xi \in V_{I_9} \subset V_{M(9)}.$ Thus 
\begin{equation} \label{eq: appendix Pappus D_1}
V_{D_1} \subset V_{M(9)}. 
\end{equation}
To prove that $V_{E_1} \subset V_{M(9)}$, take an arbitrary $\gamma \in \Gamma_{I(9)}$. It can be written as in Equation~\eqref{eq: appendix Pappus M(9)}, where the columns from left to right correspond to $\{1,2,4,5,3,6,7,8,9\}$ and the minors corresponding to the bases are non-zero. An arbitrary $\xi \in  \Gamma_{E_1}$ can be written  as $$\xi = \begin{pmatrix}
1 & 0 & 0 & 1 & 0 & 1 & 0 & 0 & 0\\
0 & 1 & 0 & 1 & 1 & 1 & 1 & 1 & 0 \\
0 & 0 & 1 & 1 & 0 & 0 & 1 & 0 & 0 
\end{pmatrix},$$
where the columns again correspond to $\{1,2,4,5,3,6,7,8,9\}$.
Let $x= \frac{1}{\epsilon}$. If $\epsilon \to 0$, then $\gamma \to \xi$. So $\xi \in V_{E_1} \subset V_{M(9)}.$ Thus 
\begin{equation} \label{eq: appendix Pappus E_1}
V_{E_1} \subset V_{M(9)}. 
\end{equation}
For $\pi_{I_9}^j$, by Theorem \ref{nil coincide}, \begin{equation} \label{eq: appendix Pappus pi}
V_{\pi_{I_9}^j} = V_{\CCC(\pi_{I_9}^j)} \subset V_{\CCC(\pi_{M(9)}^j)} = V_{\pi_{M(9)}^j}.\end{equation}
The minimal matroids over $F$ are the following, the point $9$ is a loop in each minimal matroid.
\begin{itemize}
    \item Eight matroids formed as follows: Fix a line $l$ of $F$. The resulting matroid has one line, $l$, and all the points in $[8] \setminus l$ coincide in a point not lying on $l$. By Lemma~7.6, their matroid variety is contained in $V_{I_9} \subset V_{M(9)}$.
    \item The matroids $\pi_F^i$ for $i \in \{1, \ldots, 8\}.$ These matroids are nilpotent as well, so $$V_{\pi_F^j} = V_{\CCC(\pi_F^j)} \subset V_{\CCC(\pi_{M(9)}^j)} = V_{\pi_{M(9)}^j}.$$
    \item The matroid $U_{2,9}(9)$.
    \item The matroids $F(j)$ for $j$ ranging from $1$ to $8$. $V_{\CCC(F(j))} \subset V_{\CCC(M(9,j))}.$
\end{itemize}
The matroid $F$ is not realizable. Assume for contradiction that is is realizable. Then there is an element $\xi \in \Gamma_F$ which can be written as:
$$\xi = \begin{pmatrix}
    1&0&0&1&1&0&1&-x&0 \\ 
    0&1&0&1&x&1-x&0&1-x&0 \\
    0&0&1&1&0&1&1&1-x&0 
\end{pmatrix},$$
where the columns from left to right correspond to $\{1,2,7,8,3,4,5,6,9\}.$ Moreover, since $\{3,6,7\}$ is a circuit of $F$, it follows that $x=0$. But this contradicts that $\xi \in \Gamma_F.$
So \begin{equation}  \label{eq: appendix Pappus F}
V_{\CCC(F)} \subset V_{M(9)} \cup_{j=1}^8 V_{\pi_{M(9)}^j} \cup V_{U_{2,9}(9)} \cup_{j=1}^9 V_{\CCC(M(9,j))}.\end{equation}
If $i \neq 1,4$, then by Theorem~\ref{nil coincide}, $V_{\CCC(M(9,i))}= V_{M(9,i)}.$
Moreover, by Lemma \ref{lemma: third config loop matroid variety}, $V_{M(9,i)} \subseteq V_{M(9)}$ if $i \neq 1,4.$ Thus we only need to consider $V_{\CCC(M(9,1))}$ and $V_{\CCC(M(9,4))}.$
By Example \ref{ex: decomposition circuit variety 3 concurrent lines}, it follows that $V_{\CCC(M(9,1))} = V_{M(9,1)} \cup V_{M(9,1,4)}$.
Thus we can write Equation~\eqref{eq: appendix Pappus F} as:
\begin{equation} \label{eq: appendix pappus F new}
    V_{\CCC(F)}\subset V_{M(9)} \cup V_{U_{2,9}(9)} \cup V_{M(9,1)} \cup V_{M(9,4)} \cup V_{M(9,1,4)}.
\end{equation}
So combining Equation \eqref{eq: appendix Pappus A_1}, \eqref{eq: appendix Pappus B}, \eqref{eq: appendix Pappus C_1}, \eqref{eq: appendix Pappus D_1}, \eqref{eq: appendix Pappus E_1}, \eqref{eq: appendix Pappus pi} and \eqref{eq: appendix pappus F new}, we can conclude that:
\begin{equation} \label{final one pappus appendix}
V_{\CCC(I_9)} \subset V_{M(9)} \cup V_{U_{2,9}} \cup_{j=1,4} V_{M(9,j)} \cup V_{M(9,1,4)} \cup_{j=1}^9V_{\pi_{M(9)}^j}.\end{equation}
Finally, we can notice that $V_{\pi_{M(9)}^j}$ is a redundant component.
Let $\gamma \in \Gamma_{\pi_{M(9)}^j}$. Since $M(9) \setminus \{j\}$ is nilpotent, one can lift $\gamma|_{[8] \setminus \{j\}}$ non-trivially from $\gamma_j$ to $\widetilde{\gamma}$ such that $\widetilde{\gamma} \in V_{\CCC(M(9) \setminus \{j\})}$ by Proposition \ref{nilp lift}. Let $\widetilde{\gamma}' = \widetilde{\gamma} \cup \{\gamma_j\}$. Then $\widetilde{\gamma}' \in V_{\CCC(M(9))}$, and $\widetilde{\gamma}'$ has no loops, so using the decomposition of Equation~\eqref{final one pappus appendix}, it is clear that $\gamma \in V_{M(9)}$.
Thus,
\begin{equation*} 
V_{\CCC(I_9)} \subset V_{M(9)} \cup V_{U_{2,9}(9)} \cup_{j=1,4} V_{M(9,j)} \cup V_{M(9,1,4)}.\end{equation*}

\subsection{Third configuration $9_3$}
Throughout this section, we consider $M$ to be the third configuration $9_3$, as in Figure~\ref{fig:Third Config A_1 B} (Left). First we will prove Lemma~5.2, which states some results on inclusions of matroid varieties. Then we will prove that some matroid varieties in the decomposition in Theorem~\ref{thm decompo third 9_3} are irredundant.
\subsubsection{Identifying redundancies}\label{subsubsec: identifying redundancies third} 
In this subsection, we present a proof for Lemma \ref{lemma: third config hulplemma} and Lemma \ref{lemma: inclusions third 9_3}.
\newline
\newline
{\bf Proof of $V_O \subset V_Q$ as in Lemma \ref{lemma: third config hulplemma}}
\newline
 To prove that $V_O \subset V_Q$, perform a projective transformation such that an arbitrary $\gamma \in \Gamma_Q$ can be written as $$\gamma = \begin{pmatrix} 
1 & 0 & 0 & 1 & 1 & 1+y& 1+y-\frac{y}{x} \\
0 & 1 & 0 & 1 & 0 & y & -y \\
0 & 0 & 1 & 1 & x & y & 0
\end{pmatrix},$$ 
where the minors corresponding to the bases are non-zero and the columns from left to right correspond to $\{7,6,8,9,2,3,4\}$.
An arbitrary element $\xi \in \Gamma_{O}$ can be written as $$\xi = \begin{pmatrix}
1 & 0 & 0 & 1 & 1 & 1 & 1  \\
0 & 1 & 0 & 1 & 0 & 0 & 0  \\
0 & 0 & 1 & 1 & 0 & 0 & 0 
\end{pmatrix},$$
where the minors corresponding to the bases are non-zero and the columns from left to right correspond to $\{7,6,8,9,2,3,4\}$ as well. Write $\gamma_\epsilon$ for $\gamma \in \Gamma_Q$ with 
$$\left\{
    \begin{array}{lll}
       y & = & \epsilon^2 \\ 
    x & = & \epsilon \\
    \end{array}
\right..$$
If $\epsilon \to 0$, $\gamma_\epsilon \to \xi$. So $\xi \in V_Q$. Since $\xi$ is arbitrary, it follows that $\Gamma_{O} \subset V_Q$, thus $V_{O} \subset V_Q$.

\begin{remark}
    Since we will consider many point-line configurations in this section, some symbols will be used multiple times for different point-line configurations. However, we prove five results in this section, so this section is split in five parts, denotes by \enquote{1)}, ..., \enquote{5)}, and within the same part each symbol denotes a unique point-line configuration. So unless otherwise specified, a symbol always denotes the point-line configuration defined in that part.
\end{remark}
\noindent {\bf Proof of Lemma \ref{lemma: inclusions third 9_3}} 

\medskip

1) We prove Equation~\eqref{eq: lemma 1 for third config} without loss of generality for $A_1$, obtained from $M$ by identifying $\{7,8,9\}$ and with lines $\{1,3,5\},\{1,2,6\},\{2,3,4\},\{4,5,6,7\}$, as in Figure \ref{fig:Third Config A_1 B} (Left). The minimal matroids over $A_1$ are the following, in each matroid the points 7,8 and 9 are of course identified. Automorphism means a hypergraph automorphism of $A_1$.
\begin{itemize}
    \item The matroid $B$ derived from identifying the points $\{1,2,3\}$ while placing the remaining ones on a line, as in Figure \ref{fig:Gemengd Third config Decompo A 3} (Left).
    \item The matroid obtained by identifying the points $\{1,2\}$ and $\{4,5\}$, and lines $\{1,3,4\}$ and $\{4,6,7\}$. Denote this matroid by $C_1$ and the other two matroids derived from automorphism as $C_i$, as in Figure \ref{fig:Gemengd Third config Decompo A 3} (Center).
    \item The matroid formed by identifying the points $\{1,5,6\}$. The resulting matroid has $\{2,3,4\}$ and $\{1,4,7\}$ as lines. Denote this matroid by $D_1$ and the two other matroids obtained by automorphism by $D_i$, as in Figure \ref{fig:Gemengd Third config Decompo A 3} (Right).
    \item The matroids $\pi_{A_1}^i$, with $i \in \{1,2,3,7\}$. $\pi_{A_1}^1$ is shown in Figure \ref{fig:Third config decompo A part 2} (Left).
    \item Identify $7$ with a point which is on a line with 7. This results in the quadrilateral set. Denote these matroids as $E_i$, where $E_1$ is the matroid with $4=7$ and lines $\{1,3,5\},\{1,2,6\},\{2,3,4\},\{4,5,6\},$ as in Figure \ref{fig:Third config decompo A part 2} (Center).
    \item The matroids obtained by setting a point in $[6]$ to be a loop.
\end{itemize}
\begin{figure}
    \centering
    \includegraphics[width=0.9\linewidth]{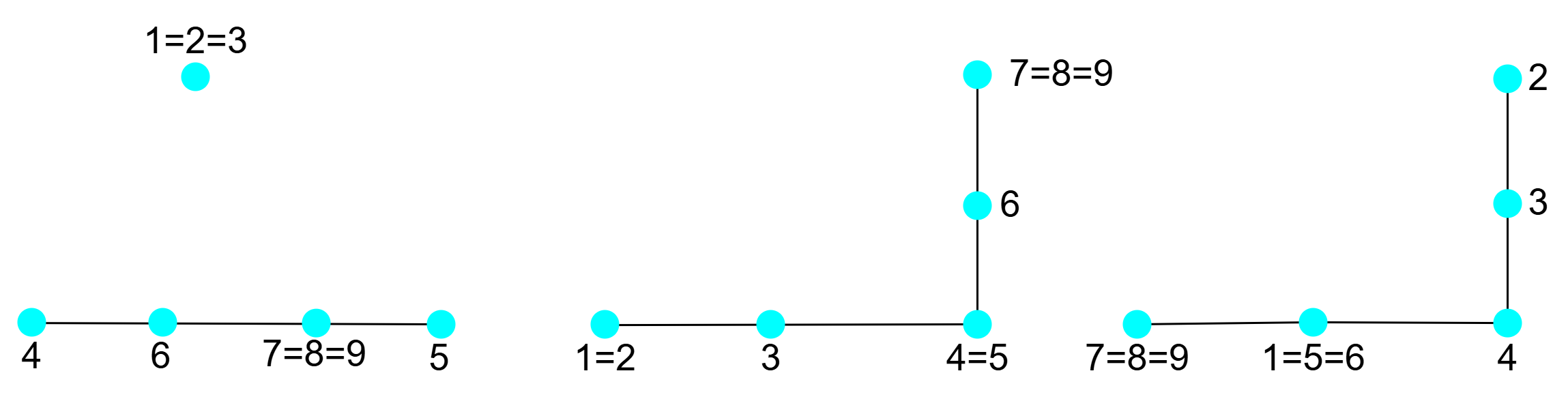}
    \caption{(Left) $B$; (Center) $C_1$, (Right) $D_1$.}
    \label{fig:Gemengd Third config Decompo A 3}
\end{figure}
\begin{figure}
    \centering
\includegraphics[width=0.9\linewidth]{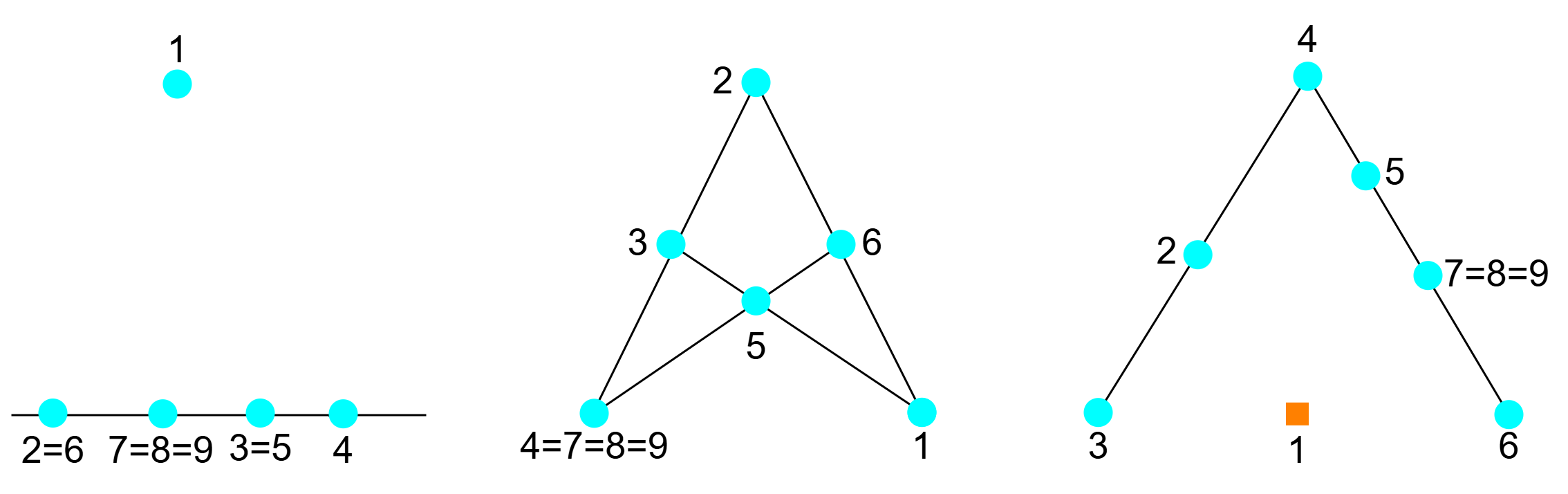}
    \caption{(Left) $\pi_{A_1}^1$; (Center) $E_1$, (Right) $A_1(1)$.}
    \label{fig:Third config decompo A part 2}
\end{figure}
From this, we derive the following decomposition.
$$V_{\CCC(A_1)} = V_{A_1} \cup V_{\mathcal{C}(B)} \cup_{i=1,2,3,7} V_{\CCC(C_i)} \cup_{i=1}^3 V_{\CCC(D_i)} \cup_{i=1}^3 V_{\CCC(\pi_N^i)} \cup_{i=1}^3 V_{\mathcal{C}(E_i)} \cup V_{\CCC(A_1(i))}.$$
Since the matroids $B,C_i,D_i,\pi_{A_1}^i, A_1(i)$ and $E_i$, are nilpotent or quasi-nilpotent, their circuit variety is contained in the union of their matroid variety and of $U_{2,9}$. 
\newline
\newline
First we prove that $V_B \subset V_{A_1}$. Let $\xi \in \Gamma_B$ arbitrarily. Redefine $\xi_3$ arbitrarily on the line through $\xi_1$ and $\xi_5$, such that $\widetilde{\xi}_3 \neq \xi_1$. Then define $\widetilde{\xi}_2 = \widetilde{\xi}_3 \xi_4 \wedge \xi_1 \xi_6$. Since $\xi$ has no loops, $\widetilde{\xi}_2 \to \xi_2$ if $\widetilde{\xi_3} \to \xi_3$ by Lemma \ref{lemma: meet of 2 extensors arbitrary close}. Let $\widetilde{\xi}_p = \xi_p$ for the points $p \in [9] \setminus \{2,3\}$.  Then, $\Gamma_B \subset V_{A_1}$, implying $V_B \subset V_{A_1}.$ 
\newline
\newline
To prove that $V_{C_i} \subset V_{A_1}$, we can use the same strategy, we may assume without loss of generality $i=1$. For an arbitrary $\xi \in \Gamma_{C_1}$, first define $\widetilde{\xi}_5$ arbitrarily on the line through the points $\xi_4,\xi_6$ and $\xi_7$, such that it does not coincide with one of that points. Then define $\widetilde{\xi}_1 = \xi_3 \widetilde{\xi_5} \wedge \xi_2 \xi_6$ and let $\widetilde{\xi}_p = \xi_p$ for the points $p \in [9] \setminus \{1,5\}$. This implies that $V_{C_1} \subset V_{A_1},$ by Lemma \ref{lemma: meet of 2 extensors arbitrary close}.
\newline
\newline
In the same way, one can prove that $V_{D_i} \subset V_{A_1}$. We can assume without loss of generality that $i=1$. Let $\xi \in \Gamma_{D_1}$ arbitrarily. First perturb $\xi_5$ to $\widetilde{\xi}_5$ on the line through $\xi_4$ and $\xi_7$, with $\xi_1 \neq \widetilde{\xi}_5 \neq \xi_6$. Then define $\widetilde{\xi}_1 = \xi_3\widetilde{\xi}_5 \wedge \xi_2 \xi_6$. Let $\widetilde{\xi}_p=\xi_p$ for $p \in [d] \setminus \{1,5\}$. Using Lemma \ref{lemma: meet of 2 extensors arbitrary close}, $V_{D_1} \subset V_{A_1}$.
 \newline
 \newline
Finally we prove that $V_{E_i} \subset V_{A_1}.$ Assume without loss of generality that $i=1$. Let $\xi \in \Gamma_{E_1}$ and perturb $\xi_7,\xi_8,\xi_9$ to the same point on the line through $\xi_4, \xi_5$ and $\xi_6$. Then the perturbed configuration is in $\Gamma_{A_1}$, using Lemma \ref{lemma: meet of 2 extensors arbitrary close}. So $V_{E_1} \subset V_N$. 
\newline
\newline
We can conclude from the above that $$V_{\mathcal{C}({A_1})} \subset V_{A_1} \cup_{i=1}^9 V_{\pi_{A_1}^i} \cup V_{U_{2,9}} \cup_{i=1}^9 V_{{A_1}(i)}.$$
This proves the lemma, since for $i \in [9]$, $V_{\pi_{A_1}^i} =  V_{\CCC(\pi_{A_1}^i)} \subset V_{\CCC(\pi_M^i)} = V_{\pi_M^i}$. In the same way, $V_{\CCC(A_1(i))} \subset V_{\CCC(M(i))}.$
\newline
\newline
2) We will prove Equation 
\eqref{eq: lemma 3 third config}. Denote $M$ the matroid of the third configuration $9_3$ and $C_i$ as in Section \ref{Decomposition of the Circuit Variety of the third configuration 93}, shown in Figure \ref{fig:Third config C pi_M^1 M(1)} (Left). We can assume without loss of generality that $i=1$.
We decompose the circuit variety of $C_1$. The points $1,5$ and $9$ are identified in each case.
The minimal matroids over $C_1$ are the following:
\begin{itemize}
    \item The matroid $A$ with $1,2,4,7$ identified and no lines, as in Figure \ref{fig:third config decompo C} (Left).     \item The matroids $\pi_{C_1}^i$, with $i$ in $\{3,6,8\}$.
    \item The matroid $D$ formed by adding the line $\{3,6,8\}$, as in Figure \ref{fig:third config decompo C} (Center).
    \item The matroids $C_1(i)$ formed by making 1,2,4 or 7 a loop. If the point 1 is for instance a loop, then the lines are $\{2,3,4\},\{4,5,6,7\},$ as in Figure \ref{fig:third config decompo C} (Right).
\end{itemize}
\begin{figure}
    \centering
    \includegraphics[width=0.9\linewidth]{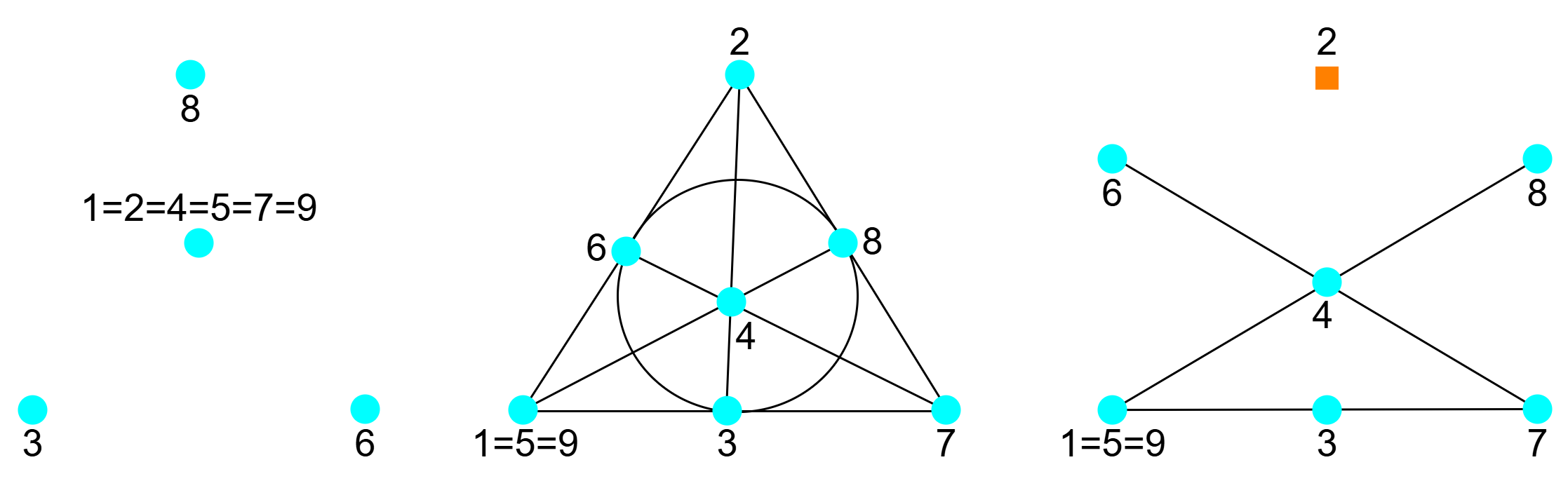}
    \caption{$A$ (Left); $D$ (Center); $C_1(2)$ (Right)}
    \label{fig:third config decompo C}
\end{figure}
Thus we can write \begin{equation} \label{eq:proof lemma 3 appendix} V_{\CCC({C_1})} = V_{C_1} \cup V_{\CCC(A)} \cup_{i=3,6,8} V_{\CCC(\pi_{C_1}^i)} \cup V_{\CCC(D)} \cup_{i=1,2,4,7} V_{\CCC({C_1}(i))}.\end{equation}
It is clear that for the point-line configurations $A,\pi_{C_1}^i$ and ${C_1}(i)$, the matroid variety equals the circuit variety by Theorem~\ref{nil coincide}.
\newline
\newline
\label{proof V_C in V_M}
First we prove that $V_{C_i} \subset V_M$. We will prove this for $C_1$, since the other proofs are similar.
\newline
\newline
An arbitrary $\gamma \in \Gamma_M$ can be written as $$\gamma = \begin{pmatrix} 
1 & 0 & 0 & 1 & 0 & 1 & 1 & 1-x(1+y) & 1\\
0 & 1 & 0 & 1 & 1 & 1 & 1+y & 1-x-xy+y & 1+y \\
0 & 0 & 1 & 1 & x & 0 & 1 & 1-x & (1+y)x
\end{pmatrix},$$ where the minors corresponding to the bases are non-zero and the columns from left to right correspond to $\{1,2,4,7,3,6,8,9,5\}.$ Moreover, the determinant of the minor corresponding to the columns 5,9,6 is 0, implying $$y (x (y + 2) - 1)=0.$$
An arbitrary $\xi \in  \Gamma_{C_1}$ can be written  as $$\xi = \begin{pmatrix}
1 & 0 & 0 & 1 & 0 & 1 & 1 & 1 & 1\\
0 & 1 & 0 & 1 & 1 & 1 & 0 & 0 & 0 \\
0 & 0 & 1 & 1 & 1 & 0 & 1 & 0 & 0 
\end{pmatrix}.$$
Define $$
\left\{
    \begin{array}{ll}
        x = \frac{1}{\epsilon +1} \\
        y = -1+\epsilon
    \end{array}
\right..
$$ If $\epsilon \to 0$, then $\gamma \to \xi$. So $\xi \in V_M.$ Thus 
\begin{equation} \label{eq: lemma 3 appendix VCi C VM}
V_{C_1} \subset V_M. 
\end{equation}

Recall the point-line configuration $B$ from Subsection \ref{Decomposition of the Circuit Variety of the third configuration 93}. Then \begin{equation} \label{eq: lemma 3 VCA}
V_{\CCC(A)} \subset V_{\CCC(B)} = V_{B}.
\end{equation}
For the matroid $D$, we again use the decomposition strategy, the minimal matroids over $D$ are:
\begin{itemize}
    \item The matroids $E_i$ with 1, 2, 4, 5, 7 and 9 identified and one line $\{3,6,8\}$, and seven other matroids of a similar form. 
    \item The matroids $\pi_D^i$, where $i$ ranges over all possible points of $D$. 
    \item The uniform matroid $U_{2,9}'$ with the points 5 and 9 identified.
    \item The matroids $D(i)$, where $i$ varies over all the points of $D$.
\end{itemize}
$V_D$ is empty, since $D$ is the point-line configuration associated to the Fano plane with two additional points 5 and 9 identified with 1, and the Fano plane is not realizable by Example \ref{ex: fano plane not realizable}.
By Theorem~\ref{nil coincide}, it follows that $V_{\CCC(E_i)}= V_{E_i}, V_{\CCC(G_i)} = V_{G_i}, V_{\CCC(U_{2,9}')} = V_{U_{2,9}}'$ and $V_{\CCC(D(i))} = V_{D(i)}.$
For $E_i$, it is clear by Lemma~7.6 of \cite{liwskimohammadialgorithmforminimalmatroids} that $V_{E_i} \subset V_{C_1}$. Finally, notice that $V_{\pi_D^i} \subset V_{\pi_{M}^i}$ and $V_{D(i)} \subset V_{\CCC(M(i))}.$
So \begin{equation} \label{eq: lemma 3 equation for VC(D)}
V_{\CCC(D)} \subset \cup_{i=1}^9 V_{\pi_{C_1}^i} \cup_{i=1}^9 V_{\CCC(C_1(i))} \cup V_{U_{2,9}}.\end{equation}
Using Equation \eqref{eq:proof lemma 3 appendix}, \eqref{eq: lemma 3 appendix VCi C VM}, \eqref{eq: lemma 3 VCA} and \eqref{eq: lemma 3 equation for VC(D)}:
\begin{equation}
    V_{\CCC(C_1)} \subset V_M \cup V_B \cup_{i=1}^9 V_{\pi_{M}^i} \cup_{i=1}^9 V_{\CCC(M(i))} \cup V_{U_{2,9}}.
\end{equation}
\label{inc V_D third}
3) We will prove Equation \eqref{eq: lemma 4 third config}. We again apply the decomposition strategy to $D$, defined in Section \ref{Decomposition of the Circuit Variety of the third configuration 93}, and identify the minimal matroids over $D$. Automorphism means a hypergraph automorphism of $D$:
\begin{itemize}
    \item The matroid formed by setting $1=2=3=8$, and two other matroids of a similar form. Every such matroid has more dependencies than the matroids $A_i$ as defined in Section \ref{Decomposition of the Circuit Variety of the third configuration 93}, see Figure \ref{fig:Third Config A_1 B} (Center), so their circuit variety is contained in the circuit variety of $A_i$.
    \item The matroid formed by setting $1=2=3$ and $6=7=9$, and two matroids of a similar form. Every such matroid has more dependencies than a matroid $A_i$ as defined in Section \ref{Decomposition of the Circuit Variety of the third configuration 93}, see Figure \ref{fig:Third Config A_1 B} (Center), so their circuit variety is contained in the circuit variety of some $A_i$.
    \item The matroid formed by setting 1=2, 4=5, 7=9. This matroid has more dependencies than a matroid $B$ as defined in Section \ref{Decomposition of the Circuit Variety of the third configuration 93}, see Figure \ref{fig:Third Config A_1 B} (Right), so its circuit variety is contained in the one of $B$.
    \item The matroid formed by setting 2=6, 3=7, 4=8. The lines of this matroid are $\{1,3,5\},\{1,4,9\},\{2,3,4\},\{9,2,5\}$, as in Figure \ref{fig:third config decompo d} (Left). Denote this matroid by $E_1$ and the two other matroid rising from automorphism by $E_i$. 
    \item The matroid formed by setting $2=3=5=8=9$. In this case the lines are $\{1,2,6\}$ and $\{4,6,7\}$, as in Figure \ref{fig:third config decompo d} (Center). Denote this matroid by $F_1$ and the other matroids rising from automorphisms by $F_2$. 
    \item The matroid formed by setting $1=3=4=6=9$ and lines $\{8,2,7\},\{8,1,5\}$, and another matroid formed by automorphism, denoted by $F'_i$. 
    \item The matroid formed by setting $1=5=6=7=8$ and lines $\{2,3,4\},\{9,3,1\}$, and another matroid formed by automorphism, denoted by $F''_i$.
    \item The matroid formed by setting $1=3=6=9$ and $2=7$, and lines $\{1,2,4\}, \{8,4,5\}$, as in Figure \ref{fig:third config decompo d} (Right). Denote this matroid by $G_1$ and the matroid formed by automorphism by $G_2$.
     \item The matroid formed by setting $1=5=6=8$ and $2=4$, and lines $\{1,2,7\}, \{9,3,7\}$, and another matroid $G_2$ formed by automorphism. Denote these matroids by $G_i'$. 
     \item The matroid formed by setting $3=4=7=8$ and $5=9$, and lines $\{1,3,5\},\{1,2,6\}$. Denote these matroids by $G_i''$. 
    \item The matroids $\pi_D^i$, where $i$ is in $\{1,\ldots,9\}.$ 
    \item The two matroids formed by identifying $\{1,5,9 \}$ or $\{2,4,7\}$. These matroids have more dependencies than the matroids $C_i$ in the decomposition of the third configuration, so their circuit variety is contained in that of $C_i$.
    \item The matroid with $8=9=6=7$. This matroid has more dependencies than the matroid $A_i$ in the decomposition of the third configuration, so its circuit variety is contained in that of $A_i$.
    \item The uniform matroid with 9 points: $U_{2,9}$.
    \item The matroids formed by making a point in $\{1,\ldots, 9\}$ a loop.
    \end{itemize}
    By the decomposition strategy,
    $$V_{\CCC(D)} \subset V_D \cup_{i=1}^6 V_{\CCC(A_i)} \cup V_{\CCC(B)} \cup_{i=1}^2 V_{\CCC(C_i)} \cup_{i=1}^9 V_{\CCC(\pi_D^i)} \cup_{i=1}^2 V_{\CCC(E_i)} \cup_{i=1}^2 V_{\CCC(F_i)} $$ $$\cup_{i=1}^2 V_{\CCC(F_i')} \cup_{i=1}^2 V_{\CCC(F_i'')} \cup_{i=1}^2 V_{\CCC(G_i)}  \cup_{i=1}^2 V_{\CCC(G_i')}  \cup_{i=1}^2 V_{\CCC(G_i'')}   \cup V_{U_{2,9}} \cup_{i=1}^9 V_{\CCC(D(i))}.$$   
   \begin{figure}
       \centering
       \includegraphics[width=0.9\linewidth]{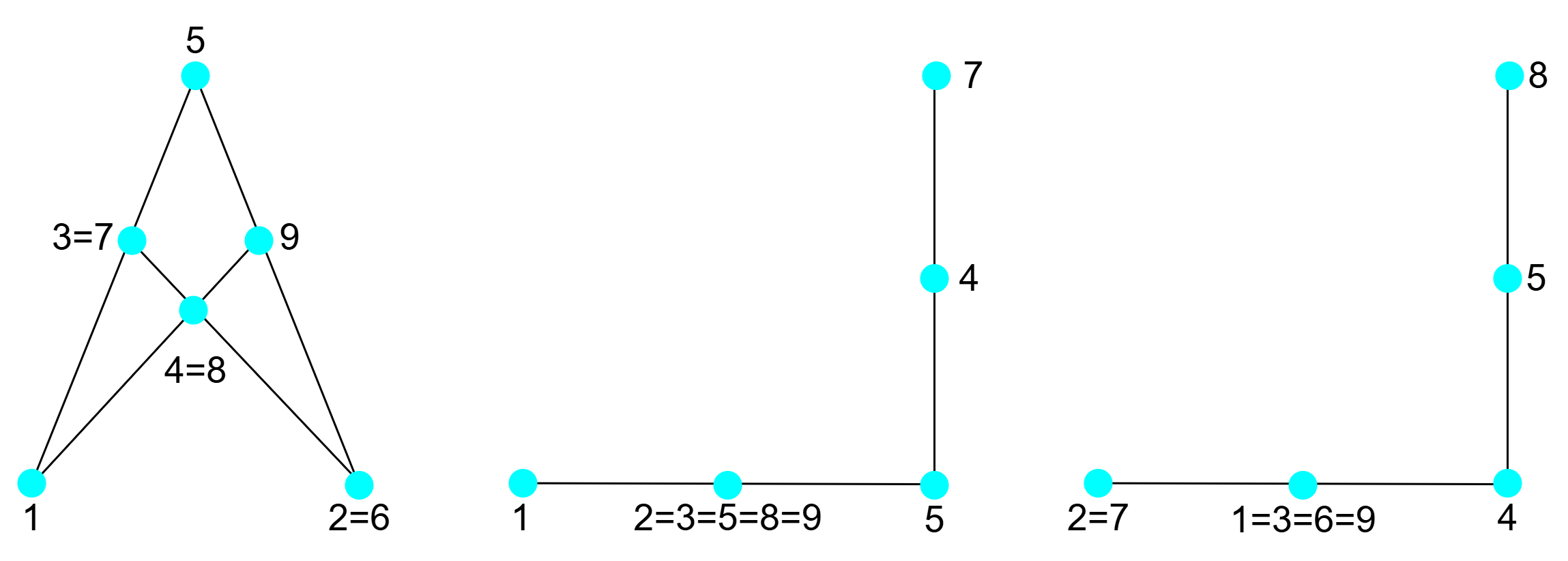}
       \caption{(Left) $E_1$; (Center) $F_1$; (Right) $G_1$.}
       \label{fig:third config decompo d}
   \end{figure} 
By Theorem~\ref{nil coincide} and Remark~\ref{remark reduced}, $V_{\CCC(E_i)} \subset V_{E_i} \cup V_{U_{2,9}},$ and similar results holds for $E_i'$ and $E_i''$. By Theorem~\ref{nil coincide}, $V_{\CCC(F_i)} = V_{F_i}, V_{\CCC(G_i)} = V_{G_i},$ $ V_{\CCC(\pi_D^i)} \subset V_{\CCC(\pi_M^i)} = V_{\pi_M^i}$ and similar results holds for $F_i'$ and $F_i'', G_i'$ and $G_i''$.
    
    An arbitrary $\xi \in \Gamma_D$ can be written as $$\xi = \begin{pmatrix} 
1 & 0 & 0 & 1 & 0 & 1 & x & -x^2 & x\\
0 & 1 & 0 & 1 & 1 & 1 & 1+x & 1-x^2& 1+x \\
0 & 0 & 1 & 1 & x & 0 & x & -x^2+x & x+x^2
\end{pmatrix},$$ where the columns from left to right correspond to $\{1,2,4,7,3,6,8,9,5\}$ and the minors corresponding to the bases are non-zero. Moreover, the determinant of the minor corresponding to the columns 5,9,6 is 0, implying $x^2=0.$ However, from this it follows that there are minors corresponding to bases which are 0, so $\Gamma_D$ is empty.
\newline
\newline
Now we prove that $V_{E_i} \subset V_M$. We prove this for $E_1$, the other proofs are similar. An arbitrary $\gamma \in \Gamma_M$ can be written as $$ \begin{pmatrix}
    1 & 0 & 0 & 1 & 1 & 1 & 1 & -x & -x \\
    0 & 1 & 0 & 1 & 1 & 0 & y+1 & 1 - (1+y)(1+x) & 1 - (1+y)(1+x) \\
    0 & 0 & 1 & 1 & 1+x & 1+x & 1 & 0 & 1 - (1+y)(1+x) 
\end{pmatrix},
$$
where the columns correspond to $\{1,2,5,8,4,3,7,6,9\}$ and the determinants of the minors corresponding to bases are non-zero. Moreover, the determinant of the minor corresponding to the points $\{3,7,9\}$ is 0. This is equivalent with $-y - 2 x y - y^2 - x y^2 = 0$. An arbitrary $\xi \in \Gamma_{E_1}$ can be written as $$\begin{pmatrix}
    1 & 0 & 0 & 1 & 1 & 1 & 1 & 0 & 0  \\
    0 & 1 & 0 & 1 & 1 & 0 & 0 & 1 & 1 \\
    0 & 0 & 1 & 1 & 1 & 1 & 1 & 0 &  1 
\end{pmatrix}.$$
Let $$
\left\{
    \begin{array}{ll}
        x = \epsilon \\
        y = \frac{1+2\epsilon}{-\epsilon-1}
    \end{array}
\right..
$$
Then $\gamma \in V_M$ for all $\epsilon$ and if $\epsilon \to 0$, then $\gamma \to \xi$. Thus $V_{E_1} \subset V_M$. The proof that $V_{E_i'} \subset V_M$ and $V_{E_i''} \subset V_M$ is analogous.
\newline
\newline
We prove that $V_{F_i} \subset V_M$. We will prove this for $F_1$, the other proofs are similar.
An arbitrary $\gamma \in \Gamma_M$ can be written as 
$$
\begin{pmatrix}
    1 & 0 & 0 & 1 & 0 & 1 & 1 & 1-x(1+y) & 1\\
0 & 1 & 0 & 1 & 1 & 1 & 1+y & 1-x-xy+y & 1+y \\
0 & 0 & 1 & 1 & x & 0 & 1 & 1-x & (1+y)x
\end{pmatrix},$$
where the columns correspond to $\{1,2,4,7,3,6,8,9,5\}$ and the determinants of the minors
corresponding to bases are non-zero. Moreover, the determinant of the minor corresponding to the columns 5,9,6 is 0, implying $$y (x (y + 2) - 1)=0.$$ An arbitrary $\xi \in \Gamma_{F_1}$ can be written as 
$$\begin{pmatrix}
    1 & 0 & 0 & 1 & 0 & 1 & 0 & 0 & 0 \\
    0 & 1 & 0 & 1 & 1 & 1 & 1 & 1 & 1 \\
    0 & 0 & 1 & 1 & 0 & 0 & 0 & 0 & 0 
\end{pmatrix},$$
where the columns again correspond to $\{1,2,4,7,3,6,8,9,5\}$  and the determinants of the minors
corresponding to bases are non-zero.
Define 
$$
\left\{
    \begin{array}{ll}
        x = \frac{\epsilon}{\epsilon +1} \\ \\
        1+y= \frac{1}{\epsilon}
    \end{array}
\right..
$$
Then $\gamma \in \Gamma_M$ and $\gamma \to \xi$ if $\epsilon \to 0$. Thus $\xi \in V_M$, and therefore $V_{F_1} \subset V_M$. The proof that $V_{F_i'} \subset V_M$ and $V_{F_i''} \subset V_M$ is analogous.
\newline
\newline
We prove that $V_{G_i} \subset V_M$. We will prove this for $G_1$, the other proofs are similar. An arbitrary $\gamma \in \Gamma_M$ can be written as 
$$ \begin{pmatrix}
    1 & 0 & 0 & 1 & 1 & 1 & 1 & -x & -x \\
    0 & 1 & 0 & 1 & 1 & 0 & y+1 & 1 - (1+y)(1+x) & 1 - (1+y)(1+x) \\
    0 & 0 & 1 & 1 & 1+x & 1+x & 1 & 0 & 1 - (1+y)(1+x) 
\end{pmatrix},
$$
where the columns correspond to $\{1,2,5,8,4,3,7,6,9\}$ and the determinants of the minors corresponding to bases are non-zero. Moreover, the determinant of the minor corresponding to the points $\{3,7,9\}$ is 0. This is equivalent with $-y - 2 x y - y^2 - x y^2 = 0.$
An element $\xi \in \Gamma_{G_i}$ can be written as $$\begin{pmatrix}
    1 & 0 & 0 & 1 & 1 & 1  & 0 & 1 & 1 \\
    0 & 1 & 0 & 1 & 1 & 0 & 1 & 0 &  0 \\
    0 & 0 & 1 & 1 & 0 & 0 & 0 & 0 & 0 \\
\end{pmatrix},$$
where the columns again correspond to $\{1,2,5,8,4,3,7,6,9\}$ as well.
Define 
$$
\left\{
    \begin{array}{ll}
        x = \frac{-1}{2\epsilon+1} \\
        1+y= \frac{1}{\epsilon}
    \end{array}
\right..
$$
Then $\gamma \in V_M$ for every $\epsilon \in \RR_0$ and if $\epsilon \to 0$, then $\gamma \to \xi$. So $V_{G_1} \subset V_M.$ The proofs that $V_{G_i'} \subset V_M$ and $V_{G_i''} \subset V_M$ are analogous.
\newline
\newline
It is clear that $V_{\CCC(D(i))} \subset V_{\CCC(M(i))}$ and $V_{\CCC(\pi_D^i)} \subset V_{\CCC(\pi_M^i)}$.
\newline
\newline
Using all the previous inclusions, we can conclude that:
$$
V_{\CCC(D)} \subset \cup_{i=1}^6 V_{\CCC(A_i)} \cup V_{B} \cup_{i=1}^2 V_{\CCC(C_i)} \cup_{i=1}^9 V_{\pi_M^i} \cup V_{U_{2,9}} \cup_{i=1}^9 V_{\CCC(M(i))}.
$$
4)
We will prove Equation \eqref{eq: lemma 5 third config}. Without loss of generality, we can assume that $i=9$.
First we find the minimal matroids over $M(9)$. Clearly, $9$ is a loop in each minimal matroid. In the following, automorphism means an automorphism of $M(9)$.
\begin{itemize}
    \item The matroid with $1=2=3$ and lines $\{1,7,8\}, \{4,5,8\}$ and $\{4,6,7\}$, as in Figure \ref{fig:decompo Third config M(9) part 1} (Left), and another matroid formed by automorphism. Denote these matroids respectively by $A_1$ and $A_2$.
    \item The matroid with $2=4=7$ and lines $\{1,3,5\},\{1,2,6\},\{8,2,5\}.$ Denote this matroid by $A'$
    \item The matroid with $2=6=7$ and lines $\{1,3,5\},\{2,3,4\},\{8,4,5\}$ and another matroid formed by automorphism. Denote these matroids respectively by $A_1''$ and $A_2''$.
    \item The matroid with $1=2, 4=5$, and lines $\{1,3,4\}, \{1,7,8\}$ and $\{4,6,7\}$: this circuit variety is contained in $V_{\CCC(B)}$, where $B$ is defined as in Section \ref{Decomposition of the Circuit Variety of the third configuration 93}, see Figure \ref{fig:Third Config A_1 B} (Right).
    \item The matroid with $1=3, 4=6$ and lines $\{1,2,4\},\{8,2,7\},\{8,4,5\}$, or $2=8,3=5$, as in Figure \ref{fig:decompo Third config M(9) part 1} (Center). Denote these matroids respectively by $C_1$ and $C_2$.
    \item The matroid with $1=2=8$ and lines $\{1,3,4,5\}$ and $\{4,6,7\}$, as in Figure \ref{fig:decompo Third config M(9) part 1} (Right), and another matroid formed by automorphism. Denote these matroids respectively by $D_1$ and $D_2$.
       \item The matroid $D'$ with $2=4=6$ and lines $\{1,3,5\},\{8,2,5,7\}$.
    \item The matroid $D''$ with $2=4=8$ and lines $\{1,3,5\},\{1,2,6,7\}.$
    \item The matroids with $1=3, 2=7$ and lines $\{1,2,4,6\}, \{4,5,8\}$, as in Figure \ref{fig:decompo Third config M(9) part 2} (Left), and another matroid by automorphism. Denote these matroids by $E_i$.
    \item The matroid formed by identifying $1=5$ and lines $\{1,2,6\},$ $\{8,2,7\},$ $\{2,3,4\},$ $\{8,1,4\},$ $\{4,6,7\}$, as in Figure \ref{fig:decompo Third config M(9) part 2} (Center). Denote this matroid by $F$.
    \item The matroids with $1=6$ or $5=8$ and lines $\{1,3,5\},$ $\{8,2,7\},$ $\{2,3,4\},$ $\{8,4,5\},$ $\{1,4,7\}$, as in Figure \ref{fig:decompo Third config M(9) part 2} (Right). Denote these matroids by $G_i$.
    \item The matroid with $3=4$ and lines $\{1,3,5,8\}, \{1,2,6\}, \{2,7,8\}$ and $\{3,6,7\}$, as in Figure \ref{fig:decompo third config part 3} (Left), and the three other matroids formed by automorphism. Denote these matroids by $H_i$.
    \item The matroid with lines $\{1,3,5,7\}, \{1,2,6\}, \{2,7,8\}, \{2,3,4\}, \{4,5,8\}$ and $\{4,6,7\}$, as in Figure \ref{fig:decompo third config part 3} (Center). Denote this matroid by $I$.
    \item The matroid with lines $\{1,3,5\}, $ $\{1,2,6\}, $ $\{2,7,8\}, $ $\{2,3,4\}, $ $\{4,5,8\}, $ $\{4,6,7\}, $ $\{3,6,8\}$, as in Figure \ref{fig:decompo third config part 3} (Right). Denote this matroid by $J$.
    \item The matroid with $1,2,4$ or $5$ a loop.
    \end{itemize}
    \begin{figure}
        \centering
        \includegraphics[width=0.9\linewidth]{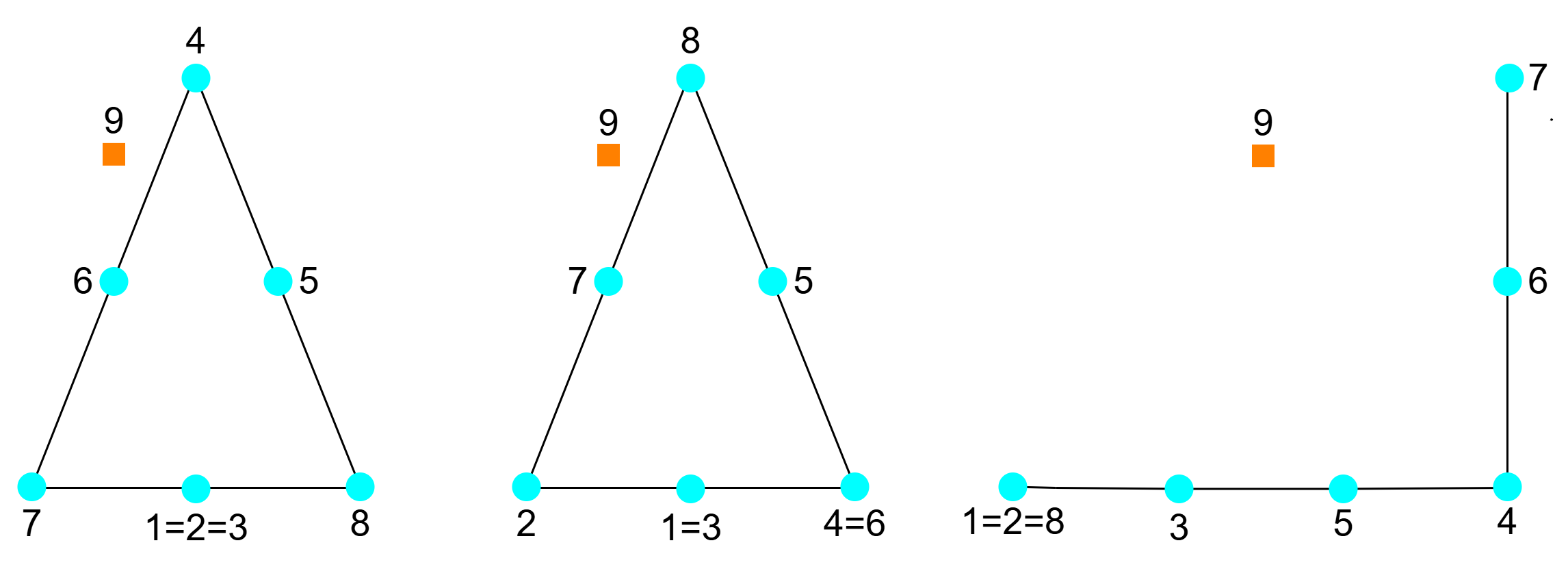}
        \caption{(Left) $A_1$; (Center) $C_1$; (Right) $D_1$.}
        \label{fig:decompo Third config M(9) part 1}
    \end{figure}
     \begin{figure}
        \centering
        \includegraphics[width=0.9\linewidth]{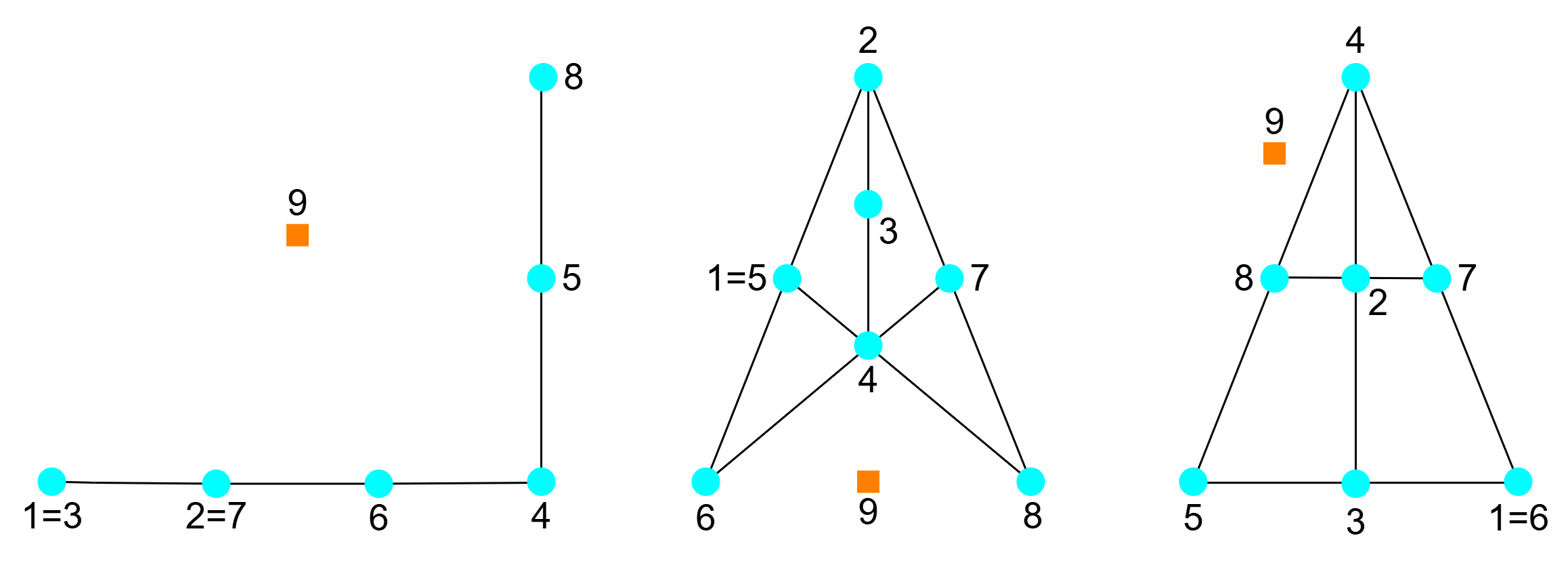}
        \caption{(Left) $E_1$; (Center) $F$; (Right) $G_1$.}
        \label{fig:decompo Third config M(9) part 2}    \end{figure}
     \begin{figure}
        \centering
        \includegraphics[width=0.9\linewidth]{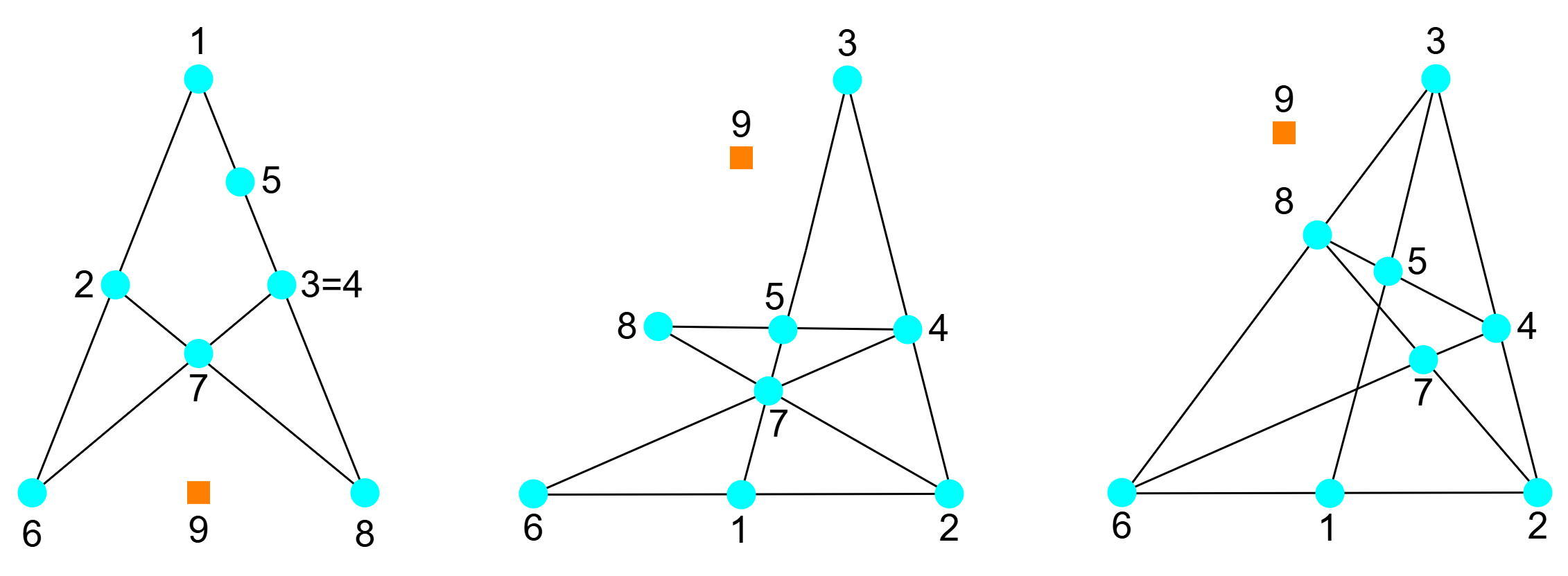}
        \caption{(Left) $H_1$; (Center) $I$; (Right) $J$.}
        \label{fig:decompo third config part 3}
    \end{figure}
By Theorem~\ref{nil coincide}, $V_{\CCC(A_i)} = V_{A_i}, $ $V_{\CCC(A)} = V_{A}, $ $V_{\CCC(A_i'')} = V_{A_i''}, $ $ V_{\CCC(B)} = V_B, $ $ V_{\CCC(C_i)} = V_{C_i},$ $V_{\CCC(D_i)} = V_{D_i}, $ $ V_{\CCC(E_i)} = V_{E_i}$. 
\newline
\newline
To prove that $V_{A_i} \subset V_{M(9)}$, we will prove this for $A_1$, since the other proofs are similar. Let $\xi \in \Gamma_{A_1}$. One can define $\widetilde{\xi}_2$ to be a perturbation of $\xi_2$ on the line through $\xi_2$ and $\xi_4$. Subsequently, one can define $\widetilde{\xi}_3 = \widetilde{\xi}_2 \xi_6 \wedge \xi_1 \xi_5$, by Lemma \ref{lemma: meet of 2 extensors arbitrary close}, $\widetilde{\xi_3} \to \xi_3$ if $\widetilde{\xi_2} \to \xi_2$. Let $\widetilde{\xi}_p = \xi_p$ for the other points. This shows that $\widetilde{\xi} \in \Gamma_{M(9)}$, implying $\xi \in V_{M(9)}$. Thus $V_{A_1} \subset V_{M(9)}.$ The proofs for $A'$ and $A''_i$ can be done in a similar way.
\newline\newline
To prove that $V_{C_i} \subset V_{M(9)}$, we will prove this for $C_1$, the other proofs are similar. Let $\xi \in \Gamma_{C_1}$. One can define $\widetilde{\xi}_4$ to be a perturbation of $\xi_4$ on the line through $\xi_5$ and $\xi_8$, such that it is not on the line through $\xi_1$ and $\xi_2$. Subsequently, one can define $\widetilde{\xi}_6 = \xi_1\xi_2 \wedge \widetilde{\xi}_4 \xi_7$. Finally define $\widetilde{\xi}_3 = \xi_2 \widetilde{\xi}_4 \wedge \xi_1 \xi_5$, and finishing the argument as above implies $V_{C_1} \subset V_{M(9)}$.
\newline 
\newline
To prove that $V_{D_i} \subset V_{M(9)}$, we will prove this for 
$D_1$, the other proofs are similar. Let $\xi \in \Gamma_{D_1}$. Then define $\widetilde{\xi}_1$ to be an arbitrary perturbation of $\xi_1$ on the line through $\xi_6$ and $\xi_2$, and the same for $\widetilde{\xi}_8$ on the line through $\xi_2$ and $\xi_7$. Then define $\widetilde{\xi}_5 = \widetilde{\xi}_1 \xi_3 \wedge \xi_4 \widetilde{\xi}_8$. Let $\widetilde{\xi}_p = \xi_p$ for the other points in the ground set. Since $\widetilde{\xi} \in \Gamma_M$, it follows by Lemma \ref{lemma: meet of 2 extensors arbitrary close} that $\xi
 \in V_{M(9)}$. Thus $V_{D_1} \subset V_{M(9)}$.
\newline
\newline
The proof for $D'$ and $D''$ is analogous.
\newline
\newline
Finally, to prove that $V_{E_i} \subset V_{M(9)}$, we will show this for $V_{E_1}$, the other proofs are similar. Let $\xi \in \Gamma_{E_1}$. Then define $\widetilde{\xi}_1$ to be an arbitrary perturbation of $\xi_1$ on the line through $\xi_3$ and $\xi_5$, and the same for $\widetilde{\xi}_7$ on the line through $\xi_2$ and $\xi_8$. Then define $\widetilde{\xi}_6 = \widetilde{\xi}_1 \xi_2 \wedge \xi_4 \widetilde{\xi}_7$. Finishing the proof as above, this shows that $V_{E_1} \subset V_{M(9)}$.

\medskip

So we can already conclude that:
$$
    V_{\CCC(M(9))} \subset V_{M(9)} \cup V_B \cup V_{\CCC(F)} \cup_{i=1}^2 V_{\CCC(G_i)} \cup_{i=1}^4 V_{\CCC(H_i)} $$
    \begin{equation} \label{eq: proof third config M(9) nilpotent ones definitive}\cup V_{\CCC(I)} \cup V_{\CCC(J)} \cup_{i=1,2,4,5} V_{\CCC(M(9,i))}.
\end{equation}
\begin{remark}
    To avoid repetition, from now on, if we only have proven $V_{Y_1} \subset V_{X}$ for the index 1, this means implicitly that the proof $V_{Y_j} \subset V_{X}$ for the other indices $j$ is either completely analogous or follows from a hypergraph automorphism.
\end{remark}
{\it Decomposing $V_{\CCC(F)}$} \newline
To decompose $V_{\CCC(F)}$, we use again the decomposition strategy. The minimal matroids over $F$ are the following; the points 1 and 5 are identified in each minimal matroid. :
\begin{itemize}
    \item The matroid $K_1$ with $1=2=8$ and lines $\{1,3,4\}$ and $\{4,6,7\}$, as in Figure \ref{fig:decompo M(9)F} (First Left), and three other matroids $K_i$ resulting from automorphisms from $F$. Denote these matroids by $K_i$.
    \item The matroid with $1=6$, $7=8$, and lines $\{2,3,4\}$ and $\{1,4,7\}$, as in Figure \ref{fig:decompo M(9)F} (Second Left), and the matroid resulting from an automorphism from $F$. Denote these matroids by $L_i$.
    \item The matroid with $2=3$, or $3=4$, as in Figure \ref{fig:decompo M(9)F} (Second Right). In the former case, the lines are $\{1,2,6\}, \{8,2,7\}, \{1,4,8\}, \{4,6,7\}$. Denote these matroids respectively by $M_1$ and $M_2$. 
    \item The matroid with $2=4$ and line $\{1,2,6,7,8\}.$ The corresponding circuit variety is contained in $V_{\CCC(\pi_M^3)} = V_{\pi_M^3}$, where $M$ denotes the matroid of the third configuration $9_3$.
    \item The matroid $F_i$ with $\{1,3,7\}$ or $\{3,6,8\}$ added as a circuit, as in Figure \ref{fig:decompo M(9)F} (First Right). Denote these matroids respectively by $N_1$ and $N_2$.
    \item The matroid with 2 or 4 being a loop. In the former case, the lines are $\{1,4,7\}$ and $\{4,6,7\}.$ Denote these matroids by $F(2)$ and $F(4)$.
\end{itemize}
\begin{figure}
    \centering
    \includegraphics[width=0.9\linewidth]{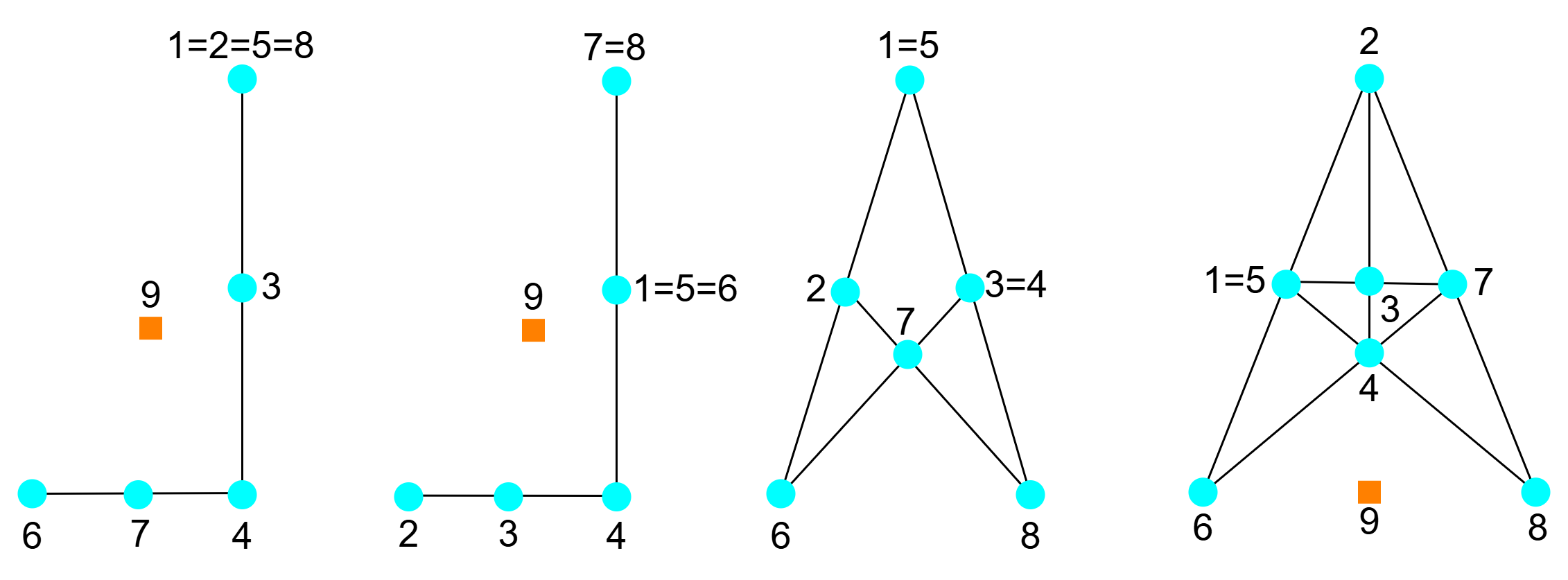}
    \caption{(First Left) $K_1$; (Second Left) $L_1$; (Second Right) $M_1$; (First Right) $N_1$.}
    \label{fig:decompo M(9)F}
\end{figure}
We prove that $V_{F} \subset V_{M(9)}$. Let $\xi \in \Gamma_{F}$ arbitrarily. Define $\widetilde{\xi}_1$ arbitrarily on the line through $\xi_2$ and $\xi_6$, such that $\widetilde{\xi}_1 \neq \xi_5$. Then define $\widetilde{\xi}_5 = \widetilde{\xi}_1 \xi_3 \wedge \xi_4 \xi_8$. By Lemma \ref{lemma: meet of 2 extensors arbitrary close}, $\widetilde{\xi}_5 \to \xi_5$ if $\widetilde{\xi_1} \to \xi_1$. Let $\widetilde{\xi}_p = \xi_p$ for the other points.  Then, $\widetilde{\xi} \in \Gamma_{M(9)}$, thus $\xi \in V_{M(9)}$. This proves that $V_F \subset V_{M(9)}.$ 
\newline
\newline
Except for $N_i$, all listed minimal matroids are nilpotent or quasi-nilpotent, thus the circuit variety of these matroids is contained in the union of their matroid variety and the matroid variety of the uniform matroid with nine points.
\newline
\newline
To prove that $V_{K_1} \subset V_F$, let $\xi \in \Gamma_{V_{K_1}}$ arbitrarily. Define $\widetilde{\xi}_1$ arbitrarily on the line through $\xi_2$ and $\xi_6$, such that $\widetilde{\xi}_1 \neq \xi_2$, and let $\widetilde{\xi}_5 = \widetilde{\xi}_1$. Then define $\widetilde{\xi}_8 = \widetilde{\xi}_1 \xi_4 \wedge \xi_2 \xi_7$. By Lemma \ref{lemma: meet of 2 extensors arbitrary close}, $\widetilde{\xi}_8 \to \xi_8$ if $\widetilde{\xi_1} \to \xi_1$. Let $\widetilde{\xi}_p = \xi_p$ for the other points. Then $\xi \in V_{F}$, thus $\widetilde{\xi} \in V_F$. This proves that $\widetilde{\xi} \in \Gamma_{F},$ thus $\xi \in V_F.$ Then, $V_{K_1} \subset V_{F}.$ 
\newline
\newline
The proof that $V_{L_i} \subset V_{F}$ is completely analogous.
\newline
\newline
To prove that $V_{M_1} \subset V_F$, let $\xi \in \Gamma_{V_{M_1}}$ arbitrarily. Perturb $\xi_3$ arbitrarily on the line through $\xi_2$ and $\xi_4$, such that it does not coincide with $\xi_2$. Then the perturbed $\xi$ lies in $\Gamma_F$, which completes the proof that $V_{M_1} \subset V_F.$ So one can conclude that for $F$:
\begin{equation} \label{eq: proof third config M(9) F}
    V_{\CCC(F)} \subset V_{M(9)} \cup V_{\CCC(M(9,2))} \cup V_{\CCC(M(9,4))} \cup_{i=1}^8 V_{\pi_M^i} \cup V_{\CCC(N_i)} \cup V_{U_{2,9}}.
\end{equation}
We have to decompose the circuit variety of $N_i$ again. Assume without loss of generality that $i=1$. The minimal matroids over $N_i$ are: 
\begin{itemize}
    \item The matroid with $1=2=4=5=7$. One can notice that the circuit variety of this matroid is contained in the circuit variety of $A'$, defined in the beginning of  part 4 of this proof, which is contained in $V_{M(9)}$.
    \item The matroid with $1=2, 4=7$ and line $\{1,3,4,8\}$, and two other matroids formed by automorphisms. The circuit variety of these matroids is contained in $V_{\CCC(\pi_M^i)} = V_{\pi_M^i}$ for some $i \in \{1, \ldots, 8\}$.
    \item The matroid $N_i$ with additional circuit $\{3,6,8\}$. Denote this matroid by $P$.
    \item The matroid with an additional loop from the set $\{1,2,4,7\}$. If the loop is 1, then the lines are $\{2,7,8\}, \{2,3,4\}, \{4,6,7\}$. The circuit varieties of these matroids are contained in $V_{M(9,i)}$, for some $i \in [8]$.
\end{itemize}
To prove that $V_{N_1} \subset V_{M(9)}$, we will show that $V_{N_1} \subset V_F$. Let $\xi \in \Gamma_{N_1}.$ Perturb $\xi_3$ arbitrarily on the line through $\xi_2$ and $\xi_4$, such that the perturbed $\xi_3$ is not on the line through $\xi_1$ and $\xi_7$ anymore. The perturbed $\xi$ lies in $\Gamma_F$. This completes the proof.
\newline
\newline
Thus for $N_i$: 
\begin{equation} \label{eq: proof third config M(9) F 2}
V_{\CCC(N_i)} \subset V_{M(9)} \cup_{i=1}^8 V_{\CCC(M(9,i))} 
\cup_{i=1}^8 V_{\pi_M^i} \cup V_{\CCC(P)}. 
\end{equation}
To decompose $V_{\CCC(P)}$, we consider the minimal matroids over $P$:
\begin{itemize}
\item The matroid with $1=2=3=8$ and circuit $\{4,6,7\}$ and six other matroids of a similar form. These circuit varieties are contained in $V_{N_1}$ or $V_{N_2}$ by Lemma~7.6 of \cite{liwskimohammadialgorithmforminimalmatroids}, which is contained in $V_{M(9)}$. Notice that we cannot use that lemma to state that the circuit varieties are contained in $V_{P}$, since realizability of the matroid is a condition for that lemma. 
\item The matroid with $1=2,8=3,4=7$ and six other matroids of a similar form. These circuit varieties are contained in $V_{\CCC(\pi_M^i)}$ with $i$ varying from 1 to 8.
\item The uniform matroid, with $1=5$ and 9 a loop.
\item The matroids with an additional loop from the set $\{1, 2, \ldots, 8\}$. 
\end{itemize}
$V_P$ is empty, since after a projective transformation, one can write $\gamma \in \Gamma_P$ as:
$$
\begin{pmatrix}
    1 & 0 & 0 & 1 & 0 & 1 & 1 & 1 & 0 \\
    0 & 1 & 0 & 1 & 1 & 1 & 0 & 0 & 0\\
    0 & 0 & 1 & 1 & 1 & 0 & 1 & 0 & 0
\end{pmatrix},
$$
where the columns correspond to $\{1,2,4,7,3,6,8,5,9\}$ and the minor corresponding to the columns $\{3,6,8\}$ is 0. This is contradiction.
Thus: 
\begin{equation} \label{eq: proof third config M(9) F 3}
V_{\CCC(P)} \subset V_{M(9)} \cup_{j=1}^8 V_{\pi_{M}^j} \cup V_{U_{2,9}} \cup V_{\CCC(M(9,i))}.
\end{equation}
Combining equations \eqref{eq: proof third config M(9) F}, \eqref{eq: proof third config M(9) F 2} and \eqref{eq: proof third config M(9) F 3}, one can conclude:
\begin{equation} \label{eq: proof third config M(9) F definitive}
    V_{\CCC(F)} \subset V_{M(9)} \cup_{i=1}^8 V_{\pi_{M}^i} \cup_{i=1}^8 V_{\CCC(M(9,i))} \cup V_{U_{2,9}}.
\end{equation}

{\it Decomposing $V_{\CCC(G_i)}$} \newline
Using Lemma~5.1 of \cite{liwskimohammadialgorithmforminimalmatroids}:
\begin{equation*} 
V_{\CCC(G_1)} = V_{G_1} \cup V_{U_{2,9}'} \cup V_{G_1(2)},\end{equation*} where $U_{2,9}'$ is obtained from the uniform matroid $U_{2,9}$ by setting the point 9 to be a loop and identifying the points 1 and 5.

We prove that $V_{G_1} \subset V_{M(9)}$. Let $\xi \in \Gamma_{V_{G_1}}$ arbitrarily. Define $\widetilde{\xi}_1$ as a perturbation of $\xi_1$ on the line through $\xi_3$ and $\xi_5$, such that $\widetilde{\xi}_1 \neq \xi_6$. Then define $\widetilde{\xi}_6 = \widetilde{\xi}_1 \xi_2 \wedge \xi_4 \xi_7$. By Lemma \ref{lemma: meet of 2 extensors arbitrary close}, $\widetilde{\xi}_6 \to \xi_6$ if $\widetilde{\xi_1} \to \xi_1$. Let $\widetilde{\xi}_p = \xi_p$ for the other points.  Then, $\widetilde{\xi} \in \Gamma_{M(9)}$, thus $\xi \in V_{M(9)}$, implying $V_{G_1} \subset V_{M(9)}.$ 
\newline
\newline
By Theorem \ref{nil coincide}, $V_{G_1(2)} = V_{\CCC(G_1(2))} \subset V_{\CCC(M(9,2))} = V_{M(9,2)}.$
Thus we can conclude that \begin{equation} \label{eq: proof third config M(9) G definitive}
V_{\CCC(G_1)} \subset V_{M(9)} \cup V_{M(9,2)} \cup V_{U_{2,9}}.
\end{equation}
\newline

{\it Decomposing $V_{\CCC(H_i)}$}
\newline
We can assume without loss of generality that $i=1$. The minimal matroids over $H_1$ are the following; automorphism means a hypergraph automorphism of $H_1$ and the points 3 and 4 are identified in each matroid:
\begin{itemize}
    \item The matroid $S_1$ with $1=2=8$ and lines $\{1,3,5\},\{3,6,7\}$ and two other matroids $S_i$ formed by automorphisms of $H_1$. Their circuit varieties are contained in those of of $D_1$, $C_1$ and $A_2''$. $V_{D_1}, V_{C_1}$ and $V_{A_2''}$ are contained in $V_{M(9)}.$ 
    \item The matroid $T_1$ with $1=3, 2=7$, and lines $\{8,1,5\},\{1,2,6\}$, and two other matroids $T_i$ formed by automorphisms of $H_1$.
    \item The matroid with $1=2$, $3=7$ and line $\{1,3,5,8\}$, and three other matroids for which the circuit variety is contained in $V_{\CCC(\pi_M^i)},$ where $i$ ranges from 1 to 8.
    \item The matroid with $2=6=7$, $3=4$ and line $\{1,3,5,8\}$. The matroid variety is contained in that of $V_{H_1}$ by Lemma~7.6 of \cite{liwskimohammadialgorithmforminimalmatroids}.
    \item The matroids with $1=5$ or $3=5$ or $5=8$. Their circuit varieties are contained in the circuit varieties of respectively $F$, $C_2$ and $G_2$.
    \item The matroid obtained from the uniform matroid $U_{2,9}$ by setting the point 9 to be a loop, and identifying the points 3 and 4.
    \item The matroids with an additional loop from the set $\{1,2,3,6,7,8\}$. 
\end{itemize}
First we prove that $V_{H_1} \subset V_{M(9)}$. Let $\xi \in \Gamma_{V_{H_1}}$ be arbitrarily. Define $\widetilde{\xi}_4$ arbitrarily on the line through $\xi_6$ and $\xi_7$, such that $\widetilde{\xi}_4 \neq \xi_3$. Then define $\widetilde{\xi}_3 = \widetilde{\xi}_4 \xi_2 \wedge \xi_1 \xi_5$. Define $\widetilde{\xi}_8 = \widetilde{\xi}_4 \xi_5 \wedge \xi_2 \xi_7$. By Lemma \ref{lemma: meet of 2 extensors arbitrary close}, $\widetilde{\xi}_3 \to \xi_3$ and $\widetilde{\xi}_8 \to \xi_8$ if $\widetilde{\xi}_4 \to \xi_4$. Let $\widetilde{\xi}_p = \xi_p$ for the other points. Then, $\widetilde{\xi} \in \Gamma_{M(9)}$, thus  $\xi \in \Gamma_{M(9)}$. This implies that $V_{H_1} \subset V_{M(9)}.$ 
\newline
\newline
To prove that $V_{T_1} \subset V_{H_1}$, let $\xi \in \Gamma_{V_{T_1}}$ be arbitrarily. Define $\widetilde{\xi}_7$ as a perturbation of $\xi_4$ on the line through $\xi_2$ and $\xi_8$, such that $\widetilde{\xi}_7 \neq \xi_3$. Then define $\widetilde{\xi}_3 = \widetilde{\xi}_7 \xi_6 \wedge \xi_1 \xi_8$ and $\widetilde{\xi}_4=\widetilde{\xi}_3$. By Lemma \ref{lemma: meet of 2 extensors arbitrary close}, $\widetilde{\xi}_3 \to \xi_3$ if $\widetilde{\xi}_7 \to \xi_7$. Let $\widetilde{\xi}_p = \xi_p$ for the other points. Then, $\widetilde{\xi} \in \Gamma_{H_1},$ then $\xi \in V_{H_1}$, implying that $V_{T_1} \subset V_{H_1}.$ 
\newline
\newline
So we can conclude that
\begin{equation} \label{eq: proof third config M(9) H definitive}
    V_{\CCC(H_1)} \subset V_{M(9)} \cup_{i=1}^8 V_{\pi_{M}^i} \cup_{i=1}^8 V_{\CCC(M(9,i))} \cup V_{U_{2,9}}.
\end{equation}

{\it Decomposing $V_{\CCC(I)}$} \label{subsec: decomposing VCI}
\newline
The minimal matroids over $I$ are the following:
\begin{itemize}
    \item The matroids with $1=2=3=7$ or $3=4=5=7$. Their circuit varieties are contained in $V_{\CCC(A_i)} \subset V_{M(9)}.$ $A_i$ is shown in Figure \ref{fig:decompo Third config M(9) part 1} (Left). 
    \item The matroids with $1=2=3$, $5=8$ and three other matroids of a similar form. Their circuit varieties are contained in $V_{\CCC(A_1)}, V_{\CCC(A_2)}, V_{\CCC(D')}$ and $ V_{\CCC(D'')}$. These circuit varieties are contained in $V_{M(9)}.$
    \item The matroids with $1=3,4=6$ or $8=2,3=5$. Their circuit varieties are contained in $V_{C_1}$ and $V_{C_2}$, which are contained in $V_{M(9)}$. 
    \item The matroid with $1=5$. Its circuit variety is contained in $V_{\CCC(F)}.$ 
    \item The matroid with $1=2, 4=7$ and another matroid for which the circuit varieties are contained in $V_{\pi_M^i},$ with $i \in \{1,2, \ldots, 8\}$.
    \item The matroids with $1=6=7$ or $8=5=7$, for which the circuit variety is contained in respectively $V_{\CCC(G)}$ and $V_{\CCC(F)}.$
    \item The matroid $U$ with lines $\{1,3,5,7\},$ $\{1,2,6\},$ $\{8,2,7\},$ $\{2,3,4\},$ $\{8,4,5\},$ $\{4,6,7\},$ $\{8,3,6\}$, as in Figure \ref{fig:decompo M(9) HIJ}. 
    \item The matroids with an additional loop from the set $\{1,2,4,5,7\}.$
\end{itemize}
\begin{figure}
    \centering
    \includegraphics[width=0.9\linewidth]{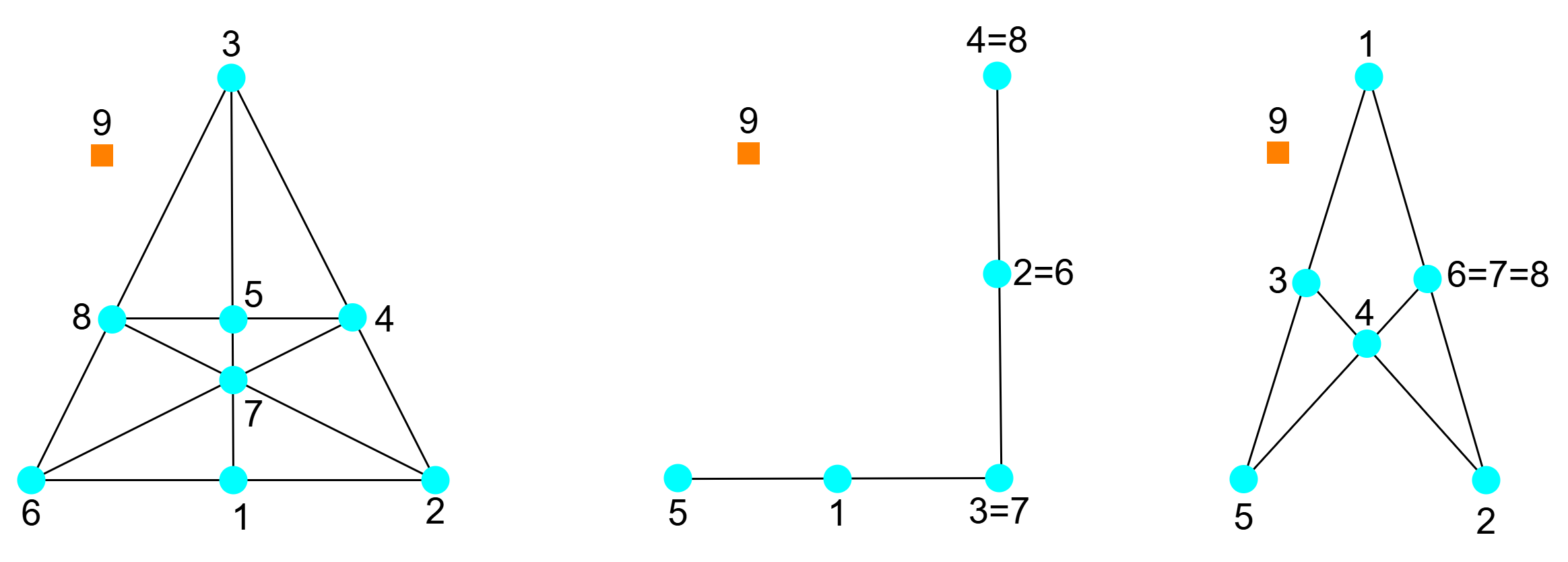}
    \caption{(Left) $U$; (Center) $V$; (Right) $X$.}
    \label{fig:decompo M(9) HIJ}
\end{figure}
$V_I \subset V_{M(9)}$: 
An arbitrary $\gamma \in \Gamma_{M(9)}$ can be written as $$\gamma = \begin{pmatrix} 
1 & 0 & 0 & 1 & 0 & -x & 1 & 1&0 \\
0 & 1 & 0 & 1 & 1 & 1-x & 1 & 0 &0\\
0 & 0 & 1 & 1 & x & 0 & 1+y & 1+y-x&0 
\end{pmatrix},$$ where the minors corresponding to the bases are non-zero and the columns from left to right correspond to $\{1,3,2,8,4,5,7,6,9\}.$ 
An arbitrary $\xi \in \Gamma_{I}$ can be written as \begin{equation} \label{eq: proof third config I matrix}
\xi = \begin{pmatrix} 
1 & 0 & 0 & 1 & 0 & -z & 1 & 1&0 \\
0 & 1 & 0 & 1 & 1 & 1-z & 1 & 0 &0\\
0 & 0 & 1 & 1 & z & 0 & 0 & -z &0
\end{pmatrix},\end{equation}
where the minors corresponding to the bases are non-zero and the columns from left to right correspond to $\{1,3,2,8,4,5,7,6,9\}$ as well.
Fix a $\xi$. Then we have to prove that $\xi \in V_{M(9)}.$ Set $$
\left\{
    \begin{array}{lll}
       1+y & = &\epsilon  \\
    x & = &z
    \end{array}
\right.
$$ in $\gamma$.
If $\epsilon \to 0$, then $\gamma \to \xi$. So $\xi \in V_{M(9)}$. 
\newline \newline
To decompose the matroid $U$, one finds the 
following minimal matroids:
\begin{itemize}
    \item Seven matroids for which the circuit variety is contained in $V_{\pi_M^i}.$ 
    \item A matroid with $1=5$, for which the circuit variety is contained in $V_{\CCC(F)}$.
    \item The matroids with $2=3=6=7$ or $3=4=7=8$ for which the circuit variety is cointained in $V_{\CCC(H_i)}$. 
    \item The matroid $V$ with $2=6,3=7,8=4$ and lines $\{1,3,5\},\{2,3,4\}$, as in Figure \ref{fig:decompo M(9) HIJ} (Center).
\end{itemize}
An arbitrary $\gamma \in \Gamma_{I}$ can be written as $$\gamma = \begin{pmatrix} 
1 & 0 & 0 & 1 & 0 & -x & 1 & 1 &0\\
0 & 1 & 0 & 1 & 1 & 1-x & 1 & 0 &0\\
0 & 0 & 1 & 1 & x & 0 & 0 & -x&0
\end{pmatrix},$$ where the columns from left to right correspond to $\{1,3,2,8,4,5,7,6,9\}.$ 
An arbitrary $\xi \in \Gamma_{U}$ can be written as $$
\xi = \begin{pmatrix} 
1 & 0 & 0 & 1 & 0 & 1 & 1 & 1&0 \\
0 & 1 & 0 & 1 & 1 & 2 & 1 & 0 &0\\
0 & 0 & 1 & 1 & -1 & 0 & 0 & 1 &0
\end{pmatrix},$$
where the minors corresponding to the bases are non-zero and the columns from left to right correspond to $\{1,3,2,8,4,5,7,6\}$ as well.
Fix such a $\xi$. Then we have to prove that $\xi \in V_{M(9)}.$ If $x \to -1$, then $\gamma \to \xi$. So $\xi \in V_{I}$. 
\newline \newline
To prove that $V_V \subset V_{I}$, let $\xi \in \Gamma_{V}$ be arbitrarily. Define $\widetilde{\xi}_7$ as perturbation of $\xi_7$ on the line through $\xi_1$ and $\xi_5$, such that $\widetilde{\xi}_7 \neq \xi_3$. Then define $\widetilde{\xi}_6 = \xi_1 \xi_2 \wedge \xi_4 \widetilde{\xi}_7$, and $\widetilde{\xi}_8 = \xi_4 \xi_5 \wedge \xi_2 \widetilde{\xi}_7$. By Lemma \ref{lemma: meet of 2 extensors arbitrary close}, $\widetilde{\xi}_6 \to \xi_6$ and $\widetilde{\xi}_8 \to \xi_8$ if $\widetilde{\xi}_7 \to \xi_7$. Let $\widetilde{\xi}_p = \xi_p$ for the other points. Then, $\widetilde{\xi} \in \Gamma_{I}$, implying $\xi \in V_{I}.$ Thus $V_V \subset V_I$.

\medskip

So $V_U \subset V_{M(9)} \cup V_{U_{2,9}} \cup_{j=1}^8 V_{\CCC(M(9,j))} \cup_{j=1}^8 V_{\pi_{M}^j}$ as well.
\newline
\newline
Combining the previous results:  
\begin{equation} \label{eq: proof third config M(9) I definitive}
V_{\CCC(I)} \subset V_{M(9)} \cup_{i=1}^8 V_{\pi_{M}^i} \cup_{i=1}^8 V_{\CCC(M(9,i))} \cup V_{U_{2,9}}.\end{equation}
{\it Decomposing $V_{\CCC(J)}$}
\newline
The minimal matroids over $J$ are the following:
\begin{itemize}
    \item The matroids with $1=2=3=8$ or $3=4=5=6$. Their circuit varieties are contained in those of $V_{\CCC(A_i)} \subset V_{M(9)}$ as defined in the beginning of part 4 of this proof.
    \item The matroids with $1=3=4=6$ or $8=2=3=5$. Their circuit varieties are contained in those of $V_{\CCC(C_i)} \subseteq V_{M(9)}$, as defined in the beginning of part 4 of this proof.
    \item The matroids with $1=2=3$, $6=7$ or $ 3=4=5,8=7$, their circuit varieties are contained in those of $V_{\CCC(A_i)} \subset V_{M(9)}$, as defined in the beginning of part 4 of this proof.
    \item The matroids with $8=3=5,4=7$ or $1=3=6, 2=7$, their circuit varieties are contained in those of $V_{\CCC(G_i)} \subset V_{M(9)} \cup V_{M(9,2)} \cup V_{U_{2,9}}$, as defined in the beginning of part 4 of this proof.
    \item The matroid with $1=2$, $4=5$, for which the circuit variety is contained in that of $V_B$. $B$ is defined in Section \ref{Decomposition of the Circuit Variety of the third configuration 93}.
    \item The matroid $W$ with $2=6, 4=8$ and lines $\{1,3,5\},\{2,3,4,7\}$.
    \item The matroid with $1=5$, for which the circuit variety is contained in that of $V_{\CCC(F_i)}$.
    \item The matroid with $1=6, 5=8$, for which the circuit variety is contained in those of $V_{\CCC(G_i)}$.
    \item The matroid with $2=4=7$. This matroid variety is contained in $V_{\CCC(A')}$. 
    \item The matroid $X$ with $6=7=8$ and lines $\{1,3,5\},$ $\{1,2,6\},$ $\{2,3,4\},$ $\{4,5,6\}$, as in Figure \ref{fig:decompo M(9) HIJ} (Right).
    \item The matroid $U$ as encountered in the decomposition of $I$ as defined in the beginning of part 4 of this proof.
    \item Matroids with an additional loop from the set $\{1,2,3,4,5,6,8\}$.
\end{itemize}
First we prove that $V_J \subset V_{M(9)}$:
An arbitrary $\gamma \in \Gamma_{M(9)}$ can be written as $$\gamma = \begin{pmatrix} 
1 & 0 & 0 & 1 & 0 & -x & 1 & 1 &0\\
0 & 1 & 0 & 1 & 1 & 1-x & 1 & 0 &0\\
0 & 0 & 1 & 1 & x & 0 & 1+y & 1+y-x &0
\end{pmatrix},$$ where the columns from left to right correspond to $\{1,3,2,8,4,5,7,6,9\}$ and the minors corresponding to the bases are non-zero. 
\newline \newline
An arbitrary $\xi \in \Gamma_{J}$ can be written as $$\gamma = \begin{pmatrix} 
1 & 0 & 0 & 1 & 0 & -z & 1 & 1 &0\\
0 & 1 & 0 & 1 & 1 & 1-z & 1 & 0 &0\\
0 & 0 & 1 & 1 & x & 0 & 1+z & 1 &0
\end{pmatrix},$$ where the columns from left to right correspond to $\{1,3,2,8,4,5,7,6,9\}$ and the minors corresponding to the bases are non-zero.
Let $$
\left\{
    \begin{array}{lll}
       y & = &z+ \epsilon  \\
    x& = &z
    \end{array}
\right..
$$
Then $\epsilon \to 0$ implies $\gamma \to \xi$. Therefore, $\xi \in V_{M(9)}$. This proves the statement.
\newline
\newline
To prove that $V_X \subset V_J$, let $\xi \in \Gamma_{X}$ be arbitrarily. Define $\widetilde{\xi}_8$ arbitrarily on the line through $\xi_4$ and $\xi_5$, such that $\widetilde{\xi}_8 \neq \xi_7$. Then define $\widetilde{\xi}_6 = \widetilde{\xi}_8 \xi_3 \wedge \xi_1 \xi_2$. Define $\widetilde{\xi}_7 = \widetilde{\xi}_6 \xi_4 \wedge \xi_2 \widetilde{\xi}_8$. By Lemma \ref{lemma: meet of 2 extensors arbitrary close}, $\widetilde{\xi}_6 \to \xi_6$ and $\widetilde{\xi}_7 \to \xi_7$ if $\widetilde{\xi}_8 \to \xi_8$. Let $\widetilde{\xi}_p = \xi_p$ for the other points. Then, $\widetilde{\xi} \in \Gamma_J$, so $\xi \in V_J$. This implies that $V_{X} \subset V_{J}.$ 
The proof that $V_W \subset V_{M(9)}$ is completely analogous to the other proofs.
\newline
\newline
From this, one can conclude that:
\begin{equation} \label{eq: proof third config M(9) J definitive}
V_{\CCC(J)} \subset V_{M(9)} \cup V_B \cup_{i=1}^8 V_{\pi_{M}^i} \cup_{i=1}^8 V_{\CCC(M(9,i))} \cup V_{U_{2,9}}.
\end{equation}

{\it Decomposing $V_{\CCC(M(9,i))}$} \newline
Finally, we have to decompose $V_{\CCC(M(9,i))}$. 
We have to distinguish four different cases due to symmetry: 
\begin{itemize}
\item Points in $\{3,6,8\}$ (all such points are on a line with 9).
\item Points in $\{1,5\}$, which are on a line with 9.
\item The point $\{7\}$, which is not on a line with 9. \item The points $\{2,4\}$, which are not on a line with 9.
\end{itemize}
The matroid $M(9,3)$ has lines $\{1,2,6\},\{8,2,7\},\{8,4,5\},\{4,6,7\}$, so it is nilpotent and has no points of degree greater than two, implying that $V_{\CCC(M(9,3))} = V_{M(9,3)}.$ Moreover, $V_{M(9,3)} \subset V_{M(9)}$, since the point 3 has degree two in $M(9)$, so that we can use Lemma \ref{lemma: third config loop matroid variety}. In conclusion:
\begin{equation*} \label{eq: appendix V(M(3)) M(3,6) definitive
}
    V_{\CCC(M(3,9))} \subset V_{M(9)}.
\end{equation*}
The matroid $M(9,1)$ has lines $\{8,2,7\},\{2,3,4\},\{8,4,5\},\{4,6,7\}$, and the point 4 has degree three. The minimal matroids in this case are the following; the points 1 and 9 are loops in each matroid: 
\begin{itemize}
\item Five nilpotent matroids for which the matroid variety is contained in $V_B$, where $B$ is defined as in Section \ref{Decomposition of the Circuit Variety of the third configuration 93}, see Figure \ref{fig:Third Config A_1 B} (Right), and in $V_{\CCC(A_1)}$
, $V_{\CCC(E_1)}$, 
$V_{\CCC(D_1)}$,  
$V_{\CCC(C_1)}$, as defined in part 4 of this proof. These circuit varieties are contained in $V_{M(9)}$.
\item The matroid with $3=4$ and lines $\{8,2,7\},\{8,3,5\},\{3,6,7\}$, and two matroids of a similar form. 
It is analogous to the previous cases to prove that their circuit varieties are contained in $V_{M(9)}.$
\item The matroid with $8=7$ and lines $\{2,3,4\},\{4,5,6,7\}$. It is analogous to the previous cases to prove that its circuit variety is contained in $V_{M(9)}.$
\item The matroid $Y$ with lines $\{8,2,7\},\{2,3,4\},\{8,4,5\},\{4,6,7\},\{8,3,6\}.$
\item The matroid $Z$ with lines $\{8,2,7\},\{2,3,4\},\{8,4,5\},\{4,6,7\},\{2,5,6\}$ and two other matroids formed by automorphisms of $M(9,1).$
\item The matroid $M(1,4,9)$.
\end{itemize}
By Lemma \ref{lemma: third config loop matroid variety}, $V_{M(9,1)} \subset V_{M(9)}$. The decomposition of $V_{\CCC(Y)}$ and $V_{\CCC(Z)}$ is completely analogous to the previous calculations, except that we have to use the following statement a couple of times: for all points $i$ except 4, $V_{M(9,1,i)} \subset V_{M(9,1)}$ by Lemma \ref{lemma: third config loop matroid variety} again. Moreover, $V_{M(9,1,4)} \subset V_{M(9,4)}.$
So $$V_{\CCC(M(9,1))} \subset V_{M(9)} \cup V_B \cup_{i=1}^8 V_{\pi_M^i} \cup V_{M(9,4)} \cup V_{U_{2,9}}.$$
For the matroid $M(9,7)$, the lines are $\{1,3,5\},\{1,2,6\},\{2,3,4\},\{8,4,5\}$, so using Lemma \ref{lemma: third config loop matroid variety}, $V_{\CCC(M(9,7))} =V_{M(9,7)} \subset V_{M(9)}$. 
\newline
\newline
Finally, for $V_{M(9,2)}$, the lines are $\{1,3,5\}, \{8,4,5\},\{4,6,7\}$, so this matroid is nilpotent and has no points of degree greater than two. Therefore $V_{\CCC(M(9,2))} = V_{M(9,2)}$.
\newline
\newline
So 
\begin{equation} \label{eq: proof third config M(9) M(9,i) definitive}
\cup_{i=1}^8 V_{\CCC(M(9,i))} \subset V_{M(9)} \cup V_{M(9,2)} \cup V_{M(9,4)} \cup V_B \cup_{i=1}^8 V_{\pi_M^i} \cup V_{U_{2,9}}.\end{equation}
So combining equations \eqref{eq: proof third config M(9) nilpotent ones definitive}, \eqref{eq: proof third config M(9) F definitive}, \eqref{eq: proof third config M(9) G definitive}, \eqref{eq: proof third config M(9) H definitive}, \eqref{eq: proof third config M(9) I definitive}, \eqref{eq: proof third config M(9) J definitive} and \eqref{eq: proof third config M(9) M(9,i) definitive}:
$$ V_{\CCC(M(i))} \subset  V_{M(i)} \cup_{j=1,2} V_{M(i,k_j^i)} \cup V_B \cup_{j=1}^9 V_{\pi_{M}^j} \cup V_{U_{2,9}},$$ where $k_j^i$ and $l_1^i$ are defined as in Equation \eqref{eq: lemma 5 third config}.
\newline
\newline
5) 
 Finally, we will prove Equation \eqref{eq: lemma 6 third config}.
This proof is very analogous to that of Equation \eqref{eq: lemma 5 third config}. The minimal matroids over $M(3)$ are the following; the point 3 is a loop in each minimal matroid:
\begin{itemize}
    \item The matroid $A_1$ with $1=2=8$ and lines $\{1,4,5\},\{4,6,7\},\{9,5,6\}$, as in Figure \ref{fig:decompo third config M(3)} (Left), and five other matroids $A_i$ formed by automorphisms of $M(3)$.
    \item The matroid $B_1$ with $1=4=6$ and lines $\{8,2,7\},\{8,1,5,9\}$, as in Figure \ref{fig:decompo third config M(3)} (Center), and four other matroids formed by automorphisms of $M(3)$.
    \item The matroid $B'$ with $7=8=9$ and lines $\{1,2,6\},\{4,5,6,7\}$.
    \item The matroid $C_1$ with $1=6, 5=8$ and lines $\{2,5,7\},\{1,4,7\},\{1,5,9\}$, as in Figure \ref{fig:decompo third config M(3)} (Right), and two other matroids formed by automorphisms of $M(3)$.
    \item The matroid $C'_1$ with $1=8, 6=7$ and lines $\{1,2,6\},\{1,4,5\},\{9,5,6\}$, and two other matroids formed by automorphisms of $M(3)$.
    \item The matroid $D_1$ with $1=2$, and lines $\{8,1,9,7\},$ $\{8,4,5\},$ $\{4,6,7\},$ $\{9,5,6\}$, as in Figure \ref{fig:decompo third config M(3) part 2} (Left), and five other matroids $D_i$ formed by automorphisms.
    \item The matroid $E_1$ with lines $\{1,2,6\},$ $\{8,2,7\},$ $\{8,4,5\},$ $\{4,6,7\},$ $\{9,5,6\},$ $\{8,1,9\},$ $\{1,5,7\}$, as in Figure \ref{fig:decompo third config M(3) part 2} (Center), and another matroid $E_2$ formed by automorphism.
    \item The matroid with an additional loop in $\{6,8\}.$
\end{itemize}
\begin{figure}
    \centering
    \includegraphics[width=0.9\linewidth]{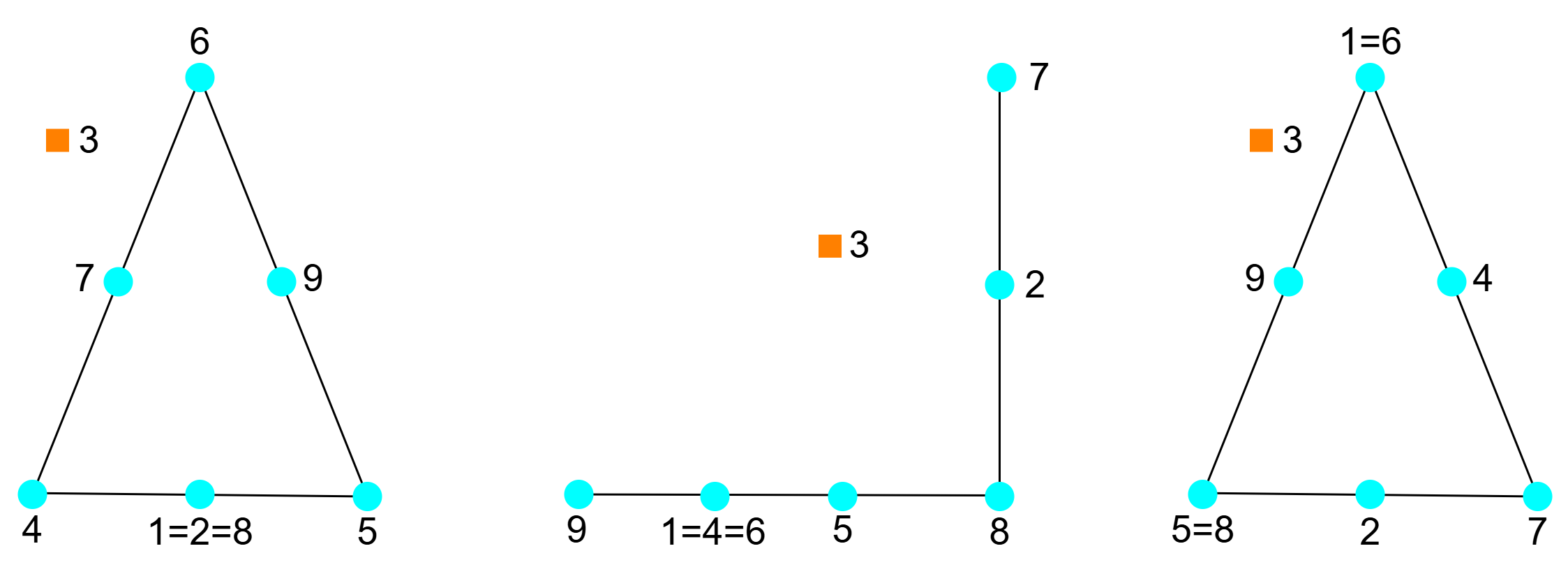}
    \caption{(Left) $A_1$; (Center) $B_1$; (Right) $C_1$.}
    \label{fig:decompo third config M(3)}
\end{figure}
\begin{figure}
    \centering
    \includegraphics[width=0.9\linewidth]{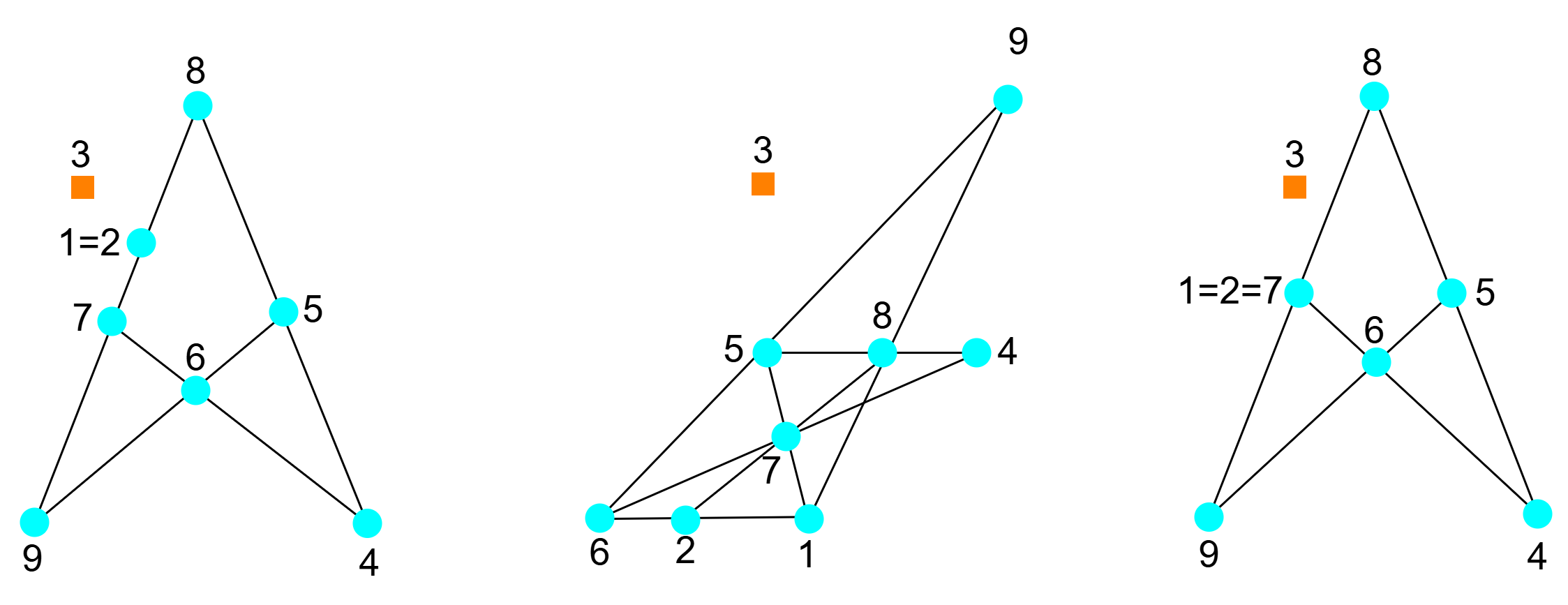}
    \caption{(Left) $D_1$; (Center) $E_1$; (Right) $F_1$.}
    \label{fig:decompo third config M(3) part 2}
\end{figure}
To prove that $V_{A_1} \subset V_{M(3)}$, let $\xi \in \Gamma_{A_1}$ be arbitrarily. Define $\widetilde{\xi}_2$ as a perturbation of $\xi_2$ on the line through $\xi_7$ and $\xi_8$, such that $\widetilde{\xi}_2 \neq \xi_2$. Then define $\widetilde{\xi}_1 = \widetilde{\xi}_2 \xi_6 \wedge \xi_8 \xi_9$. By Lemma \ref{lemma: meet of 2 extensors arbitrary close}, $\widetilde{\xi}_1 \to \xi_1$ if $\widetilde{\xi}_2 \to \xi_2$. Let $\widetilde{\xi}_p = \xi_p$ for the other points. Then, $\widetilde{\xi} \in \Gamma_M$, so $\xi \in V_M$. Thus $V_{A_1} \subset V_{M(3)}.$ This concludes the proof.
\newline
\newline
To prove that $V_{B_1} \subset V_M$, let $\xi \in \Gamma_{B_1}$ be arbitrarily. Define $\widetilde{\xi}_6$ as a perturbation of $\xi_6$ on the line through $\xi_1$ and $\xi_2$, such that $\widetilde{\xi}_6 \neq \xi_6$. Then define $\widetilde{\xi}_4$ as a perturbation of $\xi_4$ on the line through $\xi_6$ and $\xi_7$, such that $\widetilde{\xi}_4 \neq \xi_4$ and
$\widetilde{\xi}_5 = \widetilde{\xi}_6 \xi_9 \wedge \widetilde{\xi}_8 \xi_4$, and finish the proof completely analogously to the above.
\newline
\newline
The proof that $V_{B'} \subset V_M$ is completely analogous to the above.
\newline
\newline
$V_{C_1} \subset V_{M(3)}$:
An arbitrary $\gamma \in \Gamma_{M(3)}$ can be written as \begin{equation} \label{eq: appendix decomp third config matrix of M(3)}
\gamma = \begin{pmatrix} 
1 & 0 & 0 & 1 & 1+x & 1+x & 0 & 1 & 0\\
0 & 1 & 0 & 1 & x & x+y & x+y & 0 &0\\
0 & 0 & 1 & 1 & x & x & x & 1 &0
\end{pmatrix},\end{equation} where the columns from left to right correspond to $\{6,8,4,2,1,9,5,7,3\}$ and the minors corresponding to the bases are non-zero.
\newline \newline
An arbitrary $\xi \in \Gamma_{C_1}$ can be written as $$\gamma = \begin{pmatrix} 
1 & 0 & 0 & 1 & 1 & 1 & 0 & 1 & 0 \\
0 & 1 & 0 & 1 & 0 & z & 1 & 0 & 0\\
0 & 0 & 1 & 1 & 0 & 0 & 0 & 1 & 0
\end{pmatrix},$$ where the columns from left to right correspond to $\{6,8,4,2,1,9,5,7,3\}$ and the minors corresponding to the bases are non-zero. 
If $$
\left\{
    \begin{array}{lll}
       y & = &z  \\
    x& = &\epsilon
    \end{array},
\right.
$$
then $\epsilon \to 0$ implies that $\gamma \to \xi$. Therefore, $\xi \in V_{M(3)}$. This proves the statement.
\newline \newline
The proof for $V_{C_i'} \subset V_{M(3)}$ is analogous.
So we can conclude already that: 
\begin{equation} \label{eq: appendix V(M(3)) nilpotent definitive}
V_{\CCC(M(3))} \subset V_{M(3)} \cup_{i=1}^6 V_{\CCC(D_i)} \cup_{i=1}^2 V_{\CCC(E_i)} \cup_{i=6,8} V_{M(3,i)}.\end{equation}
The minimal matroids over the matroid $D_1$ are:
\begin{itemize}
    \item The matroid $F_1$ with $1=7$ and lines $\{8,1,9\},\{8,4,5\},\{1,4,6\},\{9,5,6\}$, as in Figure \ref{fig:decompo third config M(3) part 2}, and two other matroids $F_i$ formed by automorphisms of $D_1$.
    \item The matroid $G_1$ with $4=5,7=9$ and lines $\{8,1,7\},\{4,6,7\}$, as in Figure \ref{fig:third config part 3} (Left) and two other matroids $G_i$ formed by automorphism.
    \item The matroid with $4=5=6$ and line $\{8,1,9,7\}$. The circuit variety of this matroid is contained in that of $V_{A_1}$.
    \item The matroid $I_1$ with $4=7=8$ and lines $\{1,4,9\},\{9,5,6\},$ as in Figure \ref{fig:third config part 3} (Center) and two other matroids $I_i$ formed by automorphisms.
    \item Four matroids for which the circuit variety is contained in $V_{\pi_M^i}.$
    \item The matroid obtained from the uniform matroid $U_{2,9}$ by identifying the points 1 and 2 and setting the point $3$ to be a loop.
    \item The matroid obtained from $D_1$ by setting an additional loop from the set $\{4,5, 6,7,8,9\}.$
\end{itemize}
\begin{figure}
    \centering
    \includegraphics[width=0.9\linewidth]{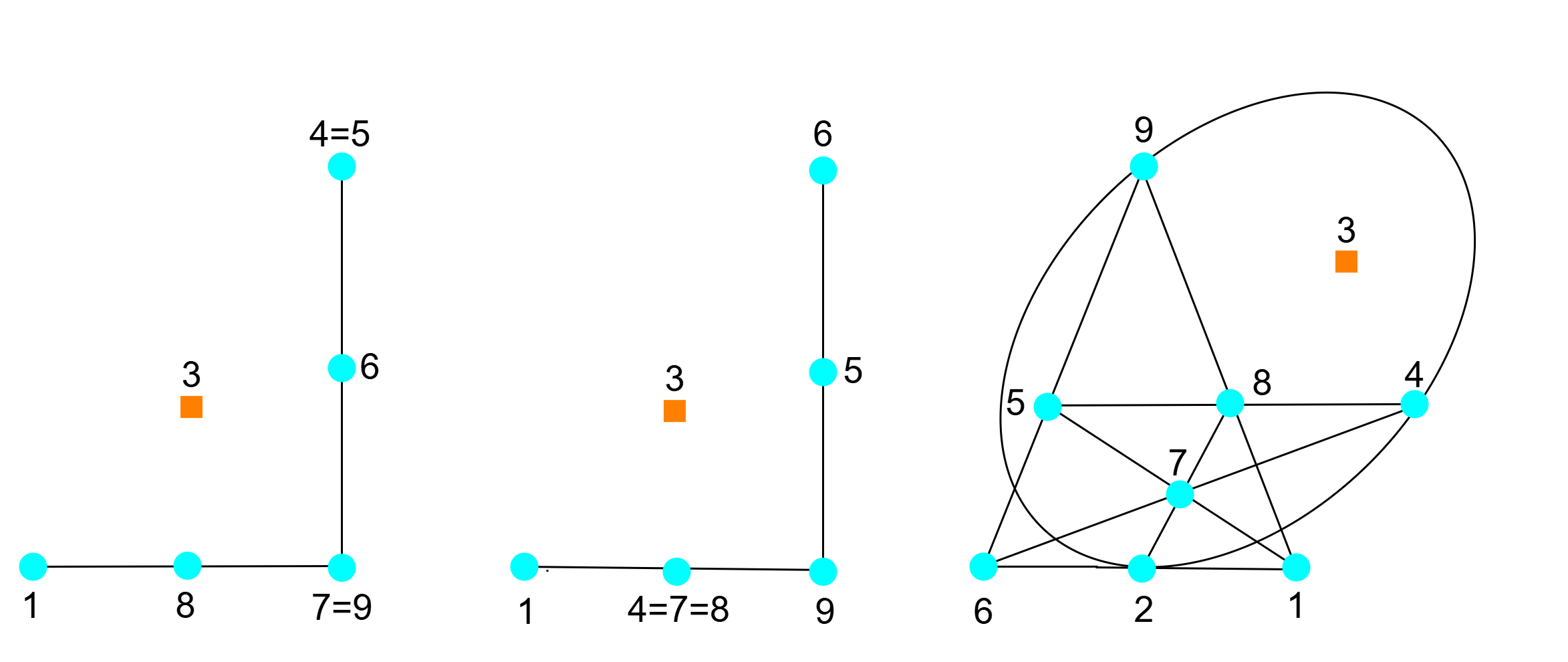}
    \caption{(Left) $G_1$; (Center) $I_1$; (Right) $K$.}
    \label{fig:third config part 3}
\end{figure}
We prove that $V_{D_1} \subset V_{M(3)}$:
An arbitrary $\gamma \in \Gamma_{M(3)}$ can be written as $$\gamma = \begin{pmatrix} 
1 & 0 & 0 & 1 & 1 & 1 & 1 & 1 & 0\\
0 & 1 & 0 & 1 & 0 & y+1 & y+1 & 1 & 0\\
0 & 0 & 1 & 1 & x & 1 & x & 0 & 0
\end{pmatrix},$$ where the columns from left to right correspond to $\{1,8,6,5,2,4,7,9,3\}$ and minors corresponding to the bases are non-zero. 
\newline \newline
An arbitrary $\xi \in \Gamma_{J}$ can be written as $$\gamma = \begin{pmatrix} 
1 & 0 & 0 & 1 & 1 & 1 & 1 & 1 & 0\\
0 & 1 & 0 & 1 & 0 & w+1 & w+1 & 1 &0\\
0 & 0 & 1 & 1 & 0 & 1 & 0 & 0 &0
\end{pmatrix},$$ where the columns from left to right correspond to $\{1,8,6,5,2,4,7,9,3\}$ and the minors corresponding to the bases are non-zero. 
If $$
\left\{
    \begin{array}{lll}
       y & = &w  \\
    x&= &\epsilon
    \end{array},
\right.
$$
then $\gamma \to \xi$, if $\epsilon \to 0$. Therefore, $\xi \in V_{M(3)}$. This proves the statement.
\newline
\newline
$V_{F_1} \subset V_{D_1}$, since for a collection of vectors $\xi \in \Gamma_{F_1}$, one can perturb the point $\xi_1$ on the line through $\xi_8$ and $\xi_9$ such that it does not coincide with $\xi_7$ anymore.
\newline
\newline
To prove that $V_{G_1} \subset V_{D_1}$, let $\xi \in \Gamma_{{G_1}}$ be arbitrarily. Define $\widetilde{\xi}_9$ as a perturbation of $\xi_9$ on the line through $\xi_1$ and $\xi_8$, such that $\widetilde{\xi}_9 \neq \xi_9$. Then define $\widetilde{\xi}_5 = \widetilde{\xi}_9 \xi_6 \wedge \xi_8 \xi_4$, and finish the proof completely analogously to the previous cases. 
\newline
\newline
To prove that $V_{I_1} \subset V_{D_1}$, let $\xi \in \Gamma_{V_{I_1}}$ be arbitrarily. Define $\widetilde{\xi}_8$ as a perturbation of $\xi_8$ on the line through $\xi_1$ and $\xi_9$, such that $\widetilde{\xi}_8 \neq \xi_8$. Then define $\widetilde{\xi}_4 = \widetilde{\xi}_7 \xi_6 \wedge \xi_5 \xi_8$, and finish the proof completely analogously to the previous cases. 
\newline
\newline 
So \begin{equation} \label{eq: appendix V(M(3)) D_1 definitive}
    V_{\CCC(D_1)} \subset V_{M(3)} \cup_{i=6,8} V_{M(3,i)} \cup V_{U_{2,9}}.
\end{equation}
The minimal matroids over $E_1$ are: 
\begin{itemize}
    \item Sixteen nilpotent matroids for which the circuit variety is contained in that of $V_{\CCC(A_i)}, V_{\CCC(B_i)}, V_{\CCC(B')}, V_{\CCC(C_i)}$ and $V_{\CCC(C')}$.
    \item Three matroids for which the circuit variety is contained in $V_{\pi_M^i}$, the point $i$ varying from 1 to 9 except for 3. 
    \item The matroid with $1=2=7$ and lines $\{8,1,9\},\{8,4,5\},\{1,4,6\},\{9,5,6\}$ and two other matroids $J_i$ formed by automorphisms. Their circuit varieties are contained in $V_{\CCC(D_1)}$ from the decomposition of $M(3)$.
    \item The matroid $K$ with lines $\{1,2,6\},$ $\{8,2,7\},$ $\{8,4,5\},$ $\{4,6,7\},$ $\{9,5,6\},$ $\{8,1,9\},$ $\{1,5,7\},$ $\{9,2,4\}$, as in Figure \ref{fig:decompo third config part 3} (Right).
    \item The matroids with an additional loop in $\{1,5,6,7,8\}$.
\end{itemize}
To prove that $V_{E_1} \subset V_M$, after a projective transformation one can write an arbitrary $\gamma \in V_{M(3)}$ as in Equation~\eqref{eq: appendix decomp third config matrix of M(3)}. An element $\xi \in V_{E_1}$ can be written after a projective transformation as \begin{equation} \label{eq: appendix decomp third config matrix of E_1}
    \begin{pmatrix} 
1 & 0 & 0 & 1 & 1 & 1 & 1 & 1 &0\\
0 & 1 & 0 & 1 & 0 & z & z & 1 & 0\\
0 & 0 & 1 & 1 & z & 1 & z & 0 & 0
\end{pmatrix},
\end{equation}
where the columns again correspond to $\{1,8,6,5,2,4,7,9,3\},$ and the minors corresponding to the bases are non-zero.
Let $$
\left\{
    \begin{array}{lll}
       y & = &w-1+\epsilon  \\
    x&= &w+\epsilon
    \end{array}
\right..
$$
If $\epsilon \to 0$, it is clear that $\gamma \to \xi$, implying $\xi \in V_{M(3)}$. This concludes the proof.
\newline
\newline
The minimal matroids over $K$ are the following:
\begin{itemize}
    \item Eight matroids for which the circuit variety is contained in $V_K$ by Lemma~7.6 of \cite{liwskimohammadialgorithmforminimalmatroids}.
    \item Eight matroids for which the circuit variety is contained in $V_{\pi_M^i}$, where $i$ varies from $1$ to 9 except 3.
    \item The uniform matroid with 3 a loop.
    \item The matroid with an additional loop in $\{1,2,4,5,6,7,8,9\}$.
\end{itemize}
$V_K \subset V_{E_1}$, since an arbitrary $\gamma \in \Gamma_{E_1}$ can be written as in Equation~\eqref{eq: appendix decomp third config matrix of E_1}, with $z$ such that no minors corresponding to non-bases are zero. $\xi \in \Gamma_K$ can be written as in Equation~\eqref{eq: appendix decomp third config matrix of E_1}, if one replaces $z$ by $\widetilde{z}$ and let $\widetilde{z}^2+\widetilde{z}+1=0$. If $z \to \widetilde{z}$, $\gamma \to x$, such that $\xi \in V_{E_1}$, finishing the proof. 
\newline 
So \begin{equation} \label{eq: appendix V(M(3)) E_1 definitive}
    V_{\CCC(E_1)} \subset V_{M(3)} \cup_{i=1}^8 V_{M(3,i)} \cup V_{U_{2,9}}
\end{equation}
The matroid $M(3,i)$ with $i \in \{1,2,4,5,7,9\}$ is already decomposed in the previous section up to a hypergraph automorphism of $M$, so we only have to consider $M(3,i)$ with $i \in \{6,8\}$. We can assume without loss of generality that $i=6$. That matroid has lines $\{8,2,7\},\{8,4,5\},\{8,1,9\}$, so by Lemma~5.5 (iv) of \cite{liwskimohammadialgorithmforminimalmatroids} \begin{equation} \label{eq: appendix V(M(3)) M(3,6) definitive}
    V_{\CCC(M(3,6))} = V_{M(3,6)} \cup V_{M(3,6,8)}. 
\end{equation}
Combining equations \eqref{eq: appendix V(M(3)) nilpotent definitive}, \eqref{eq: appendix V(M(3)) D_1 definitive}, \eqref{eq: appendix V(M(3)) E_1 definitive}, \eqref{eq: appendix V(M(3)) M(3,6) definitive}, \eqref{eq: proof third config M(9) M(9,i) definitive} and \eqref{eq: appendix V(M(3)) M(3,6) definitive}:
\begin{equation} 
    V_{\CCC(M(3))} = V_{M(3)} \cup V_{U_{2,9}} \cup_{i=1}^9 V_{\pi_M^i} \cup V_{M(3,6)} \cup V_{M(3,8)} \cup V_{M(3,6,8)}.
\end{equation}
\subsubsection{Identifying irredundancies}

In this subsection, we prove the irredundancy of Theorem \ref{thm decompo third 9_3}. For the irredundancy of $V_{U_{2,9}}$, one can proof this using similar methods as in \cite{clarke2024liftablepointlineconfigurationsdefining}. For the other proofs, we use Strategy \ref{strategy no subset}.

\medskip
We first prove the irredundancy of $V_{A_i}.$ By Strategy \ref{strategy no subset}, we only have to prove that $V_{A_i} \not \subset V_M$. We prove this for $i=1$. Define $\xi \in \Gamma_{A_1}$ as $$\xi = \begin{pmatrix} 
1 & 0 & 0 & 1 & 1 & 1 & 1 & 1 & 1\\
0 & 1 & 0 & 1 & 0 & 1 & 1 & 1 & 1 \\
0 & 0 & 1 & 1 & 3 & 3 & 0 & 1 & 1
\end{pmatrix},$$
where the columns from left to right correspond to $\{1,3,6,8,2,4,5,7,9\}$.
Suppose from the contrary that $\xi \in V_M$. Then one can find for every $\epsilon > 0$ a collection of vectors $\gamma_\epsilon \in \Gamma_M$, such that $\gamma \to \xi$ if $\epsilon \to 0$. One can perturb $\gamma_\epsilon \in \Gamma_{M}$ to $\widetilde{\gamma}_\epsilon \in \Gamma_M$, such that the vectors corresponding to $\{1,3,6,8\}$ form the standard reference frame for a projective space: 
\begin{equation} \label{eq: proof irredundancy of V_Ai}
\widetilde{\gamma}_{\epsilon} = \begin{pmatrix} 
1 & 0 & 0 & 1 & 1 & 1 & 1-z & 1 & 1-z\\
0 & 1 & 0 & 1 & 0 & w & w-z & w & w-z \\
0 & 0 & 1 & 1 & z & z & 0 & w-z(w-1) & w-z
\end{pmatrix},
\end{equation} where the columns from left to right correspond to $\{1,3,6,8,2,4,5,7,9\}$ and the minors corresponding to the bases are non-zero. Moreover, 
\begin{equation} \label{eq: irred third}
    z(1-w)(2-z)=0,
\end{equation} as the minor corresponding to the columns $3,7,9$ should vanish. So $z=2$. Since 
$\widetilde{\gamma}_{\epsilon_{2}} \to \xi_2$, it follows that $z \to 3$. From this contradiction, it follows that $\gamma \notin V_M$ and we can conclude that $V_{A_i}$ is not redundant.
\newline
\newline
To prove the irredundancy of $V_B$, we only have to prove that $V_B \not \subset V_M$, by Strategy \ref{strategy no subset}.
To prove this, consider $$\xi = \begin{pmatrix} 
1 & 0 & 0 & 1 & 1 & 1 & 1 & 1 & 1\\
0 & 1 & 0 & 1 & 0 & 3 & 3 & 3 & 3 \\
0 & 0 & 1 & 1 & 0 & 0 & 0 & 3 & 3
\end{pmatrix},$$
where the columns from left to right again correspond to $\{1,3,6,8,2,4,5,7,9\}$. Clearly, $\xi \in \Gamma_B$.
Suppose from the contrary that $\xi \in V_M$. Then one can find for every $\epsilon > 0$ a collection of vectors $\gamma_\epsilon \in \Gamma_M$, such that $\gamma \to \xi$ if $\epsilon \to 0$. One can perturb $\gamma_\epsilon \in \Gamma_{M}$ to $\widetilde{\gamma}_\epsilon \in \Gamma_M$, such that the vectors corresponding to $\{1,3,6,8\}$ form the standard reference frame for a projective space, as in Equation \eqref{eq: proof irredundancy of V_Ai}. 
As $\widetilde{\gamma}_{\epsilon_{2}} \to \xi_2$, it follows that $z \to 0$. 
However, $z=2$, so from this contradiction, it follows that $\gamma \notin V_M$ and we can conclude that $V_{B}$ is not redundant.
\newline\newline
Now we prove the irredundancy of $V_{\pi_M^i}$. Due to the symmetry of the points, it suffices to consider $i=1$ and $i=8$.
To see for instance that $V_{\pi_M^1}$ is not contained in $V_M$, one can define $\xi \in \Gamma_{\pi_M^1}$ as $$\xi = \begin{pmatrix} 
1 & 0 & 0 & 0 & 0 & 0 & 0 & 0 & 0\\
0 & 1 & 0 & 1 & 1 & 1 & 1 & 1 & 1 \\
0 & 0 & 1 & 1 & 1 & 0 & 3 & 5 &5
\end{pmatrix},$$
where the columns from left to right correspond to $\{1,2,4,3,5,6,7,8,9\}$.
Suppose that $\xi \in V_M$. Then one can find for every $\epsilon > 0$ a collection of vectors $\gamma_\epsilon \in \Gamma_M$, such that $\gamma \to \xi$ if $\epsilon \to 0$. One can perturb $\gamma_\epsilon \in \Gamma_{M}$ to $\widetilde{\gamma}_\epsilon \in \Gamma_M$, such that the vectors corresponding to $\{1,2,4,3\}$ have a fixed form: 
\begin{equation} \label{eq: proof irredundancy of pi_1^M 1}
\widetilde{\gamma}_{\epsilon} = \begin{pmatrix} 
1 & 0 & 0 & 0 & x & y & y & x & \frac{y-xz}{1-z}\\
0 & 1 & 0 & 1 & 1 & 1 & 1 & 1 & 1 \\
0 & 0 & 1 & 1 & 1 & 0 & z & \frac{xz}{y} & \frac{xz}{y}
\end{pmatrix},
\end{equation} where the minors corresponding to the bases are non-zero and the columns from left to right correspond to $\{1,2,4,3,5,6,7,8,9\}$. Moreover, the minor corresponding to \{5,6,9\} is 0, so $$\frac{z(x-y)(-zx+x+y)}{y(z-1)} = 0.$$ Since $\gamma \in \Gamma_M$, $z= \frac{x+y}{x}$. Thus we can rewrite Equation~\eqref{eq: proof irredundancy of pi_1^M 1} as:$$
\widetilde{\gamma}_{\epsilon} = \begin{pmatrix} 
1 & 0 & 0 & 0 & x & y & y & x & \frac{x^2}{y}\\
0 & 1 & 0 & 1 & 1 & 1 & 1 & 1 & 1 \\
0 & 0 & 1 & 1 & 1 & 0 & \frac{x+y}{x} & \frac{x+y}{y} & \frac{x+y}{y}
\end{pmatrix}.$$
Since $\gamma_{\epsilon} \to \xi$ if $\epsilon \to 0$, we can deduce that:
$$
\left\{
    \begin{array}{ll}
       x \to 0\\
       y \to 0 \\
       1+ \frac{y}{x} \to 3 \\ \\
       1+\frac{x}{y} \to 5 \\ \\
       \frac{x^2}{y} \to 0
    \end{array}
\right..
$$
The third and fourth condition are a contradiction, thus $\gamma \notin V_M$. By Strategy \ref{strategy no subset}, this shows that $\pi_M^1$is irredundant in the decomposition.
\newline
\newline
To show the irredundancy of $V_{\pi_M^3}$, one can define $\xi \in \Gamma_{\pi_M^3}$ as $$\xi = \begin{pmatrix} 
1 & 0 & 0 & 0 & 0 & 0 & 0 & 0 & 0\\
0 & 1 & 0 & 1 & 0 & 1 & 1 & 1 & 1 \\
0 & 0 & 1 & 1 & 1 & 0 & 3 & 5 &3 
\end{pmatrix},$$
where the columns from left to right correspond to $\{3,1,2,6,4,5,7,8,9\}$.
Suppose that $\xi \in V_M$. Then one can find for every $\epsilon > 0$ a $\gamma_\epsilon \in \Gamma_M$, such that $\gamma \to \xi$ if $\epsilon \to 0$. One can perturb $\gamma_\epsilon \in \Gamma_{M}$ to $\widetilde{\gamma}_\epsilon \in \Gamma_M$, such that the vectors corresponding to $\{3,1,2,4,6\}$ have a fixed form: 
\begin{equation} \label{eq: proof irredundancy of pi_1^M}
\widetilde{\gamma}_{\epsilon} = \begin{pmatrix} 
1 & 0 & 0 & 0 & x & y & xz & xz & \frac{x^2z^2+x^2z}{xz-y}\\
0 & 1 & 0 & 1 & 0 & 1 & 1 & 1 & 1 \\
0 & 0 & 1 & 1 & 1 & 0 & 1+z & \frac{xz-y}{x} & 1+z
\end{pmatrix},
\end{equation} where the minors corresponding to the bases are non-zero and the columns from left to right correspond to $\{3,1,2,6,4,5,7,8,9\}$. Moreover, the minor corresponding to \{5,6,9\} is 0, so $z(x+y)(xz+x-y) = 0$. Since $\gamma \in \Gamma_M$, $z= \frac{y-x}{x}$. Thus we can rewrite Equation~\eqref{eq: proof irredundancy of pi_1^M} as:$$
\widetilde{\gamma}_{\epsilon} = \begin{pmatrix} 
1 & 0 & 0 & 0 & x & y & y-x & y-x & \frac{x-y}{x}y\\
0 & 1 & 0 & 1 & 0 & 1 & 1 & 1 & 1 \\
0 & 0 & 1 & 1 & 1 & 0 & \frac{y}{x} & -1 & 1+z
\end{pmatrix}.$$
As $\widetilde{\gamma}_{\epsilon_8} \not \to \xi_8$,
it follows that $\widetilde{\gamma}_{\epsilon} \not \to \xi$, thus from this contradiction we conclude that $\gamma \notin V_M$.
\newline
\newline
Now we prove the irredundancy of $V_{M(i)}$. Due to symmetry reasons, we consider $M(1)$ and $M(3)$. To prove the irredundancy of $M(1)$, we only have to prove that $V_{M(1)} \not \subset V_M$, by Strategy \ref{strategy no subset}. Consider:
$$    \gamma = \begin{pmatrix}
        1 & 0 & 0 & 1 & 0 & 1 & 1 & 3 & 0 \\
        0 & 1 & 0 & 1 & 1 & 0 & 5 & 10 & 0  \\
        0 & 0 & 1 & 1 & 1 & 3 & 0 & 3 & 0 \\
    \end{pmatrix}
 \in \Gamma_{M(1)},$$
where the columns from left to right correspond to $\{6,8,1,4,2,5,7,9,3\}$. Notice that $\gamma_8 \gamma_9 \wedge \gamma_2 \gamma_6 = \begin{pmatrix}
    3 \\ 10 \\ -5
\end{pmatrix}$ and $\gamma_3 \gamma_5 \wedge \gamma_2 \gamma_6 = \begin{pmatrix}
    -3 \\ -8 \\ 7
    \end{pmatrix}$. The point 1 is on the lines $\{1,8,9\}, \{1,2,6\},\{1,3,5\}$ in $M$. Since $\wedge$ is a continuous operation, there is no perturbation of $\gamma$ to $\widetilde{\gamma}$ such that $\widetilde{\gamma_3} \widetilde{\gamma_5} \wedge \widetilde{\gamma_2} {\gamma_6} = \widetilde{\gamma_8} \widetilde{\gamma_9} \wedge \widetilde{\gamma_3} \widetilde{\gamma_5}.$ 
    \newline
    \newline
For $M(3)$, consider $$\gamma = \begin{pmatrix}
    1 & 0 & 0 & 1 & 1 & 1 & 1 & 0 & 0 \\
    0 & 1 & 0 & 1 & 0 & 3 & 2 & 3 & 0 \\
    0 & 0 & 1 & 1 & 2 & 1 & 2 & 1 & 0
\end{pmatrix},$$
where the columns from left to right correspond to $\{4,7,8,3,2,5,6,9,1\}$. 
\newline
\newline
Notice that $$\gamma_2\gamma_4 \wedge \gamma_1 \gamma_5 = \begin{pmatrix}
    1 \\ 3 \\ -1
\end{pmatrix}$$ and $$\gamma_2\gamma_4 \wedge \gamma_7 \gamma_9 = \begin{pmatrix}
    2 \\ 1 \\ 3\end{pmatrix}.$$ So similarly as for $M(1)$, this proves that $\gamma \notin V_M$.
    \newline
    \newline

Now we prove the irredundancy of $V_{M(i,k_j^i)}$. Due to symmetry, we only consider $V_{M(1,4)}$ and $V_{M(3,6)}$. We will construct an element $\gamma \in \Gamma_{M(1,4)}$ for which $\gamma \not \in V_M \cup V_{M(1)} \cup V_{M(4)}.$
Consider $$\gamma = \begin{pmatrix}
    1 & 0 & 0 & 1 & 1 & 1 & 5 & 0 & 0 \\
    0 & 1 & 0 & 1 & 2 & 0 & 4 & 0 & 0 \\
    0 & 0 & 1 & 1 & 0 & 3 & 7 & 0 & 0 
\end{pmatrix},$$
where the columns correspond to $\{7,2,3,5,8,9,6,1,4\}$. 
Since $$\gamma_2 \gamma_6 \wedge \gamma_3 \gamma_5 = \begin{pmatrix}
    5 \\ 5 \\ 7
\end{pmatrix}$$ and $$\gamma_8 \gamma_9 \wedge \gamma_3 \gamma_5 = \begin{pmatrix}
    2 \\ 2 \\ 3
\end{pmatrix},$$ one can conclude analogously to the previous proofs that $\gamma \notin V_M \cup V_{M(4)}$. Since $$\gamma_7 \gamma_6 \wedge \gamma_3 \gamma_2 = \begin{pmatrix}
    0 \\ 4 \\ 7
\end{pmatrix}$$ and $$\gamma_8 \gamma_5 \wedge \gamma_3 \gamma_2 = \begin{pmatrix}
    0 \\ 1 \\ -1
\end{pmatrix},$$ one can conclude analogously to the previous proofs that $\gamma \notin V_M \cup V_{M(1)}$. By Strategy \ref{strategy no subset}, this shows the irredundancy of $V_{M(i,k_j^i)}.$
\newline
\newline
The proof of the irredundancy of $V_{M(3,6,8)}$ is completely analogous to the previous cases.
\bibliographystyle{abbrv}
\bibliography{references}

\begin{thebibliography}{10}

\bibitem{alexandr2025decomposing}
Y.~Alexandr, K.~Dawson, H.~Friedman, F.~Mohammadi, P.~Semnani, and T.~Yu.
\newblock Decomposing conditional independence ideals with hidden variables.
\newblock {\em arXiv preprint arXiv:2505.02404}, 2025.

\bibitem{bland1974orientability}
R.~G. Bland.
\newblock {\em Complementary Orthogonal Subspaces of IW and Orientability of Matroids}.
\newblock Ph.d. dissertation, Cornell University, 1974.

\bibitem{BLAND197733}
R.~G. Bland.
\newblock A combinatorial abstraction of linear programming.
\newblock {\em Journal of Combinatorial Theory, Series B}, 23(1):33--57, 1977.

\bibitem{bretto2013hypergraph}
A.~Bretto.
\newblock {\em Hypergraph Theory: An Introduction}.
\newblock Springer, Cham, 2013.

\bibitem{bruns2003determinantal}
W.~Bruns and A.~Conca.
\newblock Gr{\"o}bner bases and determinantal ideals.
\newblock In {\em Commutative Algebra, Singularities and Computer Algebra}, pages 9--66. Springer Netherlands, 2003.

\bibitem{caines2022lattice}
P.~Caines, F.~Mohammadi, E.~S{\'a}enz-de Cabez{\'o}n, and H.~Wynn.
\newblock Lattice conditional independence models and {H}ibi ideals.
\newblock {\em Transactions of the London Mathematical Society}, 9(1):1--19, 2022.

\bibitem{clarke2022matroidstratificationshypergraphvarieties}
O.~Clarke, K.~Grace, F.~Mohammadi, and H.~J. Motwani.
\newblock Matroid stratifications of hypergraph varieties, their realization spaces, and discrete conditional independence models.
\newblock {\em International Mathematics Research Notices}, (22):18958--19019, 2023.

\bibitem{clarke2024liftablepointlineconfigurationsdefining}
O.~Clarke, G.~Masiero, and F.~Mohammadi.
\newblock Liftable point-line configurations: Defining equations and irreducibility of associated matroid and circuit varieties.
\newblock {\em Mathematics}, 12(19), 2024.

\bibitem{clarke2022conditional}
O.~Clarke, F.~Mohammadi, and H.~Motwani.
\newblock Conditional probabilities via line arrangements and point configurations.
\newblock {\em Linear and Multilinear Algebra}, 70(20):5268--5300, 2022.

\bibitem{corey2025singular}
D.~Corey and D.~Luber.
\newblock Singular matroid realization spaces.
\newblock {\em arXiv preprint arXiv:2307.11915}, 2025.

\bibitem{IdealsVarietiesandAlgorithms}
D.~A. {Cox}, J.~{Little}, and D.~{O’Shea}.
\newblock {\em Ideals, Varieties, and Algorithms, An Introduction to Computational Algebraic Geometry and Commutative Algebra}.
\newblock Springer, Cham, Switzerland, 4 edition, 2015.

\bibitem{diaconis1998lattice}
P.~Diaconis, D.~Eisenbud, and B.~Sturmfels.
\newblock Lattice walks and primary decomposition.
\newblock In B.~E. Sagan and R.~P. Stanley, editors, {\em Mathematical Essays in Honor of Gian-Carlo Rota}, volume 161 of {\em Progress in Mathematics}, pages 173--193. Birkh{\"a}user, Boston, 1998.

\bibitem{dolgachev631}
I.~Dolgachev.
\newblock Introduction to algebraic geometry, 2013.

\bibitem{DrtonSturmfelsSullivant09:Algebraic_Statistics}
M.~Drton, B.~Sturmfels, and S.~Sullivant.
\newblock {\em Lectures on Algebraic Statistics}, volume~39.
\newblock Birkh\"{a}user, Basel, first edition, 2009.

\bibitem{ene2013determinantal}
V.~Ene, J.~Herzog, T.~Hibi, and F.~Mohammadi.
\newblock Determinantal facet ideals.
\newblock {\em Michigan Mathematical Journal}, 62(1):39--57, 2013.

\bibitem{feher2012equivariant}
L.~M. Feh{\'e}r, A.~N{\'e}methi, and R.~Rim{\'a}nyi.
\newblock Equivariant classes of matrix matroid varieties.
\newblock {\em Commentarii Mathematici Helvetici}, 87(4):861--889, 2012.

\bibitem{gale1959transient}
D.~Gale.
\newblock Transient flows in networks.
\newblock {\em Michigan Mathematical Journal}, 6:59--63, 1959.

\bibitem{gelfand1987combinatorial}
I.~Gelfand, M.~Goresky, R.~MacPherson, and V.~Serganova.
\newblock Combinatorial geometries, convex polyhedra, and schubert cells.
\newblock {\em Advances in Mathematics}, 63(3):301--316, 1987.

\bibitem{graver1993combinatorial}
J.~E. Graver, B.~Servatius, and H.~Servatius.
\newblock {\em Combinatorial rigidity}.
\newblock Number~2 in Mathematical Sciences Series. American Mathematical Soc., 1993.

\bibitem{Connectednessandcombinatorialinterplayinthemodulispaceoflinearrangements}
B.~Guerville-Ballé and J.~Viu-Sos.
\newblock Connectedness and combinatorial interplay in the moduli space of line arrangements.
\newblock {\em arXiv preprint arXiv:2403.13718}, 2023.

\bibitem{Hartshorne}
R.~Hartshorne.
\newblock {\em Algebraic Geometry}, volume~52 of {\em Graduate Texts in Mathematics}.
\newblock Springer-Verlag, New York, 1977.

\bibitem{herzog2010binomial}
J.~Herzog, T.~Hibi, F.~Hreinsd{\'o}ttir, T.~Kahle, and J.~Rauh.
\newblock Binomial edge ideals and conditional independence statements.
\newblock {\em Advances in Applied Mathematics}, 3(45):317--333, 2010.

\bibitem{Hilbert1890}
Hilbert.
\newblock Ueber dietheorie der algebraischen formen.
\newblock {\em Mathematische Annalen}, 36:473--534, 1890.

\bibitem{hochstattler2017sticky}
W.~Hochst{\"a}ttler and M.~Wilhelmi.
\newblock Sticky matroids and kantor's conjecture.
\newblock {\em arXiv preprint arXiv:1704.08478}, 2017.

\bibitem{hocsten2004ideals}
S.~Ho{\c{s}}ten and S.~Sullivant.
\newblock Ideals of adjacent minors.
\newblock {\em Journal of Algebra}, 277(2):615--642, 2004.

\bibitem{Husimi1950-ph}
K.~Husimi.
\newblock Note on mayers' theory of cluster integrals.
\newblock {\em J. Chem. Phys.}, 18(5):682--684, 1950.

\bibitem{jackson2024maximal}
B.~Jackson and S.-i. Tanigawa.
\newblock Maximal matroids in weak order posets.
\newblock {\em Journal of Combinatorial Theory, Series B}, 165:20--46, 2024.

\bibitem{Janczewski2016}
R.~Janczewski and K.~Turowski.
\newblock An $o(n \log n)$ algorithm for finding edge span of cacti.
\newblock {\em Journal of Combinatorial Optimization}, 31(4):1373--1382, 2016.

\bibitem{atiyah}
W.~Jonsson.
\newblock Review of introduction to commutative algebra by m. f. atiyah and i. g. macdonald.
\newblock {\em Canadian Mathematical Bulletin}, 13(1):162--162, 1970.
\newblock Book review.

\bibitem{knutson2013positroid}
A.~Knutson, T.~Lam, and D.~E. Speyer.
\newblock Positroid varieties: juggling and geometry.
\newblock {\em Compositio Mathematica}, 149(10):1710--1752, 2013.

\bibitem{liwski2024pavingmatroidsdefiningequations}
E.~Liwski and F.~Mohammadi.
\newblock Paving matroids: defining equations and associated varieties.
\newblock {\em arXiv preprint arXiv:2403.13718}, 2024.

\bibitem{liwski2024solvablenilpotentmatroidsrealizability}
E.~Liwski and F.~Mohammadi.
\newblock Solvable and nilpotent matroids: Realizability and irreducible decomposition of their associated varieties.
\newblock {\em arXiv preprint arXiv:2408.12784}, 2024.

\bibitem{liwskimohammadialgorithmforminimalmatroids}
E.~Liwski and F.~Mohammadi.
\newblock Minimal matroids in dependency posets: algorithms and applications to computing irreducible decompositions of circuit varieties.
\newblock {\em arXiv preprint arXiv:2502.00799}, 2025.

\bibitem{liwski2025efficient}
E.~Liwski, F.~Mohammadi, and R.~Pr{\'e}bet.
\newblock Efficient algorithms for minimal matroid extensions and irreducible decompositions of circuit varieties.
\newblock {\em arXiv preprint arXiv:2504.16632}, 2025.

\bibitem{vandebrouckliwskimohammadi}
E.~Liwski, F.~Mohammadi, and L.~Vandebrouck.
\newblock Algebraic geometry of {C}actus, {P}ascal, and {P}appus matroids.
\newblock {\em arXiv preprint arXiv:2506.07757}, 2025.

\bibitem{sturmfelsbook}
M.~Micha{\l}ek and B.~Sturmfels.
\newblock {\em Invitation to Nonlinear Algebra}.
\newblock American Mathematical Society, Providence, RI, 2021.

\bibitem{minieka1973maximal}
E.~Minieka.
\newblock Maximal, lexicographic, and dynamic network flows.
\newblock {\em Operations Research}, 21(2):517--527, 1973.

\bibitem{Mnev1985}
N.~E. Mnëv.
\newblock On manifolds of combinatorial types of projective configurations and convex polyhedra.
\newblock {\em Soviet Mathematics Doklady}, 32:335--337, 1985.

\bibitem{Mnev1988}
N.~E. Mnëv.
\newblock The universality theorems on the classification problem of configuration varieties and convex polytopes varieties.
\newblock In {\em Topology and Geometry—Rohlin Seminar}, volume 1346 of {\em Lecture Notes in Mathematics}, pages 527--543. Springer, Berlin, 1988.

\bibitem{Olver_1999}
P.~J. Olver.
\newblock {\em Classical Invariant Theory}.
\newblock London Mathematical Society Student Texts. Cambridge University Press, Cambridge, 1999.

\bibitem{Oxley}
J.~Oxley.
\newblock {\em Matroid Theory}.
\newblock Oxford University Press, 2 edition, 2011.

\bibitem{pfister2019primary}
G.~Pfister and A.~Steenpass.
\newblock On the primary decomposition of some determinantal hyperedge ideal.
\newblock {\em Journal of Symbolic Computation}, 103:14--21, 2019.

\bibitem{piff1970vector}
M.~J. Piff and D.~J. Welsh.
\newblock On the vector representation of matroids.
\newblock {\em Journal of the London Mathematical Society}, 2(2):284--288, 1970.

\bibitem{poljak1984amalgamation}
S.~Poljak and D.~Turz{\'\i}k.
\newblock Amalgamation over uniform matroids.
\newblock {\em Czechoslovak Mathematical Journal}, 34(2):239--246, 1984.

\bibitem{victorreinerlecturematroids}
V.~Reiner.
\newblock Lectures on matroids and oriented matroids, 2005.

\bibitem{Sidman}
J.~Sidman, W.~Traves, and A.~Wheeler.
\newblock Geometric equations for matroid varieties.
\newblock {\em Journal of Combinatorial Theory, Series A}, 178:105360, 2021.

\bibitem{maximum}
M.~Sitharam and A.~Vince.
\newblock The maximum matroid of a graph.
\newblock {\em arXiv preprint arXiv:1910.05390}, 2019.

\bibitem{Studeny05:Probabilistic_CI_structures}
M.~Studený.
\newblock {\em Probabilistic conditional independence structures}.
\newblock Springer, London, 2005.

\bibitem{sturmfels1989matroid}
B.~Sturmfels.
\newblock On the matroid stratification of {G}rassmann varieties, specialization of coordinates, and a problem of {N}. {W}hite.
\newblock {\em Advances in Mathematics}, 75(2):202--211, 1989.

\bibitem{AlgorithmsInInvariantTheory}
B.~Sturmfels.
\newblock {\em Algorithms in invariant theory}.
\newblock Springer-Verlag, Berlin, Heidelberg, 1993.

\bibitem{tutte1966matroids}
W.~T. Tutte.
\newblock An introduction to the theory of matroids.
\newblock Technical report, RAND Corporation, 1966.
\newblock Research Report prepared for the U.S. Air Force Office of Scientific Research.

\bibitem{Vakil}
R.~Vakil.
\newblock {\em The Rising Sea, Foundations of Algebraic Geometry}.
\newblock Available at \url{http://math.stanford.edu/~vakil/216blog/FOAGnov1817public}, 2017.

\bibitem{Van_Lint2012-bn}
J.~H. van Lint and R.~M. Wilson.
\newblock {\em A Course in Combinatorics}.
\newblock Cambridge University Press, Cambridge, 2 edition, 2012.

\bibitem{InvariantMethodsinDiscreteandComputationalGeometry}
N.~L. White.
\newblock {\em Invariant Methods in Discrete and Computational Geometry : Proceedings of the Curaçao Conference, 13–17 June, 1994}.
\newblock Springer Netherlands, Dordrecht, 1995.

\bibitem{whiteley1996some}
W.~Whiteley.
\newblock Some matroids from discrete applied geometry.
\newblock {\em Contemporary Mathematics}, 197:171--312, 1996.

\bibitem{Hassler1935-ko}
H.~Whitney.
\newblock On the abstract properties of linear dependence.
\newblock {\em American Journal of Mathematics}, 57(3):509--533, 1935.

\end{thebibliography}

\newpage
\thispagestyle{empty} 
\mbox{}               

\newpage
\thispagestyle{empty} 
\mbox{}               

\newpage              
\includepdf[pages = {-}]{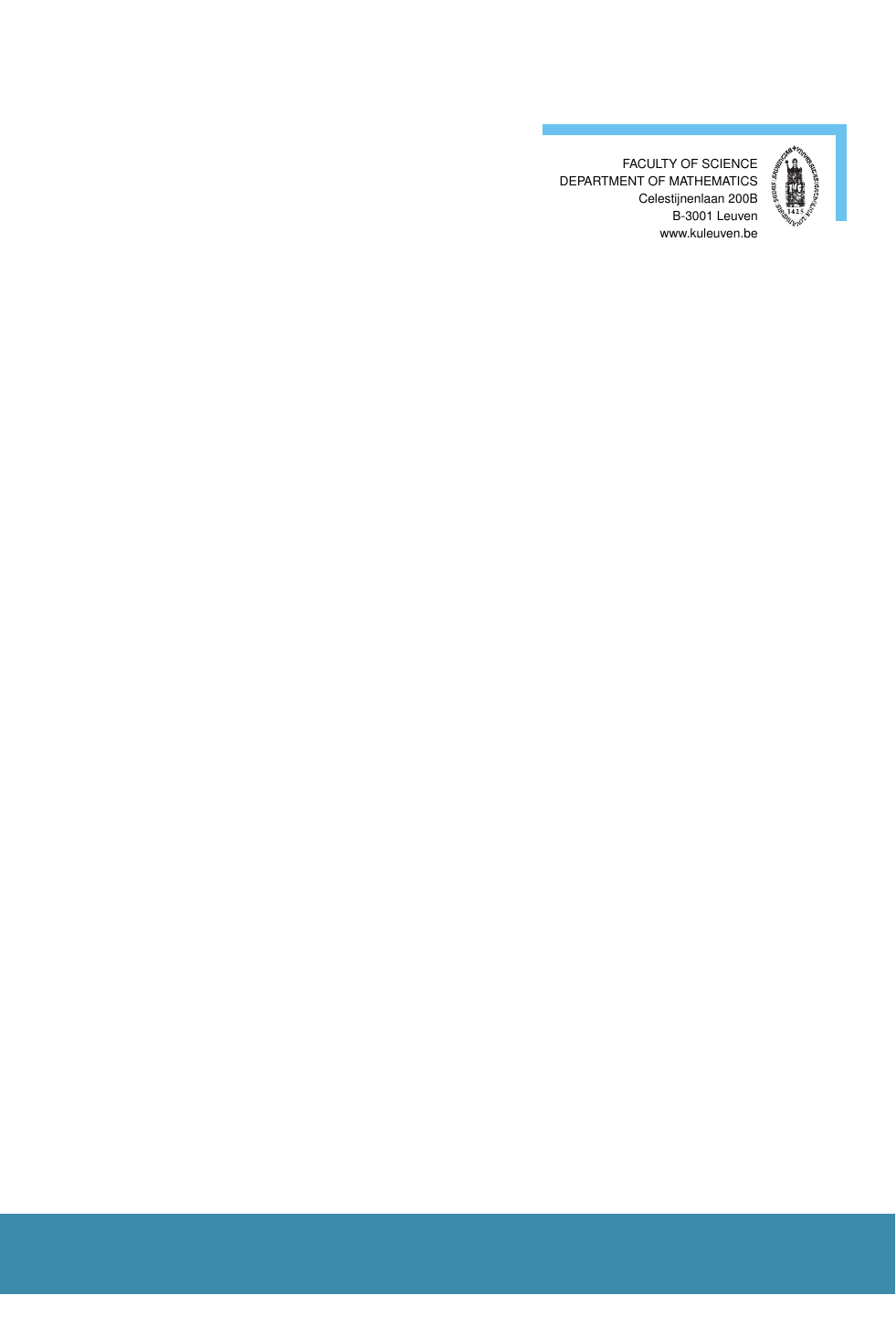}
\end{document}